\documentclass[12pt,a4paper]{article}
\usepackage{epsfig}
\usepackage[T1]{fontenc}    
\usepackage{graphics}
\usepackage{graphicx}
\usepackage{pstricks,pst-coil,pst-fill,pst-plot}
\usepackage[backend=biber, style=numeric]{biblatex}
\usepackage[fleqn]{amsmath}    
\usepackage{amssymb}    
\usepackage{amsfonts}   
\usepackage{verbatim}   
\usepackage{mathrsfs}   
\usepackage{dsfont}
\usepackage{euscript}
\usepackage{yfonts}
\usepackage{enumerate}     
\usepackage{amsthm}         
\usepackage{marvosym}
\usepackage{stmaryrd}
\usepackage{vmargin}        
\usepackage{wasysym}		
\usepackage{hyperref}

\usepackage{csquotes}
\usepackage{enumitem}

\usepackage[british]{babel}

\DeclareFieldFormat[article,inbook,incollection,inproceedings,patent,thesis,unpublished]{title}{#1\isdot}

\DeclareFieldFormat{titlecase}{%
  \ifcurrentfield{title}
    {\MakeSentenceCase*{#1}}
    {#1}}

\renewbibmacro{in:}{}

\usepackage{color}
\usepackage{bbm}

\usepackage{tikz}
\usepackage{pgfplots}

\setmarginsrb{1.8cm}{2cm}{1.8cm}{2cm}{1cm}{1cm}{1cm}{1.6cm}
 \makeatletter
 \@addtoreset{equation}{section}
 \makeatother

\providecommand{\MR}{\relax\ifhmode\unskip\space\fi MR }

\providecommand{\href}[2]{#2}

\let\tend=\rightarrow

\long\def\symbolfootnote[#1]#2{\begingroup%
\def\thefootnote{\fnsymbol{footnote}}\footnote[#1]{#2}\endgroup}

\newtheorem{theorem}{Theorem}[section]
\newtheorem{prop}[theorem]{Proposition}
\newtheorem*{theorem*}{Theorem}
\newtheorem{cor}[theorem]{Corollary}
\newtheorem{defin}[theorem]{Definition}

\newtheorem{rem}[theorem]{Remark}
\newtheorem{lemme}[theorem]{Lemma}

\newcommand\beq{\begin{equation}}
\newcommand\enq{\end{equation}}
\newcommand\bem{\begin{multline}}
\newcommand\enm{\end{multline}}

\def\beqa{\begin{eqnarray}}
\def\eeqa{\end{eqnarray}}
\def\ba{\begin{array}}
\def\ea{\end{array}}
\def\det{\operatorname{det}}

\newcommand{\f}[2]{{\ensuremath{%
    \mathchoice%
    {\dfrac{#1}{#2}}
    {\dfrac{#1}{#2}}
    {\frac{#1}{#2}}
    {\frac{#1}{#2}}
}}}
\newcommand{\tf}[2]{\ensuremath{#1/#2}}
\def\a{\alpha}

\def\be{\beta}
\def\ga{\gamma}
\def\Ga{\Gamma}

\def\de{\delta}

\def\De{\Delta}
\def\eps{\epsilon}
\def\veps{\varepsilon}
\def\la{\lambda}
\def\La{\Lambda}

\def\sg{\sigma}
\def\vsg{\varsigma}

\def\Ups{\Upsilon}
\def\ups{\upsilon}
\def\th{\theta}

\def\vth{\vartheta}

\def\om{\omega}
\def\vp{\varphi}

\newcommand{\mc}[1]{\ensuremath{\mathcal{#1}}}
\newcommand{\mf}[1]{\ensuremath{\mathfrak{#1}}}

\newcommand{\bs}[1]{\ensuremath{\boldsymbol{#1}}}

\DeclareFontFamily{OT1}{pzc}{}
\DeclareFontShape{OT1}{pzc}{m}{it}{<-> s * [1.10] pzcmi7t}{}
\DeclareMathAlphabet{\mathpzc}{OT1}{pzc}{m}{it}

\def \i{ \mathrm i}

\newcommand{\ov}[1]{\ensuremath{\overline{#1}}}
\newcommand{\wt}[1]{\ensuremath{\widetilde{#1}}}
\newcommand{\wh}[1]{\ensuremath{\widehat{#1}}}

\newcommand{\Int}[2]{\ensuremath{\int\limits_{#1}^{#2}}}
\newcommand{\Oint}[2]{\ensuremath{\oint\limits_{#1}^{#2}}}

\newcommand{\sul}[2]{\ensuremath{\sum\limits_{#1}^{#2}}}
\newcommand{\pl}[2]{\ensuremath{\prod\limits_{#1}^{#2}}}

\newcommand{\R}{\ensuremath{\mathbb{R}}}
\newcommand{\Cx}{\ensuremath{\mathbb{C}}}

\newcommand{\Dp}[1]{\ensuremath{\partial_{#1}}}

\newcommand{\limit}[2]{\ensuremath{\underset{#1 \tend #2}{\longrightarrow} }}

\newcommand{\ex}[1]{\ensuremath{\e{e}^{#1}}}

\newcommand{\op}[1]{ \boldsymbol{ \texttt{#1} } }

\newcommand{\norm}[1]{\ensuremath{  || #1 || }}

\newcommand{\dd}{\mathrm{d}}
\newcommand{\e}[1]{\ensuremath{\mathrm{#1}}}

\newcommand{\intff}[2]{\ensuremath{ [  #1 \,; #2 ] }}
\newcommand{\intfo}[2]{\ensuremath{ [  #1 \,; #2 [ }}
\newcommand{\intof}[2]{\ensuremath{ ]  #1 \,; #2 ] }}
\newcommand{\intoo}[2]{\ensuremath{ ]  #1 \,; #2 [ }}

\newcommand{\intn}[2]{\ensuremath{[\![ \, #1 \,;\, #2 \,]\!]}}

\makeatletter
\def\widebreve{\mathpalette\wide@breve}
\def\wide@breve#1#2{\sbox\z@{$#1#2$}%
     \mathop{\vbox{\m@th\ialign{##\crcr
\kern0.08em\brevefill#1{0.8\wd\z@}\crcr\noalign{\nointerlineskip}%
                    $\hss#1#2\hss$\crcr}}}\limits}
\def\brevefill#1#2{$\m@th\sbox\tw@{$#1($}%
  \hss\resizebox{#2}{\wd\tw@}{\rotatebox[origin=c]{90}{\upshape(}}\hss$}
\makeatletter

\makeatletter
\usepackage{pict2e,picture}
\pgfkeys{/csteps/inner ysep/.initial=4pt,
    /csteps/inner xsep/.initial=4pt,
    /csteps/inner color/.initial=red,
    /csteps/outer color/.initial=blue,
}
\newsavebox\csteps@CBox
\newlength\csteps@XLength \newlength\csteps@YLength \newlength\csteps@YDepth \newlength\csteps@tmplen
\def\csteps@CircledParam#1#2{\sbox\csteps@CBox{#2}%
    \csteps@XLength=\wd\csteps@CBox\advance\csteps@XLength by\pgfkeysvalueof{/csteps/inner xsep}\relax
    \csteps@tmplen=\pgfkeysvalueof{/csteps/inner ysep}\relax
    \csteps@YDepth=\dp\csteps@CBox\advance\csteps@YDepth by 0.5\csteps@tmplen\relax
    \csteps@YLength=\ht\csteps@CBox\advance\csteps@YLength by\dp\csteps@CBox\advance\csteps@YLength by\pgfkeysvalueof{/csteps/inner ysep}\relax
    \typeout{DBG:#2\space X\space\the\csteps@XLength\space Y:\the\csteps@YLength\space D:\the\csteps@YDepth}%
    \raisebox{-#1\csteps@YDepth}{%
    \ifdim\csteps@XLength>\csteps@YLength
    \makebox[\csteps@XLength]{
        \makebox(0,\csteps@YLength){%
            \color{\pgfkeysvalueof{/csteps/outer color}}\put(0,0){\oval(\csteps@XLength,\csteps@YLength)}%
        }%
    \makebox(0,\csteps@YLength){%
        \put(-.5\wd\csteps@CBox,0){\textcolor{\pgfkeysvalueof{/csteps/inner color}}{#2}}%
    }}%
    \else
    \makebox[\csteps@YLength]{%
        \makebox(0,\csteps@YLength){%
            \color{\pgfkeysvalueof{/csteps/outer color}}\put(0,0){\circle{\csteps@YLength}}%
        }%
    \makebox(0,\csteps@YLength){%
        \put(-.5\wd\csteps@CBox,0){\textcolor{\pgfkeysvalueof{/csteps/inner color}}{#2}}%
     }}%
    \fi
    }%
}
\def\Circled#1{\csteps@CircledParam{1}{#1}}
\def\CircledTop#1{\csteps@CircledParam{0}{#1}}
\makeatother

\begin{document}
\title{Large deviations of the periodic Toda chain}
\author{Tamara Grava\thanks{SISSA, via Bonomea 265, 34136, Trieste, Italy, INFN sezione di Trieste
and  School of Mathematics, University of Bristol, UK \texttt{grava@sissa.it}},  Alice Guionnet\thanks{  CNRS, Unité de Mathématiques Pures et Appliquées (UMPA),
École Normale Supérieure de Lyon, 46, allée d'Italie, Lyon, France  \texttt{alice.guionnet@ens-lyon.fr}}, Karol K. Kozlowski \thanks{  CNRS, Laboratoire de
Physique (LPENS), École Normale Supérieure de Lyon,  46, allée d'Italie, Lyon, France \texttt{karol.kozlowski@ens-lyon.fr}},
Alex Little\thanks{CNRS, Laboratoire de
Physique (LPENS), École Normale Supérieure de Lyon,  46, allée d'Italie, Lyon, France \texttt{alexander.little@ens-lyon.fr}}}

\maketitle

\abstract{This work establishes a large deviation principle for the spectral measure of the Lax matrix associated to the periodic Toda chain of $N$ particles, subject
to a generalised Gibbs measure. This large deviation principle is governed by a rate function which can be regarded as a generalisation of the free energy of the
system. Such a large deviation principle is proven both for the model when the momentum is constrained to be zero and when it is allowed to fluctuate.
Moreover, the large deviation principle is proven directly at the level of the representation of the generalised Gibbs partition function
given in terms of the variables realising the classical separation of variables, \textit{i.e.} rectify the equations of motion.
As such, this work paves the way  towards the computation of the thermodynamic limit of
dynamical correlation functions in the Toda chain subject to generalised Gibbs ensemble statistics.}

\newpage

\tableofcontents

\newpage

\section{Introduction}

In 1967, Morikazu Toda \cite{TodaIntroTodaAndClassicalSolutionTodaChain} introduced a chain of oscillators with nearest-neighbour interactions governed by an exponential potential. This system, now known
as the Toda lattice, serves as the prototypical example of a many-body classical integrable system exhibiting nonlinear waves.
Toda’s original motivation for the model was to construct a nonlinear two-body interaction that admitted stable pulses—lattice solitons—and periodic waves.
For small displacements, the lattice approximates the anharmonic
chains studied by Fermi, Pasta, Ulam, and Tsingou (FPUT) \cite{FPUT} which exhibit integrable-like behaviour on short time scales.
The complete integrability of the Toda lattice was established by Flaschka \cite{Flaschka} and Manakov \cite{Manakov}, who introduced a non-canonical change of
variables that led them to express the equations of motion in the Lax form.
The latter implies the existence of a large number of explicit conserved quantities, which permits the detailed study of the
system's evolution under various boundary conditions.

While much is known about the deterministic evolution of the Toda chain, its behaviour starting from
\textit{random} initial data   has remained  much less well understood, 
especially at a rigorous level. In standard statistical mechanics, the statistical behaviour of local observables in an isolated system is expected to
converge to that described by a Gibbs distribution, a class of measures on the phase space that is invariant under the Hamiltonian dynamics. 
Recently much progress has been made  in \cite{Aggarwal2,Aggarwal1,SpohnGGEBook} to study  the dynamics  of the Toda lattice  with random
initial data sampled from the Gibbs ensemble.   One should however stress that, for integrable systems like the Toda chain,
the existence of a large number of \textit{local} conserved quantities suggests that the statistical behaviour is instead captured
by a generalisation of the standard Gibbs measure.
Indeed, physical intuition suggests that while local parts of the system "thermalise" due to interactions with the rest of the system,
this local equilibrium must respect all conservation laws. These considerations motivated the introduction of \textit{Generalised Gibbs Ensembles} (GGEs) which are
believed to characterise this local equilibrium. These take a similar form to a Gibbs measure except that the product of the energy and the inverse temperature is
replaced by a linear combination of all the conserved quantities of the system with associated generalised inverse temperatures.
In fact, such a linear combination can be wrapped up into a single function known as the "potential" associated to the GGE.
This modified notion of local equilibrium plays a central role in the rapidly developing theory of \textit{Generalised Hydrodynamics} (GHD).
The first  precursor of the theory  was developed by Spohn in 1983 for  the hard rods model, long before the term GHD was coined.
GHD was then discovered in quantum integrable systems in \cite{BCDF,CDY} and later extended to classical integrable systems in
\cite{Bastianello_2018,Bastianello_2022,BDE22,SpohnGGEBook} and  with  rigorous mathematical results in   \cite{Aggarwal2,Aggarwal1,Croydon_Sasada,FNRW}.
For this reason, a better understanding of generalised  Gibbs ensembles is essential for this theory to progress.

In this article, we prove a Large Deviation Principle (LDP) for the distribution of eigenvalues of the Lax matrix $\mathsf{L}$ for the $N$-periodic Toda lattice
subject to GGE statistics.
The study of these ensembles was pioneered by Spohn \cite{Spohn20}, who derived the convergence of the eigenvalue distribution for polynomial
potentials using transfer matrix approaches and comparisons with Dumitriu--Edelman $\beta$-ensembles \cite{DE}.
Subsequently, Guionnet and Memin \cite{GuionnetM22} established a LDP for the empirical measure of eigenvalues for more general potentials,
though the rate function remained non-explicit.  The explicit form of the rate function itself was
proposed by Doyon \cite{Doyon} using an analogue of the Landau--Lifshitz approach \cite{LandauLifschitzStatMechBook}
for the computation of the entropy of a free gas combined with knowledge of the scattering structure deriving from the model's integrability, 
see also \cite{YangYangNLSEThermodynamics}.  Likewise, Spohn \cite{SpohnGGEBook} derived a closed expression for the rate function
in the case of the open Toda chain using a specially chosen boundary potential,  so that some of the integration steps simplified.
Our rigorous approach builds on the use of the separated, \textit{aka}, action-angle,  variables which are known to trivialise the equations of motion.
This allows us to recast the  generalised Gibbs ensemble measure in these new coordinates which then allows for a direct
analysis,  yielding the large deviation principle
for the empirical measure of the  eigenvalues of the Lax matrix $\mathsf{L}$,  along with the explicit form of the rate function.  In addition to its directness and explicitness,
an important advantage of our approach is the universality of the separated variables representation for the
classical integrable models.  In particular, we expect that our approach could be applied to other classical integrable
models  whose  spectral  curve is hyperelliptic.
We also expect that our  result would  permit a better understanding of the dynamics of the  periodic Toda lattice
with generalised Gibbs ensemble initial data, an extension of the results obtained by Aggarwal \cite{Aggarwal2,Aggarwal1},
and the study of the hydrodynamic regime of its dynamical correlation functions on rigorous grounds.

\subsection{Background}

In 1967  Morikazu  Toda \cite{TodaIntroTodaAndClassicalSolutionTodaChain,Toda}
introduced  an infinite chain of oscillators  with nearest neighbour interactions of exponential type described by the Hamiltonian
\begin{equation}
\label{Toda}
\mc{H}_{\infty}\big(\{q_a\}_{a \in \mathbb{Z}}, \{p_a\}_{a \in \mathbb{Z}} \big) \, = \, \sul{ j \in \mathbb{Z} }{} \dfrac{p^2_j}{2} \, + \,
 \sul{ j \in \mathbb{Z} }{} V( q_{j+1}-q_j) \quad \e{with} \quad V(x)\, =\,  \e{e}^{-  x}+x-1.
\end{equation}
Above,  $ q_j  \in\R $  represents the position of the $j^{\e{th}}$ particle relative to a global equilibrium   and   $p_j\in\R$ its  conjugate  momentum.
 One endows the formal phase space with the canonical position-momentum Poisson bracket which yields
the following Hamiltonian equations
\begin{equation}
\label{Hamilton}
\dfrac{\mathrm{d}q_j}{\mathrm{d}t}=\partial_{p_j}\mc{H}_{\infty}=p_j, \quad \dfrac{\mathrm{d}p_j}{\mathrm{d}t}=-\partial_{q_j}\mc{H}_{\infty}= \ex{ q_{j-1}-q_j }-\ex{q_j-q_{j+1}},\quad j\in\mathbb{Z}.
\end{equation}
Since the Hamiltonian is translationally  invariant, the formal total momentum $\sum_{j\in\mathbb{Z}}p_j$ is conserved.
 For small displacements of the particles relative to their equilibrium,  the Toda lattice can be approximated by  an anharmonic  chain of oscillators,
 that  on a relatively short time scale, displays an integrable-like behaviour, first uncovered by Fermi,
Pasta, Ulam and Tsingou  \cite{FPUT}.

The  complete  integrability of the Toda lattice was   first   derived  by     Flaschka \cite{Flaschka} and Manakov \cite{Manakov} by introducing
the   non-canonical change of variables
\begin{equation}
\label{ab}
b_{j}=p_{j},\;\quad a_{j}= \ex{ \frac{1}{2} (q_{j}-q_{j+1}) }, \; j \in \mathbb{Z}
\end{equation}
with  $b_j\in\R$ and $a_j\in\R^+$. Note that  the  variables $ \{a_j\}_{j\in\mathbb{Z}}$ only specify
the differences between consecutive entries  $\{q_j\}_{j\in\mathbb{Z}}$, so the former only determines the latter up to an
overall shift.
In this work, we consider an $N$-particle periodised reduction of this model, namely when the infinite collection of
position-momentum variables satisfy
 \begin{equation}
 \label{periodic}
(q_{N+j},p_{N+j}) \, = \, (q_j+N\ell,p_j),\quad N\in \mathbb{N}.
\end{equation}
The quantity   $\ell\in\R$ appearing above is called  the  stretch parameter.  In the Flaschka--Manakov variables,
\eqref{periodic} turns into a genuine periodicity: $b_{j+N}=b_j$ and $a_{j+N}=a_j$, $j\in\mathbb{Z}$.
The periodicity constraint \eqref{periodic} now manifests itself as
\begin{equation}
 \prod_{k=1}^{N} a_k \, = \, \mathrm{e}^{- \frac{ N \ell }{2} } \,  \equiv \, \veps_N \quad \mathrm{with} \quad \ell>0 \;.
\label{ecriture des contraintes sur  ak et definition vepsN}
\end{equation}
Since the motion of the centre of mass of the system is trivial, one may fix the overall momentum of the system to $B \in \R$, \textit{i.e.}
\begin{equation}
\sul{k=1}{N} b_k \, = \, B  \;.
\label{ecriture des contraintes sur les bk}
\end{equation}
After the $N$-periodic reduction \eqref{periodic}, the above formal  Toda chain Hamiltonian in the Flaschka-Manakov variables reduces effectively to
the $N$-periodic Toda chain Hamiltonian
\beq
\mc{H}_{\e{per}}\big( \{a_j\}_1^N, \{b_j\}_1^N \big) \, = \, \sul{j=1}{N}\Big\{ \frac{1}{2}  b^2_j  \, + \,  a_j^2 \Big\}
\enq
with the variables being constrained to evolve in the manifold defined by
\eqref{ecriture des contraintes sur  ak et definition vepsN}-\eqref{ecriture des contraintes sur les bk}.
The associated Poisson structure with respect to the coordinates of $\bs{a}_N = (a_1,\dots,a_N)\in \R^N_+$ and $\bs{b}_N = (b_1,\dots,b_N)\in \R^N$ is
\begin{equation}
\label{Poisson_ab}
\{a_i,a_j\} = 0 = \{b_i,b_j\} \, , \quad \e{and} \quad  \{ a_i , b_j \} = \frac{ a_i }{ 2 } \big( \de_{ij}- \de_{i,j-1} \big) \quad \forall i,j=1,\dots,N,
\end{equation}
 up to incorporating the natural periodic identification for the boundary indices.

Thus, the Hamilton equations \eqref{Hamilton}  take the form
\begin{equation}
\label{flaschkamanakov}
\dfrac{\mathrm{d}b_j}{\mathrm{d}t}=\{b_j,\mc{H}_{\e{per}}\}=a_j^2-a_{j-1}^2,\quad \dfrac{\mathrm{d}a_j}{\mathrm{d}t}=\{a_j,\mc{H}_{\e{per}}\}=\frac{a_j}{2}(b_{j+1}-b_j),\quad
j\in \mathbb{Z} \,.
\end{equation}
The key observation by  Flaschka and Manakov  was to show  that the system of equations \eqref{flaschkamanakov} can be expressed in the so called Lax form.
Indeed, introduce the pair of matrices
{\tiny \beq
  \mathsf{L}^{(\la)}(\bs{a}_N,\bs{b}_N)   \, = \,
\left( \begin{matrix}
b_1 & a_1 &  &    \dots & \lambda^{-1} a_N \\
a_1 & b_2 & a_2 \\
 & a_2 & b_3 &   \\
& & & \ddots \\
& & & &   a_{N-1}\\
\lambda a_N & & &   a_{N-1} & b_N
\end{matrix} \right)  \; \; \e{and} \;\;
\mathsf{M}^{(\la)}(\bs{a}_N,\bs{b}_N)  \, =  \,   \frac{1}{2}
\left( \begin{matrix}
0 & a_1 &  &    \dots & -\lambda^{-1} a_N \\
-a_1 & 0 & a_2 \\
 & -a_2 & 0 &   \\
& & & \ddots \\
& & & &   a_{N-1}\\
\lambda a_N & & &   -a_{N-1} & 0
\end{matrix}\right) \, .
\nonumber
\enq   }
Then, a direct calculation shows that \eqref{flaschkamanakov} is equivalent to the Lax equation
\beq
\frac{\mathrm{d}}{\mathrm{d} t} \mathsf{L}^{(\la)}(\bs{a}_N,\bs{b}_N)   \,  =  \,
\big[  \mathsf{L}^{(\la)}(\bs{a}_N,\bs{b}_N)  , \mathsf{M}^{(\la)}(\bs{a}_N,\bs{b}_N)    \big]\, , \qquad    \forall \lambda \in \mathbb{C}\, .
\enq
The above implies that, for any $j \in \mathbb{N}$, $\e{tr}\big\{  \big(  \mathsf{L}^{(\la)}(\bs{a}_N,\bs{b}_N) \big)^{j} \big\}$
is a conserved quantity, \textit{i.e.} does not depend on $t$. This ensures that the
characteristic polynomial $\det_N\big[ \mu \mathtt{I} - \mathsf{L}^{(\la)}(\bs{a}_N,\bs{b}_N) \big]$ is also a conserved quantity,
this for all $\lambda, \mu \in \mathbb{C}$. Furthermore, it is important to stress that each conserved quantity
$\e{tr}\big\{  \big(  \mathsf{L}^{(\la)}(\bs{a}_N,\bs{b}_N) \big)^{j} \big\}$ is "local", in the sense that it is expressed as
$\sum_{k=1}^{N} \varrho_{k;j}(\bs{a}_N,\bs{b}_N)$ with a density $\varrho_{k;j}(\bs{a}_N,\bs{b}_N)$
whose expression only involves momenta and positions with indices at distance at most $j$ from $k$, up to
$N$-periodicity.

By expanding the determinant  $\det_N\big[ \mu \mathtt{I} - \mathsf{L}^{(\la)}(\bs{a}_N,\bs{b}_N)  \big]$ with respect to $\la$,
we find there is a monic polynomial $P$ of degree $N$ such that
$$\det_N \big[ \mu \mathtt{I} -  \mathsf{L}^{(\la)}(\bs{a}_N,\bs{b}_N)  \big] \,  = P(\mu) -  \veps_N (\lambda + \lambda^{-1}) \, , $$
where $ \veps_N$ is as in \eqref{ecriture des contraintes sur  ak et definition vepsN}. When $\lambda$ is a unimodular complex number,
$ \mathsf{L}^{(\la)}(\bs{a}_N,\bs{b}_N) $ is
Hermitian, and so the characteristic polynomial has a complete set of real roots.
In this case $\lambda + \lambda^{-1} \in [-2 ; 2]$, and we have that the polynomial
$P(\mu) - 2 \veps_N s$ has a complete set of real roots for all $s \in [-1;1]$.

Note that the tridiagonal part of the Lax matrix is associated to the second order finite difference operator on $\R^{\mathbb{Z}}$:
$\op{D}[\vp]_n \, = \, a_{n-1}\vp_{n-1} \, + \,  b_{n} \vp_{n} \, + \, a_{n-1} \vp_{n+1}$. When looking for its eigenvectors
in the context of the $N$-particle Toda chain, one is interested in $N$-periodic, resp. anti-periodic,
structure in the entries $\vp_{n+N}=\vp_n$, resp. $\vp_{n+N}=-\vp_{n}$. This set up
projects the infinite  system of equations for the spectrum of the tridiagonal operator to those of
the $N\times N$ matrices
\begin{align}
\mathsf{L}^+(\bs{a}_N,\bs{b}_N) \overset{\mathrm{def}}{=}  \mathsf{L}^{(+1)}(\bs{a}_N,\bs{b}_N)  \;, \quad \e{resp}.  \quad
\mathsf{L}^-(\bs{a}_N,\bs{b}_N) \overset{\mathrm{def}}{=} \mathsf{L}^{(-1)}(\bs{a}_N,\bs{b}_N) \;.
\label{definition des matrices L plus et moins}
\end{align}
This specification then singles out the two characteristic polynomials
\begin{align}
P^+(\mu) &= \det_N\big[\mu \mathtt{I} - \mathsf{L}^+(\bs{a}_N,\bs{b}_N)\big] = P(\mu)-2\veps_N= \prod_{j=1}^N (\mu-\lambda_j^+) \label{definition P+}\\
P^-(\mu) &=\det_N\big[\mu \mathtt{I} - \mathsf{L}^-(\bs{a}_N,\bs{b}_N)\big] = P(\mu)+2\veps_N= \prod_{j=1}^N (\mu-\lambda_j^-)  \label{definition P-}
\end{align}
where
\begin{align}
\label{defP2}
    P = \frac{P^+ + P^-}{2} \, , & &\text{ and } & & P^- = P^+ + 4 \veps_N \, .
\end{align}
Note that by equation \eqref{defP2} $\bs{\la}^-_N \, = \, (\la_1^-,\dots,\la_N^-)$
depends on $\bs{\la}^+_N \, = \, (\lambda_1^+,\dots,\lambda_N^+)$ and $\veps_N$.

The explicit integration of the equations of motion for the periodic Toda chain was achieved in the seminal papers
of Van Moerbeke \cite{vanMoerbeke76} and Kac--Van Moerbeke \cite{VanMoerbeke75}, see
also the monographs \cite{Babelon,Teschl}.
The Poisson bracket \eqref{Poisson_ab} introduced above is degenerate on the whole phase space $\R^N_+\times \R^N$.
Given $B\in\R$, the level set
\begin{equation}
\label{Mc}
\mathcal{M}_B \, = \, \Big\{ (\bs{a}_N, \bs{b}_N) \in\R_+^N\times \R^N \, : \,  \sum_{j=1}^N b_j=B \quad \e{and} \quad \prod_{j=1}^Na_j=\veps_N \Big\}\, ,
\end{equation}
is a  symplectic leaf for the periodic Toda flow   with respect to the Poisson bracket \eqref{Poisson_ab}. When restricted to such a leaf,
the Toda system has $(N-1)$ degrees of freedom and  its invariant manifolds are smooth tori.

There exists a precise geometric description  of $\mathcal{M}_B$ achieved in \cite{vanMoerbeke76}. One first needs to introduce the Dirichlet spectrum
which consists of the eigenvalues $\bs{\mu}_{N-1}\, =\, \big( \mu_1 , \dots, \mu_{N-1} \big)$
of the $(N-1) \times (N-1)$ matrix  $\op{L}_2(\bs{a}_N,\bs{b}_N)$,  obtained from  $\mathsf{L}^{\pm}(\bs{a}_N,\bs{b}_N)$ by deleting its first row and column.
Since $\op{L}_2(\bs{a}_N,\bs{b}_N)$ is an $(N-1) \times (N-1)$ submatrix of $\mathsf{L}^\pm(\bs{a}_N,\bs{b}_N)$,
its spectrum must interlace with the spectra of $\mathsf{L}^\pm(\bs{a}_N,\bs{b}_N)$
as follows:
\begin{equation}\label{interlacing}
    \lambda_N^+ > \lambda_N^- \geq \mu_{N-1} \geq \lambda_{N-1}^- > \lambda_{N-1}^+ \geq \mu_{N-2} \geq \lambda_{N-2}^+ > \dots
\end{equation}
In particular, the coordinates of $\bs{\mu}_{N-1}$  are strictly ordered $\mu_1 < \dots < \mu_{N-1}$
and one has
\begin{equation}
\mu_k \in [\lambda_k^- ; \lambda_{k+1}^-] \quad \mathrm{if} \,  N-k \; \mathrm{is} \; \mathrm{odd} \quad  \mathrm{and}
\quad \mu_k \in [\lambda_k^+ ;\lambda_{k+1}^+] \quad \mathrm{if} \; N-k \; \mathrm{is}\; \mathrm{even} .
\end{equation}
In fact, it follows from \eqref{interlacing} that $\boldsymbol{\lambda}_N^+\in \R^{N}_{\leq}$, $\boldsymbol{\lambda}_N^-\in \R^{N}_{\leq}$,
resp. $\boldsymbol{\mu}_{N-1}\in \R^{N-1}_{<}$, where, for further convenience, we introduced the open and closed Weyl chambers
\begin{align*}
\mathbb{R}^N_{<} &= \{ \boldsymbol{x}_N \in \mathbb{R}^N \, : \, x_1 < \dots < x_N \} \qquad \e{and} \qquad
\mathbb{R}^N_{\leq} &= \{ \boldsymbol{x}_N \in \mathbb{R}^N \, : \, x_1 \leq \dots \leq x_N \} \, .
\end{align*}
The relation between $\boldsymbol{\lambda}_N^+$, $\boldsymbol{\lambda}_N^-$ and $\boldsymbol{\mu}_{N-1}$ is depicted for the case $N=5$ in Figure \ref{figure}.
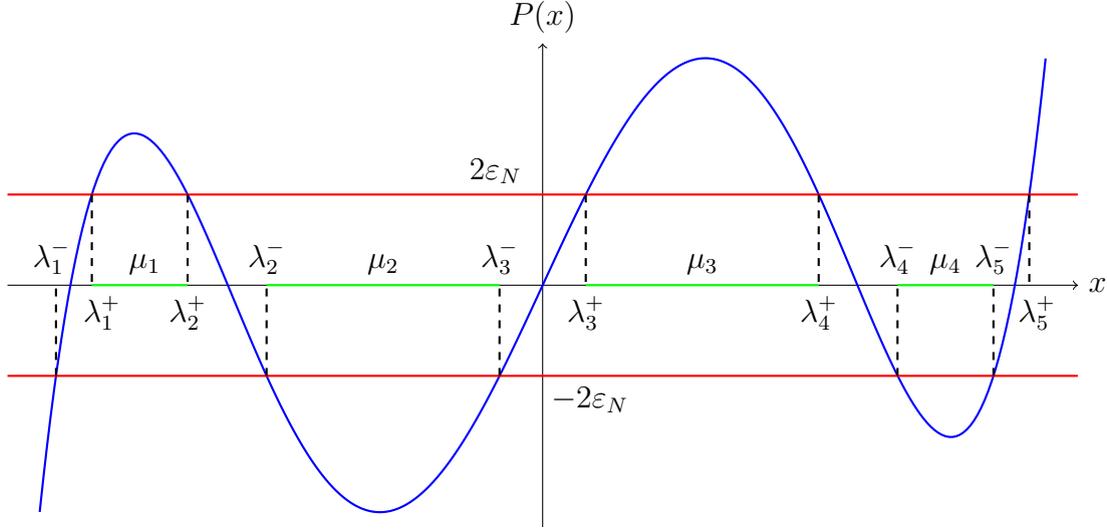
\begin{figure}\label{figure}
\centering
\begin{tikzpicture}
    \begin{axis}[
        width=0.9\textwidth,  
        height=8cm,           
        axis lines=middle,
        enlargelimits=false,
        xmin=-1.7, xmax=1.7,
        ymin=-0.4, ymax=0.4,
        xtick=\empty,
        ytick=\empty,
        xlabel={$x$},
        ylabel={$P(x)$},
        xlabel style={at={(axis description cs:1,0.5)},anchor=west},
        ylabel style={at={(axis description cs:0.5,1)},anchor=south},
        axis line style={->},
        ticks=none,
        clip=true,
    ]

    \addplot [
        domain=-1.7:1.7,
        samples=300,
        thick,
        restrict y to domain=-0.4:0.4,
        unbounded coords=discard,
        blue,
    ] {0.5*(x+1.5)*(x+1)*x*(x-1)*(x-1.5)};

    \draw [red, thick] (axis cs:-1.7,0.15) -- (axis cs:1.7,0.15);
    \draw [red, thick] (axis cs:-1.7,-0.15) -- (axis cs:1.7,-0.15);

    \draw [black, thick, dashed] (axis cs:-1.43199,0.15) -- (axis cs:-1.43199,0);
    \draw [black, thick, dashed] (axis cs:-1.12782,0.15) -- (axis cs:-1.12782,0);
    \draw [black, thick, dashed] (axis cs:0.137028,0.15) -- (axis cs:0.137028,0);
    \draw [black, thick, dashed] (axis cs:0.876926,0.15) -- (axis cs:0.876926,0);
    \draw [black, thick, dashed] (axis cs:1.54585,0.15) -- (axis cs:1.54585,0);

    \draw [black, thick, dashed] (axis cs:-1.54585,-0.15) -- (axis cs:-1.54585,0);
    \draw [black, thick, dashed] (axis cs:-0.876929,-0.15) -- (axis cs:-0.876929,0);
    \draw [black, thick, dashed] (axis cs:-0.137028,-0.15) -- (axis cs:-0.137028,0);
    \draw [black, thick, dashed] (axis cs:1.12728,-0.15) -- (axis cs:1.12728,0);
    \draw [black, thick, dashed] (axis cs:1.43199,-0.15) -- (axis cs:1.43199,0);

    \node[above] at (axis cs: -1.26391, 0) {$\mu_1$};
    \node[above] at (axis cs: -0.5069785, 0) {$\mu_2$};
    \node[above] at (axis cs: 0.506977, 0) {$\mu_3$};
    \node[above] at (axis cs: 1.279635, 0) {$\mu_4$};

    \node[below] at (axis cs: -1.40, 0) {$\lambda_1^+$};
    \node[below] at (axis cs: -1.12782, 0) {$\lambda_2^+$};
    \node[below] at (axis cs: 0.137028, 0) {$\lambda_3^+$};
    \node[below] at (axis cs: 0.876926, 0) {$\lambda_4^+$};
    \node[below] at (axis cs: 1.57, 0) {$\lambda_5^+$};

    \node[above] at (axis cs: -1.56, 0) {$\lambda_1^-$};
    \node[above] at (axis cs: -0.876929, 0) {$\lambda_2^-$};
    \node[above] at (axis cs: -0.137028, 0) {$\lambda_3^-$};
    \node[above] at (axis cs: 1.12728, 0) {$\lambda_4^-$};
    \node[above] at (axis cs: 1.43199, 0) {$\lambda_5^-$};

    \node[below] at (axis cs: 0.15, -0.15) {$-2\veps_N$};
    \node[above] at (axis cs: -0.15, 0.15) {$2\veps_N$};

    \draw [green, thick] (axis cs:-1.43199,0) -- (axis cs:-1.12782,0);
    \draw [green, thick] (axis cs:-0.876929,0) -- (axis cs:-0.137028,0);
    \draw [green, thick] (axis cs:0.137028,0) -- (axis cs:0.876926,0);
    \draw [green, thick] (axis cs:1.12728,0) -- (axis cs:1.43199,0);

    \end{axis}

\end{tikzpicture}
\caption{This figure illustrates the setting for $N = 5$. The intervals in green correspond to the domains where the variables $\mu_k$, $k\in \intn{1}{4}$
are located.}
\end{figure}

We shall denote, for short, the domain where the eigenvalues $\boldsymbol{\mu}_{N-1} \, = \, \big( \mu_1, \dots, \mu_{N-1} \big)$ live as
\begin{equation}
\mathcal{D} \, = \, \Big\{ \boldsymbol{\mu}_{N-1}\in \mathbb{R}^{N-1} \, : \; \mu_k \in [\lambda_k^{\ups_k};\lambda_{k+1}^{\ups_k}]  \Big\} \qquad \mathrm{with} \qquad
\ups_k \, = \, (-1)^{N-k} \;.
\label{definition domaine integration des mus}
\end{equation}
The domain $\mathcal{D}$ depends on $\boldsymbol{\lambda}_N^+$ (or $\boldsymbol{\lambda}_N^-$),
though we shall suppress writing this dependence explicitly in the following.

\vspace{2mm}

The restriction of the Poisson bracket \eqref{Poisson_ab} to the manifold  $\mathcal M_{B}$ is symplectic. The  canonical coordinates are constructed as follows.
For $(\bs{a}_N,\bs{b}_N) \in \mc{M}_{B}$, let $(\mf{r}_2(\mu_j),\dots,\mf{r}_N(\mu_j))^{\op{t}}$   be the eigenvector    of $\op{L}_2(\bs{a}_N,\bs{b}_N)$
associated with the eigenvalue $\mu_j$ and normalised so that $\mf{r}_2(\mu_j)=-\tf{a_N}{a_1}$. The fact that the
entries of $\bs{\mu}_{N-1}$ are strictly ordered ensures that  $ \ups_j \, \mf{r}_N(\mu_j)>0$, for $j \in \intn{1}{N}$.
This leads to the new system of canonical coordinates $\big(\bs{\mu}_{N-1}, \bs{\nu}_{N-1})$,
$\mu_j$ being canonically conjugated to $\nu_j$, with $\nu_{j}=\ln |\mf{r}_N(\mu_j)|$.
In the new coordinates, the canonical symplectic form $\omega=\sum_{a=1}^{N} \dd q_a\wedge \dd p_a$ restricted to $\mathcal{M}_{B}$  is given by the expression \cite{Date_Tanaka,Flaschka_McLaughlin,vanMoerbeke76}:
\begin{equation}
    \label{omega}
\omega_{\mid \mc{M}_{B} } \, = \, \sum_{j=1}^{N-1}  \mathrm{d}\mu_j\wedge  \mathrm{d}\nu_j \,.
\end{equation}
While $\mf{r}_N(\mu_j)$ has an explicit formula  in terms of the matrix entries of $ \mathsf{L}^\pm(\bs{a}_N,\bs{b}_N)$,
it turns out that one may express it solely in terms of
the spectral data $\bs{\la}_N^{+}$ through a relation only involving the polynomial $P$ defined in \eqref{defP2}:
\begin{equation}
    \label{vN}
\mf{r}_N(\mu_j)=\dfrac{P(\mu_j)\pm\sqrt{P^2(\mu_j)-4\veps_N^{2}}}{2\veps_N}\,,
\end{equation}
here the $\pm$ sign indicates that each $\mu_j$ selects {\it one sign} of the square root.
In other words, one should think of the $\mu_j$s as points on a Riemann surface, namely
 $  (\mu_j, w_j)  \in \Sigma$, $j \in \intn{1}{N-1}$,  where
$$\Sigma = \Big\{ (x,w) \in \mathbb{C}^2 \, : \, w^2 = P(x)^2 - 4 \varepsilon_N^2 \Big\} \, . $$
Observe that the product of the two possible choices with $+$ and $-$ in \eqref{vN} gives $1$. Hence,  away from the
boundary points $\mu_k \in \{ \la_k^{\ups_k}, \la_{k+1}^{\ups_k} \}$, one has that either $\ups_j\,\mf{r}_N(\mu_j)>1$
or $0 < \ups_j \,\mf{r}_N(\mu_j) < 1$

In fact, one can reconstruct $(\bs{a}_N,\bs{b}_N)\in \R^N_+\times \R^N$ from the data  \cite{Ferguson}
\begin{itemize}
\item $B\in\R$,
\item $\varepsilon_N>0$,
\item the vector $\bs{\mu}_{N-1} \in \R^{N-1}_{<}$,
\item the vector $\bs{\mf{n}}_{N-1} \in \R^{N-1}_{\bs{v}} = \big\{ \bs{\mf{n}}_{N-1} \in \R^N \; : \; \ups_j \mf{n}_j >0\big\}$.
\end{itemize}
%
%
%

Furthermore, there is a one-to-one analytic isomorphism between the above data and the space of $N$-periodic Jacobi matrices with $\sum_{j=1}^N b_j=B$ and
$\prod_{j=1}^Na_j=\varepsilon_N$ \cite{Korotyaev}.  The formula \eqref{vN} implies that the quantities
$$
\mf{n}_j:=\mf{r}_N(\mu_j), \quad j\in \intn{1}{N-1} \, ,
$$
are the so-called Floquet multipliers of the matrix $\mathsf{L}^{(\mf{n}_j)}(\bs{a}_N,\bs{b}_N) $ having eigenvalues $\mu_j$s, namely
$$
\det_N\big[\mu_j \mathtt{I} - \mathsf{L}^{(\mf{n}_j)}(\bs{a}_N,\bs{b}_N)  \big]  \, = \,  P(\mu_j) -  \veps_N (\mf{n}_j  +\mf{n}_j^{-1})=0.
$$
Hence, one  can reconstruct the degree $N$ monic polynomial $P$ defined in \eqref{defP2}  from
the data $\big( \varepsilon_N, B,  \bs{\mu}_{N-1}, \bs{\mf{n}}_{N-1}) \in \R \times \R_+ \times \R^{N-1}_{<} \times \R^{N-1}_{\bs{v}}  $,
by an interpolation of the conditions
 \begin{itemize}
 \item $P(\mu)=\mu^N-B\mu^{N-1}+\dots$
\item  $P(\mu_j) -  \veps_N (\mf{n}_j + \mf{n}_j^{-1})=0$, $j=1,\dots, N-1$.
\end{itemize}
Again, the symmetry $\mf{n}_j\hookrightarrow \mf{n}_j^{-1}$ leading to the $2^{N-1}$ degree is manifest.
We must stress, however, that the  above  problem has a solution only if the critical points $\vsg_1 < \vsg_2 < \dots < \vsg_{N-1}$ of the polynomial $P$
satisfy the constraint  $(-1)^{N-j}P(\vsg_j)\geq 2 \varepsilon_N$.

 Let $\bs{\mf{n}}_{N-1} \, = \, ( \mf{n}_1,\dots,\mf{n}_{N-1})$  be the Floquet multipliers.  Assuming $|\mf{n}_j|\neq 1$,
 the map from the data $( \varepsilon_N, B, \bs \mu_{N-1},\bs{\mf{n}}_{N-1} )\mapsto (\varepsilon_N,\bs{\mu}_{N-1}, \bs{\la}_N^+)$
 with the constraint $\sum_{j=1}^N\bs{\la}_j^+=B$  is a local diffeomorphism of mapping degree $2^{N-1}$ \cite{vanMoerbeke76}.

 In fact, the bottom line of  \cite{vanMoerbeke76} is that there exists an open subset $\wt{\mc{M}}_{B}$
 such that $\mc{M}_B\setminus \wt{\mc{M}}_{B} $ has measure zero and a diffeomorphism
\beq
\mc{F} \, : \,  \wt{\mc{M}}_{B} \tend \Big\{ \big( \bs{\la}_N^+ , \bs{\mu}_{N-1} \big) \;: \; \bs{\la}_N^+ \in \mc{A}_N\, ,  \;
\bs{\mu}_{N-1} \in \mc{D}  \;\; \e{and} \;\; \ov{\bs{\la}}_N^+ \, = \, B \Big\}
\enq
in which $\ov{\bs{\la}}_N^+ = \sum_{a=1}^{N} \la_a^+$ and
\beq
\mathcal{A}_N \overset{\mathrm{def}}{=} \Bigg\{ \boldsymbol{\lambda}^+_N \in \mathbb{R}^N_{<} \, : \,
 \prod_{i=1}^N ( x   - \lambda_i^+) \, + \,  4 \veps_N \text{ has all real, distinct roots}   \Bigg\} \, .
\label{definition ensembles AN}
\enq
Moreover, the symplectic form $\omega$ defined in \eqref{omega} may be explicitly recast in terms of  $\big( \bs{\la}_N^+ , \bs{\mu}_{N-1} \big)$
by taking $\bs{\la}_{N-1}^+$ to be the independent variables:
\beq
\omega=\sum_{i,j=1}^{N-1} \frac{\mathrm{d}\mu_i\wedge  \mathrm{d}\lambda_j^+}{\mu_i-\lambda_j^+}
\dfrac{\mu_i-\lambda_N^+}{\mu_i-\lambda_j^+}  \f{  \sqrt{ P^+(\mu_i) P^-(\mu_i) } } {   \prod_{j=1}^N  (\mu_i-\lambda_j^-) }\,.
\label{ecriture forme volume dans variables mu et lambda+}
\enq

\subsection{Generalised Gibbs ensemble}

In the following, we shall consider Toda chain Flaschka--Manakov variables $(\bs{a}_N,\bs{b}_N)$
sampled  from a  probability measure  invariant under the Toda time flow generated by $\mc{H}_{\e{per}}$.
Such invariant measures for the periodic Toda lattice can be considered on the whole phase space $\R^N_+\times \R^N$ (grand canonical ensemble) or on the subspace
$\mathcal{M}_{B}$   (microcanonical ensemble).
Given any such invariant measure,  a question of prime interest in physics is to estimate,
in the large-$N$ regime, the associated local correlation function,  defined as
\beq
\mathbb{E}\Big[ \pl{a=1}{\ell} r_{m_{\ell}}^{(\ell)}(t_{\ell}) \Big] \qquad \e{where} \qquad r^{(\ell)}_{a} \in \{q_a, p_a\}
\enq
and $(m_{\ell}, t_{\ell}) \in \mathbb{N} \times \R $ being kept finite in $N$.
An integrable sytem such as the Toda chain having $N$ degrees of freedom has $N$ local, functionally independent, conserved quantities.
The yet unproven physical expectation is that,  with regard to the behaviour as $N\tend +\infty$ of the local correlation functions,
one may consider many non-equivalent invariant measures, built from the tower of the local conserved quantities
$\mathrm{tr}\big\{ \big[\mathsf{L}^+(\bs{a}_N, \bs{b}_N)\big]^j \big\}$, $j\in \mathbb{N}$,  which may be conveniently combined as
$\mathrm{tr} \, \big[ V\big(   \mathsf{L}^+(\bs{a}_N,\bs{b}_N ) \big) \big]$.
This is in sharp contrast with a generic $N$-particle system, where one only has one local conserved charge given by the Hamiltonian
and one given by the total momentum--in the case of translational invariance. All others conserved quantities are non-local. The multitude
of local conserved charges forces one to consider more general classes of measures called generalised Gibbs ensembles.
 Spohn  studied  the generalised  Gibbs ensembles for the Toda lattice  \cite{Spohn20} given by the following probability densities
 \begin{equation}
 \label{eq:Gibbs_todaIntro}
\begin{split}
\mathrm{d}  \rho^{\theta,V}(\bs a,\bs b) &:=    \frac{1}{Z_{N}(\theta,V)}
	\exp\Big\{- \mathrm{tr} \,  \big[ V\big(   \mathsf{L}^+(\bs{a}_N,\bs{b}_N ) \big) \big] \Big\}
\prod_{j=1}^Na_j^{2\theta-1} \; \mathrm{d} \bs a_N  \mathrm{d} \bs b_N,\quad \theta>0\\
	\end{split}
\end{equation}
where $V:\R\to\R$ is a  continuous real   potential, bounded from below, and    with suitable growth at infinity (see section  \ref{Main_result}) where
we hereby fix the notation
\beq
\dd \bs{x}_N \, \equiv \, \pl{a=1}{N} \dd x_a \;.
\enq
 The partition function
 $$
Z_{N}(\theta,V) \, =  \hspace{-3mm}
\Int{\R_+^N\times\R^N}{} \hspace{-3mm} \exp\Big\{- \mathrm{tr} \,  \big[ V\big(   \mathsf{L}^+(\bs{a}_N,\bs{b}_N ) \big) \big] \Big\}
\prod_{j=1}^Na_j^{2\theta-1} \; \mathrm{d} \bs a_N  \mathrm{d} \bs b_N
$$
is the normalising constant of the measure.
	  The above measure is an  invariant measure for the  periodic  Toda flow
because  the volume form $\mathrm{d} \bs{a}_N  \mathrm{d} \bs{b}_N \propto \omega^{  N}$ is  time invariant.
Moreover, when $V(x)=x^2$ the variables $a_j$ and $b_j$ are independent random variables
and the partition functions can be easily computed  $Z_{N}(\theta, V(x)=x^2) \, = \, \big( \sqrt{2\pi}  \Gamma(\theta)\big)^N$ \cite{Spohn20}.

 Spohn    was able to derive the density of states of  the eigenvalues of
the Lax matrix  $\mathsf{L}^+ $ by comparison with the Dumitriu--Edelman  \cite{DE}  tri-diagonal representations of
 $\beta$-ensembles   at high temperature for the case $V(x)=x^2$ and   then using a transfer matrix approach
for  polynomial potentials  $V$. Then Guionnet and Memin \cite{GuionnetM22}
proved a large deviation principle, though with a non-explicit rate function, for the distribution of the empirical measure of the
eigenvalues   of the  Lax matrix  with general potentials having polynomial behaviour at infinity.
Similar results were also implemented  for  other integrable lattices  like the Ablowitz--Ladik lattice in \cite{MM1,MM2}.
The expression for the rate function for the Toda lattice was argued, on a theoretical physics level of rigour, independently,
in the works of  Doyon \cite{Doyon} and Spohn  \cite{SpohnGGEBook}.
In this manuscript we obtain the rate function   of the periodic Toda lattice,   by first expressing the generalised Gibbs measure
as a function of the eigenvalues of the periodic Lax matrix.
This is accomplished by using the Darboux coordinates
and then  integrating over the Dirichlet spectrum, which yields the  eigenvalue distribution of the periodic Lax matrix.
This distribution  can be interpreted  as a deformed orthogonal  ensemble where the  deformation term consists of determinants of hyperelliptic integrals.

We now fix the setting of our analysis. Consider the restriction of
the generalised Gibbs measure to the symplectic manifold $\mathcal{M}_{B=0}$ defined in \eqref{Mc}
\begin{equation}
 \label{constraint_measure}
\begin{split}
	\mathrm{d}  \rho^{V}_{\mathsf{c}}(\bs a,\bs b) &:=    \frac{1}{Z^{\mathsf{c}}_{N}(V)}
\exp\Big\{- \mathrm{tr} \,  \big[ V\big(   \mathsf{L}^+(\bs{a}_N,\bs{b}_N ) \big) \big] \Big\}
\delta(\prod_{j=1}^Na_j-\veps_N) \delta( \ov{\bs{b}}_N) \mathrm{d} \bs{a}_N   \mathrm{d} \bs{b}_N,\\
	\end{split}
		\end{equation}
where we recall that $ \ov{\bs{b}}_N  = \sum_{s=1}^{N} b_s$. Then we rewrite the above measure using the Darboux coordinates
introduced earlier, after restricting the measure to the open set $\wt{\mc{M}}_{0}$ so that the change of coordinates is well defined.
In the process, we  build on the explicit expression of the volume form $\om$ \eqref{ecriture forme volume dans variables mu et lambda+}
in terms of the variables $\bs{\mu}_{N-1}$ and $\bs{\la}_N^+$ subject to the constraints arising from $\mathcal{M}_{B=0}$,
so as to reexpress the full volume form $\frac{\omega^{N-1}}{(N-1)!}$.
Eventually, we get that the constrained  measure \eqref{constraint_measure}  expressed in the variables $\bs{\lambda}_N^+$ and $\bs{\mu}_{N-1}$
takes the form
\bem\label{GGibbs_todaIntro_c}
	 \mathrm{d}\rho_{\mathsf{c}}(\boldsymbol{\lambda}_N^+,\boldsymbol{\mu}_{N-1})
=
\frac{ \Delta(\boldsymbol{\lambda}_N^+)   \Delta(\boldsymbol{\mu}_{N-1})  \prod_{k=1}^{N}\mathrm{e}^{- V(\lambda_k^+)}  }
{ \mathcal{Z}_N^\mathsf{c} [V]\prod_{k=1}^{N-1} \sqrt{P^+(\mu_k) P^-(\mu_k)}}
\delta\big( \ov{\bs{\la}}_N^+\big)\,  \\
\times \mathbbm{1}_{ \mc{A}_N }(\bs{\la}^+_N)  \mathbbm{1}_{ \mc{D} }(\bs{\mu}_{N-1})
\mathrm{d}\boldsymbol{\mu}_{N-1} \mathrm{d}\boldsymbol{\lambda}^+_N \,.
\end{multline}
Here and in the following, we fix the convenient notation
\beq
 \ov{\bs{\la}}_N^+ \, = \, \sul{a=1}{N} \la_a^{+} \;,
\label{definition contraction N vecteur}
\enq
while $\Delta(\boldsymbol{\mu}_{N-1}) = \pl{ i < j }{ N-1}(\mu_j - \mu_i)$ is the Vandermonde determinant associated with the vector $\bs{\mu}_{N-1}$
while $\Delta(\boldsymbol{\lambda}_N^+)$ the one associated with $\boldsymbol{\lambda}_N^+$. Finally, $\mathcal{Z}_N^\mathsf{c} [V]$ is the normalising constant
that we shall specify shortly.

\begin{rem}
Note that $P^+$ and $P^-$ have the same sign everywhere on $\mathbb{R}$ except for the small intervals between $\lambda_k^+$ and $\lambda_k^-$.  Thus in particular
$P$ also takes the same sign as $P^+$ and $P^-$ on the complement of the small intervals between $\lambda_k^+$ and $\lambda_k^-$.
\end{rem}
By taking the $\bs{\mu}_{N-1}$-integral over the set
 $\mc{D}$ defined in \eqref{definition domaine integration des mus} we obtain the integral
\begin{align}
\mc{I}(\bs{\la}^+_N ; \veps_N) = \Int{\mathcal{D}}{} \frac{  \Delta(\boldsymbol{\mu}_{N-1})  }{\prod_{k=1}^{N-1} \sqrt{P^+(\mu_k) P^-(\mu_k)}}
\, \mathrm{d}\boldsymbol{\mu}_{N-1}
\label{density}
\end{align}
We arrive at the joint density of the roots $\bs{\la}_N^+$, \textit{i.e.} the distribution of eigenvalues of $\mathsf{L}^+(\bs{a}_N,\bs{b}_N)$,
in the constrained model:
\begin{equation}\label{constraineddensity}
\rho^+_{\mathsf{c}}(\boldsymbol{\lambda}_N^+) \, \mathrm{d}\boldsymbol{\lambda}_N^+ = \frac{1}{\mathcal{Z}_{N}^\mathsf{c} [V]}
\delta\big(  \ov{\bs{\la}}_N^+ \big) \, \mathbbm{1}_{ \mc{A}_N }(\bs{\la}^+_N)  \, \mc{I}(\bs{\la}^+_N ; \veps_N) \Delta(\boldsymbol{\lambda}_N^+)
   \pl{k=1}{N} \mathrm{e}^{-  V(\lambda_k^+)}\, \mathrm{d}\boldsymbol{\lambda}^+_N \;.
\end{equation}

Alternatively we can relax the constraint $ \sum_{k=1}^N \lambda_k^+=0$ and consider the joint density of the roots $\bs{\la}_N^+$ in the
unconstrained model
\begin{equation}\label{unconstraineddensity}
    \rho^+_{\mathsf{u}}(\boldsymbol{\lambda}_N^+) \, \mathrm{d}\boldsymbol{\lambda}_N^+ = \frac{1}{\mathcal{Z}_{N}^\mathsf{u} [V]}
\, \mathbbm{1}_{ \mc{A}_N }(\bs{\la}^+_N)  \, \mc{I}(\bs{\la}^+_N ; \veps_N) \Delta(\boldsymbol{\lambda}_N^+) \pl{k=1}{N}\mathrm{e}^{-  V(\lambda_k^+)}
\, \mathrm{d}\boldsymbol{\lambda}^+_N
\end{equation}

The respective normalisation constants, or partition functions, are given by
\begin{align*}
\mathcal{Z}_N^\mathsf{u} [V] &= \Int{\mathcal{A}_N}{} \mc{I}(\bs{\la}^+_N ; \veps_N)
\Delta(\boldsymbol{\lambda}_N^+) \pl{k=1}{N}\mathrm{e}^{-  V(\lambda_k^+)}\, \mathrm{d}\bs{\lambda}^+_N \\
\mathcal{Z}_N^\mathsf{c} [V] &= \Int{\mathcal{A}_N}{} \delta\big(  \ov{\bs{\la}}_N^+ \big) \mc{I}(\bs{\la}^+_N ; \veps_N)
\Delta(\boldsymbol{\lambda}_N^+) \pl{k=1}{N}\mathrm{e}^{-  V(\lambda_k^+)}\, \mathrm{d}\bs{\lambda}^+_N \, .
\end{align*}

For further purpose, let us define the monic polynomial
\begin{equation}\label{defP1}
P(x) \, = \, \frac{ P^+(x)   +  P^-(x) }{ 2 } \, = \, \pl{a=1}{N}(x-\eta_a)  \;.
\end{equation}
Since $P(\la_a^{\pm})=\pm 2 \veps_N$, it follows that  $\boldsymbol{\eta}_{N} \in \mathbb{R}^N_{ <  }$.
This leads to the inversion formula
\beq
P^\pm(x) \, =  \, P(x) \mp 2\veps_N  = \prod_{k=1}^N (x - \lambda_k^\pm) \, .
\label{definition P pm}
\enq

\subsection{Statement of the result}
\label{Main_result}

In order to state our results, we shall need to give a more precise definition of the probability measures which underlie what we refer to as the
constrained and unconstrained models.
Let $\mathcal{M}_1(\mathbb{R})$ be the space of Borel probability measures on $\mathbb{R}$ equipped with the weak topology. More precisely, given any (signed) Borel
measure $\mu$ on $\mathbb{R}$ of finite total variation, let us define the norm
\begin{align}
\norm{  \mu  } = \sup_{\substack{f : \mathbb{R} \to \mathbb{R} \\ \norm{  f }_{\mathrm{BL}}\leq 1}} \Big| \Int{\mathbb{R}}{}f( x ) \, \mathrm{d}\mu(x) \Big|
\end{align}
where
\beq
\norm{  f }_{\mathrm{BL}} = \sup_{x \in \mathbb{R}}|f(x)| + \sup_{\substack{x,y \in \mathbb{R} \\ x \neq y}} \left| \frac{f(x)-f(y)}{x-y}\right| \, .
\enq
Then for any $\mu, \nu \in \mathcal{M}_1(\mathbb{R})$ we impose the metric $\op{d}_{\mathrm{BL}}(\mu,\nu) = \norm{  \mu - \nu}$. This metrises weak convergence.  We let $ \mathfrak{B}(\mathcal{M}_{1}(\mathbb{R}))$  denote the Borel sets of $ \mathcal{M}_{1}(\mathbb{R})$.
We shall denote the open ball of radius $\de>0$ around $\mu$ with respect to this distance as
\beq
B(\mu,\de) \, = \, \Big\{ \nu \in \mc{M}_1(\R) \; : \; \op{d}_{\mathrm{BL}}(\mu,\nu)<\de \Big\} \;.
\label{definition boule BL de rayon delta en mu}
\enq

\begin{defin}[Empirical measure]
Given a vector $\bs{x}_N \in \mathbb{R}^N$, the associated empirical measure corresponds to
\begin{align}
\op{L}_N^{(\bs{x}_N)} \overset{\mathrm{def}}{=} \frac{1}{N}\sum_{k=1}^N \delta_{x_i} \, .
\end{align}
Note that the map $\iota : \mathbb{R}^N \to \mathcal{M}_1(\mathbb{R})$, $\iota(\bs{x}_N) := \op{L}_N^{(\bs{x}_N)}$ is continuous.
In particular this means that if $E \subset \mathcal{M}_1(\mathbb{R})$ is a Borel set then
$\iota^{-1}(E) = \big\{ \bs{x}_N \in \mathbb{R}^N \, : \, \op{L}_N^{(\bs{x}_N)} \in E \big\}$ is a Borel set.
\end{defin}

Let us introduce the two models we will study.
\begin{defin}[Unconstrained model] Given a Borel set $E \subset \mathcal{M}_1(\mathbb{R})$ define the mass function by
\begin{align*}
&\Pi_{N}: \mathfrak{B}(\mathcal{M}_1(\mathbb{R})) \longrightarrow \intff{0}{1} \\
    &\Pi_{N}[E] = \frac{1}{\mathcal{Z}_N^\mathsf{u} [V]}\Int{  \mathcal{A}_N }{}
\mc{I}(\bs{\la}^+_N ; \veps_N) \Delta(\boldsymbol{\lambda}_N^+) \pl{k=1}{N} \Big\{ \mathrm{e}^{- V(\lambda_k^+)} \Big\}
\, \mathbbm{1}_{E}(\op{L}_N^{(\boldsymbol{\lambda}_N^+)} )    \, \mathrm{d}\bs{\lambda}^+_N \, .
\end{align*}
Likewise we define the un-normalised mass function,
\begin{align*}
&\overline{\Pi}_{N}: \mathfrak{B}(\mathcal{M}_1(\mathbb{R})) \longrightarrow \intfo{0}{+\infty} \\
&\overline{\Pi}_{N}[E] = \frac{ (N-1)! }{  |\ln(2\veps_N)|^{N-1}}\Int{  \mathcal{A}_N }{} \mc{I}(\bs{\la}^+_N ; \veps_N)
\Delta(\boldsymbol{\lambda}_N^+) \pl{k=1}{N} \Big\{ \mathrm{e}^{- V(\lambda_k^+)} \Big\}
\, \mathbbm{1}_{E}(\op{L}_N^{(\boldsymbol{\lambda}_N^+)} )    \, \mathrm{d}\bs{\lambda}^+_N \, .
\end{align*}
\end{defin}

\begin{defin}[Constrained model] Define
\beq
\mathcal{M}_{1,\mathsf{c}}(\mathbb{R}) = \Big\{ \mu \in \mathcal{M}_1(\mathbb{R}) \, : \,
\Int{\mathbb{R}}{} |x| \, \mathrm{d}\mu(x) < +\infty \text{ and } \Int{\mathbb{R}}{} x \, \mathrm{d}\mu(x) = 0 \Big\}
\label{definition espace probas contraint}
\enq
equipped with the subspace topology.  In fact, this is a dense subset of $\mathcal{M}_1(\mathbb{R})$, though we shall not need this fact.
Given a Borel set $E \subset \mathcal{M}_{1,\mathsf{c}}(\mathbb{R})$ define the mass function by
\begin{align*}
&\Pi_{N,\mathsf{c}}: \mathfrak{B}(\mathcal{M}_{1,\mathsf{c}}(\mathbb{R})) \longrightarrow \intff{0}{1} \\
    &\Pi_{N,\mathsf{c}}[E] = \frac{1}{\mathcal{Z}_N^\mathsf{c} [V]} \Int{  \mathcal{A}_N }{} \delta\big(  \ov{\bs{\la}}_N^+\big)
\mc{I}(\bs{\la}^+_N ; \veps_N) \Delta(\boldsymbol{\lambda}_N^+) \pl{k=1}{N} \Big\{ \mathrm{e}^{- V(\lambda_k^+)} \Big\}
\, \mathbbm{1}_{E}(\op{L}_N^{(\boldsymbol{\lambda}_N^+)} )    \, \mathrm{d}\bs{\lambda}^+_N  \, .
\end{align*}
where $ \ov{\bs{\la}}_N^+ $ is as given in \eqref{definition contraction N vecteur}. Likewise we define the un-normalised mass function,
\begin{align*}
&\overline{\Pi}_{N,\mathsf{c}}: \mathfrak{B}(\mathcal{M}_{1,\mathsf{c}}(\mathbb{R})) \longrightarrow \intfo{0}{+\infty} \\
&\overline{\Pi}_{N,\mathsf{c}}[E] = \frac{ (N-1)! }{  |\ln(2\veps_N)|^{N-1} }\Int{  \mathcal{A}_N }{}
\delta\big(   \ov{\bs{\la}}_N^+ \big) \mc{I}(\bs{\la}^+_N ; \veps_N)
\Delta(\boldsymbol{\lambda}_N^+) \pl{k=1}{N} \Big\{ \mathrm{e}^{- V(\lambda_k^+)} \Big\}
\, \mathbbm{1}_{E}(\op{L}_N^{(\boldsymbol{\lambda}_N^+)} )    \, \mathrm{d}\bs{\lambda}^+_N \, .
\end{align*}
\end{defin}
In particular
$$\overline{\Pi}_{N}\big[\mathcal{M}_1(\mathbb{R}) \big] =
\frac{ (N-1)!}{  |\ln(2\veps_N)|^{N-1}} \mathcal{Z}_N^\mathsf{u} [V] \mbox{ and }\overline{\Pi}_{N,\mathsf{c}}\big[ \mathcal{M}_{1,\mathsf{c} }(\mathbb{R}) \big] =
    \frac{ (N-1)!}{ |\ln(2\veps_N)|^{N-1}} \mathcal{Z}_N^\mathsf{c} [V] \, .
    $$
Let us state our hypotheses on $V$.
\begin{enumerate}[label=(V\arabic*)]
\item \label{potentialh1} $V: \mathbb{R} \longrightarrow \mathbb{R}$ is differentiable and is such that there exists two finite non-negative constants $C_1,C_2$
so that, for every $x\in\mathbb R$,

$$\sup_{u\in [-1;1]}|V'(x+u)|\le C_1 V(x)+C_2\,.$$

\item \label{potentialh2} $V$ is continuous and grows algebraically at infinity, \textit{i.e.} there exists a $\theta > 0$ such that $$ 0 < \liminf_{|x| \to \infty} \frac{V(x)}{|x|^\theta} \leq \limsup_{|x| \to \infty}\frac{V(x)}{|x|^\theta} < +\infty \, .$$
For the constrained model we require that $\theta > 1$.
\end{enumerate}
Moreover, we will assume without loss of generality that the infimum of $V$ is zero, up to replacing $V$ by $V-\inf_{\mathbb{R}} V$. 
Let us now introduce the rate function which will appear in our upcoming large deviation principle. We assume $I[\mu]=+\infty$ if
$\Int{\mathbb{R}}{} V \, \mathrm{d}\mu = +\infty$ and otherwise it is given by
\begin{equation}\label{ratefunction}
\begin{split}
    &I : \mathcal{M}_1(\mathbb{R}) \longrightarrow \intof{-\infty}{+\infty} \\
    &I[\mu] \overset{\mathrm{def}}{=} \Int{\mathbb{R}}{} V(x) \, \mathrm{d}\mu(x) -
    \Int{\mathbb{R}}{} \ln \max\Big\{ 0 ,  1+ \frac{2}{\ell} \Int{\mathbb{R}}{} \ln|x-y| \, \mathrm{d}\mu(y) \Big\} \, \mathrm{d}\mu(x)
- \mathrm{Ent}[\mu]
\end{split}
\end{equation}
where
\begin{align}
\mathrm{Ent}[\mu] =     \begin{cases}
   - \Int{\mathbb{R}}{} \ln  \Big\{ \frac{\mathrm{d}\mu(x)}{\mathrm{d}x} \Big\} \, \mathrm{d}\mu(x) & \mu \ll \mathrm{d}x \\
    -\infty & \text{otherwise.}
    \end{cases}
\end{align}
The reader might worry that the integral formula for $\mathrm{Ent}[\mu]$ has a signed integrand, and there may be situations where the integral is indeterminate. This can be resolved by re-expressing $\Int{\mathbb{R}}{} V \, \mathrm{d}\mu - \mathrm{Ent}[\mu]$ in terms of the \textit{relative} entropy between $\mu$ and the probability measure $(\Int{\mathbb{R}}{} \mathrm{e}^{-V(y)} \, \mathrm{d}y)^{-1} \mathrm{e}^{-V(x)} \, \mathrm{d}x$. Such a relative entropy is always well defined and this explains our convention that if $\Int{\mathbb{R}}{} V \, \mathrm{d}\mu = +\infty$ then $I[\mu]= +\infty$.

\begin{rem}
We remark that a necessary condition for $I[\mu] < +\infty$ is that $\mu \in \mathcal{C}$ where
\begin{align}
\mathcal{C} = \Big\{ \mu \in \mathcal{M}_1(\mathbb{R}) \, : \,
1 + \frac{2}{\ell} \Int{\mathbb{R}}{} \ln|x-y| \, \mathrm{d}\mu(y) \geq 0 \text{ for } \mu-\mathrm{a.e.} \,  x \in \mathbb{R} \Big\} \, .
\end{align}
This condition arises from the requirement that eigenvalues live in $\mathcal{A}_N$, see Section \ref{section:condroots} of the Appendix.

\end{rem}

All the properties of this rate function relevant for our purposes will be established  in Section \ref{studygrf}.
We have collected together the  properties showing that it is a "good" rate function

\begin{prop}\label{propositiongrf} Assume Hypothesis \ref{potentialh2} holds.
Then $I : \mathcal{M}_1(\mathbb{R}) \longrightarrow \intof{-\infty}{+\infty}$ is lower semi-continuous and its level sets are compact. If moreover  Hypothesis \ref{potentialh2} holds with some $\theta>1$, then the restriction of $I$ to
$\mathcal{M}_{1,\mathsf{c}}(\mathbb{R})$ \eqref{definition espace probas contraint} also has compact level sets.

\end{prop}
Our main theorem is a large deviation principle
for the un-normalised measures $ \overline{\Pi}_{N}$ and $\overline{\Pi}_{N,\mathsf{c}}$.

\begin{theorem}[Large deviation principle for the un-normalised mass function]\label{maintheorem}
The un-normalised mass function of the unconstrained model $\overline{\Pi}_{N}$ on $ \mathcal{M}_1(\mathbb{R})$
obeys a large deviation principle  at speed $N$ and rate function $I$. That is, for any open set $O \subset \mathcal{M}_1(\mathbb{R})$,
\begin{align}\label{unnormalisedLDPlowerbound}
\liminf_{N \to +\infty} \frac{1}{N}\ln \overline{\Pi}_{N}[O] \geq - \inf_{\mu \in O}I[\mu]
\end{align}
and, for any closed set $F \subset \mathcal{M}_1(\mathbb{R})$,
\begin{align}\label{unnormalisedLDPupperbound}
\limsup_{N \to +\infty} \frac{1}{N}\ln \overline{\Pi}_{N}[F] \leq - \inf_{\mu \in F}I[\mu] \, .
\end{align}
Likewise the un-normalised mass function of the constrained model $\overline{\Pi}_{N,\mathsf{c}}$ on  $\mathcal{M}_{1,\mathsf{c}}(\mathbb{R})$
also obeys a large deviation principle at speed $N$ and rate function $I$. That is, for any open set $O \subset \mathcal{M}_{1,\mathsf{c}}(\mathbb{R})$,
\begin{align}\label{ldlbtheo}
    \liminf_{N \to +\infty} \frac{1}{N}\ln \overline{\Pi}_{N,\mathsf{c}}[O] \geq - \inf_{\mu \in O}I[\mu]
\end{align}
and, for any closed set $F \subset \mathcal{M}_{1,\mathsf{c}}(\mathbb{R})$,
\begin{align}
    \limsup_{N \to +\infty} \frac{1}{N}\ln \overline{\Pi}_{N,\mathsf{c}}[F] \leq - \inf_{\mu \in F}I[\mu] \, .
\end{align}
\end{theorem}

Since $\frac{1}{N}\ln \left[ \frac{(N-1)!}{ |\ln(2\veps_N)|^{N-1}} \right]$ has a finite limit as $N \to +\infty$,
the LDP shows that $\frac{1}{N} \ln \mathcal{Z}_N^\mathsf{u} [V]$ and $\frac{1}{N} \ln \mathcal{Z}_N^\mathsf{c} [V]$ converge as $N \to +\infty$.
Their limits are related to the infimum of the relevant rate function.
Since $\Pi_{N}[E] = \frac{\overline{\Pi}_{N}[E]}{\overline{\Pi}_{N}[\mathcal{M}_1(\mathbb{R})]}$ and
$\Pi_{N,\mathsf{c}}[E] = \frac{\overline{\Pi}_{N,\mathsf{c}}[E]}{\overline{\Pi}_{N,\mathsf{c}}[\mathcal{M}_{1,\mathsf{c}}(\mathbb{R})]}$,
an immediate corollary to Theorem \ref{maintheorem}
is the following large deviation principle for the normalised measures $ \Pi_{N}$
and $ \Pi_{N,\mathsf{c}}$.
\begin{cor}[Large deviation principle for the normalised mass function]
The normalised mass function of the unconstrained model $\Pi_{N}$ obeys a large deviation principle
at speed $N$ and rate function $I - \!\! \underset{\mu \in \mathcal{M}_1(\mathbb{R})}{\inf}\!\! I[\mu ]$.
That is, for any open set $O \subset \mathcal{M}_1(\mathbb{R})$,
\begin{align}
\liminf_{N \to +\infty} \frac{1}{N}\ln \Pi_{N}[O] \geq
- \Big\{ \inf_{\mu \in O}I[\mu] - \inf_{\mu \in \mathcal{M}_1(\mathbb{R})} \hspace{-3mm} I[\mu] \Big\}
\end{align}
and, for any closed set $F \subset \mathcal{M}_1(\mathbb{R})$,
\begin{align}
    \limsup_{N \to +\infty} \frac{1}{N}\ln \Pi_{N}[F] \leq
- \Big\{ \inf_{\mu \in F}I[\mu]-\inf_{\mu \in \mathcal{M}_1(\mathbb{R})} \hspace{-3mm} I[\mu] \Big\} \, .
\end{align}
Likewise the normalised mass function of the constrained model $\Pi_{N,\mathsf{c}}$ also obeys a large deviation principle
on  $ \mathcal{M}_{1,\mathsf{c} }(\mathbb{R})$  at speed $N$ and rate function $I- \underset{ \mu \in \mathcal{M}_{1,\mathsf{c}}(\mathbb{R})}{\inf}\hspace{-3mm}I[\mu]$.
\end{cor}

Let us gather the main properties of the minimisers of our rate functions, proved in Section \ref{studygrf}.
\begin{prop}\label{propositiongrfmin} Assume Hypothesis \ref{potentialh2} holds.
Then $I : \mathcal{M}_1(\mathbb{R}) \longrightarrow \intof{-\infty}{+\infty}$  is strictly convex and achieves its minimum value  at a unique probability measure $\nu_{{\ell}} \in \mathcal{M}_1(\mathbb{R})$. If moreover  Hypothesis \ref{potentialh2} holds for some $\theta>1$, then the restriction of $I$ to
$\mathcal{M}_{1,\mathsf{c}}(\mathbb{R})$ \eqref{definition espace probas contraint} is strictly convex and it  achieves its minimal
value at a unique probability measure $\nu_{{\ell}, \mathsf{c}} \in \mathcal{M}_{1,\mathsf{c}}(\mathbb{R})$.
\end{prop}
Finally, let us relate the minimiser of $I$ on $\mathcal{M}_1(\mathbb{R}) $  to the minimiser $\mu_P$ of the free energy of a high-temperature $\beta=P/N$-ensemble,
\begin{equation}\label{def:CoulombEntropy} J^C_P[\mu] \overset{\mathrm{def}}{=} \Int{\mathbb{R}}{} V(x) \, \mathrm{d}\mu(x) - P
    \Int{\mathbb{R}^2}{}  \ln|x-y| \, \mathrm{d}\mu(y) \,   \mathrm{d}\mu(x) - \mathrm{Ent}[\mu] \, .\end{equation}
    
According to Lemmata 3.2 and 3.6  of \cite{GuionnetM22}, $J^C_P$ achieves its minimum value at a  unique probability measure $\mu_P$ which satisfies
$$ \dd\mu_P(x) = \frac{1}{\mathsf{Z}_P} \exp\Big\{ -V(x)+2P\Int{\mathbb{R}}{} \ln|x-y|\dd\mu_P(y) \Big\} \dd x$$
where $\mathsf{Z}_P$ is a normalisation constant such that $\mu_P$ is  a probability measure. In the case where the hard constraint,
see \eqref{ecriture des contraintes sur  ak et definition vepsN},
$\de \Big( \prod_{j=1}^N a_j- \mathrm{e}^{-\frac{N\ell}{2}}  \Big)$
is  replaced by a soft constraint $(\prod_{j=1}^N a_j)^{2P}$, which amounts to a Laplace transform of
our un-normalised density with respect to $\ell$, \cite{Spohn20} showed that the equilibrium measure of the Toda chain is given by $\partial_P(P\mu_P)$ when the
potential $V$ is a polynomial. This was generalised to the case where $V$  is continuous and scales like $ax^{2k}$ at infinity, with $k$ an integer and $a$ a positive
real number, see Lemma 4.5 of \cite{GuionnetM22}.
In order to use results from this paper, we hereafter assume that $V$ satisfies this growth condition at infinity
(although we do not expect this to be a relevant assumption). In particular, thanks to the results of Section 4 in \cite{GuionnetM22}, we know that
  $P\mapsto  P\int f \, \dd \mu_P$ is continuously differentiable for every bounded and continuous function $f$. We can extend  the deep relation between $\beta$-ensembles and the GGE for the Toda chain discovered in \cite{Spohn20} as follows:
\begin{lemme}\label{relBeta} Let $V$ be continuous and satisfy $V(x) = (1+\mathrm{o}(1)) a x^{2k}$ as $|x| \to \infty$ for some $a > 0$, and let $\nu_{{\ell}}$ be the unique minimiser of
$I$ in $\mc{M}_1(\mathbb R)$. Then, for every bounded continuous function $f$
$$\Int{\mathbb{R}}{} f \,  \dd \nu_{{\ell}}=\partial_s\Big( s\Int{\mathbb{R}}{} f \,  \dd \mu_{s}\Big)\Big|_{s=\frac{1}{\ell}\mathsf{m}(\ell)}
$$
with 
$$\mathsf{m}(\ell) = \Int{\mathbb{R}}{} \frac{\mathrm{d}\nu_{\ell}(x)}{1 + \frac{2}{\ell}\Int{\mathbb{R}}{} \ln |x-y| \, \mathrm{d}\nu_{\ell}(y)} \, .$$
\end{lemme}
This lemma is proven at the end of  Section \ref{sec:proofmin}.  By way of comparison to \cite{Spohn20}, note that we must evaluate $\partial_s(s\mu_s)$ at $s=\mathsf{m}(\ell)/\ell$ rather than $s=P$, the parameter of the soft constraint imposed on the periodicity in the $q_i$'s. It is natural that such a change of variables needs to be done since under the soft constraint imposed in \cite{Spohn20} the mean of the $r_i$'s does not go to $P$ but rather the derivative of the limiting free energy with respect to the parameter $P$, \textit{i.e.} $J^C_P[\mu_P]$.

\subsection{Heuristics}
Since the complete proof of Theorem \ref{maintheorem} involves many technicalities,  the reader may find it helpful to outline the overall strategy and to explain \textit{morally} why the result is true, so that the subsequent manipulations appear less arbitrary. The main theorem is established by the standard "three step procedure" to prove LDPs. One starts by establishing, as $N\tend + \infty$, both lower and upper bounds on the log-probabilities of balls
$B(\mu,\de)$ of shrinking radius $\delta \searrow 0$. When one has such upper and lower bounds on balls, this implies a \textit{weak large deviation principle}, see Theorem 4.1.11 of \cite{DemboZ01}. Hence, Propositions \ref{lowerboundonballs} and \ref{upperboundonballs} given below imply that $\big\{ \overline{\Pi}_{N,\mathsf{c}} \big\}_{N\ge 1}$ and $\big\{ \overline{\Pi}_{N} \big\}_{N\ge 1}$ satisfy a weak large deviation principle with speed $N$ and rate function $I$ on their respective spaces.

\begin{prop}[Lower bound on balls]\label{lowerboundonballs} Let
$B(\mu, \delta)$ be as in \eqref{definition boule BL de rayon delta en mu}
and let $B_{\mathsf{c}}(\mu, \delta) = B(\mu, \delta) \cap \mathcal{M}_{1,\mathsf{c}}(\mathbb{R})$,
\textit{c.f.} \eqref{definition espace probas contraint}. Then, for any $\mu \in \mathcal{M}_1(\mathbb{R})$, we have
\begin{align}
\liminf_{\delta \searrow 0} \liminf_{N \to +\infty} \frac{1}{N} \ln \overline{\Pi}_{N}[B(\mu,\delta)] \geq - I[\mu] \, .
\end{align}
Similarly, for any $\mu \in \mathcal{M}_{1,\mathsf{c}}(\mathbb{R})$, we have
\begin{align}\label{wldlbtheo}
    \liminf_{\delta \searrow 0} \liminf_{N \to +\infty} \frac{1}{N} \ln \overline{\Pi}_{N,\mathsf{c}}[B_{\mathsf{c}}(\mu,\delta)] \geq - I[\mu] \, .
\end{align}
\end{prop}
\begin{prop}[Upper bound on balls]\label{upperboundonballs} For any $\mu \in \mathcal{M}_1(\mathbb{R})$, we have
\begin{align}\label{prop:wldub}
\limsup_{\delta \searrow 0} \limsup_{N \to +\infty} \frac{1}{N} \ln \overline{\Pi}_{N}[B(\mu,\delta)] \leq - I[\mu] \, .
\end{align}
Similarly, for any $\mu \in \mathcal{M}_{1,\mathsf{c}}(\mathbb{R})$, we have
\begin{align}\label{wldubco}
\limsup_{\delta \searrow 0} \limsup_{N \to +\infty} \frac{1}{N}\ln \overline{\Pi}_{N,\mathsf{c}}[B_{\mathsf{c}}(\mu,\delta)] \leq - I[\mu] \, .
\end{align}
\end{prop}

The lower bound on balls is proven in Section \ref{section:ldlb} whilst the upper bound on balls is proven in Section \ref{section:ldub}. The final step is to prove \textit{exponential tightness} of the sequence of probability measures; this is done in  Section \ref{section:ldub}.

\begin{prop}[Exponential tightness]\label{exponentialtightness}

There exists a sequence of compact sets $K_L \subset \mathcal{M}_1(\mathbb{R})$ ($L > 0$) such that

\begin{align}
\limsup_{L \to +\infty} \limsup_{N \to +\infty} \frac{1}{N} \ln \overline{\Pi}_{N}[K_L^c] = - \infty \, .
\end{align}
Likewise, if we let $K_{L,\mathsf{c}} = K_L \cap \mathcal{M}_{1,\mathsf{c}}(\mathbb{R})$
then (we claim, according to our earlier choice of $K_L$) $K_{L,\mathsf{c}}$
is compact in the subspace topology of $\mathcal{M}_{1,\mathsf{c}}(\mathbb{R})$
\begin{align}
\limsup_{L \to +\infty} \limsup_{N \to +\infty} \frac{1}{N} \ln \overline{\Pi}_{N,\mathsf{c}}[K_{L,\mathsf{c}}^c] = - \infty \, .
\end{align}
where here the complement is taken in $\mathcal{M}_{1,\mathsf{c}}(\mathbb{R})$.
\end{prop}

It is well known (see Lemma 1.2.8 of \cite{DemboZ01}) that when both exponential tightness and a weak large deviation principle hold, then a full large deviation
principle follows.  Hence, Propositions \ref{lowerboundonballs},  \ref{upperboundonballs} and \ref{exponentialtightness} given below imply Theorem \ref{maintheorem}.

To prove Propositions \ref{lowerboundonballs} and \ref{upperboundonballs} the central point  is to estimate the joint density of the eigenvalues for the constrained
\eqref{constraineddensity} and unconstrained models \eqref{unconstraineddensity}, which contains the highly non-trivial $(N-1)$-fold integral
$\mc{I}(\bs{\la}^+_N ; \veps_N) $. If this term were replaced by a constant, the large deviations would already be known, see \textit{e.g.} Theorem 1.1 of
\cite{Garcia19} and Theorem 3.1 of \cite{GuionnetM22}. Estimating $\mc{I}(\bs{\la}^+_N ; \veps_N) $ is in fact quite challenging.
Its integrand is given by a ratio of products each involving $N^2$ terms. It so happens that there are fine cancellations when $N\tend + \infty$
between the numerator and denominator which result, ultimately, in an exponential-scale behaviour of order $N$ for $\Delta(\bs{\la}^+_N) \mc{I}(\bs{\la}^+_N ; \veps_N) $.
We stress that, on top of the mentioned structure, the denominator of the integrand of $\mc{I}(\bs{\la}^+_N ; \veps_N)$ vanishes as a square root of the distance
to $ \Dp{}\mc{D}$ and exhibits a coalescence as $N\tend +\infty$ of two square roots generating a pole in the limit.
This makes it impossible to use any kind of rough estimates in the intermediate steps.

The core idea to understanding these cancellations is that we expect the $\lambda_a^+$, $\lambda_a^-$ and $\eta_a$ to all be very close to each other, indeed
exponentially close, as they are respectively the zeroes of $P^\pm$ and $P$, where $P^\pm$ and $P$ differ by a constant of order $\veps_N$ which goes to zero
exponentially fast. Hence in the large $N$ limit, the two inverse square root singularities at $\lambda_a^+$ and $\lambda_a^-$ tend towards a single non-integrable
singularity. When $N$ is large but finite, $|\lambda_{j}^+-\lambda_j^-|$ and $|\lambda_{j+1}^+-\lambda_{j+1}^-|$ are typically much smaller than
$|\lambda_{j+1}^\pm-\lambda_j^\pm|$, and so we have an accumulation of mass at the two endpoints of the integration interval of $\mu_j$. Hence let us assume that $\mu_j$
is very close to one of the endpoints, and discard the rest of the integration region. Next, we observe that if $\mu_j$ is near the right endpoint, and $\mu_{j+1}$ is
near the left endpoint, $|\mu_j - \mu_{j+1}|$ is small, and hence the Vandermonde $\Delta(\boldsymbol{\mu}_{N-1})$ is small, and so such configurations will be
penalised. Hence the dominant configurations will be those in which each $\mu_j$ pairs up one-to-one with a $\lambda_a^\pm$.
Since there are $N-1$ of the former and $N$ of the latter, one $\lambda_k^\pm$ will be left over. For such a setting,
\beq
\mu_j \left\{ \ba{c} j=1,\dots, k-1 \\
   j = k, \dots, N-1  \ea \right. \;\;  \e{will} \, \e{be} \, \e{close} \, \e{to} \;\; \left\{\ba{c}  \la_j^{\ups_j}  \vspace{1mm} \\ \la_{j+1}^{\ups_j} \ea \right. \;.
\nonumber
\enq

Putting all these ideas together, if we assume that, for some $p>1$,
$|\lambda_a^{+} - \lambda_a^{-}| \ll N^{-p} \ll |\lambda_{b+1}^{\pm} - \lambda_b^{\pm}|$ for all $a \in \intn{1}{N}$ and $b \in \intn{1}{N-1}$,  we expect
that in the leading order approximation
\begin{align*}
\mc{I}(\bs{\la}^+_N ; \varepsilon_N) \simeq \sum_{k=1}^N \prod_{j=1}^{k-1} \bigg\{  \Int{\lambda_j^{\ups_j}}{\lambda^{\ups_j}_j + N^{-p}} \hspace{-3mm} \mathrm{d}\mu_j \bigg\}
\prod_{j=k+1}^{N} \bigg\{  \Int{\lambda_j^{\ups_j} - N^{-p}}{\lambda^{\ups_j}_j} \hspace{-3mm} \mathrm{d}\mu_j \bigg\}
\frac{\Delta(\boldsymbol{\mu}_{N-1})}{\prod_{ a=1  }^{N}\sqrt{P^+(\mu_a) P^-(\mu_a)}} \, .
\end{align*}
Under the integral sign, given the expected structure of integration variables, we expect
$$\Delta(\boldsymbol{\mu}_{N-1}) \simeq \Delta(\boldsymbol{\lambda}_N^+) \prod_{\substack{j=1 \\ j \neq k}}^N |\lambda_k^+ - \lambda_j^+|^{-1} \, .$$
Furthermore, we roughly expect
\beq
P^\pm(\mu_a) \simeq \left\{ \ba{cc} (P^{\pm})^\prime(\lambda_a^+)(\mu_a - \lambda_a^{\pm})   & a=1,\dots, k-1   \vspace{2mm} \\
 (P^{\pm})^\prime(\lambda_{a+1}^+)(\mu_a - \lambda_{a+1}^{\pm})   & a=k,\dots, N-1   \ea \right. \, .
\enq
$P^+$ and $P^-$ only differ by a constant, hence $(P^+)^\prime = (P^-)^\prime$. As a consequence
\begin{align*}
\mc{I}(\bs{\la}^+_N ; \varepsilon_N) \simeq \sum_{k=1}^N \frac{\Delta(\boldsymbol{\lambda}_N^+)}{\prod_{j=1}^N |(P^+)^\prime(\lambda_j^+)|} &
\prod_{j=1}^{k-1} \bigg\{  \Int{\lambda_j^{\ups_j}}{\lambda^{\ups_j}_j + N^{-p}} \hspace{-3mm}
\frac{1}{\big\{ |\mu  - \lambda_j^+| |\mu - \lambda_j^-|\big\}^{1/2 } } \,  \mathrm{d}\mu
\bigg\} \\
&\times \prod_{j=k+1}^{N} \bigg\{  \Int{\lambda_j^{\ups_j} - N^{-p}}{\lambda^{\ups_j}_j}  \hspace{-2mm}
\frac{1}{ \big\{ |\mu  - \lambda_j^+| |\mu - \lambda_j^-|\big\}^{1/2 } }
\,  \mathrm{d}\mu  \bigg\} \, .
\end{align*}
Then, we observe that
\begin{align*}
\Int{\lambda_j^{\ups_j}}{\lambda^{\ups_j}_j + N^{-p}}  \hspace{-3mm}
\frac{1}{ \big\{ |\mu  - \lambda_j^+| |\mu - \lambda_j^-|\big\}^{1/2 } } \,  \mathrm{d}\mu  =
\ln |\lambda_j^+ - \lambda_j^-|^{-1} + \mathrm{O}(\ln N) \, .
\end{align*}
$|\lambda_j^+ - \lambda_j^-|$ is exponentially small hence the first term is of scale $N$. We may further simplify the presumed leading order approximant
to $\mc{I}(\bs{\la}^+_N ; \varepsilon_N) $
by observing that $\frac{\Delta(\boldsymbol{\lambda}_N^+)}{\prod_{j=1}^N |(P^+)^\prime(\lambda_j^+)|} = \frac{1}{\Delta(\boldsymbol{\lambda}_N^+)}$. If we hypothesise that we are in the regime where $\mathrm{e}^{-CN} \leq |\lambda_j^+ - \lambda_j^-| \leq \mathrm{e}^{-cN}$ for some $c, C > 0$, then we have
\begin{align*}
\mc{I}(\bs{\la}^+_N ; \varepsilon_N) \simeq \frac{\mathrm{e}^{\mathrm{O}(\ln N)}}{\Delta(\boldsymbol{\lambda}_N^+)} \prod_{j=1}^N \ln |\lambda_j^+ - \lambda_j^-|^{-1} \, .
\end{align*}
From the relation $P^+ - P^- = -4\varepsilon_N$ we have
\begin{align*}
    \ln |\lambda_j^+ - \lambda_j^-| =\ln 4 \varepsilon_N -\sum_{\substack{k=1 \\ k\neq j}}^N \ln |\lambda_k^+ - \lambda_j^-| \,.
\end{align*}
Hence, all-in-all, we expect something like
\beq
\mc{I}(\bs{\la}^+_N ; \varepsilon_N) \simeq \frac{|\ln \varepsilon_N |^N}{\Delta(\boldsymbol{\lambda}_N^+)} \exp\Big\{ N \Int{x \neq y}{}
\ln \Big[ 1+ \frac{2}{\ell} \Int{\R}{} \ln|x-y| \, \mathrm{d}\op{L}_N^{(\bs{\lambda}_N^+)}(y) \Big]
\, \mathrm{d}\op{L}_N^{(\bs{\lambda}_N^-)}(x)  + \mathrm{o}(N) \Big\}  \, .
\enq
The Vandermonde will cancel with the Vandermonde appearing in the density, while $|\ln \varepsilon_N |^N$
is $\bs{\lambda}_N^+$ independent so we are free to divide it out. In this way, we are left with something at exponential scale $N$ and we see how the the double-
logarithmic term present in the rate function \eqref{ratefunction} arises, since the measures $\op{L}_N^{(\bs{\lambda}_N^+)}$ and $\op{L}_N^{(\bs{\lambda}_N^-)}$ will be
close.

Let us now sketch how to turn these rough ideas into a rigorous argument.
For the lower bound, we show in Lemma \ref{lambdaexpansion} that under a hypothesis which roughly implies  that the zeroes $(\eta_a)_{1\le a\le N}$ of $P$ are
sufficiently far apart, see \eqref{ecriture hypotheses sur les quantites petites}, an event that will be shown to have a large probability, we have
\beq
\la_{a}^{\ups}\, = \, \eta_a \, +  \, 2 \ups \f{ \veps_N }{P^{\prime}(\eta_a) } (1+\mathrm{o}(1)) \, ,  \quad \e{for} \; \e{any} \quad
\ups \in \{\pm\} \; \; \e{and} \;\; a \in \intn{1}{N}\, .
\enq
To estimate $\mc{I}$ in this regime, we first show in Proposition \ref{exactlowerbound} that
$$
\mc{I}(\bs{\la}^+_N ; \veps_N)  \,  \geq  \, N    \frac{ \det_{N-1}[ \op{A} ] }{ \De(\bs{\eta}_N) } \, ,
$$
where  $\op{A}$ is the $(N-1)\times (N-1)$ matrix with entries
$
A_{ks} = \ln\Big|  \big(\lambda_{k+1}^{\ups_k} - \eta_s \big) / \big(\lambda_{k}^{\ups_k} - \eta_s \big) \Big|$.
We can then lower bound  $\mc{I}$ in this regime (see the proof of Corollary \ref{detlowerbound2}) roughly by
$$    \frac{ N   }{ \De(\bs{\eta}_N) } |\ln (2\veps_N)|^{N-1} \pl{j=1}{N}  \bigg( 1- \frac{   \ln|P^\prime(\eta_j)|  }{ \ln(2\veps_N) } - C \frac{\ln N}{N} \bigg)^{1- \frac{1}{N}}$$
from which the desired estimate follows, as we expect that when $ \op{L}^{(\bs{\eta}_N)}_N \in B(\mu, \de)$ we roughly have

$$1-\frac{   \ln|P^\prime(\eta_j)|  }{ \ln(2\veps_N) }\simeq
1+\frac{2}{\ell} \Int{\mathbb{R}}{} \ln|\eta_j-y| \, \dd\mu(y)\,.$$

To derive the complementary upper bound, a key step is Proposition \ref{integralupperbound} where we show that for any fixed $\gamma>0$ there exists a constant $C$ finite so that
$$
\mc{I}(\bs{\la}^+_N ; \veps_N) \leq \mathrm{e}^{CN^{1- \frac{1}{4}}} \sqrt{\frac{ \mc{T}(\boldsymbol{\lambda}^+_N) \mc{T}(\boldsymbol{\lambda}^-_N)}{\Delta(\boldsymbol{\lambda}_N^+) \Delta(\boldsymbol{\lambda}_N^-)} } \cdot  \prod_{j=1}^N \max\{ \ln |\lambda_j^+ - \lambda_j^-|^{-1}   , \gamma N \}
$$
where $\mc{T}$ is an explicit function \eqref{Tdef} which, when integrated over, will be ultimately shown to provide subleading contributions
in comparison to the scale of the large deviations, see Lemmata \ref{ABbound}, \ref{Acond} and \ref{Bcond}. The parameter $\gamma$
plays the role of a regularisation parameter which we will take to $0$ at the end of our analysis. The proof of the large deviation upper bound then requires dealing
with the factor $\Delta(\bs{\lambda}_N^+)/ \Delta(\bs{\lambda}_N^-)$ which in fact can be eliminated by noticing that this is the Jacobian of the change of
variables $\bs{\lambda}_N^-\mapsto \bs{\lambda}_N^+$, see Proposition \ref{jacobian}.

\section{Lower bound on balls}\label{section:ldlb}

In this section we prove Proposition \ref{lowerboundonballs} for the constrained model. The proof for the unconstrained model follows exactly the same procedure, save
for missing a few extra steps which are needed to handle the conditioning. We do this in three steps: firstly we establish a fully general lower bound on the joint
density of $\boldsymbol{\eta}_{N}$, a set of roots related to $\boldsymbol{\lambda}^+_{N}$ through \eqref{defP1}-\eqref{definition P pm}. Secondly we simplify  this
lower bound further, conditionally to certain hypotheses; and, thirdly, we apply these conditional lower bounds, showing that these hypotheses obtain with high
probability  that $\boldsymbol{\lambda}^+_{N}$ is well approximated by $\boldsymbol{\eta}_{N}$ what results in a lower bound on the $\boldsymbol{\lambda}^+_{N}$
density.

\subsection{Exact lower bound on the density}

Let us begin by establishing the following preliminary  lower bound on the integrand of $ \mc{I}(\bs{\la}^+_N ; \veps_N) $. 
\begin{prop}\label{exactlowerbound} Let $\mathcal{I}$ be defined by \eqref{density}. Then, one has
\begin{align}
\mc{I}(\bs{\la}^+_N ; \veps_N)  \,  \geq  \, N    \frac{ \det_{N-1}[ \op{A} ] }{ \De(\bs{\eta}_N) } \, .
\end{align}
where  $\op{A}$ is the $(N-1)\times (N-1)$ matrix with entries
\begin{equation}
A_{ks} = \ln\left| \frac{ \lambda_{k+1}^{\ups_k} - \eta_s }{ \lambda_{k}^{\ups_k} - \eta_s} \right|  \; , \quad \bs{\eta}_N \in \R^N_{<}\,
\label{definition matrice A}
\end{equation}
being the roots introduced in \eqref{defP1}.
\end{prop}
\begin{proof}

By the
 inequality of arithmetic and geometric means (for short, the AM-GM inequality) we have
\beq
\sqrt{P^+(\mu)P^-(\mu)} \, \leq  \, \Big| \frac{P^+(\mu) + P^-(\mu)}{2} \Big| \, = \, (-1)^{N-k}  P(\mu) \quad  \e{for} \quad  \mu \in \intff{ \lambda_k^{\ups_k} }{ \lambda_{k+1}^{\ups_k} } \, ,
\enq
with $P$ as introduced in \eqref{defP1}. This leads to the lower bound involving the roots of $P$:
\bem
\mc{I}(\bs{\la}^+_N ; \veps_N) \, \geq (-1)^{\frac{N(N-1)}{2}} \, \pl{k=1}{N-1} \bigg\{ \Int{\lambda_k^{\ups_k}}{\lambda_{k+1}^{\ups_k}} \hspace{-2mm} \mathrm{d}\mu_k   \bigg\}
\f{ \Delta( \bs{\mu}_{N-1} )  }{  \pl{k=1}{N-1} \pl{u=1}{N}  (\mu_k-\eta_u)    } \\
\, = \, (-1)^{\frac{N(N-1)}{2}} \det\bigg[  \Int{ \lambda_k^{\ups_k} }{ \lambda_{k+1}^{\ups_k} }    \f{ \mu^{j-1} }{  \pl{u=1}{N}  (\mu-\eta_u)  }    \dd\mu   \bigg]
\, = \, (-1)^{\frac{N(N-1)}{2}} \det_{N-1}\big[ \op{B} \big] \;.
\end{multline}
In the last line, we have performed elementary row operations to replace $\mu^{j-1}$ with the monic polynomial $\prod_{k=1}^{j-1} (\mu-\eta_k)$, leading to
\begin{align}
B_{kj}  =   \Int{ \lambda_k^{\ups_k} }{ \lambda_{k+1}^{\ups_k} }
\pl{s=j}{N} \Big\{ \frac{1}{ \mu-\eta_s }  \Big\} \, \dd \mu \qquad \e{with} \qquad  k,j \in \intn{1}{ N-1 }\;.
\label{definition matrice tile Bkj}
\end{align}
Note that the integration is over an interval in which $P^+(\mu)P^-(\mu) \geq 0$, whereas one has $P^+(\eta_s)P^-(\eta_s) = - 4 \veps_N^2 < 0$, hence the singularity at $\mu = \eta_s$ is strictly outside of the integration domain and so the integral is well defined.


By a partial fraction decomposition we may write
\begin{align*}
\prod_{s=j}^{N}(\mu-\eta_s)^{-1} = \sum_{s=j}^N \frac{1}{\mu-\eta_s} \prod_{\substack{\ell=j  \\ \ell \neq s}}^N \frac{1}{(\eta_s - \eta_\ell)} \;,
\end{align*}
which permits explicit integration leading to
\begin{align*}
B_{kj} &= \sum_{s=j}^N \ln\left| \frac{\lambda_{k+1}^{\ups_k} - \eta_s}{\lambda_{k}^{\ups_k}  - \eta_s} \right| \prod_{\substack{\ell=j  \\ \ell \neq s}}^N \frac{1}{(\eta_s - \eta_\ell)}
                                            & & k,j \in \intn{1}{ N-1 } \, .
\end{align*}
If, by an abuse of notation, we let $\op{A}$ be the $(N-1) \times N$ matrix with matrix elements given by \eqref{definition matrice A}, and we let
$\op{C}$ be the $N \times (N-1)$ matrix with entries
\begin{align}
%
%
%
C_{sj} &= \mathbbm{1}_{s \geq j} \prod_{\substack{\ell=j  \\ \ell \neq s}}^N \frac{1}{(\eta_s - \eta_\ell)}  \label{definition matrice C}
\end{align}
with $j\in \intn{1}{N-1}$ and $s \in\intn{1}{N}$,  then $\op{B} = \op{A}\op{C}$.  Note that $\op{A}, \op{C}$ are rectangular matrices.
However we may extend $\op{C}$ to a $N\times N$ matrix $\wh{\op{C}}$, by defining $\wh{C}_{sj}$ as in \eqref{definition matrice C}, with now $s,j\in \intn{1}{N}$
and upon understanding an empty product as $1$.

Next, we would like to extend $\op{A}$ to an $N \times N$ square matrix $\wh{\op{A}}$ in a determinant preserving way,
\textit{viz}. $\det_{N-1}\!\big[ \op{B} \big]  =  \det_{N}\!\big[ \wh{\op{B}} \big]$
with $\wh{ \op{B} } = \wh{ \op{A} } \, \wh{ \op{C} }$. We define the $N\times N$ matrix $\wh{ \op{B} }$ as
\beq
\wh{B}_{Nj} \, = \, \de_{Nj} \quad \e{for} \quad  j \in \intn{1}{N} \qquad \e{and} \qquad \wh{B}_{kj} \, = \,B_{kj}
\enq
 with $k,j \in \intn{1}{N-1}$, and $B_{kj}$ still given by \eqref{definition matrice tile Bkj}.
It remains to define  $\wh{B}_{kN}$ for $k \in \intn{1}{N-1}$. We set $\wh{A}_{ks}=A_{ks}$ for $k \in \intn{1}{N-1} $ and $s \in \intn{1}{N} $. Further,
the expression for $\wh{C}_{sj}$ then unambigously fixes
\begin{align*}
\wh{B}_{kN} \,=\, \wh{A}_{kN} C_{NN} \, = \, \wh{A}_{kN} & & k\in \intn{1}{N-1} \;,
\end{align*}
hence fixing the expression for $ \wh{B}_{kN} $ with $k=\intn{1}{N-1}$. The remaining equation
\begin{align*}
\delta_{N j} \, = \, \wh{B}_{Nj} \, = \,  \sum_{s=j}^N \wh{A}_{Ns} \wh{C}_{sj}  & & j\in \intn{1}{N}
\end{align*}
can be satisfied by taking $\wh{A}_{Nk} \, = \,  1$ for $k=\intn{1}{N}$. This holds because
\begin{align*}
\delta_{N j} = \frac{1}{2\pi \i } \Oint{ \Ga(\{\eta_a\} ) }{}  \pl{s=j}{N} \frac{ 1 }{z-\eta_s }\, \mathrm{d}z  \, = \,
\sum_{s=j}^N  \pl{ \substack{\ell=j \\ \ell \neq s } }{N} \frac{ 1 }{ \eta_s-\eta_\ell } & & \forall j\in \intn{1}{N}
\end{align*}
where $\Ga(\{\eta_a\} )$ is an index $1$ loop around $\eta_1, \dots, \eta_N$ and the integral can be taken by evaluating the
residues of the poles located either inside or outside of $\Ga(\{\eta_a\} )$.

These manipulations allow us to factorise the determinant and lead to
\begin{align}
\mc{I}(\bs{\la}^+_N ; \veps_N)  \,  \geq  \,  (-1)^{\frac{N(N-1)}{2}} \det_N\!\big[ \wh{\op{B}}\,  \big] \,  = \,    \frac{ \det_N[ \, \wh{\op{A}} \, ] }{ \De(\bs{\eta}_N) }
\end{align}
where $\wh{\op{A}}$ is now our modified matrix
\begin{align}
\wh{A}_{kj} &= \ln\left| \frac{\lambda_{k+1}^{ \ups_k } - \eta_j }{\lambda_{k}^{\ups_k} - \eta_j} \right| & & k\in \intn{1}{N-1}, \quad j\in \intn{1}{N} \\
\wh{A}_{Nj} &= 1 & & j \in \intn{1}{N} \, .
\end{align}
Finally, by Lemma \ref{columnoperations} one may simplify $\det_N[ \, \wh{\op{A}}]$ further by carrying out column operations, yielding Proposition \ref{exactlowerbound}. \end{proof}
\begin{lemme}\label{columnoperations}
We have the relation,
\beq
\det_N[ \, \wh{\op{A}}]  \, = \, N \det_{N-1}[ \op{A} ]
\enq
with $\op{A}$ as introduced in \eqref{definition matrice A}, $\wh{A}_{kj}=A_{kj}$ and $\wh{A}_{Nj}=1$ where $k \in \intn{1}{N-1}$, $j\in \intn{1}{N}$.

\end{lemme}

\begin{proof}

We first establish that $\sum_{j=1}^N A_{kj} = N \delta_{k N}$.

For $k = N$ this is trivial, so assume $k \in \intn{1}{N-1}$. Then
\begin{align*}
\sum_{j=1}^N A_{kj} \,  = \, \sum_{j=1}^N  \ln\left| \lambda_{k+1}^{\ups_k} - \eta_j  \right| - \sum_{j=1}^N  \ln\left|\lambda_{k}^{ \ups_k } - \eta_j\right|
= \ln |P(\lambda_{k+1}^{ -\ups_{k+1}} )| - \ln |P(\lambda_{k}^{\ups_k})| \, .
\end{align*}
But by the equations $P = P^+ - 2\veps_N = P^- + 2\veps_N$, $|P(\lambda_{k+1}^{-\ups_{k+1}})| = |P(\lambda_{k}^{\ups_k})|  = 2\veps_N$ which establishes the desired identity.

Next, on the level of $\det_N[ \, \wh{\op{A}} \, ]$ we perform the column operation $\bs{C}_N \hookrightarrow \sul{s=1}{N} \bs{C}_s$, which yields the claim.

\end{proof}

%




\subsection{Estimates on the roots}\label{prelim}

Consider the polynomial $P(x) \, = \,  \pl{a=1}{N}( x  -\eta_a)$. In the following, we make two assumptions on its roots $\eta_1, \dots, \eta_N$. We assume there exists constants $\varkappa_{N}$ and $\vth_N$ such that
\beq
\sg_{N;a}\, \overset{\mathrm{def}}{=} \,  \f{ \veps_N }{   | P^{\prime}(\eta_a)  |  }  \, \leq \, \varkappa_{N} \qquad \e{and} \qquad
\f{ \sg_{N;a} }{ |\eta_{ka} | } \, \leq \,\vth_N \;,
\label{ecriture hypotheses sur les quantites petites}
\enq
for any $k\not=a$ and where we fix the shorthand notation $\eta_{ba} \, = \, \eta_b - \eta_a$.
The upper bounds should be at least such that $N \vth_N \tend 0$ and $\varkappa_N \tend 0$. By Proposition \ref{Pzetalowerbound2}, this implies that 
\begin{equation}\label{Pzetalowerbound3}
    \min_{k\in \intn{1}{N-1}}|P(\zeta_k)| \geq \frac{1}{4}\vth_N^{-1} \veps_N > 2 \veps_N
\end{equation}
for $N$ sufficiently large. Then by Lemma \ref{etalambdaequivalence}, this implies that the polynomials $P(x) \mp 2 \veps_N$ both have a complete set of a real simple roots, which we denote  $\big\{ \la_{a}^{\pm} \big\}_{a=1}^N$,
\beq
P(x) \mp 2 \veps_N \; = \; \pl{a=1}{N}\big( x -\la_a^{\pm} \big) \;.
\enq

\begin{lemme}\label{lambdaexpansion}

Under the Hypothesis \eqref{ecriture hypotheses sur les quantites petites},  we have that for any $a\in  \intn{1}{N}$, $\ups \in \{\pm\}$,
\beq
\la_{a}^{\ups}\, = \, \eta_a \, +  \, 2 \ups \f{ \veps_N }{P^{\prime}(\eta_a) } \, + \, \e{O}\big(\sg_{N;a} N \vth_N \big) \;,
\enq
with a remainder that is uniform in $a$ and $N$.
\end{lemme}

\begin{proof}

One first observes that owing to the second bound in \eqref{ecriture hypotheses sur les quantites petites}, $\eta_{a}\not=\eta_b$ if $a\not=b$.
This implies that $P$ is invertible in some neighbourhood of each of the roots $\eta_a$. Now, we obtain an a priori bound on the \textit{loci} of the roots $\la_a^{\pm}$.

Consider the auxiliary polynomial
\begin{align*}
    P_s(x) \, = \, P(x) \, - \,  2 s  \veps_N \qquad \e{with} \qquad  s  \in \intff{ - 1 }{ 1 } \;.
\end{align*}
Then, one has that
\begin{align*}
    P_s\Big(\eta_{a} \pm 4  \tfrac{ \veps_N }{ |P^{\prime}(\eta_a)| }  \Big) &= \pm   \tfrac{ 4 \veps_N }{ |P^{\prime}(\eta_a)| }
\pl{ \substack{ b=1 \\b  \not=a} }{N} \bigg( \eta_{ab} \, \pm\, \     \tfrac{ 4 \veps_N }{ |P^{\prime}(\eta_a)| } \bigg) \, - \, 2 s  \veps_N \\
&= \pm   \tfrac{ 4 \veps_N }{ |P^{\prime}(\eta_a)| } |P^{\prime}(\eta_a)| (-1)^{N-a}
\pl{ \substack{ b=1 \\ b \not=a} }{N} \bigg( 1 \, \pm\, \     \tfrac{ 4 \veps_N }{ \eta_{ab} |P^{\prime}(\eta_a)| } \bigg) \, - \, 2 s  \veps_N \\
&= 4 \veps_N \Big\{ \pm (-1)^{N-a}\, - \, \f{s}{2} \, + \, \e{O}\big( N \vth_N \big) \Big\} \;.
\end{align*}
Thus, $P_s$ changes sign at least once on
\beq
I_{a}\, = \, \intff{\eta_{a} - 4  \tfrac{ \veps_N }{ |P^{\prime}(\eta_a)| } }{ \eta_{a} + 4  \tfrac{ \veps_N }{ |P^{\prime}(\eta_a)| } }
\enq
and thus admits one root on this interval. Observe that $I_{a}\cap I_{b} \, =\,  \emptyset$ for $a\not=b$. Indeed, by symmetry one may suppose that $b>a$. Then
\beq
\eta_{b} - 4  \tfrac{ \veps_N }{ |P^{\prime}(\eta_b)| } \, -\, \eta_{a} - 4  \tfrac{ \veps_N }{ |P^{\prime}(\eta_a)| }  \, > \,
\eta_{ba} \, \Big( 1 \, - \, \e{O}\big(\vth_N\big) \Big)
\enq
Thus, $P_s$ admits at least $N$ roots $\la_a^{s}$, $a\in \intn{1}{N}$, with  $\la_a^{s} \in I_a$.
Since $P_s$ has exactly $N$ roots, this entails that it has exactly one root in $I_a$.

In particular, by setting $s=\pm 1$, one gets that $\la^{\pm}_a \in I_a$ and that $P$ is a smooth diffeomorphism on the interval $I_a$. Then, by the Taylor series expansion up to the second order, one gets
that

\beq
\la_a^{\ups} \,  = \, P^{-1}(2 \ups \veps_N) \, = \, P^{-1}(0) \,  +\, \f{ 2 \ups \veps_N  }{ P^{\prime}\big( P^{-1}(0) \big) }
\, + \, \e{O}\bigg( \veps_N^2 \sup_{u \in \intff{0}{2 \ups \veps_N} }\big|   \big( P^{-1} \big)^{\prime\prime}(u) \big|  \bigg) \;.
\label{Taylor rest Linfty pour lambda pm}
\enq
Now, one has through explicit calculations
\begin{align*}
 \big( P^{-1} \big)^{\prime\prime}(u) \; = \; - \f{  P^{\prime\prime}\big(  P^{-1}(u) \big) }{  \big[  P^{\prime}\big(  P^{-1}(u) \big) \big]^3 } \;.
\end{align*}
Since, for any $u \in  \intff{-2  \veps_N }{2  \veps_N }$,  $P^{-1}(u) \in I_a$, one has that
\begin{align*}
\norm{ \big( P^{-1} \big)^{\prime\prime} }_{ L^{\infty}\big( \intff{0}{2 \ups \veps_N} \big) }  \; \leq \;
 \Big| \Big| \tfrac{  P^{\prime\prime}  }{  (  P^{\prime}  )^3 }  \Big| \Big|_{ L^{\infty}(  I_a )}   \;.
\end{align*}
Now, for any $t\in \intff{-1}{1}$, it holds
\begin{align*}
    &P^{\prime}\big(\eta_a+ 2 t \sg_{N;a} \big) \; = \; \sul{ \substack{k=1 \\ k \not= a} }{N} 2 t \sg_{N;a}\pl{ \substack{s=1 \\ s \not= a,k} }{N} \Big( \eta_{as} + 2 t \sg_{N;a} \Big)
\; + \; \pl{ \substack{s=1 \\ s \not= a } }{N} \Big( \eta_{as} + 2 t \sg_{N;a} \Big) \\
 \, &= \, P^{\prime}(\eta_a) \Big\{  \pl{ \substack{s=1 \\ s \not= a } }{N} \Big(1 +  \tfrac{ 2 t \sg_{N;a} }{  \eta_{as} } \Big)
 \, + \,  \sul{ \substack{k=1 \\ k \not= a} }{N} \f{ 2 t \sg_{N;a} }{ \eta_{ak} }  \pl{ \substack{s=1 \\ s \not= a,k} }{N} \Big(1 +  \tfrac{ 2 t \sg_{N;a} }{  \eta_{as} } \Big) \Big\} 
 \, = \, P^{\prime}(\eta_a)  \Big( 1+\e{O}\big(N\vth_N\big) \Big) \;.
\end{align*}

Next one has
\begin{align*}
    P^{\prime\prime}(u) \, = \, \sul{ \substack{k, \ell=1 \\ k\not= \ell } }{N} \pl{ \substack{s=1 \\ s \not= k,\ell} }{N}\hspace{-2mm} \big( u-\eta_{s} \big)
 \, = \, \sul{ \substack{k, \ell=1 \\ k\not= \ell, \\ k, \ell\not= a } }{N} (u-\eta_a) \hspace{-2mm} \pl{ \substack{s=1 \\ s \not= k,\ell,a} }{N} \hspace{-2mm} \big( u-\eta_{s} \big)
\, + \, 2   \sul{ \substack{k=1 \\ k \not= a} }{N}  \pl{ \substack{s=1 \\ s \not= a,k} }{N} \big( u - \eta_{s} \big) \;.
\end{align*}
This decomposition thus leads to
\begin{align*}
P^{\prime\prime}\big(\eta_a+ 2 t \sg_{N;a} \big) \; = \;
\sul{ \substack{k, \ell=1 \\ k\not= \ell, \\ k, \ell\not= a } }{N} 2 t \sg_{N;a}\pl{ \substack{s=1 \\ s \not= a,k,\ell} }{N} \Big( \eta_{as} + 2 t \sg_{N;a} \Big)
\, + \, 2   \sul{ \substack{k=1 \\ k \not= a} }{N}  \pl{ \substack{s=1 \\ s \not= a,k} }{N}\Big( \eta_{as} + 2 t \sg_{N;a} \Big)  \\
\; = \;  2   \sul{ \substack{k=1 \\ k \not= a} }{N}  \f{ P^{\prime}\big(\eta_a)   }{ \eta_{ak} } \pl{ \substack{s=1 \\ s \not= a,k} }{N}\Big( 1 +  \tfrac{ 2 t \sg_{N;a} }{  \eta_{as} } \Big)
\, + \, \sul{ \substack{k, \ell=1 \\ k\not= \ell, \\ k, \ell\not= a } }{N}  \f{2  t \sg_{N;a} P^{\prime}\big(\eta_a)   }{ \eta_{ak} \eta_{a\ell} }
\pl{ \substack{s=1 \\ s \not= a,k,\ell } }{N}\Big( 1 +  \tfrac{ 2 t \sg_{N;a} }{  \eta_{as} } \Big)
\end{align*}
Thus, direct bounds yield
\begin{equation*}
 \Big| \f{ P^{\prime\prime} }{ P^{\prime} }\big(\eta_a+ 2 t \sg_{N;a} \big) \Big| \, \leq \, 2 \Big( 1+\e{O}\big(N\vth_N\big) \Big) \sul{ \substack{k=1 \\ k\not= a} }{N}  \f{ 1  }{ |\eta_{ak}| }  \;.
\end{equation*}
Then, inserting the above into \eqref{Taylor rest Linfty pour lambda pm} allows one to conclude.

\end{proof}

There are a few consequences of the expansion for the roots $\la_{a}^{\pm}$ provided by Lemma \ref{lambdaexpansion}. First of all,
the $\la_{a}^{\pm}$ are increasingly ordered,
\beq
\la_1^{\ups_1}\, <\, \cdots \, < \, \la_N^{\ups_N} \;, \quad \e{for} \quad \ups_{a}\in \{ \pm \} \;.
\enq
Indeed, for $N$ large enough, we have
\begin{eqnarray*}
 \la_{a+1}^{\ups} - \la_a^{\ups^{\prime}}  \, & \geq& \,  \big( \eta_{a+1}-\eta_{a} \big) \cdot \bigg\{ 1 \, - \,
 \f{ 2 \veps_N }{ | (\eta_{a+1}-\eta_{a}) P^{\prime}(\eta_a)|}
\, - \, \f{ 2 \veps_N }{ | (\eta_{a+1}-\eta_{a}) P^{\prime}(\eta_{a+1})|}  \, + \, \e{O}(N \vth_N^2)  \bigg\}  \\
& =& \, \big( \eta_{a+1}-\eta_{a} \big) \cdot \big\{ 1 \, - \,    \e{O}( \vth_N)  \big\} >0 \;.
\end{eqnarray*}
Further, owing to
\beq
P^{\prime}\big(\eta_a\big) \, = \, (-1)^{N-a} |P^{\prime}\big(\eta_a\big)|
\enq
the roots $\la_{a}^{\pm}$ interlace in a manner depending on $N$.
%
%
%
%
%
%
%
%
For example, for $N$ odd, the roots are ordered as
\beq\label{ecriture entrelacement intervalles des lambdas}
\la_1^{-}\, < \, \la_1^{+} \, < \, \la_2^{+}\, < \, \la_2^{-} \, < \, \la_{3}^{-} \, < \, \cdots \, < \,
\la_N^{-}\, < \, \la_N^{+} \;.
\enq
This interlacing means that, again for $N$ odd,
\beq
\intoo{ \la_{2k-1}^{-} }{  \la_{2k}^{-} } \supset  \intoo{ \la_{2k-1}^{+} }{  \la_{2k}^{+} }
\qquad \e{and} \qquad
\intoo{ \la_{2k}^{-} }{  \la_{2k+1}^{-} } \subset  \intoo{ \la_{2k}^{+} }{  \la_{2k+1}^{+} }  \;
\label{ecriture entrelacement intervalles des mus}
\enq
whereas for $N$ even the inclusions are reversed.

\subsection{Asymptotic lower bound on the density}

In this section, we shall assume that the entries of $\bs{\eta}_N$
satisfy the hypotheses \eqref{ecriture hypotheses sur les quantites petites}. This will allow us to build on Lemma \ref{lambdaexpansion}
so as to obtain an explicit lower bound on $\det_{N-1}[ \op{A} ] $.

\begin{prop}\label{detlowerbound2}
Let $\kappa>0$ and $\gamma >1$. Assume that $\bs{\eta}_N\in \R^N$ is such that
\begin{itemize}
\item[i)] $1- \frac{ \ln|P^\prime(\eta_k)| }{ \ln(2\veps_N) }  \, \geq \,  \kappa$ for all $k \in \intn{1}{N}$;
\item[ii)] $|\eta_a - \eta_b| \, \geq \, N^{-\ga}$ for all $a \neq b \in \intn{1}{N}$;
\item[iii)] $N^{\f{3}{2}} \vartheta_N \to 0$ as $N \to +\infty$.
\end{itemize}

 Let $\op{A}$ be as defined in \eqref{definition matrice A}. Then, there exists $N_0 \in \mathbb{N}$ such that, for all $N \geq N_0$, the following lower bound holds
\bem
\det_{N-1}[ \op{A} ] \, \geq  \,  | \ln(2\veps_N)|^{N-1} (1 - \mathrm{O}( N^{\f{3}{2}}  \vartheta_N)) \\
\times \prod_{j=1}^N  \Big( 1- \frac{ \ln|P^\prime(\eta_j)| }{\ln(2\veps_N)}    - \mathrm{O}\big( \frac{\ln N}{N}\big) - \mathrm{O}\big( N^{\f{3}{2}} \vartheta_N \big) \Big)^{1- \frac{1}{N}} \;.
\end{multline}
\end{prop}

Observe first of all that the hypotheses of Proposition \ref{detlowerbound2} imply the bounds \eqref{ecriture hypotheses sur les quantites petites}.
We thus start the proof by inserting the estimates of Lemma \ref{lambdaexpansion} into the definition
of $\op{A}$ from \eqref{definition matrice A}. We find, for $j \in \intn{1}{N-1} \setminus \{k,k+1\}$,
$$A_{kj} = \ln|\eta_{k+1} - \eta_j| - \ln|\eta_k - \eta_j| + \mathrm{O}(\vartheta_N) \, .$$
For $j = k+1 \in \intn{2}{N-1}$ we find
$$A_{k k+1} = \ln (2 \veps_N ) -  \ln|P^\prime(\eta_{k+1})| - \ln \left| \eta_{k} - \eta_{k+1}\right| +\mathrm{O}( N \vartheta_N)$$
and for $j = k \in \intn{1}{N-1}$
$$A_{kk} = -\ln (2 \veps_N) + \ln|P^\prime(\eta_k) | + \ln \left| \eta_{k} - \eta_{k+1}\right| + \mathrm{O}( N \vartheta_N) \, .$$
Let $\op{E}$ be an $(N-1) \times (N-1)$ matrix with the following entries
\begin{align*}
E_{kj} =
\begin{cases}
\ln|\eta_{k+1}-\eta_j| - \ln|\eta_k - \eta_j| \;, & j \in \intn{1}{N-1} \setminus \{ k,k+1\}  \vspace{3mm} \\
 \ln |P^\prime(\eta_{k})| + \ln \left| \eta_{k} - \eta_{k+1}\right|  \; , & j = k \in  \intn{1}{N-1}  \vspace{3mm} \\
 - \ln |P^\prime(\eta_{k+1})| - \ln \left| \eta_{k} - \eta_{k+1}\right| \;, & j = k+1 \in  \intn{2}{N-1} \\
\end{cases}
\end{align*}
and let $\op{F}$ be the $(N-1) \times (N-1)$ matrix given by
\begin{align*}
\op{F} = \left( \begin{matrix}
-1 & 1 \\
& -1 & 1  \\
& & -1 & 1 \\
& & &&  \ddots &1  \\
& & & & & -1  \\
\end{matrix} \right)
\end{align*}
so that
$$A_{kj} = \ln(2\veps_N) F _{kj}+ E_{kj} + \mathrm{O}\Big( \big[1 + N ( \de_{kj}+\de_{k+1j} )  \big]\vartheta_N\Big) \, .$$
Clearly $\det_{N-1}[ \op{F}] =(-1)^{N-1}$. We then find that
\beq
\det_{N-1} [ \op{A} ] \,  = \,   (-1)^{N-1} \det_{N-1}\big[ \ln(2\veps_N) \op{I}_{N-1}  + \op{F}^{-1} \op{E} + \mathrm{O}(N \vartheta_N)  \big]    \;.
\label{ecriture det A relie a det M}
\enq
We stress that the remainder is to be understood entrywise. The inverse of $\op{F}$ may be computed as $(\op{F}^{-1})_{kj} = - \mathbbm{1}_{k \leq j}$.  Then
\begin{align}
( \op{F}^{-1}\op{E} )_{kj} = - \ln|\eta_N -\eta_j| + \mathbbm{1}_{k \neq j} \ln|\eta_k - \eta_j| - \delta_{kj} \ln|P^\prime(\eta_j)| \, .
\label{ecriture F moins 1 E}
\end{align}
\begin{lemme}
\label{Lemme reecriture determinant A comme perturbation de diagonale}

The following relation holds
\bem
\det_{N-1}\big[ \op{A} \big]  \, = \,   | \ln(2\veps_N)|^{N-1}
   \det_N\bigg[  \Big( 1- \frac{1}{\ln(2\veps_N)}  \ln|P^{\prime}(\eta_j)| \Big) \delta_{kj}  + \frac{ N }{ \ln(2\veps_N) } M_{kj} \\
\, +  \, \mathrm{O}\Big( \big( 1+ N \de_{jN} \mathbbm{1}_{k \leq N-1}   \big)\vartheta_N \Big) \bigg]  \, .
\end{multline}
in which the matrix $\op{M}$ is defined as
\begin{align}
M_{kj} \,  =  \, \f{1}{N}\mathbbm{1}_{k \neq j} \ln| \eta_k -\eta_j| \, = \,
    \begin{cases}
\frac{1}{N} \ln|\eta_k - \eta_j| & k \neq j  \vspace{2mm}\\
0 & k = j
\end{cases} \, .
\label{definition matrices M}
\end{align}

\end{lemme}

\begin{proof}

Starting from \eqref{ecriture det A relie a det M}, after adding an extra line and column, one obtains
\begin{align*}
\det_{N-1} [- \op{A} ] \,  = \,  \frac{1}{\ln(2\veps_N)}   \det_{N}\left[  \ba{cc} \ln(2\veps_N) \de_{kj}  + (\op{F}^{-1} \op{E})_{kj} + \mathrm{O}(N \vartheta_N) & 0 \\
                                                    \ln |\eta_{Nj}|     &      \ln(2\veps_N)  \ea \right]    \;.
\end{align*}
Then, after performing the line operations $\bs{L}_k \hookrightarrow \bs{L}_k+\bs{L}_N$ and using \eqref{ecriture F moins 1 E}, one gets
\begin{align*}
\det_{N-1} [- \op{A}] \,  = \,   \frac{1}{\ln(2\veps_N)}   \det_{N}\left[  \ba{cc} \ln\Big( \f{ 2\veps_N }{ |P^{\prime}(\eta_{j}) | } \Big) \de_{kj}  + N M_{kj} + \mathrm{O}(N \vartheta_N) &  \ln(2\veps_N) \\
                                                    \ln |\eta_{Nj}|     &      \ln(2\veps_N) \ea  \right]    \;.
\end{align*}
Now, we perform the column operation $\bs{C}_N \hookrightarrow \bs{C}_N   -   \sul{j=1}{N-1} \bs{C}_{j} \, = \,  \wt{\bs{C}}_N $, which yields
$$
\big( \wt{\bs{C}}_N  \big)_N \, = \, \ln(2\veps_N) \, - \, \sul{k=1}{N-1} \ln |\eta_{Nk}|  \, = \, \ln\Big( \f{ 2\veps_N }{ |P^{\prime}(\eta_{N}) | } \Big) \de_{NN} \, + \, N M_{NN} \,+ \,  \mathrm{O}(N \vartheta_N) \, 
$$
 and also for $j<N$
\begin{eqnarray*}
\big( \wt{\bs{C}}_N  \big)_j \, &=& \, - \ln\Big( \f{ 2\veps_N }{ |P^{\prime}(\eta_{j}) | } \Big)
\, - \, \sul{ \substack{k=1 \\  k\not= j} }{N-1} \ln |\eta_{jk}| \,- \,  \mathrm{O}(N^2 \vartheta_N) +\ln(2\veps_N)  \, = \, \ln |\eta_{jN}| \,+ \,  \mathrm{O}(N^2 \vartheta_N)  \\
&=& NM_{jN}+ \,  \mathrm{O}(N^2 \vartheta_N)  \\
\end{eqnarray*}
Thus, we get
\beq
\det_{N-1} [- \op{A}] \,  = \,   \frac{1}{\ln(2\veps_N)}   \det_{N}\bigg[    \ln\Big( \f{ 2\veps_N }{ |P^{\prime}(\eta_{j}) | } \Big) \de_{kj}  + N M_{kj}   \,+ \,   N^{\delta_{jN}\mathbbm{1}_{k\leq N-1}}\mathrm{O}(N \vartheta_N)  \bigg]    \;.
\enq
Pulling out $\ln(2\veps_N)$ from each line, using that $\ln(2\veps_N) = - |\ln(2\veps_N)|$ is of order $N$  yields the claim.

\end{proof}

We now establish that the matrix $\op{M}$ introduced in \eqref{definition matrices M} is almost negative definite when projected onto appropriate subspaces of $\R^N$.
 This will allow us to develop an effective lower bound of the determinant obtained in Lemma \ref{Lemme reecriture determinant A comme perturbation de diagonale}.

\begin{lemme}\label{almostposdef}
Let $\big(\bs{y}_N,\bs{x}_N\big)=\sul{s=1}{N}y_ax_a$, $\|\bs{x}\|^{2}=\big(\bs{x}_N,\bs{x}_N\big)$ and set 
\beq
\mc{V}^{(0)}_N \, = \,  \Big\{ \bs{x}_N \in \mathbb R^N \, : \, \sul{k=1}{N} x_k \, = \,  0 \Big\} \, .
\label{definition sous espace V0}
\enq
Let $\ga > 1$ and $\bs{\eta}_N \in  \R^N$ be such that $|\eta_a - \eta_{b}| \geq N^{- \ga}$ for all $a \neq b$.
Then, there is a constant $C > 0$ such that the $N \times N$ matrix $\op{M}$ introduced in \eqref{definition matrices M} satisfies
\beq
\big( \bs{x}_N,  \op{M} \, \bs{x}_N  \big) \, \leq \,  C \frac{\ln N}{N} \norm{ \bs{x}_N }^2 \quad   for   \,  all  \quad  \bs{x}_N \in \mc{V}^{(0)}_N \, .
\enq
\end{lemme}

\begin{proof}
Without loss of generality let us assume $\| \bs{x}_N \|^2 = 1$. Consider the function
$$f_{\bs{x}_N}(\xi) = \frac{1}{2} N^{\gamma + 1} \sum_{s=1}^N x_s \mathbbm{1}_{ \intff{ -N^{-\ga - 1} }{ N^{-\ga - 1}  }  }\big( \xi-\eta_s \big)  \, .$$
Then
\begin{align*}
  \Xi[ f_{\bs{x}_N} ] &:=  \Int{\mathbb R^2}{} \ln|s- t| f_{\bs{x}_N}(s) f_{\bs{x}_N}(t) \, \mathrm{d}s \, \mathrm{d}t \\
  &= \sum_{\substack{k,j=1 \\ k \neq j}}^N x_k x_j \ln|\eta_k - \eta_j|
                + \frac{1}{4} \sum_{\substack{k,j=1 \\ k \neq j}}^N x_k x_j
                \Int{-1}{1} \ln\left| 1+  N^{- \ga - 1} \frac{u - v }{\eta_k - \eta_j} \right|\, \mathrm{d}u \, \mathrm{d}v  \\
&\quad  - (\gamma+1) \ln N  + \frac{1}{4} \Int{-1}{1}  \ln|u -  v|\, \mathrm{d}u \, \mathrm{d}v \\
&= N \big( \bs{x}_N, \op{M} \bs{x}_N \big)  \,  + \, \mathrm{O}(\ln N)\, .
\end{align*}
Finally we claim that $\Xi[ f_{\bs{x}_N} ]\leq 0$ which would prove the claim.  Lemma 1.8 of \cite{SaffT97} states that given two compactly supported, finite mass,
positive Borel measures $\mu_1, \mu_2$ having equal mass on $\Cx$, the following holds. Assume that each has a finite logarithmic energy:
\beq
-\Int{\Cx^2}{} \ln|z-w| \dd \mu_a(z) \dd \mu_a(w) \, < \, +\infty
\enq
Then, the logarithmic energy $\Xi$  of $\mu=\mu_1-\mu_2$ is non-negative. In our case, we may set
\beq
\dd \mu_{\pm}(\xi) \, = \, f_{\bs{x}_N}^{(\pm)}(\xi) \dd \xi  \;, \quad
f_{\bs{x}_N}^{(\pm)}(\xi) = \frac{1}{2} N^{\ga + 1} \sul{ \substack{s=1 \\ \pm x_s>0 } }{N} |x_s| \mathbbm{1}_{ \intff{ -N^{-\ga - 1} }{ N^{-\ga - 1}  }  }\big( \xi-\eta_s \big) \;.
\enq
The measures $\mu_{\pm}$ are Lebesgue continuous, positive and compactly supported. They have finite logarithmic energy owing to the Lebesgue integrability of the logarithm. Moreover, they have the same mass as $\bs{x}_{N}\in \mc{V}^{(0)}_N$.
Then, the measure $\dd \mu(\xi) \, = \, f_{\bs{x}_N}(\xi) \dd \xi \, = \, \dd \mu_+(\xi)-\dd \mu_{-}(\xi)$, satisfies the hypotheses of the mentioned lemma, in particular has zero mass since
 $\bs{x}_N \in \mc{V}^{(0)}_N$. Thus, $\Xi[ f_{\bs{x}_N} ]\leq 0$.
\end{proof}
\begin{lemme}\label{detlowerbound}

Let $\op{D} = \mathrm{diag}(\La_1, \dots, \La_N)$ be a diagonal matrix with positive eigenvalues $\La_k > 0$ for all $k \in \intn{1}{N}$ and let $\mc{V}^{(0)}_N$ be as in \eqref{definition sous espace V0}
and $\big(\mc{V}^{(0)}_N\big)^{\perp}$ be its orthogonal complement, so that $\R^N \, = \, \mc{V}^{(0)}_N  \overset{\perp}{\oplus} \big(\mc{V}^{(0)}_N\big)^{\perp} $.
The matrix $\op{D}$ admits the block decomposition with respect to this direct sum decomposition
\beq
\op{D} \, = \, \left( \ba{cc}  \op{D}_{\mid  \mc{V}^{(0)}_N}  & \op{D}_{12} \\
                                        \op{D}_{21}         &    \op{D}_{\mid \big(\mc{V}^{(0)}_N\big)^{\perp} } \ea \right) \;.
\enq
Then, it holds
$$\det_{N-1}\big[ \op{D}_{\mid  \mc{V}^{(0)}_N} \big] \, \geq \,  (\La_1 \dots \La_N)^{1 - \frac{1}{N}} \, .$$
\end{lemme}

In other words, the determinant of $\op{D}$ restricted to $\mc{V}^{(0)}_N$ may be approximately bounded from below by the determinant on the full space $\mathbb{R}^N$.

\begin{proof}

Let $\bs{w}_1, \dots, \bs{w}_{N-1}$ be an orthonormal basis for $\mc{V}^{(0)}_N$ and let $\bs{v}_k \, = \, \bs{e}_k - \bs{e}_N \in \mc{V}^{(0)}_N$ where $\bs{e}_k \in \mathbb{R}^N$
is the unit vector with one in the $k^{\e{th}}$ entry and zero in all other entries. Further, we set $\ov{\bs{e}}_N=\sul{k=1}{N} \bs{e}_k$. Finally, we denote by  $\bs{w}_k^{*}$
the linear dual to $\bs{w}_k$ with respect to the canonical scalar product of $\R^N$. There exists $\op{W} \in GL_N(\R)$ such that
\beq
\bs{v}_k \, = \, \op{W} \bs{w}_k\;, \quad k \in \intn{1}{N-1} \quad \e{and} \quad
\ov{\bs{e}}_N \, = \, \op{W} \, \ov{\bs{e}}_N \;.
\enq
We denote by $\op{P}_{\perp}$ the orthogonal projector on $\mc{V}^{(0)}_N$. One has
\beq
\op{W} \, = \, \sul{a=1}{N-1} \bs{v}_a \otimes \bs{w}_a^{*} \, + \,  \ov{\bs{e}}_N \otimes \ov{\bs{e}}_N^{\, *} \quad \e{and} \quad
\op{P}_{\perp} \, = \, \sul{a=1}{N-1} \bs{w}_a \otimes \bs{w}_a^{*}\;.
\enq
Further,
\beq
\det_{N-1}\big[ \op{D}_{\mid  \mc{V}^{(0)}_N} \big]  \, = \, \det_{N}\big[  \overline{\bs{e}}_N \otimes \overline{\bs{e}}_N^* \, + \,  \op{P}_{\perp}  \op{D}  \op{P}_{\perp} \big]
\enq
Next, one observes that
\beq
\op{P}_{\perp}  \op{D}  \op{P}_{\perp}  \, = \, \sul{a,b=1}{N-1}  \bs{w}_a \otimes \bs{w}_b^{*}  \, \big(  \bs{w}_a, \op{D} \bs{w}_b \big)
\enq
and that $ \op{P}_{\perp} \op{W} = \op{W} \op{P}_{\perp} $. Thus, denoting by $\op{W}^{\mathtt{t}}$ the transpose of $\op{W}$, one has
\begin{eqnarray*}
\det_{N-1}\big[ \op{D}_{\mid  \mc{V}^{(0)}_N} \big] \, &=& \,
\f{ \det_{N}\big[  \op{W}^{\mathtt{t}} \big( \overline{\bs{e}}_N \otimes \overline{\bs{e}}_N^* \, + \,  \op{P}_{\perp}  \op{D}  \op{P}_{\perp} \big) \op{W} \big] }
{  \det_{N}\big[  \op{W}^{\mathtt{t}} \op{W} \big]  }
\, = \, \f{ \det_{N-1}\big[ \op{P}_{\perp}  \op{W}^{\mathtt{t}}  \op{D} \op{W}   \op{P}_{\perp} \big]  }{   \det_{N}\big[  \op{W}^{\mathtt{t}} \op{W} \big]   }  \\
\, &=& \, \f{ \det_{N-1}\big[ \big(  \bs{w}_a, \op{W}^{\mathtt{t}}  \op{D} \op{W}  \bs{w}_b \big) \big]  }{   \det_{N-1}\big[ \big(  \bs{w}_a, \op{W}^{\mathtt{t}}  \op{W}  \bs{w}_b \big) \big] }\, = \,
\, = \, \f{ \det_{N-1}\big[ \big(  \bs{v}_a,  \op{D}  \bs{v}_b \big) \big]  }{   \det_{N-1}\big[ \big(  \bs{v}_a,  \bs{v}_b \big) \big] } \;.
\end{eqnarray*}
The matrix entries of the above ratio of determinants can be computed explicitly
\beqa
 \big(  \bs{v}_a,  \op{D}  \bs{v}_b \big) \, = \,  \big(  \bs{e}_a-\bs{e}_N,  \La_b \bs{e}_b - \La_N \bs{e}_N \big) \, = \, \La_{b} \de_{ab} \, + \, \La_N \\
 \big(  \bs{v}_a,   \bs{v}_b \big) \, = \, \de_{ab} \, + \, 1 \;.
\eeqa
By using that for rank 1 matrices $\op{X}$, $\det_{N-1}\big[\op{I}_{N-1} \, + \, \op{X} \big]   \, = \, 1 \, + \, \e{tr}\big[ \op{X} \big] $, one gets
\beqa
\det_{N-1}\big[ \big(  \bs{v}_a,  \bs{v}_b \big) \big] & = & 1+N-1 \, = \, N  \\
\det_{N-1}\big[ \big(  \bs{v}_a,  \op{D}  \bs{v}_b \big) \big] & = & \pl{k=1}{N-1} \La_k \cdot\Big( 1+ \La_N \sul{k=1}{N-1}\La_k^{-1} \Big) \, = \, \det_N[ \op{D} ] \sul{k=1}{N} \f{1}{\La_k} \;.
\eeqa
The AM-GM inequality yields
\beq
\f{1}{N} \sul{k=1}{N} \f{1}{\La_k}  \, \geq \,  \Big(  \pl{k=1}{N} \f{1}{\La_k} \Big)^{ \f{1}{N} }
\enq
hence leading to the claim.
\end{proof}
We also state the well known lemma:
\begin{lemme}
\label{Lemma lower bound on determinant of sum of positive definite}
Let $\op{S}_1,\op{S}_2 \geq 0$ be two real valued, positive semi-definite $N\times N$ matrices such that $\op{S}_1 \, = \, \op{S}_1^{\, \mathtt{t}}$ is symmetric.
Then
\beq
\det_N\big[ \op{S}_1 \, + \,  \op{S}_2 \big] \, \geq \, \det_N[ \op{S}_1] \, .
\enq
\end{lemme}

\vspace{3mm}

\begin{proof}[Proof of Proposition \ref{detlowerbound2}]

Denote by $\op{X}$ the matrix arising in the \textit{rhs} of  Lemma \ref{Lemme reecriture determinant A comme perturbation de diagonale}:
\begin{align*}
X_{jk} \, = \, \bigg( 1  - \frac{ \ln|P^{\prime}(\eta_j)| }{ \ln(2\veps_N) }     \bigg) \delta_{kj} \,  + \,
\frac{ N }{ \ln(2\veps_N) }   M_{kj}   \, + \, B_{jk}
\end{align*}
with $\op{B}_{jk} \, = \, \mathrm{O}\Big( \big( 1+ N \de_{jN} \mathbbm{1}_{k \leq N-1}   \big)\vartheta_N \Big)$. Then,
\beq
\norm{ \hspace{-2mm} \op{ B} }_{\e{HS}}^2 \, \leq \, C^{\prime} \sul{j,k=1}{N} \vartheta_N^2   \big( 1+ N^2 \de_{jN} \mathbbm{1}_{k \leq N-1}   \big)
\, \leq  \,  \wt{C}^2  N^3 \vartheta_N^2 \;.
\label{norme matrice B HS}
\enq
for some $\wt{C}>0$. $\op{X}$ may be recast  as $\op{X}\, = \, \op{D} \, + \, \wt{\op{M}} \, + \,  \op{G}$, where
\begin{align*}
D_{jk} & =  \bigg( 1- \frac{  \ln|P^{\prime}(\eta_j)| }{\ln(2\veps_N)}   \, - \, C\f{\ln N }{ N} \, - \, \wt{C} N^{\f{3}{2}} \vth_N \bigg) \delta_{kj} \\
\wt{M}_{jk} & = \frac{ N }{ \ln(2\veps_N) }   M_{kj} \, + \, C\f{\ln N }{ N} \de_{jk}
\end{align*}
Here, $C$ is the constant arising in Lemma \ref{almostposdef} while $\wt{C}$ is the constant arising in an upper bound on the Hilbert-Schmidt norm, \textit{c.f.} \eqref{norme matrice B HS}, of the remainder
present in the expression of $\op{X}$. Finally, let
\beq
G_{jk} \, = \, B_{jk} \, + \,  \wt{C}   N^{\f{3}{2}} \vartheta_N  \de_{jk}
\enq
so that $G_{jk}$ is positive semi-definite by \eqref{norme matrice B HS}. Further, we set $\op{X}_{\e{red}}\,=\, \op{D}\, +\, \wt{\op{M}}$. First, observe that
\begin{equation*}
\op{X}_{\e{red}} \cdot \ov{\bs{e}}_N \, = \,  \Big( 1 \, - \, \wt{C}  N^{\f{3}{2}} \vartheta_N  \Big)\cdot \ov{\bs{e}}_N \qquad \e{with} \qquad  \ov{\bs{e}}_N \, = \, \sul{k=1}{N} \bs{e}_k
\end{equation*}
and $\bs{e}_k$ being the vector with $1$ in $k^{\e{th}}$ position and $0$ elsewhere.
Now, upon adopting the notations introduced in Lemma \ref{detlowerbound} and recalling the definition \eqref{definition sous espace V0} of $\mc{V}_N^{(0)}$,
one has that $\op{X}_{\e{red}}$ has the block form decomposition
\begin{equation*}
\op{X}_{\e{red}}  \, = \, \left( \ba{cc} \big( \op{D}\, +\, \wt{\op{M}} \big)_{ \mid \mc{V}_N^{(0)} }  & 0 \\
                                                                        0                       &   1 \, - \, \wt{C}  N^{\f{3}{2}} \vartheta_N    \ea \right) \;.
\end{equation*}

By virtue of Lemma \ref{almostposdef}, $ \wt{\op{M}}_{ \mid \mc{V}_N^{(0)} }$ is positive definite.
Next, one has that $\op{D}_{ \mid \mc{V}_N^{(0)} }\, =\, \op{P}_{\perp} \op{D} \op{P}_{\perp} $ with $\op{P}_{\perp} $ the orthogonal projector on $\mc{V}_N^{(0)}$.
Observe that $\op{D}$ is positive definite for $N$ large enough by hypothesis. Then,
for any $\bs{x}_N \in \R^N$, $\big( \bs{x}_N , \op{D}_{ \mid \mc{V}_N^{(0)} } \bs{x}_N  \big) \, = \, \big(  \op{P}_{\perp}\bs{x}_N , \op{D} \op{P}_{\perp}\bs{x}_N  \big) \geq 0$.
This entails that $\big( \op{D}\, +\, \wt{\op{M}} \big)_{ \mid \mc{V}_N^{(0)} } $ is positive definite. It is also symmetric.
Hence, by Lemma \ref{Lemma lower bound on determinant of sum of positive definite}
\begin{align*}
\det_N[\op{X} ] \, &\geq \, \det_N\big[\op{X}_{\e{red}} \big] \, = \,  \Big( 1 \, - \, \wt{C}  N^{\f{3}{2}} \vartheta_N  \Big) \cdot
\det_{N-1}\big[  \big( \op{D}\, +\, \wt{\op{M}} \big)_{ \mid \mc{V}_N^{(0)} }  \big] \\
\, &\geq \Big( 1 \, - \, \wt{C}  N^{\f{3}{2}} \vartheta_N  \Big) \,\det_{N-1}\big[  \op{D}_{ \mid \mc{V}_N^{(0)} }  \big]  \, \geq \,  \Big( 1 \, - \, \wt{C}  N^{\f{3}{2}} \vartheta_N  \Big) \cdot \big\{ \det_{N }[   \op{D} ] \big\}^{ \f{N-1}{N} }
\end{align*}
where we have used, again, Lemma \ref{Lemma lower bound on determinant of sum of positive definite}, and finally invoked Lemma \ref{detlowerbound}.
This entails the claim.

\end{proof}

\subsection{Proof of the lower bound on balls centered at good measures}

In this section, we will prove the lower bound on balls centred at good probability measures that satisfy the following hypotheses. 
\begin{defin}\label{defgood} For $\kappa>0$, we let $\mathcal G_{\kappa}$ be the set of probability measures $\Ups$ on $\mathbb R$ which are absolutely continuous with respect to 
 Lebesgue measure with density $\dd \Ups(x) \, = \, \varrho(x) \dd x$ such that
\begin{enumerate}
\item $\Ups$ is compactly supported with bounded density, $\varrho \in L^{\infty}(\R)$;
\item $\Int{\mathbb{R}}{} \ln(1+|x|)^4 \, \varrho(x)\, \mathrm{d}x  < +\infty$;
\item $1+ \frac{2}{\ell} \Int{\mathbb{R}}{} \ln |x-y| \, \varrho(y) \, \mathrm{d}y \geq \kappa$ for all $x\in \mathbb R$. 
\end{enumerate}
\end{defin}

\subsubsection{Properties of good probability measures}

\begin{lemme}\label{concentration} Let  $\kappa>0$ and $\Ups\in \mathcal G_{\kappa}$. Then, for $N \geq 2$,
\begin{enumerate}
\item For all $\delta > 0$
$$\Ups^{\otimes N} \Big[ \big\{\bs{x}_N \in \R^N \, : \,  |x_k - x_j| \leq N^{-1 - \delta} \quad \text{ for some } \quad k \neq j \big\} \Big] \, \leq \,  \norm{\varrho}_{ L^{\infty}(\R) } N^{1-\delta} \, .$$
\item There is a constant $C^\prime > 0$ such that
$$\Ups^{\otimes N}\Big[ \Big\{ \bs{x}_N \in \R^N \, : \, 1+ \frac{2}{\ell N} \sum_{\substack{j=1 \\ j \neq k} }^N \ln|x_k - x_j| \, \leq \frac{\kappa}{2}
            \quad \text{ for some } \quad  k \in \intn{1}{ N } \Big\} \Big] \,  \leq \, C^\prime N^{-1}  \, .$$
\item Let  $\delta \geq 2$ and set $\mc{F} = \mc{F}_1 \cap \mc{F}_2 $ where
\begin{align*}
\mc{F}_1 &= \big\{  \bs{x}_N \in \mathbb{R}^N \, : \, |x_k - x_j | \geq N^{-1-\delta} \quad  \forall k \neq j \in  \intn{1}{N} \big\}  \vspace{4mm} \\
\mc{F}_2 &= \Big\{ \bs{x}_N \in \mathbb{R}^N \, : \, 1 \, - \,  \sul{\substack{j=1 \\ j \neq k}}{N} \frac{ \ln|x_k - x_j|  }{ \ln(2\veps_N) }  \, \geq  \, \frac{\kappa}{2} \quad \forall k \in \intn{1}{N} \Big\} \, .
\end{align*}
Then
\begin{align*}
&\Int{\mathbb{R}^N}{} \mathbbm{1}_{ \mc{F} }(\bs{x}_N) \, \ln\bigg( 1 - \sul{\substack{j=1 \\ k \neq j}}{N}  \frac{  \ln|x_k - x_j| }{ \ln(2\veps_N) } \bigg) \,\dd^{N}\Ups(\bs{x}_N) \\
&= \Int{\mathbb{R}}{} \ln\left( 1 - \frac{N-1}{\ln(2\veps_N)} \Int{\mathbb{R}}{} \ln|x - y| \, \dd \Ups(y) \right) \, \dd \Ups(x) + \mathrm{O}(N^{-\frac{1}{2}}) \, .
\end{align*}
\end{enumerate}
\end{lemme}
\begin{proof}
The first claim follows from a direct union bound. Let us then proceed to establishing the second claim. For any $k \not= j  \in \intn{1}{N}$,
we introduce
\beq
Y_{kj} \, = \,  \ln|x_k - x_j| - \Int{\mathbb{R}}{} \ln|x_k - y| \varrho(y) \, \mathrm{d}y \, .
\enq
We first establish that for  any $\eta > 0$ there exists a $C_\eta > 0$ such that
$$\Ups^{\otimes N}\Big[  \Big|  \frac{1}{N} \sul{\substack{j=1 \\ j \neq k} }{N} Y_{kj} \Big| \geq \eta \Big] \leq C_\eta N^{-2} \, .$$
Indeed, by Markov's inequality
$$\Ups^{\otimes N}\Big[  \Big|  \frac{1}{N} \sul{\substack{j=1 \\ j \neq k} }{N} Y_{kj} \Big| \geq \eta \Big] \leq
\frac{1}{(\eta N)^4} \mathbb{E}_{\Ups^{\otimes N}} \Big[ \Big(\sul{\substack{j=1 \\ j \neq k} }{N} Y_{kj} \Big)^{4}\Big] \, .$$
Note that by independence and centering, 
$
\mathbb{E}_{\Ups^{\otimes N}} \Big[ Y_{kj_1} \dots  Y_{kj_{\ell}} \Big] \, =\, 0 $ if there is  $j_b\not \in \{ j_a \}_{a\not=}^{\ell} $.
Thus, expanding the product and using Cauchy-Schwartz  yields
$$
 \mathbb{E}_{\Ups^{\otimes N}} \Big[ \Big(\sul{\substack{j=1 \\ j \neq k} }{N} Y_{kj} \Big)^{4}\Big]
%
%
%
\leq \, \big( 3 (N-1)(N-2)+(N-1) \big)\mathbb{E}_{\Ups^{\otimes 2}} \big[ Y_{12}^4 \big] \;.
$$
By using $|a+b|^{k}\leq 2^{k-1}(|a|^k+|b|^k)$, one has that
\begin{align*}
Y_{12}^4 &\leq 2^3 \Big[ \big( \ln|x_1-x_2| \big)^4 + \Big( \Int{\mathbb{R}}{} \ln|x_1 - y| \, \varrho(y) \, \mathrm{d}y\Big)^4 \Big] \; .
\end{align*}
This ensures that
\beq
\mathbb{E}_{\Ups^{\otimes 2}}[Y_{12}^4] \,\leq \, 2^4 \Int{\mathbb{R}^2}{} \big( \ln|x_1-x_2| \big)^4  \, \mathrm{d}\Ups(x_1) \, \mathrm{d}\Ups(x_2)  \, .
\enq
We finally claim that the right-hand side is finite. Indeed,  because of (ii) in Definition \ref{defgood},
\begin{align*}
\Int{ \mathbb{R}^2 }{} \left| \ln|x-y| \right|^k \varrho(x) \varrho(y) \, \mathrm{d}x \, \mathrm{d}y &=
                \Int{\mathbb{R}^2}{} \mathbbm{1}_{|x-y| \leq 1}\left| \ln|x-y| \right|^k \varrho(x) \varrho(y) \, \mathrm{d}x \, \mathrm{d}y \\
    &\quad + \Int{\mathbb{R}^2}{} \mathbbm{1}_{|x-y| > 1}\left| \ln|x-y| \right|^k \varrho(x) \varrho(y) \, \mathrm{d}x \, \mathrm{d}y \\
    &\leq 2 k! \norm{\varrho}_{L^{\infty}(\R) }  + 2^k\Int{\mathbb{R}}{} \ln(1+|x|)^k \varrho(x) \, \mathrm{d}x < +\infty \, .
\end{align*}
Now, one observes that
\bem
\Ups^{\otimes N}\Big[ \Big\{ \bs{x}_N \in \R^N \, : \, \exists k\in \intn{1}{N} \; \e{such} \, \e{that} \;  1+ \frac{2}{\ell N} \sul{\substack{j=1 \\ j \neq k} }{N} \ln|x_k - x_j| \, \leq \frac{\kappa}{2}  \Big\} \Big]  \\
\leq \, \sul{k=1}{N}
\Ups^{\otimes N}\Big[ \Big\{ \bs{x}_N \in \R^N \, : \,  1+ \frac{2}{\ell N} \sum_{\substack{j=1 \\ j \neq k} }^N \ln|x_k - x_j| \, \leq \frac{\kappa}{2}  \Big\} \Big]  \, .
\end{multline}
Setting $q(x) \, = \, 1+\tfrac{2}{\ell} \Int{\R}{} \ln|x-y| \dd \Ups(y)$, we get by (iii)
\beq
1+ \frac{2}{\ell N} \sul{\substack{j=1 \\ j \neq k} }{N} \ln|x_k - x_j| \, = \, q(x_k) \, + \, \f{2}{\ell N}\sul{\substack{j=1 \\ j \neq k} }{N} Y_{kj}
\, \geq  \, \kappa \,  +\, \f{2}{\ell N}\sul{\substack{j=1 \\ j \neq k} }{N} Y_{kj} \,.
\enq
  Thus,
\bem
\Big\{ \bs{x}_N \in \R^N \, : \,  1+ \frac{2}{\ell N} \sul{\substack{j=1 \\ j \neq k} }{N} \ln|x_k - x_j| \, \leq \frac{\kappa}{2} \; \;  \dd \Ups(x_k) \; \e{a.e.} \Big\} \\
\,  \subset  \, \Big\{ \bs{x}_N \in \R^N \, : \,  \f{1}{N} \Big| \sul{\substack{j=1 \\ j \neq k} }{N} Y_{kj} \Big| \, \geq \frac{\kappa \ell }{4 }  \; \;  \dd \Ups(x_k) \; \e{a.e.} \Big\}  \;.
\end{multline}
The above thus establishes the second property owing to the previous estimates.
We finally prove the third point. 
Firstly, we use the basic inequality, for $x,y > 0$, $|\ln x - \ln y | \leq \frac{|x-y|}{\min\{ x, y \}}$. From this and by taking $N$ large enough, we have
\bem
\bigg| \Int{\mathbb{R}^N}{} \mathbbm{1}_{ \mc{F} }(\bs{x}_N)  \, \Big\{  \ln\Big( 1 -  \sum_{\substack{j=1 \\ k \neq j}}^N \frac{ \ln|x_k - x_j| }{\ln(2\veps_N)} \Big)
- \ln\Big( 1 - \frac{N-1}{\ln(2\veps_N)} \Int{\mathbb{R}}{} \ln|x_k - y|  \mathrm{d}\Ups(y) \Big) \Big\}   \,\dd^{N}\Ups(\bs{x}_N) \bigg| \\
\leq  \frac{ \tf{2}{\kappa} }{|\ln(2\veps_N)|} \Int{\mathbb{R}^N}{}    \bigg|  \sul{\substack{j=1 \\ k \neq j}}{N} \Big\{Y_{kj}
\Big\} \bigg|
 \,\dd^{N}\Ups(\bs{x}_N)  
    = \mathrm{O}(N^{-\frac{1}{2}})
\end{multline}
by the same manipulations that appear  to prove the previous point. To prove the claim, it remains to bound
\begin{align*}
\Xi^{\e{c}}  = & \Int{\mathbb{R}^N}{}  \mathbbm{1}_{ \mc{F}^{\e{c}} }(\bs{x}_N) \Big| \ln\Big( 1-\frac{N-1}{\ln(2\veps_N)} \Int{\mathbb{R}}{} \ln|x_k - y| \mathrm{d}\Ups(y) \Big)   \Big| \,\dd^{N}\Ups(\bs{x}_N)  \\
 \leq &\sqrt{\Ups^{\otimes N}[ \mc{F}^{\e{c}}] } \sqrt{\Int{\mathbb{R}}{} \Big| \ln\Big( 1-\frac{N-1}{\ln(2\veps_N)} \Int{\mathbb{R}}{} \ln|x  - y| \mathrm{d}\Ups(y) \Big)   \Big|^2 \, \mathrm{d}\Ups(x) }  \, .
\end{align*}
One observes that $\mathrm{d}\Ups(x)$ a.e.,
\bem
1-\frac{N-1}{\ln(2\veps_N)} \Int{\mathbb{R}}{} \ln|x  - y| \mathrm{d}\Ups(y)   \\
\,= \,  \f{N-1}{ N - \tf{ 2 \ln 2 }{\ell} } \Big( 1 + \f{2}{\ell} \Int{\mathbb{R}}{} \ln|x  - y| \mathrm{d}\Ups(y) \Big)
\, + \, 1 \, - \, \f{N-1}{ N - \tf{ 2 \ln 2 }{\ell} } \geq \f{\kappa}{2}
\end{multline}
provided that $N$ is large enough. Hence, this leads to
\beq
\Xi^{\e{c}} \,  \leq   \,\sqrt{\Ups^{\otimes N}[ \mc{F}^{\e{c}}] }  \sqrt{ 2  | \ln \f{\kappa}{2}|^2 \, + \, \f{32}{\ell^3}   \Int{\mathbb{R}}{} (\ln|x  - y|)^2 \mathrm{d}^2\Ups(x,y) }
\enq
By a union bound $\Ups^{\otimes N}[ \mc{F}^{\e{c}} ]  \, \leq \, \Ups^{\otimes N}[\mc{F}_1^{\e{c}}] \, + \, \Ups^{\otimes N}[ \mc{F}_2^{\e{c}}]\, = \,  \mathrm{O}(N^{-1})$
by virtue of the estimates from $\e{i)}$ and $\e{ii)}$  since $\delta \geq 2$.
\end{proof}

We now introduce the events $\mc{E}_{1}, \mc{E}_{2},\mc{E}= \mc{E}_1 \cap \mc{E}_2$ similar to $\mc{F}_{1}, \mc{F}_{2}, \mc{F}$ of the previous lemma  but restricted to ordered particles: 
\begin{align*}
\mc{E}_1 & = \big\{  \bs{\eta}_N \in \R^N_{<} \, : \,  |\eta_k - \eta_j | \geq N^{-3} \quad  \forall k \neq j \in \intn{1}{N}   \big\} \\
\mc{E}_2 & = \Big\{  \bs{\eta}_N \in \R^N_{<} \, : \, \, 1 - \frac{1}{\ln(2\veps_N)} \sul{\substack{j=1 \\ j \neq k}}{N} \ln|\eta_k - \eta_j| \geq \frac{\kappa}{4} \quad \forall k \in \intn{1}{N}  \Big\} \; .
\end{align*}
Our first claim is that the event $\mc{E}$ is contained within the integration domain $\mc{A}_N$ introduced in \eqref{definition ensembles AN}.
\begin{lemme}
\label{Lemme map etaN to La+N et ses ptes}
For $N$ sufficiently large (depending on $\kappa > 0$) there exists $c>0$ such that, for all $\bs{\eta}_N\in \mc{E}$ we have
\beq
\sg_{N;a}\, = \, \f{ 2 \veps_N }{  |P^{\prime}(\eta_a) | } \,  \leq  \, \ex{-c N} \qquad \e{and} \qquad \f{\sg_{N;a}}{ |\eta_k-\eta_a|} \, \leq \,  \ex{-c N}
\enq
this uniformly in $a , k\in \intn{1}{N}$, $a \neq k$.
Given $P(X)\, = \, \pl{a=1}{N}(X-\eta_a)$, the polynomials $P_{+} \, = \, P \, - \, 2 \veps_N$
admits $N$ real roots $[\bs{\La}_N^{+}(\bs{\eta}_N)]_1\, < \, \cdots \, < \, [\bs{\La}_N^{+}(\bs{\eta}_N)]_N $. The map
\beq
\bs{\La}_N^{+} \; : \; \Big\{ \ba{ccc }
 \mc{E} \tend \R^N \\
\bs{\eta}_N &\mapsto &  \La_N^{+}(\bs{\eta}_N) \in \R^N  \ea
\enq
whose coordinates are such that
\beq
\pl{a=1}{N}(X-\eta_a)\, - \, 2\veps_N \, =\,  \pl{a=1}{N}\big( X \, - \, [\bs{\La}_N^{+}(\bs{\eta}_N)]_a \big)
\enq
is well defined, satisfies $  [\bs{\La}_N^{+}(\bs{\eta}_N)]_1 \, < \, \cdots \, < \,  [\bs{\La}_N^{+}(\bs{\eta}_N)]_N   $ for any $\bs{\eta}_N\in \mc{E}$.
Furthermore, we have
\beq
\bs{\La}_N^{+}(\mc{E}) \subset \mc{A}_N\, ,
\enq
$\bs{\La}_N^{+}$ is a smooth diffeomorphism from $\mc{E}$ onto $\bs{\La}_N^{+}(\mc{E})$.
\end{lemme}

\begin{proof}

We establish the results by showing that any $ \bs{\eta}_N \in  \mc{E}$ is such that the control parameters $\varkappa_N$ and $\vth_N$ introduced in \eqref{ecriture hypotheses sur les quantites petites}
may be taken exponentially small.
One has
\beq
2 \sg_{N;a}\, = \, \f{ 2 \veps_N }{  |P^{\prime}(\eta_a) | } \, = \, \ex{\ln (2\veps_N)  - \sul{ j\not=a }{N} \ln |\eta_{ja}| } \, \leq \, \ex{ \f{\kappa}{2} \ln (2\veps_N)  }
\, \leq \, 2^{\f{\kappa}{2}} \ex{-\f{N\ell}{4}\kappa }
\enq
Likewise,
\beq
\f{\sg_{N;a}}{ |\eta_k-\eta_a|} \, \leq \, N^{3} 2^{\f{\kappa}{2}-1} \ex{-\f{N\ell}{4}\kappa } \;.
\enq
Thus, one may take $\varkappa_N\, = \, \ex{-cN}$ and $\vth_N \, = \, \ex{-c^{\prime}N}$.
The rest of the claim follows directly from Propositions \ref{openness}-\ref{diffeo}.

\end{proof}

We now introduce several auxiliary sets that will be of use in establishing the lower bound.
\beqa
\mc{L}^+_{\mu, \de} & = & \Big\{ \bs{\la}^+_N \in \mc{A}_N \; : \; \op{L}^{(\bs{\la}_N^+)}_N \in B(\mu, \de) \Big\}  \\
\mc{N}_{\mu, \de} & = & \Big\{ \bs{\eta}_N\in \big(\bs{\La}^+_N\big)^{-1}\big(\mc{A}_N \big)
\; : \; \op{L}^{(\bs{\eta}_N)}_N \in B(\mu, \de) \Big\} \, .
\eeqa

\begin{lemme}\label{lemmainc}

Let $\mu \in \mc{M}_{1}(\R)$. For every $\de>0$, there exists $N_0$ such that $N\geq N_0$ implies that
$$ \bs{\La}_N^+\big( \mc{N}_{\mu, \f{\de}{2} } \big) \cap \bs{\La}_N^+\big(\mc{E}\big) \subset \mc{L}^+_{\mu, \de} \; .$$
\end{lemme}
\begin{proof}

If $ \bs{\la}_N^+ \in \bs{\La}_N^+\big( \mc{N}_{\mu, \f{\de}{2} } \big) \cap \bs{\La}_N^+\big(\mc{E}\big) $, then, by Lemma \ref{Lemme map etaN to La+N et ses ptes}  there exists
a unique $\bs{\eta}_N \in \mc{E}\cap \mc{N}_{\mu, \f{\de}{2} } $ such that $\bs{\la}_N^+ \, = \,\bs{\La}_N^+\big( \bs{\eta}_N \big)$.
Furthermore, one has that
\beq
\la_a^+ - \eta_a \, = \, \e{O}\big( \ex{-cN} \big) \, ,  \qquad \e{so} \; \e{that} \quad \e{d}_{\e{BL}}\big( \op{L}^{(\bs{\la}_N^+)}_N  , \op{L}^{(\bs{\eta}_N)}_N \big) \,= \, \e{O}(\ex{-c N} )\leq \tf{\de}{2} \;,
\enq
provided that $N$ is large enough. Since, by construction,  $\e{d}_{\e{BL}}\big( \mu  , \op{L}^{(\bs{\eta}_N)}_N \big)  \leq \tf{\de}{2}$, the claim follows.

\end{proof}

\subsubsection{Weak large deviation lower bound}

We are now in a position to establish the weak large deviation lower bound for $\mu\in \mathcal{G}_{\kappa}$, $\kappa>0$, \textit{c.f.}
Definition \ref{defgood}.
We concentrate on the constrained model and therefore assume that
$\int  x \,  \dd\mu(x)=0$, since in this case the proof is more complex because some additional arguments are needed to deal with the constraint.
Dealing with the unconstrained case afterwards is direct.
Finally we assume that $\mu$ is compactly supported. We then prove in this section that for such measures, \eqref{wldlbtheo} holds, namely:
\begin{lemme}\label{ldlbgood}
Let $\mu\in \mathcal{G}_{\kappa}$ for some $\kappa > 0$, and also $\int x \, \dd\mu(x)=0$. Then
\begin{align}\label{ldlblemma}
\lim_{\delta \searrow   0}  \liminf_{N \to +\infty} \frac{1}{N}\ln \overline{\Pi}_{N,\mathsf{c}}[B_{\mathsf{c}}(\mu,\delta)] \geq - I[\mu] \;,
\end{align}
with $B_{\mathsf{c}}(\mu,\delta)$ as given in the statement of Proposition \ref{lowerboundonballs}.

\end{lemme}

We start the proof of this lemma by simplifying our previous lower bounds. First,  by Lemma \ref{lemmainc}, we have
for $N$ large enough
\beq
\bs{\La}_N^+\big( \mc{N}_{ \mu, \frac{\delta}{2} } \big) \cap \bs{\La}_N^+\big(\mc{E}\big)  \subset   \mc{L}^+_{\mu, \de} \;.
\enq
Next, let us change variables in \eqref{density}, this after reducing the integration domain to
$\bs{\La}_N^+\big( \mc{N}_{\mu, \f{\de}{2} }  \cap  \mc{E}\big) $.  Proposition \ref{jacobian} then yields
\bem
\overline{\Pi}_{N,\mathsf{c}}\big[ B_{\mathsf{c}}(\mu,\delta) \big]
\geq \frac{ (N-1)!}{  | \ln(2\veps_N)|^{N-1}} \hspace{-4mm} \Int{   \mc{E} \cap \mc{N}_{ \mu, \frac{\delta}{2} }    }{} \hspace{-4mm}
 \Delta\big(  \bs{\eta}_N  \big)\,   \mc{I}\big( \bs{\La}^+_N\big( \bs{\eta}_N \big); \veps_N \big)
 \pl{k=1}{N} \Big\{ \mathrm{e}^{-  V ( [\bs{\La}^+_N( \bs{\eta}_N )]_k  )  }  \Big\} \\
\times  \delta\big(  \ov{\bs{\eta}}_N \big) \, \mathrm{d} \bs{\eta}_N \; .
\end{multline}
Above, $ \ov{\bs{\eta}}_N = \sul{a=1}{N} \eta_a$ and we have used the fact that since the polynomials $P$ and $P^+$ differ by only a constant,
we must have $\ov{\bs{\eta}}_N =  \ov{\bs{\la}}_N^+ $. Moving forward, we invoke the lower bound of Proposition \ref{exactlowerbound} which allows for the replacement
\beq
 \mc{I}\big( \bs{\La}^+_N\big( \bs{\eta}_N \big); \veps_N \big)  \, \geq \, \f{ N }{  \De ( \bs{\eta}_N ) }  \det_{N-1}\big[ \op{A}(\bs{\eta}_N) \big]
\enq
with the $(N-1)\times (N-1)$ matrix $\op{A}$ as introduced in \eqref{definition matrice A}.
Above, we have stressed its dependence on $\bs{\eta}_N$.
Then, since $\bs{\eta}_N \in \mc{E}$, by invoking the lower bound on $\det_{N-1}\big[ \op{A}(\bs{\eta}_N) \big]$ obtained in
Corollary \ref{detlowerbound2}, we get for some $C,c>0$
\bem
 \overline{\Pi}_{N,\mathsf{c}}\Big[ B_{\mathsf{c}}(\mu,\delta) \Big]  \geq  N!  \, \Big( 1 - \e{O}(\ex{-cN}) \Big) \, \hspace{-4mm} \Int{   \mc{E} \cap \mc{N}_{ \mu, \frac{\delta}{2} }    }{} \hspace{-4mm}
 \pl{j=1}{N}  \bigg( 1- \frac{   \ln|P^\prime(\eta_j)|  }{ \ln(2\veps_N) } - C \frac{\ln N}{N} \bigg)^{1- \frac{1}{N}}
\\
\times  \pl{k=1}{N} \Big\{ \mathrm{e}^{-  V ( [\bs{\La}^+_N( \bs{\eta}_N )]_k  )  }  \Big\}  \delta\big( \ov{\bs{\eta}}_N\big)
\cdot \mathrm{d}\bs{\eta}_N
\end{multline}
We now symmetrise the integral. For that, we observe that given $\bs{\eta}_N\in \R^N$ with pairwise distinct coordinates, there exists a unique $\sg \in \mf{S}_N$,
such that $\eta_{\sg(1)}<\cdots < \eta_{\sg(N)}$. We then set $\bs{\eta}_N^{\sg}\,= \, \big( \eta_{\sg(1)}, \dots , \eta_{\sg(N)} \big)$
which allows us to define
\beq
\big[ \La_N^{+}(\bs{\eta}_N)  \big]_{k} \, = \, \big[ \La_N^{+}(\bs{\eta}_N^{\sg} )  \big]_{ \sg^{-1}(k) }
\label{definition transfo pour coord permutes}
\enq
Further, we set    $\wt{\mc{E}}  = \wt{\mc{E}}_1 \cap \wt{\mc{E}}_2$ where
\begin{align}
\wt{\mc{E}}_1 & = \{  \bs{\eta}_N \in \R^N \, : \,  |\eta_k - \eta_j | \geq N^{-3} \quad  \forall k \neq j \in \intn{1}{N}  \}  \label{definition E tile 1} \\
\wt{\mc{E}}_2 & = \Big\{  \bs{\eta}_N \in \R^N \, : \, \, 1 - \frac{1}{\ln(2\veps_N)} \sul{\substack{j=1 \\ j \neq k}}{N} \ln|\eta_k - \eta_j| \geq \frac{\kappa}{4} \quad \forall k \in \intn{1}{N}   \Big\}
\label{definition E tile 2}
\end{align}
and adopt the convention
\beq
\wt{\mc{N}}_{\mu, \de} \,  = \,  \Big\{ \bs{\eta}_N\in \R^N  \; : \; \op{L}^{(\bs{\eta}_N)}_N \in B(\mu, \de) \Big\} \;.
\label{definition N tile mu delta}
\enq
With these notations at hand, we arrive at
\bem
 \overline{\Pi}_{N,\mathsf{c}}\big[ B_{\mathsf{c}}(\mu,\delta) \big]  \geq    \Big( 1 - \e{O}(\ex{-cN}) \Big) \,    \Int{  \R^{N-1}   }{}
\mathbbm{1}_{  \wt{\mc{E}} \cap \wt{\mc{N}}_{ \mu, \frac{\delta}{2} }  }(\bs{\eta}_N)
   \\
\times \pl{j=1}{N} \bigg\{ \mathrm{e}^{-  V ( [\La^+_N( \bs{\eta}_N )]_j  )  }  \Big( 1-  \sul{ \substack{ s=1 \\ s\not= j} }{N}\frac{   \ln|\eta_j-\eta_s|  }{\ln(2\veps_N)}    - C \frac{\ln N}{N} \Big)^{1- \frac{1}{N}}
 \bigg\}_{ \mid \eta_N=-\sul{s=1}{N-1} \eta_s }  \mathrm{d} \bs{\eta}_{N-1} \;.
\end{multline}
Note that all the functions building up the integrand are well-defined by \eqref{definition transfo pour coord permutes}
and Lemma \ref{Lemme map etaN to La+N et ses ptes}. Also, above,
we agree upon $\bs{\eta}_N \, = \, \big( \bs{\eta}_{N-1}, -\sul{s=1}{N-1} \eta_s  \big)$ in view of the constraint.
Let $ \rho$ be the density of $\mu$.  We introduce an extra variable to deal with this constraint   by writing
$$1 = \frac{1}{N}\Int{\mathbb{R}}{} \varrho \Big(\frac{1}{N}\xi - \sul{j=1}{N-1} \eta_j\Big) \, \mathrm{d}\xi \, .$$
Then make the change of variables $(\bs{\eta}_{N-1},\xi) \hookrightarrow \bs{x}_N$ with $\bs{x}_N = X_N(\bs{\eta}_{N-1},\xi)$ where
\beq
X_N(\bs{\eta}_{N-1},\xi) \, = \, \Big( \eta_1+\f{\xi}{N}, \dots, \eta_{N-1} + \f{\xi}{N}, -\sul{s=1}{N-1} \eta_s +\f{\xi}{N} \Big) \;.
\enq
It is direct to check that $X_N:\R^N \tend \R^N$ is a smooth diffeomorphism having unit Jacobian, \textit{viz}. $\det_{N}\big[ \op{D}_{ (\bs{\eta}_{N-1},\xi)  }X_N \Big]=1$,
and such that
\beq
X_N^{-1}(\bs{x}_{N}) \, = \, \Big( x_1- \f{ \ov{\bs{x}}_N }{N}, \dots, x_{N-1}- \f{ \ov{\bs{x}}_N}{N}, \ov{\bs{x}}_N \Big)  \qquad \e{with} \qquad \ov{\bs{x}}_N
\, = \, \sul{s=1}{N} x_s\; .
\enq
We set
\beq
\mc{W}_{\de} \, = \, X_N \Big( \big\{ \bs{\eta}_{N-1} \in \R^{N-1} \, : \, \big(  \bs{\eta}_{N-1} , - \sul{s=1}{N-1} \eta_s \big) \in \wt{\mc{E}} \cap \wt{\mc{N}}_{ \mu, \frac{\delta}{2} }  \big\}  \times \R  \Big)
\enq
which leads to the lower bound
\bem
 \overline{\Pi}_{N,\mathsf{c}}\big[ B_{\mathsf{c}}(\mu,\delta) \big]  \geq    \f{ 1 - \e{O}(\ex{-cN})  }{ N }  \,    \Int{  \mc{W}_{\de}  }{}   \varrho \big( x_N \big)
   \\
\times \pl{j=1}{N} \bigg\{ \mathrm{e}^{-  V ( [\bs{\La}^+_N( \wh{\bs{x}}_N )]_j  )  }  \Big( 1-  \sul{ \substack{ s=1 \\ s\not= j} }{N}\frac{   \ln|x_j-x_s|  }{\ln(2\veps_N)}    - C \frac{\ln N}{N} \Big)^{1- \frac{1}{N}}
 \bigg\}    \cdot \mathrm{d} \bs{x}_N \;.
\end{multline}
Here, we have employed the shorthand notation $\boldsymbol{\eta}_{N}(\boldsymbol{x}_N) \equiv \wh{\bs{x}}_N \, = \, \bs{x}_N \, - \, \f{ \ov{\bs{x}}_N  }{ N }\ov{\bs{e}}_N$ with $\ov{\bs{e}}_N=\sul{k=1}{N} \bs{e}_k$
and $\bs{e}_k$ being the vector having unity in the $k^{\e{th}}$ coordinate. We further restrict the integration to the domain
\beq
 \mc{W}_{\de}  \cap \mc{Q}_N  \qquad \e{with} \qquad \mc{Q}_N \, = \, \big(  \e{supp}\big[ \mu \big]  \big)^N
\enq
and tilt the measure by writing
$$1 =  \varrho(x_1)\dots  \varrho(x_{N-1})  \cdot \pl{k=1}{N-1}\Big\{ \mathrm{e}^{- \ln  \varrho (x_k)}  \Big\}\, .$$
We now reduce again the integration domain to a more suitable one. First of all, we introduce the subset
\beq
\mc{H} \, =\,  \Big\{  \Big| \frac{  \ov{ \bs{x}}_N }{ N }  \Big| \, \leq \,  N^{-\frac{1}{4}} \Big\}\, ,
\label{ecriture definiton H}
\enq
with an entrywise remainder. In particular, the control on the centre of mass of $\bs{x}_N$ entails that
$$\bs{\La}_N^{+}(\wh{\bs{x}}_N)\, = \, \wh{\bs{x}}_N \, + \, \e{O}\big( \ex{-c N} \big)\, = \, \bs{x}_N+ \e{O}\big( N^{-\f{1}{4}} \big) \,. $$
Furthermore, if $\bs{x}_N \in \mc{W}_{\de}$, then, by construction,
$\wh{\bs{x}}_N = \bs{\eta}_N \in X_N^{-1}(\bs{x}_N) \in  \wt{\mc{E}} \cap \wt{\mc{N}}_{ \mu, \frac{\delta}{2} }$
and is such that $\eta_N \, = \, - \sul{s=1}{N-1}\eta_s$. Then,  because of $|\eta_a-\eta_b|=|x_a-x_b|$, one has that
\beq\label{rec}
\f{\kappa}{4} \leq \, 1 \, - \,  \sul{ \substack{ s=1 \\ s\not= j} }{N}\frac{   \ln|\eta_j-\eta_s|  }{\ln(2\veps_N)}
\, = \, 1 \, - \,  \sul{ \substack{ s=1 \\ s\not= j} }{N}\frac{   \ln|x_j - x_s|  }{\ln(2\veps_N)}
\enq
and $|x_a-x_b|\geq N^{-3}$, so that $\bs{x}_N \in \wt{\mc{E}}_2\cap\wt{\mc{E}}_1=\wt{\mc{E}}$. Furthermore,
it is direct to check for such $\bs{x}_N$s that $\e{d}_{\e{BL}}\big( \op{L}_N^{(\wh{\bs{x}}_N)},  \op{L}_N^{(\bs{x}_N)} \big) \leq \tf{\de}{4}$
provided that $N$ is large enough. Hence, if
\beq
\e{d}_{\e{BL}}\big(\mu,  \op{L}_N^{(\bs{x}_N)} \big) \leq \tf{\de}{4} \quad \e{then} \quad
\e{d}_{\e{BL}}\big( \op{L}_N^{(\wh{\bs{x}}_N)}, \mu \big) \leq \tf{\de}{2} \;.
\enq
Thus,
\beq
\mc{H}\cap \mc{W}_{\de} \,  \supset \, \mc{H}\cap  \wt{\mc{E}} \cap \wt{\mc{N}}_{ \mu, \f{\de}{4} }  \;.
\enq
All-in-all,  
this leads to the lower bound with $\mc{D}(\de) \, := \,  \mc{H}\cap  \wt{\mc{E}} \cap \wt{\mc{N}}_{ \mu, \f{\de}{4} }  \cap \mc{Q}_N $, 
\bem
 \overline{\Pi}_{N,\mathsf{c}}\big[ B_{\mathsf{c}}(\mu,\delta) \big]  \geq    \f{ 1 - \e{O}(\ex{-cN})  }{ N }  \ex{- \f{c^{\prime}}{ \kappa} \ln N}  \Int{  \mc{D}(\de)  }{}  \pl{s=1}{N} \Big\{  \varrho ( x_s ) \Big\}
\cdot \pl{j=1}{N-1} \Big\{ \ex{- \ln  \varrho ( x_j ) } \Big\}
   \\
\times \pl{j=1}{N} \bigg\{ \mathrm{e}^{-  V ( x_j+\e{O}(N^{-\tf{1}{4}}) )    }  \Big( 1-  \sul{ \substack{ s=1 \\ s \not= j} }{N}\frac{   \ln|x_j-x_s|  }{\ln(2\veps_N)}    \Big)^{1- \frac{1}{N}}
 \bigg\}    \cdot \mathrm{d} \bs{x}_N \;.
\end{multline}
Finally, we apply Jensen's inequality to the probability measure on $\R^N$
\beq
\dd \nu_{N}(\bs{x}_N) \, := \, \mathbbm{1}_{\mc{D}(\de) }(\bs{x}_N)  \cdot \f{    \dd^N \mu(\bs{x}_N)   }
                    {   \mu^{\otimes N}\big[ \mc{D}(\de) \big]     }
\enq
which yields
%
%
\bem
 \ln \overline{\Pi}_{N,\mathsf{c}}\big[ B_{\mathsf{c}}(\mu,\delta) \big]  \geq    - \e{O}(\ex{-cN}) -\ln  N - \f{c^{\prime}}{ \kappa} \ln N
 \, + \, \ln \Big(  \big(\mu\big)^{\otimes N}\big[ \mc{D}(\de) \big]  \Big)  \\
 \,- \, N \Int{}{}  V \big(  x_1+\e{O}(N^{-\tf{1}{4}}) \big)     \dd  \nu_{N}(\bs{x}_N)
\, - \, (N-1) \Int{}{}    \ln  \varrho ( x_1 )  \dd  \nu_{N}(\bs{x}_N)
   \\
\, + \,(N-1) \Int{}{} \ln \bigg( 1-  \sul{ s=1  }{N-1}\frac{   \ln|x_N-x_s|  }{\ln(2\veps_N)}      \bigg)\dd  \nu_{N}(\bs{x}_N)  \;.
 \label{ecriture borne inf sur mesure des boules en mu rayon delta}
\end{multline}

We now establish an auxiliary lemma that will allow us to understand the $N \to +\infty$ scaling of the \textit{rhs} of the above equation.

\begin{lemme}\label{probabilitybound} Let $\kappa, \delta  >0$ and let $\mu\in \mathcal{G}_{\kappa}$, which in particular means $\mu$ has bounded density and is compactly supported. Assume also that $\int x \, \dd\mu(x)=0$. Then,
given $\wt{\mc{N}}_{\mu,\de}$ as in \eqref{definition N tile mu delta}, $\wt{\mc{E}} = \wt{\mc{E}}_1\cap \wt{\mc{E}}_2$ with $\wt{\mc{E}}_a$ as introduced in \eqref{definition E tile 1}-\eqref{definition E tile 2}
and $\mc{H}$ as in \eqref{ecriture definiton H}, there exists $C > 0$ such that
$$ \mu^{\otimes N}[\mc{D}(\delta)]= \mu^{\otimes N}\Big[\wt{\mc{N}}_{ \mu, \f{\de}{4} } \cap  \wt{\mc{E}} \cap  \mc{H}  \Big] \, \geq  \, 1 - \frac{C}{\sqrt{N}} \, .$$
\end{lemme}
\begin{proof}
We prove this by a union bound. By Lemma \ref{concentration}, $ \mu^{\otimes N}  \big[ \wt{\mc{E}}^{\e{c}} \big] \,  \leq  \, \frac{C^\prime}{N}$ for some constant $C^\prime > 0$. Next,
it holds
\beq
 \mu^{\otimes N} \big[ \mc{H}^{\e{c}} \big]  \, \leq \,  \sqrt{N}  \Int{\R^N}{} \Big( \f{ \ov{\bs{x}}_N }{ N } \Big)^2 \,\mathrm{d}^N \mu(\bs{x}_N)
\, = \, \f{1}{\sqrt{N}  } \Int{\R}{} x^2 \,\mathrm{d} \mu(x)  
\label{ecriture estimee Mrkov sur Hc}
\enq
by Markov's inequality and the fact that $\mu$ has zero first moment. Finally, by Sanov's theorem (Theorem 6.1.3 of \cite{DemboZ01}), there exists a finite constant $C$ and $c_\delta>0$ for $\delta>0$ so that
\begin{equation}\label{sanov}
\mu^{\otimes N} \big[ \wt{\mc{N}}_{ \mu, \f{\de}{4} }^{\e{c}}  ] \,  \le \,  C  \ex{-c_\delta N}\,.
\end{equation}
\end{proof}

We now continue estimating the large-$N$ behaviour of the building blocks in \eqref{ecriture borne inf sur mesure des boules en mu rayon delta}. For $N$ large enough, if $x \in \mathrm{supp}\, \mu$ then $x+\e{O}( N^{-\tf{1}{4}} ) \in \mathfrak{K} := \{ x \in \mathbb{R}\, : \, \mathrm{dist}(x,\mathrm{supp}\, \mu) \leq 1\}$. Thus, one has
\bem
\Int{}{}  V \big(  x_1+\e{O}(N^{-\f{1}{4}}) \big)     \dd  \nu_N(\bs{x}_N) \, = \,
\e{O}\bigg( \norm{V}_{L^{\infty}( \mathfrak{K} ) }
\f{  \mu^{\otimes N}\big[ (\mc{D}(\de))^{\e{c}} \big]     } {   \mu^{\otimes N}\big[ \mc{D}(\de) \big]     } \bigg)  \\
\ +\,
\Int{}{}  \f{   V \big(  x+\e{O}(N^{-\f{1}{4}}) \big)     }{   \mu^{\otimes N}\big[ \mc{D}(\de) \big]     } \dd   \mu(x)   \limit{N}{+\infty} \Int{}{}  V (  x )     \dd   \mu(x)
\end{multline}
by Lemma \ref{probabilitybound} and dominated convergence since $V$ is uniformly continuous on compact sets.
Next,
$$\Int{}{}    \ln \big[ \varrho ( x_1 ) \big]  \dd  \nu_N(\bs{x}_N) \\
= \Int{}{}    \ln  \big[ \varrho ( x  )  \big] \f{ \dd   \mu(x ) }{  \mu^{\otimes N}\big[ \mc{D}(\de) \big]     }
\, - \,  \Int{}{}    \ln  \big[\varrho ( x_1 ) \big]  \f{ \mathbbm{1}_{ (\mc{D}(\de))^{\e{c}} }(\bs{x}_N) }{  \mu^{\otimes N}\big[ \mc{D}(\de) \big]     }
\dd^N   \mu(\bs{x}_N) \;.
$$
This second term may be estimated by Cauchy--Schwarz
$$\bigg|  \Int{}{}     \ln  \big[\varrho ( x_1 )\big]  \mathbbm{1}_{ (\mc{D}(\de))^{\e{c}} }(\bs{x}_N)   \dd^N   \mu(\bs{x}_N)  \bigg| \\
\leq \Big\{ \mu^{\otimes N}\big[ (\mc{D}(\de))^{\e{c}} \big] \Big\}^{ \f{1}{2} } \cdot
\Big\{ \Int{}{}       |\ln  \varrho ( x )|^2    \varrho ( x ) \dd x \Big\}^{ \f{1}{2} } \, .
$$
%
%
%
The second term in the product may be bounded by a constant by Definition \ref{defgood}(i), since $x \mapsto |\ln  \varrho ( x )|^2    \varrho ( x )$ is a bounded function supported on a compact set, while the first term is $\mathrm{O}(N^{-\frac{1}{4}})$
by Lemma \ref{probabilitybound}. Hence, all-in-all,
$$\Int{}{}    \ln  \varrho ( x_1 )  \dd  \nu_N(\bs{x}_N) \overset{N\to+\infty}{\longrightarrow}
                \Int{}{}    \ln  \varrho ( x  )   \dd   \mu(x )   \, .$$
Finally, we estimate the behaviour of the last contribution in \eqref{ecriture borne inf sur mesure des boules en mu rayon delta}.
\begin{align*}
    &\Int{}{} \ln \bigg( 1-  \sul{ s=1  }{N-1}\frac{   \ln|x_N-x_s|  }{\ln(2\veps_N)}      \bigg)\dd  \nu_N(\bs{x}_N) \\
    &=  \frac{1}{ \mu^{\otimes N}\big[ \mc{D}(\de) \big] }  \Int{}{} \ln \bigg( 1-   \frac{ N-1    }{\ln(2\veps_N) }
    \Int{}{}  \ln|x-y|  \dd  \mu(y)    \bigg)\dd  \mu(x) + \mathrm{O}(N^{-\frac{1}{2}}) \\
    &\; -  \frac{1}{ \mu^{\otimes N}\big[ \mc{D}(\de) \big] } \Int{}{} \ln \bigg( 1-  \sul{ s=1  }{N-1}\frac{   \ln|x_N-x_s|  }{\ln(2\veps_N)}      \bigg)
     \mathbbm{1}_{ (\mc{D}(\de))^{\e{c}} }(\bs{x}_N) \, \dd^N  \mu( \bs{x}_N ) \\
     &\geq \Big( 1+ \mathrm{O}(N^{-\frac{1}{2}}) \Big)
     \Int{}{} \ln \bigg( 1-   \frac{ N-1    }{\ln(2\veps_N) } \Int{}{}  \ln|x-y|  \dd  \mu(y)    \bigg)\dd  \mu(x) + \mathrm{O}(N^{-\frac{1}{2}}) \\
    &\; - \Big(1+ \mathrm{O}(N^{-\frac{1}{2}}) \Big) \mu^{\otimes N}\big[ (\mc{D}(\de))^{\e{c}} \big]
    \ln \bigg( 1-  \frac{ (N-1)   \ln|C|  }{\ln(2\veps_N)}      \bigg)
\end{align*}
where $C = \sup \{ |x-y| \, : \, x,y \in \mathrm{supp}\, \mu \}$. 
%
%
%
%
%
%
%
%
%
    \\
%
%
%
%
%
%
%
%

%
%
%
Here, we have invoked Lemma \ref{concentration}(iii) and Lemma \ref{probabilitybound}.
One may directly observe, thanks to the property (iii) of Definition \ref{defgood}, that  the terms present in the final lines of the bound tend to
$$  \Int{\mathbb{R}}{} \ln\bigg( 1 + \frac{2}{\ell} \Int{\mathbb{R}}{} \ln |x - y|  \, \mathrm{d} \mu(y) \bigg)  \, \mathrm{d} \mu(x) \; . $$
From all this, we conclude from \eqref{ecriture borne inf sur mesure des boules en mu rayon delta}
 that for any $\delta>0$, 
\begin{align}
    \liminf_{N \to +\infty}  \frac{1}{N} \ln \overline{\Pi}_{N,\mathsf{c}}\big[ B_{\mathsf{c}}(\mu,\delta) \big] \, \geq  \, -  I[  \mu ]  \, ,
\label{equation finale preuve lower bound good measures ldp}
\end{align}
in which $I$ is the rate function introduced in \eqref{ratefunction}. Taking $\delta \searrow 0$ concludes the proof in the case of "good"
measures in the sense of definition \ref{defgood}.

\subsection{Weak large deviation lower bound for general measures}

We are now in position to establish the weak large deviation lower bound for general measures, namely \eqref{wldlbtheo}. To do that, we approximate any measure with finite free energy by a sequence of measures which are good in the sense of Definition \ref{defgood}. In fact, we prove in Section \ref{sec:continuity along  good subsequences} the following lemma: 

\begin{lemme}\label{lemmeapprox}
Let $\mu$ be a probability measure such that $I[\mu] <+\infty$.
Then there exists a sequence of probability measures $(\mu_p)_{p\ge 1} $ so that
\begin{enumerate}
    \item For every $p\in \mathbb{N}^*$, $d_{\mathrm{BL}}(\mu,\mu_p)\le 1/p$,
    \item For $p$ large enough, there exists $\kappa_p>0$
    \begin{align*}
        &1+\frac{2}{\ell}\int\ln|x-y|\mathrm{d}\mu_p(y)\ge \kappa_p \, , &  &\forall x\in \mathbb{R} \; .
    \end{align*}
    \item $\mu_p$ is compactly supported.
    \item If $\int x \, \mathrm{d}\mu(x)=0$, then $\int x \, \mathrm{d}\mu_p(x)=0$ for every integer number $p$.
    \item $\mu_p$ has bounded density with respect to Lebesgue measure.
    \item $I[\mu_p] \longrightarrow I[\mu]$ as $p \to +\infty$.
\end{enumerate}

\end{lemme}
It is enough to prove the weak large deviation lower bound \eqref{wldlbtheo}  for measures $\mu$ such that $I[\mu]$ is finite, since otherwise the bound is trivial. Taking $\mu$ so that $I[\mu]$ is finite, Lemma \ref{lemmeapprox} shows that for any $\delta>0$, we can find a $p_0 \in \mathbb{N}^*$ such that for all $p\geq p_0$, we have $\mu_{p}\in B_{\mathsf{c}}(\mu,\frac{\delta}{2})$ and where $\mu_{p}$ satisfies all the hypotheses of Lemma \ref{ldlbgood}.
We observe that $B_{\mathsf{c}}(\mu_{p},\frac{\delta}{2})\subset B_{\mathsf{c}}(\mu,\delta)$, and so \eqref{equation finale preuve lower bound good measures ldp},
\textit{c.f.} the proof of Lemma \ref{ldlbgood}, implies that
$$\liminf_{N\rightarrow +\infty}\frac{1}{N}\ln \overline{\Pi}_{N,\mathsf{c}}\big[ B_{\mathsf{c}}(\mu,\delta)\big] \ge
\liminf_{N\rightarrow +\infty}\frac{1}{N}\ln \overline{\Pi}_{N,\mathsf{c}}\big[ B_{\mathsf{c}}\big(\mu_{p},\tfrac{\delta}{2}\big)\big] \ge -I[\mu_{p}]\,.$$
Sending $p \to +\infty$, the last point of Lemma \ref{lemmeapprox} implies that our lower bound goes to $I[\mu]$. Finally, we notice that \textit{rhs} does not depend on $\delta$, hence we may send $\delta \searrow 0$, which yields \eqref{wldlbtheo}.

\section{Upper bound on balls and exponential tightness}\label{section:ldub}

\subsection{Upper bound on the density}

In this section we prove the following upper bound on the joint eigenvalue density, more precisely on
$\mc{I}(\bs{\la}^+_N;\veps_N)$ introduced in \eqref{density}. This provides the first step
towards establishing the upper bound on balls.
\begin{prop}\label{integralupperbound}
Let $V$ satisfy \ref{potentialh1} and \ref{potentialh2} and fix $\gamma > 0$.
Then there is a constant $C >0$ (which may depend on $\gamma$) such that, for $N$ sufficiently large, the upper bound holds
\begin{align}
\mc{I}(\bs{\la}^+_N ; \veps_N) \leq \mathrm{e}^{CN^{1- \frac{1}{4}}} \sqrt{\frac{ \mc{T}(\boldsymbol{\lambda}^+_N) \mc{T}(\boldsymbol{\lambda}^-_N)}{\Delta(\boldsymbol{\lambda}_N^+) \Delta(\boldsymbol{\lambda}_N^-)} } \cdot  \prod_{j=1}^N \max\{ \ln |\lambda_j^+ - \lambda_j^-|^{-1}   , \gamma N \} \;.
\end{align}
Above, given $\boldsymbol{x}_N\in \mathbb R^{N}$ we agree upon:
\begin{align}\label{Tdef}
    \mc{T}(\boldsymbol{x}_N) =  \prod_{k=1}^N \Big( 1 + \frac{1}{\gamma N} \ln^2(1+|x_k|) \Big)^5 \prod_{j=1}^N \Big( 1 + \frac{6}{\gamma N^{\frac{5}{4}}} \sum_{\substack{ i = 1 \\ i \neq j}}^N \ln^2|x_i - x_j| \Big)^4  \, .
\end{align}

\end{prop}

The proof is intricate, and proceeds in several steps, which we describe in the following. We start by obtaining a first structured upper bound on
$\mc{I}(\bs{\la}^+_N ; \veps_N) $. Prior to that, we observe that the integration domain
$\intff{ \lambda_j^{\ups_k} }{ \lambda_{j+1}^{\ups_k} }$  for $\mu_j$ introduced in \eqref{definition domaine integration des mus},
may be decomposed into the left half and right half, \textit{i.e.} $\mathcal{K}_j^0 \cup \mathcal{K}_j^1$, where
\beq
\mathcal{K}_j^0  \, =  \, \Big[ \lambda_j^{\ups_j} ; \tfrac{\lambda_j^{\ups_j} + \lambda_{j+1}^{\ups_j}}{2}\Big]  \qquad \e{and} \qquad
\mathcal{K}_j^1 \, = \,  \Big[ \tfrac{\lambda_j^{\ups_j} + \lambda_{j+1}^{\ups_j}}{2}; \lambda_{j+1}^{\ups_j}\Big] \, .
\label{ecriture decomposition sur Kj sigma}
\enq

\begin{lemme}
\label{Lemme First upper bound on I}
One has the upper bound
\beq
\mc{I}(\bs{\la}^+_N ; \veps_N)  \, \leq \,  \pl{\vsg = \pm}{} \Big\{  \f{ \mc{S}^{\vsg}   }{ \Delta(\bs{\lambda}_N^{\vsg}) } \Big\}^{\f{1}{2}}
\label{Ibound}
\enq
with $\mc{S}^{\vsg} $ expressed as a combinatorial sum of determinants
\begin{align}\label{Sdef}
\mc{S}^{\vsg}\overset{\mathrm{def}}{=} \hspace{-5mm}
\sul{ \bs{\sg}_{N-1} \in \{ 0, 1\}^{N-1} }{} | \det_N\big[ \op{M}_{ \bs{\sg}_{\small N-1} }^{\, \vsg} \big]  |
\end{align}
involving the $N\times N$ matrix  $\op{M}_{ \bs{\sg}_{\small N-1} }^{\, \vsg}$ having entries
\beq
\big( \op{M}_{\bs{\sg}_{N-1}}^{\, \vsg}\big)_{ij} \overset{\mathrm{def}}{=} \left\{ \ba{c c c }
\Int{\mathcal{K}_j^{\sigma_j}}{} \frac{\mu - \lambda_{j+\sigma_j}^{\vsg} }{ \mu - \lambda_i^{\vsg} }
\frac{1}{\big\{ (\mu - \lambda_{j+\sigma_j}^+) (\mu - \lambda_{j+\sigma_j}^-) \big\}^{ 1/2 }  } \, \mathrm{d}\mu   & for & j \leq N-1   \vspace{3mm} \\
1 & for & j = N    \ea \right.  \, .
\label{definition matrice M sigma}
\enq
\end{lemme}

Recall that the integration domain is such that  $(\mu - \lambda_{j+\sigma_j}^+) (\mu - \lambda_{j+\sigma_j}^-)  > 0$ on its interior,
so that the square root appearing in the definition of $\op{M}_{ \bs{\sg}_{\small N-1} }^{\, \vsg}$ is well-defined and strictly positive.

\begin{proof}

By using the decomposition of the integration domain into the sets introduced in \eqref{ecriture decomposition sur Kj sigma}, one gets
\beq
\mathcal{I}(\bs{\la}^+_N ; \veps_N)  \, =  \hspace{-3mm} \sul{ \bs{\sg}_{N-1} \in \{ 0, 1\}^{N-1} }{} \hspace{-3mm}  \mathcal{I}_{ \bs{\sg}_{N-1}} (\bs{\la}^+_N ; \veps_N)
\enq
with 
\beq
\mathcal{I}_{ \bs{\sg}_{N-1} }(\bs{\la}^+_N ; \veps_N)   \, = \,  \Int{ \R^N }{}\Delta(\boldsymbol{\mu}_{N-1})
\pl{s=1}{N-1}\bigg\{  \frac{ \mathbbm{1}_{ \mc{K}_s^{\sg_s} }(\mu_s)  }{ \sqrt{P^+(\mu_s) P^-(\mu_s)}}  \bigg\}
  \, \mathrm{d}\boldsymbol{\mu}_{N-1} \, .
\enq
For fixed $\bs{\sg}_{N-1}\in \{ 0, 1\}^{N-1}$, one has the factorisation
$$\pl{k=1}{N-1} \sqrt{P^+(\mu_k) P^-(\mu_k)} \,  = \,
\pl{k=1}{N-1}  \sqrt{P_{k+\sigma_k}^+(\mu_k) P_{k+\sigma_k}^-(\mu_k)}
\pl{k=1}{N-1}   \sqrt{(\mu_k - \lambda_{k+\sigma_k}^+) (\mu_k - \lambda_{k+\sigma_k}^-)} $$
where
$$P_k^\pm ( X ) \overset{\mathrm{def}}{=}  \prod_{\substack{\ell = 1 \\ \ell \neq k}}^N( X  - \lambda_\ell^\pm) \, .$$
Also, we stress that all square roots are well-defined and non-negative for $\mu_k\in \mc{K}_k^{\sg_k}$.
Next,  let us define
\begin{align}
\mathcal{I}_{ \bs{\sg}_{N-1} }^{\pm} (\bs{\la}^+_N ; \veps_N)  & =  \hspace{-2mm} \Int{\R^{N-1} }{} \hspace{-3mm}
\Delta(\boldsymbol{\mu}_{N-1}) \prod_{k=1}^{N-1} \Bigg\{ \frac{  \mathbbm{1}_{ \mc{K}_k^{\sg_k} }(\mu_k)  }
{ P_{k+\sigma_k}^\pm (\mu_k)  \big\{ (\mu_k - \lambda_{k+\sigma_k}^+)  (\mu_k - \lambda_{k+\sigma_k}^-) \big\}^{1/2} }
 \Bigg\}  \, \mathrm{d}\boldsymbol{\mu}_{N-1}  \, .
\end{align}
Then, by Cauchy--Schwarz's inequality, we get
 \begin{equation}
\mathcal{I}_{ \bs{\sg}_{N-1} }^{\pm} (\bs{\la}^+_N ; \veps_N)
\leq \sqrt{ \big| \mathcal{I}_{ \bs{\sg}_{N-1} }^{+} (\bs{\la}^+_N ; \veps_N)  \,   \mathcal{I}_{ \bs{\sg}_{N-1} }^{-} (\bs{\la}^+_N ; \veps_N)  \big| } \;.
\end{equation}
For any monic polynomials $q_{i}$ of degree $i$, we can write, by representing $\Delta(\boldsymbol{\mu}_{N-1})$
as a Vandermonde determinant, 

\beq
\mathcal{I}_{ \bs{\sg}_{N-1} }^{\pm} (\bs{\la}^+_N ; \veps_N)  \, = \,  \det_{N-1}\bigg[ \Int{\mathcal{K}_j^{\sigma_j}}{}
\frac{ q^{\pm}_{i-1}(\mu) }{ P^{\pm}(\mu) }
\frac{ \mu - \lambda_{j+\sigma_j}^{\pm} }{  \big\{  (\mu - \lambda_{j+\sigma_j}^+) (\mu - \lambda_{j+\sigma_j}^-) \big\}^{1/2}  }
\, \mathrm{d}\mu   \bigg] \, .
\enq
We will choose 
\beq
q_i^\pm(X) = \prod_{\ell=1}^i( X  - \lambda_\ell^\pm)\,.
\label{definition polynome qi pm}
\enq
Now applying Cauchy--Schwarz on the level of the sum over $\bs{\sg}_{N-1}$, one gets
 \beq
\mathcal{I} (\bs{\la}^+_N ; \veps_N)  \, \leq \,
\bigg\{ \sul{\bs{\sg}_{N-1} \in \{ 0, 1\}^{N-1} }{} \hspace{-5mm} \big| \mathcal{I}_{ \bs{\sg}_{N-1} }^{+} (\bs{\la}^+_N ; \veps_N)   \big|  \bigg\}^{ \f{1}{2} }
\, \cdot \,
\bigg\{ \sul{\bs{\sg}_{N-1} \in \{ 0, 1\}^{N-1} }{} \hspace{-5mm} \big| \mathcal{I}_{ \bs{\sg}_{N-1} }^{-} (\bs{\la}^+_N ; \veps_N)  \big| \bigg\}^{ \f{1}{2} } \, .
\enq
For $\bs{\sg}_{N-1}\in\{ 0,1\}^{N-1}$, we define the $N \times N$ matrix
\begin{align*}
\big( Q^{ \pm }_{ \bs{\sg}_{N-1} } \big)_{ij} \, = \,
\left\{ \ba{cc} \Int{\mathcal{K}_j^{\sigma_j}}{} \frac{q_{i-1}^\pm(\mu)}{P^\pm(\mu)}
\frac{\mu - \lambda_{j+\sigma_j}^\pm }{
\big\{ (\mu - \lambda_{j+\sigma_j}^+) (\mu - \lambda_{j+\sigma_j}^-) \big\}^{1/2}
} \, \mathrm{d}\mu & j \leq N-1 \\
\delta_{iN} & j = N  \ea \right.    \;.
\end{align*}
 Because the final column is all zeroes except for the final entry, we have
\begin{align}
\mathcal{I}_{ \bs{\sg}_{N-1} }^{\pm} (\bs{\la}^+_N ; \veps_N)  \, = \, \det_N \big[ \op{Q}^{ \pm }_{ \bs{\sg}_{N-1} } \big] \, .
\end{align}
Given $\bs{x}_N \in \mathbb{R}^N$ having pairwise distinct entries, define the $N \times N$ matrix $\op{A}(\boldsymbol{x}_N)$ as
\beq
A(\boldsymbol{x}_N)_{i j} \overset{\mathrm{def}}{=} \mathbbm{1}_{i \leq j}  \prod_{\substack{m= i \\ m \neq j }}^N \frac{1}{x_j - x_m}  \, .
\label{definition matrice A2}
\enq
By virtue of Lemma \ref{Lemme ecriture action inverse A}, it holds
\begin{align}
  \det_N\big[ \op{Q}^{ \pm }_{ \bs{\sg}_{N-1} } \big] = \det_N[\op{A}(\boldsymbol{\lambda}_N^\pm)]
\det_N\big[\op{A}(\boldsymbol{\lambda}_N^\pm)^{-1}\op{Q}^{ \pm }_{ \bs{\sg}_{N-1} } \big] = \frac{  (-1)^{ N \f{N-1}{2} } }{\Delta(\boldsymbol{\lambda}^\pm_N)}
\det_N[ \op{M}_{\bs{\sg}_{N-1}}^{\, \vsg}]
\end{align}
where $ \op{M}_{\bs{\sg}_{N-1}}^{\, \vsg}$ is as defined in \eqref{definition matrice M sigma}.
Thus, overall, we have \eqref{Ibound}.
\end{proof}

To proceed with the bounds, we need a few auxiliary results.

\begin{lemme}\label{exactMformula}
The matrix entry $\big( M_{\bs{\sg}_{N-1}}^{\, \vsg} \big)_{ j+\sigma_j, j }$ can be evaluated in closed form
$$\big( M_{\bs{\sg}_{N-1}}^{\, \vsg} \big)_{ j+\sigma_j, j } \,  =  \,
\ln\left|\Phi\left(\frac{1}{2}\frac{\Delta_j}{\delta_{j+\sigma_j}}\right)\right| \qquad \text{for} \qquad j \in \intn{1}{N-1} \;. $$
Here, we introduced $\Phi(x) := 2x+1+2\sqrt{x(1+x)}$, and have set
\beq
\delta_j \, =  \, |\lambda_j^+ - \lambda_j^-|  \qquad and \qquad  \Delta_j \, = \,  \lambda_{j+1}^{\ups_j}  \, - \,  \lambda_{j}^{\ups_j} \, .
\enq
\end{lemme}
\begin{proof}
By translating and rescaling,
$$\big( M_{\bs{\sg}_{N-1}}^{\, \vsg} \big)_{ j+\sigma_j, j } \,  =  \,
\Int{0}{ \tf{ \Delta_j }{ (2\delta_{j+\sigma_j})} } \frac{1}{\sqrt{\mu(\mu+1)}} \, \mathrm{d}\mu \, . $$
At this stage, it remains to observe that $\frac{1}{\sqrt{\mu(\mu+1)}} = \frac{\mathrm{d}}{\mathrm{d}\mu} \ln |\Phi(\mu)|$.
\end{proof}


\begin{cor}\label{matrixelementbound2}
One has the lower and upper bounds
$$-\ln (N+1) - \ln |\delta_{j+\sigma_j}| + \ln|\lambda_{j+1}^{\vsg} - \lambda_j^{\vsg}|
\leq \big( M_{\bs{\sg}_{N-1}}^{\, \vsg} \big)_{ j+\sigma_j, j }
\leq \ln 3 - \ln |\delta_{j+\sigma_j}| + \sul{s=0}{1} \ln(|\lambda_{j+s}^{\vsg}| + 1)  \, .$$
Moreover, $\big( M_{\bs{\sg}_{N-1}}^{\, \vsg} \big)_{ j+\sigma_j, j } \geq 0$.
\end{cor}

\begin{proof}

Positivity of the matrix elements follows directly from their integral representation.
Furthermore, one has the lower and upper bounds $4x + 1 \leq \Phi(x) \leq 4x+3$.
Let us prove the lower bound for the case $\vsg=+$, the case $\vsg=-$ being similar. Using Lemma \ref{exactMformula} we see that $\big( M_{\bs{\sg}_{N-1}}^{\, +} \big)_{ j+\sigma_j, j } \geq \ln| 2 \frac{\Delta_j}{\delta_{j+\sigma_j}}+1|$. In the case $\ups_j = +$, we may then immediately bound $\ln| 2 \frac{\Delta_j}{\delta_{j+\sigma_j}}+1| \geq \ln| \frac{\Delta_j}{\delta_{j+\sigma_j}}|$ which proves the claim. Hence let us assume $\ups_j = -$. Following the same lower bound, we find
$$\big( M_{\bs{\sg}_{N-1}}^{\, +} \big)_{ j+\sigma_j, j } \geq \ln|\Delta_j + \delta_{j+\sigma_j} | - \ln|\delta_{j+\sigma_j}| \, .$$
In particular we have
$$\Delta_j + \delta_{j+\sigma_j} = \begin{cases}
\lambda_{j+1}^+ - \lambda_j^- & \sigma_j = 1\\
\lambda_{j+1}^- - \lambda_j^+ & \sigma_j = 0
\end{cases} \geq \begin{cases}
\lambda_{j+1}^+ - \zeta_j & \sigma_j = 1\\
\zeta_j - \lambda_j^+ & \sigma_j = 0
\end{cases}$$
where $ \zeta_j $ is the unique zero of $P^\prime$ that lies in between the two consecutive roots $\lambda_j^- $ and $ \lambda_{j+1}^-$. By Lemma \ref{Lemme espacement racine lambda pm et zeta j} we see that
$$\lambda_{j+1}^+ - \lambda_j^+ = |\lambda_{j+1}^+ - \zeta_j| + |\lambda_{j}^+ - \zeta_j| \leq (N+1) \min\{ |\lambda_{j+1}^+ - \zeta_j|  , |\lambda_{j}^+ - \zeta_j| \}$$
which completes the proof of the lower bound. The upper bound proceeds similarly since Lemma \ref{exactMformula} implies  that $\big( M_{\bs{\sg}_{N-1}}^{\, +} \big)_{ j+\sigma_j, j } \leq \ln| 2 \frac{\Delta_j}{\delta_{j+\sigma_j}}+3|$. If $\nu_j = -$, then clearly $$\Delta_j + \delta_{j+\sigma_j} \leq \lambda_{j+1}^+ - \lambda_j^+ \leq (1+|\lambda_j^+|) (1+|\lambda_{j+1}^+|) \, .$$
If $\nu_j = +$ then we may use the fact that for $N$ sufficiently large $\delta_{j+\sigma_j} \leq 1$ (by Proposition \ref{kappeler}). Hence $\Delta_j + \delta_{j+\sigma_j} \leq |\lambda_{j+1}^+ - \lambda_{j}^+|+1 \leq (1+|\lambda_j^+|) (1+|\lambda_{j+1}^+|)$.
\end{proof}

 There is a convenient way to parameterise an element $ \bs{\sg}_{N-1} \in \{ 0, 1\}^{N-1}$ in terms of "up/down" steps.

\begin{defin}[Level change, step up, step down] Let $\bs{\sg}_{N-1} \in \{ 0,1\}^{N-1}$.  One says that $\sigma_j$ with $j \in  \intn{1}{N-2}$
is a "level change" if $\sigma_j \neq \sigma_{j+1}$. Level changes come in two forms:
\begin{itemize}
\item "steps up" corresponding to $(\sigma_j, \sigma_{j+1}) = (0,1)$
\item "steps down" corresponding to $(\sigma_j, \sigma_{j+1}) = (1,0)$.
\end{itemize}
There is a natural identification $\bs{\sg}_{N-1} \mapsto \big(\sigma_1, X(\bs{\sg}_{N-1}) \big)$
where $X(\bs{\sg}_{N-1}) \subset \intn{1}{N-2}$ is the set of level changes.  That is, $\sigma_1$ is the first element in the sequence and
$$X(\bs{\sg}_{N-1}) \,  = \, \big\{ j \in \intn{1}{N-2} \, : \,  \sigma_j \neq \sigma_{j+1} \big\} \, .$$
Note that one has the partitioning $X(\bs{\sg}_{N-1}) \, = \,  X_\downarrow(\bs{\sg}_{N-1}) \cup X_{\uparrow}(\bs{\sg}_{N-1})$ where
$X_\downarrow(\bs{\sg}_{N-1})$ is the set of down steps and $X_{\uparrow}(\bs{\sg}_{N-1})$ is the set of up steps of $\bs{\sg}_{N-1}$.
\end{defin}

\begin{prop}
\label{Proposition upper bound on S via A et B mathfrak}

Recall $\delta_j = |\lambda_j^+ - \lambda_j^-|$. Then,  for every $\gamma>0$, $\mc{S}^{\vsg}$ introduced in \eqref{Sdef} can be bounded from above by
\bem
\mc{S}^{\vsg} \, \leq \,  \frac{1}{\gamma \sqrt{N} } \prod_{j=1}^{N}   \Big\{  \ln |\delta_{j}|^{-1} \vee \gamma N \Big\}
 \hspace{-3mm} \sum_{ \bs{\sg}_{N-1} \in \{ 0, 1\}^{N-1} } \hspace{-3mm}  (\gamma N)^{- \frac{1}{2}|X_\downarrow( \bs{\sg}_{N-1} )|}  \\
\times \prod_{\substack{j=1 \\ j \not\in X_\downarrow( \bs{\sg}_{N-1} )}}^{N-1}  \hspace{-3mm}
\Big\{ \mathfrak{A}_j^{\vsg}(\bs{\sg}_{N-1}) \Big\}^{\f{1}{2}}
 \prod_{\substack{j=1 \\ j \in X_\downarrow( \bs{\sg}_{N-1} )}}^{N-1}  \hspace{-3mm}
\Big\{ \frac{ \mathfrak{B}_j^{\vsg}(\bs{\sg}_{N-1}) }{\gamma N }    \Big\}^{\f{1}{2}}
\label{Sbound}
\end{multline}
where
\beq
 \mathfrak{A}_j^{\vsg}(\bs{\sg}_{N-1})  \, = \,
 \left( 1 + \frac{\ln 3}{\gamma N }  + \frac{\ln(|\lambda_{j+1}^{\vsg} | + 1)+\ln(|\lambda_{j}^{\vsg}| + 1)}{\gamma N }  \right)^2
+ \sum_{\substack{i=1 \\ i \neq j+\sigma_j}}^N \frac{ \big|\big( M_{\bs{\sg}_{N-1}}^{\, \vsg} \big)_{ij} \big|^2}{\gamma^2 N^2}
\label{definition mf A j vsg de sigma N-1}
\enq
and
\beq
 \mathfrak{B}_j^{\vsg}(\bs{\sg}_{N-1}) \, = \, \sum_{i=1}^N \big| \big( R_{\bs{\sg}_{N-1}}^{\, \vsg} \big)_{ij}\big|^2 \;.
\enq
 The matrix $\op{R}_{\bs{\sg}_{N-1}}^{\, \vsg} $ appearing above is defined as
\beq
\big( R_{\bs{\sg}_{N-1}}^{\, \vsg} \big)_{ij} \overset{\mathrm{def}}{=}
\big( M_{\bs{\sg}_{N-1}}^{\, \vsg} \big)_{ij} - \mathbbm{1}_{j \in X_\downarrow(\bs{\sg}_{N-1})} \big( M_{\bs{\sg}_{N-1}}^{\, \vsg} \big)_{ij+1} \, .
\label{definition matrice R}
\enq

\end{prop}

\begin{proof}

Starting from the representation \eqref{Sdef},
one performs the following  column operations on $\op{M}_{\bs{\sg}_{N-1}}^{\, \vsg}$ that leave the determinant invariant.
For each  down step $j \in X_{\downarrow}(\bs{\sg}_{N-1})$, we replace the $j^{\e{th}}$ column with the $j^{\e{th}}$
minus the $(j+1)^{\e{th}}$ column. This operation is independent of the order in which rows are subtracted since $j$ being a step down
prevents  $j+1$ from being a step down and leads to the replacement

\begin{align}
\mc{S}^{\vsg}\overset{\mathrm{def}}{=} \hspace{-5mm}
\sul{ \bs{\sg}_{N-1} \in \{ 0, 1\}^{N-1} }{} \hspace{-4mm} \big| \det_N\big[ \op{R}_{ \bs{\sg}_{  N-1} }^{\, \vsg} \big]  \big|
\end{align}
with $\op{R}_{ \bs{\sg}_{  N-1} }^{\, \vsg}$ as introduced in  \eqref{definition matrice R}.
At this stage, one applies Hadamard's inequality to find
\beqa
  \mc{S}^{\vsg}  & \leq & \sqrt{N} \hspace{-5mm} \sul{ \bs{\sg}_{ N-1}  \in \{ 0, 1\}^{N-1} }{}
  \pl{j=1}{N-1} \bigg\{\sum_{i=1}^N \big| \big( R_{ \bs{\sg}_{  N-1} }^{\, \vsg} \big)_{ij} \big|^2 \bigg\}^{\f{1}{2} }  \\
& =&  \sqrt{N} \hspace{-5mm}  \sul{ \bs{\sg}_{ N-1}  \in \{ 0, 1\}^{N-1} }{}\hspace{-5mm}
\prod_{\substack{j=1 \\ j \not\in X_\downarrow(\bs{\sg}_{  N-1}) }}^{N-1} \hspace{-5mm}
            \bigg\{\sum_{i=1}^N \big| \big( R_{ \bs{\sg}_{  N-1} }^{\, \vsg} \big)_{ij} \big|^2 \bigg\}^{\f{1}{2} }
\hspace{-5mm}  \prod_{\substack{j=1 \\ j \in X_\downarrow(\bs{\sg}_{  N-1}) }}^{N-2}\hspace{-5mm}
                    \bigg\{\sum_{i=1}^N \big| \big( R_{ \bs{\sg}_{  N-1} }^{\, \vsg} \big)_{ij} \big|^2 \bigg\}^{\f{1}{2} }  \\
&=   & \sqrt{N} \hspace{-5mm}   \sul{ \bs{\sg}_{ N-1}  \in \{ 0, 1\}^{N-1} }{} \hspace{-5mm}
\prod_{\substack{j=1 \\ j \not\in X_\downarrow(\bs{\sg}_{  N-1}) }}^{N-1}  \hspace{-5mm}
\bigg\{  \big| \big( M_{ \bs{\sg}_{  N-1} }^{\, \vsg} \big)_{j+\sg_j \, j} \big|^2 +
\sum_{\substack{i=1 \\ i \neq j+\sigma_j}}^N \big| \big( M_{ \bs{\sg}_{  N-1} }^{\, \vsg} \big)_{ij} \big|^2 \bigg\}  \\
&& \times \hspace{-5mm}  \prod_{\substack{j=1 \\ j \in X_\downarrow(\bs{\sg}_{  N-1}) } }^{N-2}\hspace{-5mm}
\bigg\{\sum_{i=1}^N \big| \big( R_{ \bs{\sg}_{  N-1} }^{\, \vsg} \big)_{ij} \big|^2 \bigg\}^{\f{1}{2} }    \;.
\eeqa
Note that $\sqrt{N}$ comes from the contribution of the $N^{\e{th}}$ column to the Hadamard bound.
Fix $\gamma  > 0$. By Corollary \ref{matrixelementbound2}, one has the upper bound
$$ \big| \big( M_{ \bs{\sg}_{  N-1} }^{\, \vsg} \big)_{j+\sg_j \, j} \big|  \, \leq  \,
\ln 3 + \big\{  \ln |\delta_{j+\sigma_j}|^{-1} \vee \gamma N \big\} + \ln(| \lambda_j^{\vsg} | + 1) + \ln(|\lambda_{j+1}^{\vsg} | + 1) $$
leading to
$$ \big| \big( M_{ \bs{\sg}_{  N-1} }^{\, \vsg} \big)_{j+\sg_j \, j} \big|^2  \, \leq  \,
 \big\{  \ln |\delta_{j+\sigma_j}|^{-1}  \vee \gamma N \big\}^2
\cdot \bigg[ 1 + \frac{\ln 3 + \ln(| \lambda_j^{\vsg}| + 1) + \ln(|\lambda_{j+1}^{\vsg} | + 1)}{\gamma N }  \bigg]^2  \, .$$
The latter then yields
$$
\mc{S}^{\vsg}   \leq   \sqrt{N} \hspace{-5mm}   \sul{ \bs{\sg}_{ N-1}  \in \{ 0, 1\}^{N-1} }{} \hspace{-5mm}
\prod_{\substack{j=1 \\ j \not\in X_\downarrow(\bs{\sg}_{  N-1})  }}^{N-1}
\hspace{-6mm}  \big\{  \ln |\delta_{j+\sigma_j}|^{-1}  \vee \gamma N \big\}   \times
\pl{ \substack{j=1 \\ j \not\in X_\downarrow(\bs{\sg}_{  N-1}) }}{N-1} \hspace{-6mm}  \Big\{\mathfrak{A}_j^{\vsg}(\bs{\sg}_{N-1})
\Big\}^{\f{1}{2}}  \prod_{\substack{j=1 \\ j \in X_\downarrow(\bs{\sg}_{  N-1}) } }^{N-2}\hspace{-5mm}
\Big\{\mathfrak{B}_j^{\vsg}(\bs{\sg}_{N-1})  \Big\}^{\f{1}{2} }    \;.
$$

Now observe that if $j,k \not\in X_\downarrow(\bs{\sg}_{  N-1})$ and $j \neq k$, then one also has $j+\sigma_j \neq k+\sigma_k$.
Indeed, one may assume that $j < k$ without loss of generality. Then if $j+\sigma_j = k+\sigma_k$, the only way this can happen is if $k = j+1$, $\sigma_j = 1$ and $\sigma_{j+1}  = 0$. But then $j$ would be step down, contradicting the assumption.
This property entails
$$\pl{ \substack{j=1 \\ j \not\in X_\downarrow(\bs{\sg}_{  N-1})} }{N-1} \big\{ \ln |\delta_{j+\sigma_j}|^{-1} \vee \gamma N \big\}
\, \leq \,  (\gamma N)^{-1-|X_\downarrow(\bs{\sg}_{  N-1})|} \pl{j=1}{N} \big\{  \ln |\delta_{j}|^{-1} \vee  \gamma N \big\}  \, ,  $$
thus providing one with the last bound necessary to conclude.
\end{proof}




We will now obtain upper bounds, first for $\mathfrak{A}_j^{\vsg}(\bs{\sg}_{N-1})$  and then $ \mathfrak{B}_j^{\vsg}(\bs{\sg}_{N-1})$.
This requires us to obtain a few auxiliary estimates on the matrix entries  $\big( M_{ \bs{\sg}_{  N-1} }^{\, \vsg} \big)_{i j}$.
Let us denote, for short,
\beq
\big( M_{ \bs{\sg}_{  N-1} }^{\, \vsg} \big)_{i j} \Big|_{ \mid \sg_j = a}\,  = \, \big( \mf{m}_{a}^{\vsg}\big)_{ij} \;,
\quad \e{with} \quad a \in \{0,1 \} \, .
\label{definition element de matrice mathfrak m a ij}
\enq
Thus, more explicitly, it holds for $i \neq j$
\begin{align*}
\big( \mf{m}_{0}^{\vsg}\big)_{ij}  &:= \Int{\mathcal{K}_j^0}{}
\frac{\mu -\lambda_j^{\vsg}}{\mu - \lambda_i^{\vsg}} \frac{1}{\sqrt{(\mu - \lambda_j^+)(\mu -\lambda_j^-)}} \, \mathrm{d}\mu
\end{align*}
while, for $i \neq j+1$,
\begin{align*}
\big( \mf{m}_{1}^{\vsg}\big)_{ij} &:=
\Int{\mathcal{K}_j^1}{} \frac{\mu -\lambda_{j+1}^{\vsg}}{\mu - \lambda_i^{\vsg}} \frac{1}{\sqrt{(\mu - \lambda_{j+1}^+)(\mu -\lambda_{j+1}^-)}}
\, \mathrm{d}\mu \, .
\end{align*}

\begin{prop}[Bounds on the matrix elements]\label{matrixelementbounds}
The following bounds hold,
\begin{align*}
\big| \big( \mf{m}_{0}^{\vsg}\big)_{ij} \big|  &\leq
\left\{  \ba{cc} \ln|\lambda_{j+1}^{\vsg} - \lambda_i^{\vsg}|  - \ln|\lambda_j^{\vsg} - \lambda_i^{\vsg}|  + \pi &  j > i   \vspace{2mm}  \\
\ln 2 + \mathbbm{1}_{\{ \ups_j = -\vsg \}}
\sqrt{ \Big| \tfrac{   \lambda_j^{-\vsg} - \lambda_j^{\vsg}  }{ \lambda_i^{\vsg} \, - \lambda_j^{-\vsg} }  \Big| }
                        \ln\frac{\sqrt{2}+1}{\sqrt{2}-1}   &  j < i   \ea \right.  \vspace{2mm} \\
\big| \big( \mf{m}_{1}^{\vsg}\big)_{ij} \big| &\leq
\left\{  \ba{cc} \ln|\lambda_{i}^{\vsg} - \lambda_j^{\vsg}|  - \ln|\lambda_i^{\vsg} - \lambda_{j+1}^{\vsg}|  + \pi &  i >j+1   \vspace{2mm}\\
\ln 2  + \mathbbm{1}_{\{ \ups_j = -\vsg \}} \sqrt{ \Big| \tfrac{  \lambda_{j+1}^{\vsg} - \lambda_{j+1}^{-\vsg}   }
            {\lambda_{j+1}^{-\vsg} - \lambda_{i}^{\vsg} }  \Big| }
                                                            \ln \frac{\sqrt{2}+1}{\sqrt{2}-1} &  i < j+1 \, . \ea \right.
\end{align*}

\end{prop}

\begin{proof}

\vspace{1mm}

  $\bullet${\it Case  $\sg_j=0$ and $\ups_j = \vsg$.}

One has
\beq
 \big( \mf{m}_{0}^{\vsg}\big)_{ij}  = \hspace{-2mm} \Int{\lambda_j^{\vsg} }{\lambda_j^{\vsg} + \frac{1}{2}\Delta_j}\hspace{-2mm}
 \frac{\mu -\lambda_j^{\vsg} }{\mu - \lambda_i^{\vsg} } \frac{1}{\sqrt{(\mu - \lambda_j^+)(\mu -\lambda_j^-)}} \, \mathrm{d}\mu
\nonumber
\enq
where $\Delta_j = \lambda_{j+1}^{\vsg} - \lambda_j^{\vsg}$. Then, with $i\not=j$,
\beq
\big|  \big( \mf{m}_{0}^{\vsg}\big)_{ij}  \big| \, \leq \,
\bigg|   \Int{\lambda_j^{\vsg} }{\lambda_j^{\vsg} + \frac{1}{2}\Delta_j}\hspace{-2mm}  \frac{  \mathrm{d}\mu  }{ \mu - \lambda_i^{\vsg} }  \,  \bigg|
\, =  \,
\Big|  \ln|\lambda_j^{\vsg}  - \lambda_i^{\vsg}  + \frac{1}{2}\Delta_j| - \ln|\lambda_j^{\vsg}  - \lambda_i^{\vsg}  | \Big|  \, .
\nonumber
\enq
In the case $j > i$ we can bound
$\lambda_j^{\vsg}  - \lambda_i^{\vsg}  +  \frac{1}{2}\Delta_j  \, \leq  \, \lambda_{j+1}^{\vsg} - \lambda_i^{\vsg}$.
Conversely if $i > j$ then $\lambda_i^{\vsg} - \lambda_j^{\vsg} - \frac{1}{2}\Delta_j \, \geq \, \frac{1}{2}(\lambda_i^{\vsg} - \lambda_j^{\vsg})$.
This covers the case of interest.

\vspace{2mm}
  $\bullet${\it Case $\sg_j=0$ and $\ups_j = -\vsg$.}
\vspace{2mm}

Now, one has
\begin{align*}
 \big( \mf{m}_{0}^{\vsg}\big)_{ij}    &=
 \Int{\lambda_j^{-\vsg} }{\lambda_j^{-\vsg} + \frac{1}{2}\Delta_j}\hspace{-2mm}  \frac{\mu -\lambda_j^{\vsg} }{\mu - \lambda_i^{\vsg}}
\frac{1}{\sqrt{(\mu - \lambda_j^+)(\mu -\lambda_j^-)}} \, \mathrm{d}\mu
\end{align*}
where $\Delta_j = \lambda_{j+1}^{-\vsg}   - \lambda_{j}^{-\vsg} $.
Then
\begin{align*}
\big|  \big( \mf{m}_{0}^{\vsg}\big)_{ij}  \big|  &=
\bigg| \hspace{-2mm}   \Int{\lambda_j^{-\vsg} }{\lambda_j^{-\vsg} + \frac{1}{2}\Delta_j}\hspace{-2mm}
\frac{1}{ \mu - \lambda_i^{\vsg} } \Big\{ 1 + \tfrac{\lambda_j^{-\vsg} - \lambda_j^{\vsg} }{\mu -\lambda_j^{-\vsg}} \Big\}^{\f{1}{2}}
\, \mathrm{d}\mu  \bigg|
\leq
\bigg| \hspace{-2mm} \Int{\lambda_j^{-\vsg} }{\lambda_j^{-\vsg} + \frac{1}{2}\Delta_j}\hspace{-4mm}
\frac{1}{\mu - \lambda_i^{\vsg}} \bigg( 1 + \sqrt{\tfrac{\lambda_j^{-\vsg} - \lambda_j^{\vsg} } {\mu -\lambda_j^{-\vsg}} } \, \bigg)
\, \mathrm{d}\mu  \bigg|  \\
&= \bigg| \,
\ln \Big| \tfrac{ \lambda_j^{-\vsg} - \lambda_i^{\vsg} + \frac{1}{2}\Delta_j }{ \lambda_j^{-\vsg} - \lambda_i^{\vsg} } \Big| \,  + \,
2\sqrt{\lambda_j^{-\vsg} - \lambda_j^{\vsg}}
\Int{0}{\sqrt{ \tf{\Delta_j}{2}}} \hspace{-2mm} \frac{1}{\mu^2 - (\lambda_i^{\vsg} - \lambda_j^{-\vsg}) }  \, \mathrm{d}\mu \bigg| \, .
\end{align*}
First, consider the case $j > i$. Then one bounds the log contribution by using that
$$\lambda_j^{-\vsg} - \lambda_i^{\vsg} \geq \lambda_j^{\vsg} - \lambda_i^{\vsg}  \qquad \e{and} \qquad
\lambda_{j+1}^{-\vsg}  - \lambda_i^{\vsg} \leq \lambda_{j+1}^{\vsg} - \lambda_i^{\vsg}\, , $$
so that
$$
\ln \bigg| \frac{\lambda_j^{-\vsg} - \lambda_i^{\vsg} + \frac{1}{2}\Delta_j}{\lambda_j^{-\vsg} - \lambda_i^{\vsg}} \bigg|  \, \leq \,
\ln \bigg| \frac{\lambda_j^{-\vsg} - \lambda_i^{\vsg} + \Delta_j}{\lambda_j^{-\vsg} - \lambda_i^{\vsg}} \bigg|
\leq
\ln|\lambda_{j+1}^{\vsg} - \lambda_i^{\vsg}|  - \ln|\lambda_j^{\vsg} - \lambda_i^{\vsg}| \, .$$
The integral term gives
\begin{align*}
2 \sqrt{ \lambda_j^{-\vsg} - \lambda_j^{\vsg} } \Int{ 0 }{ \sqrt{ \tf{\Delta_j}{2} } }
\frac{1}{ \mu^2 + \lambda_j^{-\vsg} - \lambda_i^{\vsg}  }  \, \mathrm{d}\mu  \, = \,
2 \underbrace{  \sqrt{ \tfrac{  \lambda_j^{-\vsg} - \lambda_j^{\vsg} }{  \lambda_j^{-\vsg} - \lambda_i^{\vsg} } } }_{ \leq 1 }
\underbrace{ \Int{0}{ \sqrt{  \tf{\Delta_j}{ [2(\lambda_j^{-\vsg} - \lambda_i^{\vsg})]} }} \hspace{-4mm} \frac{ \mathrm{d}\mu }{\mu^2 + 1 }
   }_{ \leq \frac{\pi }{2} } \leq \pi \, .
\end{align*}
Now, consider the case $i > j$. Then, as before, one can bound $\lambda_i^{\vsg} - \lambda_j^{-\vsg} \geq \Delta_j$, thus ensuring that
$$\ln\bigg| \frac{\lambda_i^{\vsg} -\lambda_j^{-\vsg} }{ \lambda_i^{\vsg} -\lambda_j^{-\vsg} - \frac{1}{2}\Delta_j} \bigg| \leq \ln 2 \, .$$
Finally, the integral term goes as
\begin{align*}
2\sqrt{ \frac{\lambda_j^{-\vsg} - \lambda_j^{\vsg} }{  \lambda_i^{\vsg} - \lambda_j^{-\vsg} }  } \hspace{6mm}
\bigg|  \hspace{-8mm} \Int{0}{ \sqrt{ \tf{\Delta_j }{ [2(  \lambda_i^{\vsg} - \lambda_j^{-\vsg})] }  } }   \hspace{-8mm}
\frac{ \mathrm{d}\mu  }{ \mu^2 - 1 }  \bigg|
\leq 2\sqrt{ \frac{\lambda_j^{-\vsg} - \lambda_j^{\vsg} }{  \lambda_i^{\vsg} - \lambda_j^{-\vsg} }  } \,
\bigg|  \Int{0}{ \frac{1}{\sqrt{2}} }  \frac{   \mathrm{d}\mu  }{ \mu^2 - 1 }   \bigg|
= \sqrt{ \frac{\lambda_j^{-\vsg} - \lambda_j^{\vsg} }{  \lambda_i^{\vsg} - \lambda_j^{-\vsg} }  }  \, \ln\frac{\sqrt{2}+1}{\sqrt{2}-1}
\end{align*}

\vspace{2mm}
  $\bullet$ {\it Case $\sg_j=1$ and $\ups_j = \vsg$.}
\vspace{2mm}
One has, by a similar argument,
\begin{align*}
\big|  \big( \mf{m}_{1}^{\vsg}\big)_{ij}  \big|  \leq
\bigg| \Int{ \lambda^{\vsg}_{j+1} - \frac{\Delta_j}{2} }{ \lambda^{\vsg}_{j+1} }
\frac{\mathrm{d}\mu  }{\mu - \lambda_i^{\vsg} } \bigg|
= \Bigg|  \ln \bigg| \frac{ \lambda_{j+1}^{\vsg} - \lambda_i^{\vsg} }
                { \lambda_{j+1}^{\vsg} - \lambda_i^{\vsg} - \frac{1}{2}\Delta_j } \bigg|\, \Bigg| \, .
\end{align*}
When $j+1 > i$, one can bound $\lambda_{j+1}^{\vsg}  - \lambda_i^{\vsg} \,  \geq \,  \Delta_j$ which entails that
$$\big|  \big( \mf{m}_{1}^{\vsg}\big)_{ij}  \big|  \leq \ln 2 \, .$$

\noindent Finally, if $i > j+1$, one bounds
$$ \big|  \big( \mf{m}_{1}^{\vsg}\big)_{ij}  \big|  \, \leq  \,
\ln|\lambda_i^{\vsg} - \lambda_j^{\vsg} | - \ln| \lambda_i^{\vsg} - \lambda_{j+1}^{\vsg} |  \, .$$

\vspace{2mm}
  $\bullet${\it Case  $\sg_j=1$ and $\ups_j = -\vsg$.}
\vspace{2mm}
To start with, it holds
\begin{align*}
 \big|  \big( \mf{m}_{1}^{\vsg}\big)_{ij}  \big|  \,  \leq \,
 \Bigg| \ln \bigg| \frac{\lambda_{j+1}^{-\vsg} - \lambda_i^{\vsg} }{ \lambda_{j+1}^{-\vsg} - \lambda_i^{\vsg} - \frac{1}{2}\Delta_j} \bigg|
- 2\sqrt{\lambda_{j+1}^{\vsg} - \lambda_{j+1}^{-\vsg}}  \hspace{-3mm}
\Int{0}{ \sqrt{  \tf{\Delta_j}{2} }  }  \hspace{-3mm} \frac{ \mathrm{d}\mu }{\mu^2 + \lambda_i^{\vsg} - \lambda_{j+1}^{-\vsg}}   \, \Bigg| \, .
\end{align*}
If $i > j+1$, then in the same manner as before
$$  \big|  \big( \mf{m}_{1}^{\vsg}\big)_{ij}  \big| \leq \ln|\lambda_i^{\vsg} - \lambda_j^{\vsg} |
- \ln|\lambda_i^{\vsg} - \lambda_{j+1}^{\vsg}|  +  \pi \, .$$
If $j+1 > i$ then $\Delta_j \leq \lambda_{j+1}^{-\vsg} - \lambda_i^{\vsg}$, and so by the same methods as before
$$ \big|  \big( \mf{m}_{1}^{\vsg}\big)_{ij}  \big|  \, \leq \, \ln 2  + \sqrt{ \tfrac{ \lambda_{j+1}^{\vsg} - \lambda_{j+1}^{-\vsg}}
{ \lambda_{j+1}^{-\vsg} - \lambda_{i}^{\vsg} } }  \ln \frac{ \sqrt{2} + 1 }{ \sqrt{2} - 1 } \, .$$
This concludes the proof.
\end{proof}

We are  now finally in position to bound $\mathfrak{A}^\vsg_j$.

\begin{prop}\label{Mprodbound}
There exists $C > 0$, possibly depending on $\gamma > 0$, such that uniformly in $\bs{\sg}_{N-1} \in \{ 0, 1\}^{N-1}$
\beq
\prod_{\substack{j=1 \\ j \not\in X_\downarrow(\bs{\sg}_{N-1})} }^{N-1} \hspace{-6mm}
\Big\{ \mathfrak{A}^{\vsg}_j(\bs{\sg}_{N-1}) \Big\}^{\f{1}{2}}
\leq C \prod_{j=1}^N \bigg( 1 + \frac{\ln(|\lambda_{j}^{\vsg}| + 1)}{\gamma N }  \bigg)^2
\prod_{j=1}^{N} \bigg( 1 + \frac{3}{\gamma N^2} \sum_{\substack{i=1 \\ i \neq j}}^N \ln^2|\lambda_i^{\vsg} - \lambda_j^{\vsg}| \bigg) \, .
\enq
\end{prop}

\begin{proof}

Using $1+a+b \leq (1+a)(1+b)$ for $a,b\geq 0$ allows one to bound \eqref{definition mf A j vsg de sigma N-1} as
\bem
\mathfrak{A}^{\vsg}_j(\bs{\sg}_{N-1})   \leq
\bigg( 1 + \frac{\ln 3}{\gamma N } \bigg)^2
\bigg( 1 + \frac{\ln(|\lambda_{j}^{\vsg}| + 1)}{\gamma N }  \bigg)^2  \\
\times \bigg( 1 + \frac{\ln(|\lambda_{j+1}^{\vsg}| + 1)}{\gamma N }  \bigg)^2
\bigg( 1 + \sum_{\substack{i=1 \\ i \neq j+\sigma_j}}^N \hspace{-3mm} \frac{ | (M^{\vsg}_{\bs{\sg}_{N-1}} )_{ij} |^2}{\gamma^2 N^2}  \bigg) \, .
\end{multline}
Hence, using $(1+x)\le \mathrm{e}^{x}$, we get
\begin{equation}\label{bbn}
\prod_{\substack{j=1 \\ j \not\in X_\downarrow(\bs{\sg}_{N-1})}}^{N-1} \hspace{-6mm}
\Big\{ \mathfrak{A}^{\vsg}_j(\bs{\sg}_{N-1})  \Big\}^{\f{1}{2}}
\leq \mathrm{e}^{\frac{\ln 3}{\gamma}} \prod_{j=1}^N \left( 1 + \frac{\ln(|\lambda_{j}^{\vsg}| + 1)}{\gamma N }  \right)^2
\hspace{-4mm}\prod_{\substack{j=1 \\ j \not\in X_\downarrow(\bs{\sg}_{N-1})}}^{N-1} \hspace{-4mm}
\bigg\{ 1 + \sum_{\substack{i=1 \\ i \neq j+\sigma_j}}^N\hspace{-3mm}  \frac{  | (M^{\vsg}_{\bs{\sg}_{N-1}} )_{ij} |^2 }{\gamma^2 N^2}  \bigg\}^{\f{1}{2}} \, .
\end{equation}
Next,  one bounds the last term in the above right-hand side
\begin{align*}
\sum_{\substack{i=1 \\ i \neq j+\sigma_j}}^N | (M^{\vsg}_{\bs{\sg}_{N-1}} )_{ij} |^2 \, =  \,
\sum_{\substack{i = 1 \\ i \neq j,j+1}}^N  | (M^{\vsg}_{\bs{\sg}_{N-1}} )_{ij} |^2
\, + \,  | (M^{\vsg}_{\bs{\sg}_{N-1}} )_{j+1-\sg_j , j } |^2 \, .
\end{align*}
For $i \not\in \{ j,j+1 \}$, one has
\beq
| (M^{\vsg}_{\bs{\sg}_{N-1}} )_{ij} |^2  \leq \big|  \big( \mf{m}_{0}^{\vsg}\big)_{ij}  \big|^2
\, + \, \big|  \big( \mf{m}_{1}^{\vsg}\big)_{ij}  \big| ^2 \,,
\enq
with $ \mf{m}_{a}^{\vsg}$ as introduced in \eqref{definition element de matrice mathfrak m a ij}.
Further
\beq
| (M^{\vsg}_{\bs{\sg}_{N-1}} )_{j+1-\sigma_j, j}|^2 = \left\{ \ba{ccc}
\big|  \big( \mf{m}_{0}^{\vsg}\big)_{j+1j}  \big|^2 & \e{if} & \sigma_j = 0  \vspace{2mm}\\
\big|  \big( \mf{m}_{1}^{\vsg}\big)_{jj}  \big|^2 &   \e{if} & \sigma_j = 1 \ea \right.  \; ,
\enq
thus, leading to the bound
$| (M^{\vsg}_{\bs{\sg}_{N-1}} )_{j+1-\sigma_j, j}|^2 \leq \big|  \big( \mf{m}_{0}^{\vsg}\big)_{j+1j}  \big|^2 +
\big|  \big( \mf{m}_{1}^{\vsg}\big)_{jj}  \big|^2$.
Inserting the estimates obtained in Proposition \ref{matrixelementbounds}, one gets
\begin{align*}
&\sum_{\substack{i=1 \\ i \neq j+\sigma_j}}^N | (M^{\vsg}_{\bs{\sg}_{N-1}} )_{ij} |^2  \\
&\leq \sum_{i=1}^{j-1} \left[ \ln|\lambda_{j+1}^{\vsg} - \lambda_i^{\vsg}| - \ln|\lambda_j^{\vsg} - \lambda_i^{\vsg}| + \pi \right]^2
+ \sum_{i=j+2}^{N} \left[ \ln 2 + \mathbbm{1}_{\{ \ups_j = -\vsg\} } \sqrt{\tfrac{|\lambda_j^{-\vsg} - \lambda_j^{\vsg}|}{\lambda_i^{\vsg} - \lambda_j^{-\vsg} }} \ln \frac{\sqrt{2}+1}{\sqrt{2}-1} \right]^2 \\
&+ \sum_{i=1}^{j-1} \left[ \ln 2 + \mathbbm{1}_{\{ \ups_j = -\vsg \} } \sqrt{\tfrac{|\lambda_{j+1}^{-\vsg} - \lambda_{j+1}^{\vsg}|}
{\lambda_{j+1}^{-\vsg} - \lambda_i^{\vsg} }} \ln \frac{\sqrt{2}+1}{\sqrt{2}-1} \right]^2
+ \sum_{i=j+2}^{N} \left[  \ln|\lambda_i^{\vsg} - \lambda_j^{\vsg}| - \ln|\lambda_i^{\vsg} - \lambda_{j+1}^{\vsg}| + \pi \right]^2  \\
&+ \bigg( \ln 2 + \mathbbm{1}_{\{ \ups_j = -\vsg \} } \sqrt{\tfrac{|\lambda_j^{-\vsg} - \lambda_j^{\vsg}|}{ \lambda_{j+1}^{\vsg} - \lambda_j^{-\vsg} }}
\ln \frac{\sqrt{2}+1}{\sqrt{2}-1} \bigg)^2
+ \bigg( \ln 2 + \mathbbm{1}_{\{ \ups_j = -\vsg\} } \sqrt{\tfrac{|\lambda_{j+1}^{-\vsg} - \lambda_{j+1}^{\vsg}|}{\lambda_{j+1}^{-\vsg} - \lambda_j^{\vsg}}}
\ln \frac{\sqrt{2}+1}{\sqrt{2}-1} \bigg)^2 \, .
\end{align*}
Combining terms and applying Jensen's inequality (so that $(\sul{j=1}{k} |a_j|)^2 \leq k \sul{j=1}{k} a_j^2$) we have
\begin{align*}
\sum_{\substack{i=1 \\ i \neq j+\sigma_j}}^N| (M^{\vsg}_{\bs{\sg}_{N-1}} )_{ij} |^2   &  \leq
3 \sum_{\substack{i=1 \\ i \neq j+1}}^{N}  \ln^2|\lambda_{j+1}^{\vsg} - \lambda_i^{\vsg}|
+  3 \sum_{\substack{i=1 \\ i \neq j}}^{N} \ln^2|\lambda_{j}^{\vsg} - \lambda_i^{\vsg}|   \\
&+2 \mathbbm{1}_{\{ \ups_j = - \vsg\} } \bigg( \ln \frac{\sqrt{2}+1}{\sqrt{2}-1} \bigg)^2
\bigg[ \sum_{i=j+1}^{N} \frac{|\lambda_j^{-\vsg} - \lambda_j^{\vsg} | }{ \lambda_i^{\vsg} - \lambda_j^{-\vsg} }
+ \sum_{i=1}^{j} \frac{ |\lambda_{j+1}^{-\vsg} - \lambda_{j+1}^{\vsg} | }{ \lambda_{j+1}^{-\vsg} - \lambda_i^{\vsg}  } \bigg]  \\
&+ \big(3  \pi^2 + 2 (\ln 2)^2 \big) N   \, .
\end{align*}

By invoking Lemma \ref{reciprocalbound}, we deduce that there exists a $\bs{\sg}_{N-1}$-independent $C > 0$  such that
\begin{align*}
\sum_{\substack{i=1 \\ i \neq j+\sigma_j}}^N | (M^{\vsg}_{\bs{\sg}_{N-1}} )_{ij} |^2 &\leq
3 \sum_{\substack{i=1 \\ i \neq j+1}}^{N}  \ln^2|\lambda_{j+1}^{\vsg} - \lambda_i^{\vsg}|
+  3 \sum_{\substack{i=1 \\ i \neq j}}^{N}  \ln^2|\lambda_{j}^{\vsg} - \lambda_i^{\vsg}| +C N   \, .
\end{align*}
Plugging this estimate into \eqref{bbn} proves the claim.

\end{proof}

It remains to bound $\mathfrak{B}_j^{\vsg}(\bs{\sg}_{N-1})$.

\begin{prop}
\label{Proposition Borne sup sur produit des B sigma}

There exists $C > 0$ large enough, in particular satisfying $\frac{C}{\gamma} \geq 1$, so that
 uniformly in $\bs{\sg}_{N-1} \in \{0,1\}^{N-1}$
\begin{equation}\label{Rprodbound}
\hspace{-6mm} \prod_{j \in X_\downarrow(\bs{\sg}_{N-1})} \hspace{-4mm} \Big\{ \frac{ \mathfrak{B}_j^{\vsg}(\bs{\sg}_{N-1}) }{\gamma N} \Big\}^{\f{1}{2}}
  \leq   \hspace{-4mm} \prod_{j \in X(\bs{\sg}_{N-1})}  \hspace{-5mm} \bigg\{ \frac{C}{\gamma}
+  \frac{6}{\gamma N} \sul{s=0}{2} \Big[ \ln^2(|\lambda^{\vsg}_{j+s} | + 1) + \,
\sul{\substack{i = 1 \\ i \neq j+s }}{N} \ln^2 |\lambda_i^{\vsg}  - \lambda_{j+s}^{\vsg}|  \Big] \bigg\} \;.
\end{equation}

\end{prop}

\begin{proof}

For any $j \in X_\downarrow(\bs{\sg}_{N-1})$ one has $(\sigma_j , \sigma_{j+1}) = (1 , 0)$. One thus separates
\begin{align*}
\mathfrak{B}_j^{\vsg}(\bs{\sg}_{N-1}) = \sum_{\substack{i=1 \\ i \neq j+1}}^N | (R^{\vsg}_{\bs{\sg}_{N-1}} )_{ij} |^2
\, +  \,  | (R^{\vsg}_{\bs{\sg}_{N-1}} )_{j+1 j} |^2
\end{align*}
where
\begin{align*}
(R^{\vsg}_{\bs{\sg}_{N-1}} )_{j+1j}    \, =  \, (M^{\vsg}_{\bs{\sg}_{N-1}} )_{j+1 j} \,  -  \,  (M^{\vsg}_{\bs{\sg}_{N-1}} )_{j+1 j+1}
\, =  \,
\ln\left|\Phi\left(\frac{1}{2}\frac{\Delta_j}{\delta_{j+1}}\right)\right|
- \ln\left|\Phi\left(\frac{1}{2}\frac{\Delta_{j+1}}{\delta_{j+1}}\right)\right| \, .
\end{align*}
Corollary \ref{matrixelementbound2} thus ensures that
\begin{equation}\label{matrixelementbound3}
\begin{split}
\big| (R^{\vsg}_{\bs{\sg}_{N-1}} )_{j+1j}   \big|
& \leq   \ln 3(N+1)  + \ln(|\lambda^{\vsg}_{j+2}| + 1) + \ln(|\lambda^{\vsg}_{j+1}| + 1) + \ln(|\lambda^{\vsg}_{j}| + 1)  \\
%
&\quad + |\ln|\lambda_{j+2}^{\vsg} - \lambda_{j+1}^{\vsg}||+   |\ln|\lambda_{j+1}^{\vsg} - \lambda_{j}^{\vsg}|| \, .
\end{split}
\end{equation}
The key point here is that the $\ln|\delta_{j+1}|$ term which is $\mathrm{O}(N)$ cancels, leaving only $\mathrm{O}(\ln N)$.
This is what motivated the column operation. Then
\beqa
\mathfrak{B}_j^{\vsg}(\bs{\sg}_{N-1}) & \leq  &
  2\sum_{\substack{i=1\\ i \neq j+1}}^N
  \Big[   \big|(M^{\vsg}_{\bs{\sg}_{N-1}} )_{i j}\big|^2 \, + \,  \big|(M^{\vsg}_{\bs{\sg}_{N-1}} )_{i j+1} \big|^2 \Big]
  \, + \,  \big| (R^{\vsg}_{\bs{\sg}_{N-1}} )_{j+1j}   \big|^2 \\
&=&  2\sum_{\substack{i=1\\ i \neq j+1}}^N \Big[   |(\mf{m}^{\vsg}_{1})_{ij}|^2 + | (\mf{m}^{\vsg}_{0})_{ij+1} |^2 \Big]
  \, + \,  \big| (R^{\vsg}_{\bs{\sg}_{N-1}} )_{j+1j}   \big|^2  \nonumber
%
%
%
%
\eeqa
Then, inserting the bounds provided by Proposition \ref{matrixelementbounds} and those given in (\ref{matrixelementbound3}), one gets
\begin{scriptsize}
\begin{align*}
\mathfrak{B}_j^{\vsg}(\bs{\sg}_{N-1})  &\leq  2\sum_{i=1}^j \bigg[   \Big|\ln 2
+ \mathbbm{1}_{\{ \ups_j = -\vsg\}} \sqrt{ \tfrac{|\lambda_{j+1}^{\vsg} - \lambda_{j+1}^{-\vsg} | }{ \lambda_{j+1}^{-\vsg} \, - \, \lambda_{i}^{\vsg} } }
\ln \frac{\sqrt{2}+1}{\sqrt{2}-1}\Big|^2
+ \Big| \ln|\lambda_{j+2}^{\vsg} - \lambda_i^{\vsg}|  - \ln|\lambda_{j+1}^{\vsg} - \lambda_i^{\vsg}|  + \pi \Big|^2 \bigg] \\
&+ 2\sum_{i=j+2}^N \bigg[   \Big| \ln|\lambda_{i}^{\vsg} - \lambda_j^{\vsg}|  - \ln|\lambda_i^{\vsg} - \lambda_{j+1}^{\vsg}|  + \pi  \Big|^2
+ \Big| \ln 2 + \mathbbm{1}_{\{ \ups_{j+1}= - \vsg \}}
\sqrt{ \tfrac{|\lambda_{j+1}^{-\vsg} - \lambda_{j+1}^{\vsg}| }{  \lambda_i^{\vsg} - \lambda_{j+1}^{-\vsg} }}
\ln\frac{\sqrt{2}+1}{\sqrt{2}-1}  \Big|^2 \bigg] \\
&+ \bigg[  \ln [ 3(N+1)]  \, + \, \sul{s=0}{2} \ln(|\lambda^{\vsg}_{j+s} | + 1)
\, +\, \sul{s=0}{1}\big| \ln|\lambda_{j+1+s}^{\vsg} - \lambda_{j+s}^{\vsg}| \big|  \bigg]^2 \\
%
%
%
%
%
%
%
%
%
%
%
%
%
%
&\leq 2\sum_{i=1}^j \bigg[   2\mathbbm{1}_{\{ \ups_j = -\vsg\}} \tfrac{ |\lambda_{j+1}^{\vsg} - \lambda_{j+1}^{-\vsg} |}
{\lambda_{j+1}^{-\vsg} - \lambda_{i}^{\vsg}}
\Big( \ln \tfrac{\sqrt{2}+1}{\sqrt{2}-1}\Big)^2
\, +\,  3 \sul{s=1}{2}\ln^2|\lambda_{j+s}^{\vsg} - \lambda_i^{\vsg}|   \bigg] \\
& + 2 \sum_{i=j+2}^N \bigg[   3 \sul{s=0}{1}  \ln^2|\lambda_{i}^{\vsg} - \lambda_{j+s}^{\vsg}| 
+ 2\mathbbm{1}_{\{ \ups_{j+1}= -\vsg \}} \tfrac{|\lambda_{j+1}^{-\vsg} - \lambda_{j+1}^{\vsg}|}{\lambda_i^{\vsg} - \lambda_{j+1}^{-\vsg} }
\Big( \ln\frac{\sqrt{2}+1}{\sqrt{2}-1}\Big)^2    \bigg]\\
&+ CN +   6 \sul{s=0}{2} \ln^2 (|\lambda^{\vsg}_{j+s} | + 1)
+ 6 \sul{s=0}{1}\ln^2|\lambda_{j+1+s}^{\vsg} - \lambda_{j+s}^{\vsg}|     \\
&\leq CN + 6 \sul{s=0}{2} \sul{\substack{i = 1 \\ i \neq j + s }}{N}
\ln^2|\lambda_i^{\vsg} - \lambda_{j+s}^{\vsg}|    \, +\,  6 \sul{s=0}{2}\ln^2 (|\lambda^{\vsg}_{j+s} | + 1)  \\
&+ 4 \Big( \ln\tfrac{\sqrt{2}+1}{\sqrt{2}-1}\Big)^2 \bigg[ \mathbbm{1}_{\{ \ups_j = -\vsg \}}
\sum_{i=1}^j  \frac{|\lambda_{j+1}^{\vsg} - \lambda_{j+1}^{-\vsg} |}{\lambda_{j+1}^{-\vsg} - \lambda_{i}^{\vsg}}
+  \mathbbm{1}_{\{ \ups_{j+1}= -\vsg\}}
\sum_{i=j+2}^N \frac{|\lambda_{j+1}^{-\vsg} - \lambda_{j+1}^{\vsg}|}{\lambda_i^{\vsg} - \lambda_{j+1}^{-\vsg} }  \bigg] \, .
\end{align*}
\end{scriptsize}
Thus, by invoking Lemma \ref{Lemme borne sup somme lambda pm pour borne des Bj sigma}, one concludes that
there exists a pure constant $C > 0$ such that for any $j \in X_\downarrow( \bs{\sg}_{N-1} )$
\beq
\mathfrak{B}_j^{\vsg}( \bs{\sg}_{N-1} ) \, \leq  \,   CN
\, + \, 6\sul{s=0}{2} \bigg[   \ln^2 (|\lambda^{\vsg}_{j+s} | + 1) \, + \,
 \sum_{\substack{i = 1 \\ i \neq j+s }}^N \ln^2|\lambda_i^{\vsg} - \lambda_{j+s}^{\vsg}|  \bigg] \, .
\enq

It remains to take the square root of the bound and then the product over $j \in X_\downarrow( \bs{\sg}_{N-1} )$.
The latter can be bounded by a product over $X( \bs{\sg}_{N-1} )$, since all the terms in the product are greater than $1$.
For the same reason, one may drop the square roots arising there.

\end{proof}

\begin{prop}
\label{Proposition Upper bound on S}
Let $\mc{S}^{\vsg}$ be as given by \eqref{Sdef} and let $\gamma > 0$. Then, there exists a constant $C > 0$, possibly depending on
$\gamma$, such that for $N$ large enough
\bem
\mc{S}^{\vsg}  \leq \mathrm{e}^{C N^{1 - \frac{1}{4}}} \pl{j=1}{N} \big\{ \ln |\lambda_j^+ - \lambda_j^-|^{-1} \vee \gamma N \big\} \\
\times \prod_{j=1}^N  \bigg( 1 + \frac{6}{\gamma N  } \ln^2(|\lambda_{j}^{\vsg}| + 1) \bigg)^5
 \; \prod_{j=1}^{N} \bigg( 1 +
\frac{6}{(\gamma N)^{\frac{5}{4}}} \sum_{\substack{i=1 \\ i \neq j}}^N \ln^2|\lambda_i^{\vsg} - \lambda_j^{\vsg}| \bigg)^4   \, .
\end{multline}
\end{prop}
\begin{proof}
One starts by observing the bound
$|X_\downarrow(\bs{\sg}_{N-1})| \geq \frac{1}{2}|X(\bs{\sg}_{N-1})| - \frac{1}{2}$ and note the identity
$$\sum_{X \subset \intn{1}{N-2}} (\gamma N)^{-\frac{1}{4}|X|} \prod_{j\in X} a_j
\, =  \, \prod_{j=1}^{N-2} \Big(1 + (\gamma  N)^{- \frac{1}{4}} a_j \Big) \, .$$
Then, starting from Proposition \ref{Proposition upper bound on S via A et B mathfrak},
further invoking Propositions \ref{Mprodbound} and  \ref{Proposition Borne sup sur produit des B sigma},
along with the trival bounds $1+a + b \leq (1+a)(1+b)$ for $a,b\geq 0$, allows one to conclude.
\end{proof}

\begin{proof} [Proof of Proposition \ref{integralupperbound}]

Starting from the upper bound given in Lemma \ref{Lemme First upper bound on I}, one invokes
Proposition \ref{Proposition Upper bound on S} to conclude.

\end{proof}

\subsection{Proof of the upper bound on balls}

To prove the upper bound on balls we must now integrate our bound in Proposition \ref{integralupperbound}. We will first get preliminary upper bounds on the density and finally derive the weak large deviations upper bounds.

\subsubsection{Preliminary upper bounds} 

A key result that we will use is due to Henrici and Kappeler (Proposition B.1 of \cite{Kappeler})
which reads as follows.
\begin{prop}[Henrici-Kappeler]\label{kappeler} Let $\bs{\la}_N^+\in \mc{A}_N$. Then, for every $k\in \intn{1}{N}$,

$$|\lambda_k^+-\lambda_k^-|\le \frac{2\pi \mathrm{e}^{-\ell/2}}{N}\,.$$
\end{prop}
From this  Proposition \ref{kappeler}  and  \ref{potentialh1}, we have the following corollary (recall that we assumed without loss of generality that $V$ is non-negative). 
\begin{cor}\label{Vrelation}
For $N$ sufficiently large, there are constants $C, \tilde{C} > 0$ such that
$$\sul{k=1}{N} |V(\lambda_k^+) - V(\lambda_k^-)|  \leq \frac{C}{N} \sul{k=1}{N}
\big\{ {V}(\lambda_k^+) + V(\lambda_k^-) \big\} + \tilde{C} \, .$$
\end{cor}
\begin{proof}
Using  \ref{potentialh1} with $x= \lambda_k^+$ and $\lambda_k^-$ we have
$$ |V(\lambda_k^+) - V(\lambda_k^-)|  \leq \bigg(\frac{C_{1} }{2} \{ V(\lambda_k^+) + V(\lambda_k^-) \} +C_{2} \bigg)| \lambda_k^+-\lambda_k^-|  $$
and the result follows from Proposition \ref{kappeler} with $C= \pi \mathrm{e}^{-\ell/2} C_{1} $ and $\tilde{C} =2\pi \mathrm{e}^{-\ell/2} C_{2} $.
\end{proof}
Moving forward, fix $0 < \eps \leq \frac{1}{2}$ and a bounded Lipschitz function $\phi : \mathbb{R} \to \mathbb{R}$ and define
\begin{align}\label{Wchoice}
    W(x) = \eps V(x) - \phi(x) + \ln\Big( \Int{\mathbb{R}}{} \mathrm{e}^{\phi(y) - \eps V(y)}\, \mathrm{d}y  \Big) \, .
\end{align}
Note that by construction $\Int{\mathbb{R}}{} \mathrm{e}^{-W(x)} \, \mathrm{d}x = 1$. We will later introduce this density $\mathrm{e}^{-W}$ as an auxiliary construct to make sure that some  integrals converge (which is why we take $\eps>0$) but also to get bounds that are uniform in the function $\phi$, allowing ultimately to optimise over this function to obtain the relative entropy, in a way similar to the proof of Sanov's theorem (Theorem 6.2.10 of \cite{DemboZ01}).
Next, let $M > 0$, then define the regularised logarithm
\beq
\ln_{M} : \mathbb{R}\longrightarrow \mathbb{R}
\quad \e{with} \quad \ln_{M}(x) = \ln \max \big\{ x, \frac{1}{M} \big\} \, .
\label{definition ln M regularise}
\enq
Note that $\ln_{M}$ is defined for all real arguments, and $| \ln_{M}(x) - \ln_{M}(y)| \leq M |x-y|$.
\begin{lemme}\label{Jcontinuous}
Let  $\eta, \kappa, M_0, M_1, M_2,  M_3> 0$, and assume that  $\kappa < \frac{1}{2}$.
Fixing the notation $\mathbf{M}=(M_0,\dots, M_3)$, define the continuous and bounded function
$$f_{\kappa, \mathbf{M}}(x,y) \overset{\mathrm{def}}{=}
\max\Big\{   \ln_{M_0} |x - y| - \frac{\ell \kappa}{4M_1} V(x)  - \frac{\ell \kappa}{4M_1} V(y),
- M_{2} \Big\}  \, .$$
Next, given $\bs{\mf{r}} \, = \, (\eta, \kappa,\mathbf{M})$
define the functional $J_{\bs{\mf{r}}} : \mathcal{M}_1(\mathbb{R}) \longrightarrow \mathbb{R}$
%
%
\beqa
J_{\bs{\mf{r}}}[\mu]  &\overset{\mathrm{def}}{=} & \Int{\mathbb{R}}{}\ln_{M_1}\Big( 1+ \frac{2}{\ell}
\Int{\mathbb{R}}{} f_{\kappa, \mathbf{M}}(x,y) \, \mathrm{d}\mu (y)+ \eta \Big) \, \mathrm{d}\mu(x) \nonumber \\
 && \qquad + \Int{\mathbb{R}}{}\max\Big\{  -(1-\kappa) V(x)  + W(x),  - M_3\Big\} \,  \mathrm{d}\mu(x) \, .
\label{definition fnelle Jr}
\eeqa
Then $J_{\bs{\mf{r}}}$ is continuous with respect to the weak topology.
\end{lemme}
\begin{proof}
$$\mu \mapsto \Int{\mathbb{R}}{}\max\Big\{  - (1-\kappa) V(x)  + W(x),  - M_3\Big\} \,  \mathrm{d}\mu(x)$$
is manifestly continuous since $\max\big\{  - (1-\kappa) V   + W,  - M_3\big\}$ is bounded and continuous
(by virtue of $\eps+\kappa <1 $).
$$\mu \mapsto
\Int{\mathbb{R}}{}\ln_{M_1}\Big( 1+ \frac{2}{\ell} \Int{\mathbb{R}}{} f_{\kappa, \mathbf{M}}(x,y)
\, \mathrm{d}\mu (y)+ \eta \Big) \, \mathrm{d}\mu(x)$$
is continuous by the same argument as in the proof of Lemma \ref{Itildecontinuous}.
\end{proof}
\begin{lemme}\label{regularisedbound}
Pick  $\eta, \kappa, M_0, M_1, M_2,  M_3> 0$, with  $\kappa < \frac{1}{2}$, set $\mathbf{M}=(M_0,\dots, M_3)$,
denote  $\bs{\mf{r}} \, = \, (\eta, \kappa,\mathbf{M})$. Let  $\gamma = \frac{\ell}{2 M_1}$. Then, for $N$ sufficiently large,
it holds
\beq
\prod_{j=1}^{N}    \Big\{  \big\{  \ln |\lambda_j^+ - \lambda_j^- |^{-1}  \vee \gamma N \big\} \,
\cdot \mathrm{e}^{- V(\lambda_j^\pm) + \frac{\kappa}{2}V(\lambda_j^\pm) + W(\lambda_j^\pm)}  \Big\} \, \leq \,
|\ln(\veps_N)|^{N} \,  \mathrm{e}^{N J_{\bs{\mf{r}}}[ \op{L}_N^{(\boldsymbol{\lambda}_N^\pm)}] } \, .
\enq
with $J_{\bs{\mf{r}}}$ as given by in \eqref{definition fnelle Jr}.

\end{lemme}
\begin{proof}
Starting from $P^+-P^-=4\veps_N$, one gets
\begin{align*}
&\pl{j=1}{N}     \big\{  \ln |\lambda_j^+ - \lambda_j^-|^{-1} \vee  \gamma N \big\} \\
&= |\ln(\veps_N)|^N \pl{j=1}{N}     \max\Big\{ 1  - \frac{\ln 4}{|\ln(\veps_N)|} +
\frac{1}{| \ln(\veps_N)|}
\sul{\substack{k=1 \\ k \neq j}}{N} \ln |\lambda_j^{\ups} - \lambda_k^{-\ups}|,  \frac{\gamma N}{|\ln(\veps_N)|} \Big\} \\
&\leq   |\ln(\veps_N)|^N \pl{j=1}{N}     \max\Big\{
1 + \frac{1}{| \ln(\veps_N)|} \sul{\substack{k=1 \\ k \neq j}}{N}
\ln |\lambda_j^{\ups} - \lambda_k^{-\ups} |,  \frac{\gamma N}{|\ln(\veps_N)|} \Big\} \\
&\leq   |\ln(\veps_N)|^N \pl{j=1}{N}     \max\Big\{ 1  +  \frac{1}{| \ln(\veps_N)|}
\sul{\substack{k=1 \\ k \neq j}}{N} \ln_{M_0} |\lambda_j^{\ups} - \lambda_k^{-\ups} |,  \frac{\gamma N}{|\ln(\veps_N)|} \Big\} \\
&=   |\ln(\veps_N)|^N \pl{j=1}{N}     \max\Big\{ 1  + \frac{1}{| \ln(\veps_N)|} \sul{\substack{k=1 }}{N}
\ln_{M_0} |\lambda_j^{\ups} - \lambda_k^{-\ups} |  + \frac{\ln M_0}{ |\ln(\veps_N)|},  \frac{\gamma N}{|\ln(\veps_N)|} \Big\} \;.
\end{align*}
Note that, in the intermediate steps, we have invoked  Proposition \ref{kappeler} which ensures that for $N$ sufficiently large, one has
$|\lambda_j^{\ups} - \lambda_j^{-\ups}| \leq \frac{1}{M_0}$ .  Again, by Proposition \ref{kappeler} and the Lipschitz property of $\ln_{M_0}$, we have
$$ \frac{1}{N}\sul{\substack{k=1 }}{N} \ln_{M_0} |\lambda_j^{\ups} - \lambda_k^{-\ups}| \, \leq \,
\Int{\mathbb{R}}{} \ln_{M_0} |\lambda_j^{\ups} - y| \, \mathrm{d}\op{L}_N^{(\boldsymbol{\lambda}_N^{\ups})}(y) + \frac{CM_0}{N} \, .$$
Further, given any fixed $\eta > 0$, one has that $\frac{ C M_0 }{ |\ln(\veps_N)| }  \, +  \, \frac{ \ln M_0 }{ |\ln(\veps_N)| }  \, \leq  \, \eta$
for $N$ sufficiently large. Hence
\bem
\prod_{j=1}^{N}     \max\{  \ln |\lambda_j^+ - \lambda_j^-|^{-1} , \gamma N \}   \\
\leq  |\ln(\veps_N)|^N
\exp\bigg\{  N\Int{\mathbb{R}}{}    \ln_{M_1} \Big(1 + \frac{2}{\ell} \Int{\mathbb{R}}{} \ln_{M_0} |x - y|\, \mathrm{d}
\op{L}_N^{(\boldsymbol{\lambda}_N^{\ups})}(y) + \eta\Big)  \, \mathrm{d}\op{L}_N^{(\boldsymbol{\lambda}_N^{\ups})}(x) \bigg\}\, .
\nonumber
\end{multline}
Finally, by using that the Lipschitz constant of $\ln_{M_1}$ is $M_1$ and that this function is defined on $\R$, one observes that
\begin{align*}
& \Int{\mathbb{R}}{} \ln_{M_1} \Big(1 + \frac{2}{\ell} \Int{\mathbb{R}}{} \ln_{M_0} |x - y|\,
\mathrm{d}\op{L}_N^{(\boldsymbol{\lambda}_N^{\ups})}(y) + \eta\Big)\, \mathrm{d}\op{L}_N^{(\boldsymbol{\lambda}_N^{\ups})} (x) \\
&\leq \Int{\mathbb{R}}{} \ln_{M_1} \Big(1 + \frac{2}{\ell}
\Int{\mathbb{R}}{} \Big[ \ln_{M_0} |x - y| -  \frac{\ell \kappa}{4M_1}V(x) -  \frac{\ell \kappa}{4M_1} V(y) \Big] \,
\mathrm{d} \op{L}_N^{(\boldsymbol{\lambda}_N^{\ups})}(y)  + \eta\Big) \, \mathrm{d}\op{L}_N^{(\boldsymbol{\lambda}_N^{\ups})} (x) \\
 &\quad + \frac{\kappa}{2} \Int{\mathbb{R}}{} V(x) \, \mathrm{d}\op{L}_N^{(\boldsymbol{\lambda}_N^{\ups})}(x) \, .
\end{align*}
\end{proof}
To move forward, we need to introduce a few auxiliary integrals. Given $U\in \mc{C}^0(\R)$ growing fast enough at infinity
and exponentially integrating to unity
\beq
U(x) \,  \geq \,   \frac{4}{7}  \ln^{14} (|x|+1)  + 2 \ln(|x|+1)+ \tilde{C} \quad \e{and} \quad
\Int{\mathbb{R}}{} \mathrm{e}^{-U(x)}\, \mathrm{d}x = 1 \, ,
\label{ecriture hypothese sur U}
\enq
for some $\tilde{C} \in \R$, for any $p \in \mathbb{N}$, consider the two $N$-fold integrals
\beqa
A_N^{(p)}[U] & = & \Int{ \mathbb{R}^N }{} \prod_{j=1}^N \Big\{ 1 + \frac{6}{\gamma N  } \ln^2(|x_j| + 1)  \Big\}^p  \pl{k=1}{N} \mathrm{e}^{-  U(x_k)}\, \mathrm{d}\boldsymbol{x}_N
\label{definition integrale AN libre}\\
B_N^{(p)}[U]  & = & \Int{ \mathbb{R}^N }{} \prod_{j=1}^{N} \Big\{ 1 + \frac{6}{(\gamma N)^{\frac{5}{4}}} \sul{\substack{i=1 \\ i \neq j}}{N} \ln^2|x_i - x_j| \Big\}^p
\pl{k=1}{N} \mathrm{e}^{-  U(x_k) }\, \mathrm{d}\boldsymbol{x}_N  \;.
\label{definition integrale BN libre}
\eeqa
These have constrained counterparts
\beqa
A_{N, \mathsf{c}}^{(p)}[U] &= &
\Int{ \mathbb{R}^N }{}   \de \big( \ov{\bs{x}}_N  \big)   \, \prod_{j=1}^N \Big\{ 1 + \frac{6}{\gamma N  } \ln^2(|x_j| + 1)  \Big\}^p  \pl{k=1}{N}\mathrm{e}^{-  U (x_k) }
\, \mathrm{d}\boldsymbol{x}_N  \label{definition integrale AN contrainte}\\
B_{N, \mathsf{c}}^{(p)}[U]  & = &
\Int{ \mathbb{R}^N }{}   \de\big( \ov{\bs{x}}_N  \big)   \prod_{j=1}^{N} \Big\{ 1 + \frac{6}{(\gamma N)^{\frac{5}{4}}} \sul{\substack{i=1 \\ i \neq j}}{N} \ln^2|x_i - x_j| \Big\}^p
\pl{k=1}{N} \mathrm{e}^{-  U (x_k) }\, \mathrm{d}\boldsymbol{x}_N
\label{definition integrale BN contrainte}
\eeqa
with $\ov{\bs{x}}_N=\sul{a=1}{N} x_a$.
The estimates of the upper bound on balls will strongly rely on having appropriate upper bounds on the large-$N$ behaviour of the above integrals.

\begin{lemme}\label{ABbound} Fix $p \in \mathbb{N}$, $6^{\frac{4}{5}} \geq \ga>0$ and $U \in \mc{C}^0(\R)$ satisfying \eqref{ecriture hypothese sur U}.
Then there exists $C > 0$ such that, for $N$-large enough, the $N$-fold integrals defined in \eqref{definition integrale AN libre}-\eqref{definition integrale BN libre}
are upper bounded as
\beq
A_N^{(p)}[U] \, \leq \, C \qquad and  \qquad B_{N}^{(p)}[U]  \, \leq \, \mathrm{e}^{C N^{\frac{7}{8}}  } \; .
\enq
\end{lemme}
\begin{proof}

The first integral factorises directly into a product of one-dimensional integrals
\begin{align*}
A_N^{(p)}[U]  &= \Big[ \Int{\mathbb{R}}{} \Big\{ 1 + \frac{6}{\gamma N  }\ln^2(|x| + 1)  \Big\}^{p} \mathrm{e}^{-U(x)} \, \mathrm{d}x \Big]^N \, .
\end{align*}
The upper bounds on $U$ ensure that $\Int{\mathbb{R}}{} \ln^m(1+|x|) \, \mathrm{e}^{-U(x)} \, \mathrm{d}x < +\infty$ for every $m \geq 0$. Hence,
by using a binomial expansion one infers that $A_N^{(p)}[U]   \leq \big( 1 + \tf{C}{N} \big)^N \, \leq \, \ex{C}$.

To deal with $ B_{N}^{(p)}[U] $, one first decomposes $\ln^2|x| = \mc{L}\e{n}^2_0|x| + \mc{L}\e{n}^2_1|x|$ wherein
$$ \mc{L}\e{n}_0|x| = \mathbbm{1}_{\{ |x| \leq 1\}  } \ln|x|  \qquad \e{and} \qquad
\mc{L}\e{n}_1|x| = \mathbbm{1}_{\{ |x| \geq 1\} } \ln|x| \, . $$
Then, one has that $\mc{L}\e{n}^2_1|x-y| \, \leq \, 2 \ln^2(1+|x|) +  2\ln^2(1+|y|)$.
Since  $6\gamma^{-\frac{5}{4}} \geq 1$, $(1+x)^{6\gamma^{-\frac{5}{4}} }\ge 1+6\gamma^{-\frac{5}{4}} x$ for $x\ge 0$, so that the following bounds hold
\begin{align*}
&\pl{j=1}{N}\Big\{ 1 + \frac{6}{(\gamma N)^{\frac{5}{4}}} \sum_{\substack{i=1 \\ i \neq j}}^N \ln^2|x_i - x_j| \Big\}^p \, \leq \,
\pl{j=1}{N} \Big\{ 1 + \frac{1}{N^{\frac{5}{4}}} \sum_{\substack{i=1 \\ i \neq j}}^N  \ln^2|x_i - x_j| \Big\}^{6p \gamma^{-\frac{5}{4}} } \\
&\leq \pl{j=1}{N}  \Big\{ 1 + \frac{1}{N^{\frac{5}{4}}} \sum_{\substack{i=1 \\ i \neq j}}^N \mc{L}\e{n}_0^2|x_i - x_j| \Big\}^{ 6p\gamma^{-\frac{5}{4}} }
   \\
& \times \Big\{ 1 + \frac{2}{N^{\frac{5}{4}}} \sum_{\substack{i=1 }}^N  \ln^2(1+|x_i|) \Big\}^{ 6N p\gamma^{-\frac{5}{4}} }
\pl{j=1}{N} \Big\{ 1 + \frac{2}{N^{\frac{1}{4}}}   \ln^2(1+|x_j|) \Big\}^{ 6p\gamma^{-\frac{5}{4}} }  \, .
\end{align*}
By applying $1+\sul{s=1}{N} a_s \leq \pl{s=1}{N}(1+a_s)$ and using $\big( 1 + x \big)^{\ga} \, \leq \, \big( 1 + x / \be \big)^{\ga \be} $ one gets that
\bem
 \Big\{ 1 + \frac{2}{N^{\frac{5}{4}}} \sum_{\substack{i=1 }}^N  \ln^2(1+|x_i|) \Big\}^{ 6 N p\gamma^{-\frac{5}{4}} }
  \pl{j=1}{N} \Big\{ 1 + \frac{2}{N^{\frac{1}{4}}}   \ln^2(1+|x_j|) \Big\}^{ 6p\gamma^{-\frac{5}{4}} }   \\
\, \leq \,  \pl{j=1}{N} \Big\{ 1 + \frac{4}{N^{\frac{5}{4}}} \ln^2(1+|x_j|) \Big\}^{ 6 N p\gamma^{-\frac{5}{4}} } \;.
\end{multline}

The Cauchy--Schwarz inequality leads to
$B_{N}^{(p)}[U] \leq \Big\{ C_{N}^{(12 p\gamma^{-\frac{5}{4}})}[U] \, D_{N}^{(12 p \gamma^{-\frac{5}{4} },\frac{5}{4})}[U] \Big\}^{\f{1}{2}}$ where
\beq
C_{N}^{(q)}[U] \, = \, \Int{ \mathbb{R}^N }{}  \pl{j=1}{N}\Big( 1 + \frac{1}{N^{\frac{5}{4}}} \sum_{\substack{i=1 \\ i \neq j}}^N \mc{L}\e{n}_0^2|x_i - x_j| \Big)^{ q }
\pl{k=1}{N} \mathrm{e}^{-  U (x_k) }\, \mathrm{d}\boldsymbol{x}_N
\label{definition integrale multiple CNp}
\enq
is upper bounded by Proposition \ref{CNbound} as  $C_{N}^{(q)}[U] \,  \leq \, \mathrm{e}^{C N^{\frac{3}{4}} \ln^2 N}$ for some $C > 0$.
Further, one has for $q,\alpha>0$
\beq
D_{N}^{(q,\alpha)}[U] = \Int{\mathbb{R}^N}{}  \pl{i=1 }{N} \Big\{ 1 + \frac{4}{N^{\alpha}} \ln^2(1+|x_i|) \Big\}^{N q} \pl{k=1}{N} \mathrm{e}^{ - U(x_k)} \, \mathrm{d}\boldsymbol{x}_N \;,
\nonumber
\enq
where $q = 12 p \gamma^{-\frac{5}{4} }$.
By applying the $N$-product version of H\"older's inequality, one gets,
\beq
D_{N}^{(q,\alpha)}[U] \, \leq  \, \Int{\mathbb{R}}{} \Big\{ 1 + \frac{4}{N^{\alpha}}  \ln^2(1+|x|) \Big\}^{N^2 q} \mathrm{e}^{-U(x)} \, \mathrm{d}x \;.
\enq
Substituting the lower bound on $U$, yields, after the change of variables $u=\ln (1+|x|)$
\begin{align*}
D_{N}^{(q,\alpha)}[U] \leq 2\mathrm{e}^{-\tilde{C}} \Int{\R^+}{ } \mathrm{e}^{4 q N^{2-\alpha} u^2 - \frac{4}{7} u^{14} - u  } \, \mathrm{d}u
\leq 2\mathrm{e}^{-\tilde{C}}  \mathrm{e}^{\sup_{u \geq 0 }[4 q N^{2-\alpha} u^2 - \frac{4}{7} u^{14} ]} \, .
\end{align*}
The function $u \mapsto 4 q N^{2-\alpha} u^2 - \frac{4}{7} u^{14}$ admits its maximum on $\R^+$ at $u^* = (q N^{2-\alpha})^{\frac{1}{12}}$, thereby yielding 
$\sup_{u \geq 0 }\Big\{ 4 q N^{2-\alpha} u^2 - \frac{4}{7} u^{14} \Big\} = \frac{24}{7} q^{\frac{7}{6}}N^{\frac{7}{6}(2-\alpha)} \,$ and therefore
\begin{equation}\label{boundD}
D_{N}^{(q,\alpha)}[U] \leq 2\mathrm{e}^{-\tilde{C}} \mathrm{e}^{\frac{24}{7} q^{\frac{7}{6}}N^{\frac{7}{6}(2-\alpha)}}\,.\end{equation}
In particular $D_{N}^{(12 p \gamma^{-\frac{5}{4} }, \frac{5}{4})}[U]  \leq \mathrm{e}^{C N^{\frac{7}{8}}}$ for some constant $C >0$, $N$ being sufficiently large, which gives the desired  bound 
on $B_{N}^{(p)}[U] $.

\end{proof}
\begin{lemme}\label{Acond}
Let $p \in \mathbb{N}$ and $\ga>0$ be fixed while  $U\in \mc{C}^0(\R)$ satisfies  \eqref{ecriture hypothese sur U}.
Then there exists a $C> 0 $ such that the $N$-fold integral introduced in \eqref{definition integrale AN contrainte}
\begin{align*}
&A_{N, \mathsf{c}}^{(p)}[U]   \leq C \, .
\end{align*}
\end{lemme}

\begin{proof}
Evaluating the $\delta$-function over the final variable $x_N$, applying the Cauchy--Schwarz inequality, and observing that $\mathrm{e}^{-U}$ is bounded, one finds
\begin{align*}
A_{N, \mathsf{c}}^{(p)}[U]  &\leq C \Bigg\{ \Int{\mathbb{R}^{N-1} }{} \prod_{j=1}^{N-1} \Big\{ 1 + \frac{6}{\gamma N  } \ln^2(|x_j| + 1)  \Big\}^{2p}
\pl{k=1}{N-1} \mathrm{e}^{-  U(x_k)}\, \mathrm{d}\boldsymbol{x}_{N-1}  \Bigg\}^{\f{1}{2}} \\
    &\quad \times \Bigg\{ \Int{\mathbb{R}^{N-1} }{} \Big\{ 1 + \frac{6}{\gamma N  } \ln^2(|x_1 + \dots + x_{N-1}| + 1)  \Big\}^{2p}
    \pl{k=1}{N-1}\mathrm{e}^{- U(x_k)}\, \mathrm{d}\boldsymbol{x}_{N-1} \Bigg\}^{\f{1}{2}}  \, .
\end{align*}
The first integral may be upper bounded by virtue of Lemma \ref{ABbound}.  Thus, one only needs to bound the second integral.
Consider a sequence of real-valued iid random variables $\{ X_i \}_{i=1}^N$, each distributed with density $\mathrm{e}^{-U}$.
Then introduce $Y = |X_1 + \dots + X_{N-1}|$ and define $f(x) = \big( 1 + \frac{6}{\gamma N} \ln^2(1+x) \big)^p$. $f$ is smooth on $\R^+$.
Then the second integral arising in the Cauchy--Schwarz estimate corresponds to $\mathbb{E}[f(Y)]$.
Moreover, for $x>(N-1) \mathbb{E}[X_1] $, Chebyshev's inequality implies
\begin{eqnarray*}
\mathbb{P}[ \, Y \geq x ]&\le& \mathbb{P} \Big[  \big| \sul{a=1}{N-1} X_a -  \mathbb{E}[X_a] \big| \geq x -(N-1) \mathbb{E}[X_1]  \Big]  \\
&\le  & \frac{1 }{(x-(N-1)  \mathbb{E}[X_1] )^{2}} \cdot  \mathrm{var}\big( \sul{a=1}{N} X_a \big) \, = \, \frac{  N }{(x-(N-1)  \mathbb{E}[X_1] )^{2}}   \mathrm{var}\big( X_1 \big)\,\end{eqnarray*}
where we note that $\mathrm{var}(X_1)<+\infty$ owing to the hypotheses on $U$. Moreover, 
an integration by parts yields
\begin{eqnarray*}
 \mathbb{E}[ f(Y)] &=  &
 \Int{0}{+\infty} f^\prime(x) \mathbb{P}\big[ Y  \geq x  \big] \, \mathrm{d}x + f(0) \\
 &\le & f((N-1) \mathbb{E}[X_1]+N )+\Int{(N-1) \mathbb{E}[X_1]+N}{+\infty} f'(x)   \, \frac{  N  \mathrm{var}\big( X_1 \big)}{(x-(N-1)  \mathbb{E}[X_1] )^{2}}  \, \mathrm{d}x
\, .\end{eqnarray*}
where we bounded $\mathbb{P}\big[ Y  \geq x  \big]$ from above by one on $[0, (N-1) \mathbb{E}[X_1]+N]$, using that $f'$ is non-negative.
Furthermore, $f(CN)$ is uniformly bounded in $N$ and the upper bound 
$f^\prime(x) \leq \frac{12p}{\gamma N} (1+ \frac{6}{\gamma N} \ln^2(1+x))^{p-1} $ shows that the last term in the above right-hand side is also uniformly bounded. 
\end{proof}

\begin{lemme}\label{Bcond}
Let $p \in \mathbb{N}$, $0 < \eps \leq \frac{1}{2}$ $\kappa>0$, let $V$ satisfy \ref{potentialh1} and \ref{potentialh2} and let $W$
be as introduced in \eqref{Wchoice}. 
Then, there exists $C>0$ such that the $N$-fold integral introduced in \eqref{definition integrale BN contrainte}
admits the upper bound $B_{N, \mathsf{c}}^{(p)}[W+\tfrac{\kappa}{4}V]\,  \leq \,  \mathrm{e}^{ C N^{ \frac{7}{8} }  }$ for $N$ large enough. This $C$ may depend on $p$.
\end{lemme}
\begin{proof}
One begins by evaluating the $\delta$-function with respect to $x_N$.
Then, for $h \in L^1(\R)$ being the density of a probability measure, one obtains
\bem
B_{N, \mathsf{c}}^{(p)}[W+\tfrac{\kappa}{4}V]  \, = \,
\Int{\mathbb{R}^{N-1}\times \mathbb{R} }{}
\prod_{j=1}^{N-1} \bigg\{ 1 + \frac{6}{(\gamma N)^{\frac{5}{4}}} \sum_{\substack{i=1 \\ i \neq j}}^{N-1} \ln^2|x_i - x_j|
+ \frac{6}{(\gamma N)^{\frac{5}{4}}} \ln^2\big| x_j + \ov{\bs{x}}_{N-1} \big|   \bigg\}^p \\
         \times \Big\{ 1 + \frac{6}{(\gamma N)^{\frac{5}{4}}} \sum_{i=1}^{N-1} \ln^2\big|x_i + \ov{\bs{x}}_{N-1} \big|  \Big\}^p
    \pl{k=1}{N-1} \mathrm{e}^{ - ( W + \frac{\kappa}{4} V)(x_k) } \\
        \times \mathrm{e}^{- ( W + \frac{\kappa}{4} V)(-\ov{\bs{x}}_{N-1} ) } \, h(\xi) \, \mathrm{d}\boldsymbol{x}_{N-1} \, \mathrm{d}\xi \, ,
\end{multline}
where we fix the notation $\ov{\bs{x}}_ {r}=\sul{s=1}{r} x_s$.
At this stage, one makes the substitutions: $x_i \to x_i -  \xi/N$ for all $i \in \intn{1}{N-1}$, followed by  $\xi \to \xi +  \ov{\bs{x}}_ {N-1} $.
Finally, one relabels $\xi$ as $x_N$. All-in-all, one gets
\bem
B_{N, \mathsf{c}}^{(p)}[W+\tfrac{\kappa}{4}V]  \, = \,
  \Int{\mathbb{R}^{N}  }{}  \prod_{j=1}^{N} \Big\{ 1 + \frac{6}{(\gamma N)^{\frac{5}{4}}} \sum_{\substack{i=1 \\ i \neq j}}^N \ln^2|x_i - x_j| \Big\}^p \\
\times \pl{k=1}{N}\mathrm{e}^{-  (W + \frac{\kappa}{4} V)(x_k - \f{ \overline{\bs{x}}_N}{N} ) }  h(\overline{\bs{x}}_N) \,  \mathrm{d}\boldsymbol{x}_N \;.
\end{multline}
At this stage, one makes the choice $h(\xi) = N^2 \mathbbm{1}_{\{ |\xi| \leq \frac{1}{2}N^{-2} \} }$, thus implying that the integration domain is restricted to
$|\overline{\bs{x}}_N| \leq N^{-2}/2$.
By \ref{potentialh1}, it holds that on this domain
\begin{align*}
-(W  + \frac{\kappa}{4}V)(x_k - \f{ \overline{\bs{x}}_N}{N}  ) \,  \leq \,  - W(x_k)
+ \frac{(\eps + \frac{\kappa}{4})C_2 + \| \phi \|_{\mathrm{Lip}}}{N}  + \Big( \frac{(\eps + \frac{\kappa}{4})C_1}{N} - \frac{\kappa}{4} \Big) V(x_k) \, .
\end{align*}
For  sufficiently large $N$, one furthermore has that $ \frac{(\eps + \frac{\kappa}{4})C_1}{N} - \frac{\kappa}{4} \leq 0$.
Inserting this above shows that  that $B_{N, \mathsf{c}}^{(p)}[W+\tfrac{\kappa}{4}V]  \leq C N^2 B_{N}^{(p)}[W+\tfrac{\kappa}{4}V]  $.
The bound for $B_{N}^{(p)}[U]$ provided by  Lemma \ref{ABbound} completes the proof.
\end{proof}

\subsubsection{Weak large deviation upper bound for the constrained model}

We first derive a preliminary  weak  large deviation upper bound for the constrained model. 

\begin{prop}\label{wldubc}
 Let $\mu \in   \mathcal{M}_{1,\mathsf{c}}(\mathbb{R})$. For any $\kappa, \eta, \tau>0$ small enough and $M_a>0$  large enough, there exists $C>0$ such that for $N$ large enough and $\delta$ small enough 
\beq
\overline{\Pi}_{N,\mathsf{c}}[B_\mathsf{c}(\mu,\delta)] \leq   \mathrm{e}^{C N^{\frac{7}{8}} }   \mathrm{e}^{N (J_{\bs{\mf{r}}}[\mu]+ \tau)}
\enq
for some $C >0$ and for $N$ sufficiently large. Here, $J_{\bs{\mf{r}}}$ is as introduced in \eqref{definition fnelle Jr} while $\bs{\mf{r}} \, = \, (\eta, \kappa,\mathbf{M})$

\end{prop}

\begin{proof}

Proposition \ref{integralupperbound} and Corollary \ref{Vrelation} ensure that there exists  $C > 0$ such that for $N$ sufficiently large
\beqa
\overline{\Pi}_{N,\mathsf{c}}[B_\mathsf{c}(\mu, \delta)] &\leq& \frac{ (N-1)! \,  \mathrm{e}^{CN^{ \frac{3}{4}} } }{  |\ln(2\veps_N)|^{N-1} }
\Int{\mathcal{A}_N }{}   \Big\{ \frac{\Delta(\boldsymbol{\lambda}_N^+)}{\Delta(\boldsymbol{\lambda}_N^-)} \Big\}^{\f{1}{2}}
\Big\{ \mc{T}(\boldsymbol{\lambda}^+_N) \mc{T}(\boldsymbol{\lambda}^-_N) \Big\}^{\f{1}{2}}
\delta\Big( \ov{\bs{\lambda}}_N^+  \Big)  \mathbbm{1}_{B(\mu, \delta)}\big( \op{L}_N^{(\boldsymbol{\lambda}_N^+)} \big) \nonumber \\
    &&\times \prod_{j=1}^N \max\big\{ \ln |\lambda_j^+ - \lambda_j^-|^{-1}   , \gamma N \big\}
\pl{k=1}{N} \pl{\ups=\pm}{}  \Big\{ \mathrm{e}^{  - \frac{1}{2}V(\lambda_k^{\ups}) + \frac{\tilde{C}}{N} V(\lambda_k^{\ups}) } \Big\}
\, \mathrm{d}\boldsymbol{\lambda}_N^+ \nonumber  \\
    &\leq & (N-1)!  |\ln(2\veps_N)|
\, \mathrm{e}^{CN^{\frac{3}{4}} } \Int{\mathcal{A}_N }{} \Big\{ \frac{\Delta(\boldsymbol{\lambda}_N^+)}{\Delta(\boldsymbol{\lambda}_N^-)} \Big\}^{\f{1}{2}}
\Big\{ \mc{T}(\boldsymbol{\lambda}^+_N) \mc{T}(\boldsymbol{\lambda}^-_N) \Big\}^{\f{1}{2}}
 \pl{\ups=\pm}{} \mathrm{e}^{\frac{N}{2} J_{\mf{r}}[\op{L}_N^{(\boldsymbol{\lambda}_N^{\ups})}] }  \nonumber  \\
    & & \times \pl{k=1}{N} \pl{\ups=\pm}{}
   \Big\{  \mathrm{e}^{- \frac{1}{2} ( W  +  \frac{\kappa}{4}   V ) (\lambda_k^{\ups}) }  \Big\}
\delta\Big( \ov{\bs{\lambda}}_N^+  \Big)  \mathbbm{1}_{B(\mu, \delta)}\big( \op{L}_N^{(\boldsymbol{\lambda}_N^+)} \big) \, \mathrm{d}\boldsymbol{\lambda}_N^+ \, .
\label{PiBarBound1}
\eeqa
In the second line, we have invoked Lemma \ref{regularisedbound} and used that, for $N$ sufficiently large, $\frac{\kappa}{4} - \frac{\tilde{C}}{N} \geq \frac{\kappa}{8}$. Moreover, we inserted artificially the functions $W$ by "tilting" the measure while removing  it in the function $J_{\mf{r}}$ as in the proofs of Sanov's and Cramer's  theorems \cite{DemboZ01}.
By virtue of Proposition \ref{kappeler}, observe that for $N$ sufficiently large one may symmetrise the integration domain 
\bem
\Big\{\bs{\la}_N^+ \in \mc{A}_N \, : \,  \op{L}_N^{(\boldsymbol{\lambda}_N^+)} \in  B(\mu, \delta) \Big\}  \\
\subset   \Big\{ \bs{\la}_N^+ \in \mc{A}_N \, : \, \op{L}_N^{(\boldsymbol{\lambda}_N^+)} \in  B(\mu, 2\delta) \Big\}
\bigcap \Big\{ \bs{\la}_N^+ \in \mc{A}_N \, : \, \op{L}_N^{(\boldsymbol{\lambda}_N^-)} \in  B(\mu, 2\delta) \Big\} \, .
\nonumber
\end{multline}
By the above and Lemma \ref{Jcontinuous}, for any $\tau > 0$, there is a $\delta > 0$ sufficiently small that
\beq
\frac{1}{2} J_{\mf{r}}[\op{L}_N^{(\boldsymbol{\lambda}_N^+)}]+\frac{1}{2} J_{\mf{r}}[\op{L}_N^{(\boldsymbol{\lambda}_N^-)}] \, \leq  \, J_{\mf{r}}[\mu] + \tau
\qquad  \text{ for all }  \quad  \op{L}_N^{(\boldsymbol{\lambda}_N^+)} \in  B(\mu, \delta) \, .
\label{ecriture upper bound Jr des L pm}
\enq

At this stage, it is useful to remind oneself that $\de\big( \ov{\bs{\la}}_N^+\big)$ is to be understood as the constraint $\la_N^+=-\ov{\bs{\la}}_{N-1}^+$
with all remaining variables $\ov{\bs{\la}}_{N-1}^+$ integrated over $\R^{N-1}$. Thus, there is no "distributional" problem to apply the Cauchy--Schwarz inequality
in the upper bound \eqref{PiBarBound1}, this once that the upper bound \eqref{ecriture upper bound Jr des L pm} is implemented.
Then Proposition \ref{jacobian} allows one to trade the integration over $\bs{\la}_N^+$ by one over $\bs{\la}_N^-$ while cancelling the ratio of Vandermondes
by the Jacobian of this change of variables. One then observes that $\mc{T}$ is symmetric which allows one to extend the integral to all of $\mathbb{R}^N$, for the price of a $\frac{1}{N!}$ factor.
\beq
\overline{\Pi}_{N,\mathsf{c}}[B_{\mathsf{c}}(\mu,\delta)]
\leq  \mathrm{e}^{CN^{\frac{3}{4}}}   \frac{|\ln(\veps_N)|}{N } \mathrm{e}^{N [J_{_{\bs{\mf{r}}}}[\mu]+ \tau]}
\Int{\mathbb{R}^N}{} \de( \ov{\boldsymbol{x}}_N ) \mc{T}(\bs{x}_N) \,  \pl{k=1}{N} \mathrm{e}^{- (W+\frac{\kappa}{4}V)(x_k) }\, \mathrm{d}\boldsymbol{x}_N \, .
\label{PiBarUBound1}
\enq
The function $W$ introduced in \eqref{Wchoice} does enjoy \eqref{ecriture hypothese sur U}.
Thus applying Cauchy--Schwarz at this stage allows one to conclude that, for some $C>0$,
\beq
\overline{\Pi}_{N,\mathsf{c}}[B_{\mathsf{c}}(\mu,\delta)]
\leq  \mathrm{e}^{CN^{\frac{3}{4}}}   \ \mathrm{e}^{N [J_{{\bs{\mf{r}}}}[\mu]+ \tau]}
\sqrt{ A_{N, \mathsf{c}}^{(10)} [W]\,  B_{N, \mathsf{c}}^{(8)}\big[W + \tfrac{\kappa}{2} V \big]  }  \, .
\label{PiBarUBound2}
\enq
Lemmata \ref{Acond}-\ref{Bcond} then yield the claim.

\end{proof}

\begin{rem}
The importance of Proposition \ref{jacobian} stems from the fact that the change of variables from $\boldsymbol{\lambda}_N^+ \to \boldsymbol{\lambda}_N^-$  eliminates the rather tricky
ratio $\frac{\Delta(\boldsymbol{\lambda}_N^+)}{\Delta(\boldsymbol{\lambda}_N^-)}$ which would have been difficult to bound.
\end{rem}

We are finally in position to prove the weak large deviation upper bound for the constrained model, namely \eqref{wldubco} of Proposition \ref{upperboundonballs}:
\begin{lemme}\label{lemme:wldubco}

 For every $\mu \in \mathcal{M}_{1,\mathsf{c}}(\mathbb{R})$ we have
\begin{align*}%
\limsup_{\delta \searrow 0} \limsup_{N \to +\infty} \frac{1}{N}\ln \overline{\Pi}_{N,\mathsf{c}}[B_{\mathsf{c}}(\mu,\delta)] \leq - I[\mu] \, .
\end{align*}
\end{lemme}
\begin{proof}
In view of Proposition \ref{wldubc}, we have for every $\mu\in   \mathcal{M}_{1,\mathsf{c}}(\mathbb{R})$ and for every $\tau>0$, provided the parameters  $\kappa,\eta$ are small enough and the parameters $M_{a}$ large enough,
$$ \limsup_{\delta \searrow 0} \limsup_{N \to +\infty} \frac{1}{N} \ln  \overline{\Pi}_{N,\mathsf{c}}[B_{\mathsf{c}}(\mu,\delta)] \leq   J_{{\bs{\mf{r}}}}[\mu]+ \tau \, .$$
If $\Int{\mathbb{R}}{} V(x) \, \mathrm{d}\mu(x) = +\infty$, then we may send $M_{3} \to +\infty$ and conclude that
\begin{align*}
\limsup_{\delta \searrow 0} \limsup_{N \to +\infty} \frac{1}{N} \ln  \overline{\Pi}_{N,\mathsf{c}}[B_{\mathsf{c}}(\mu,\delta)] \leq   -\infty \, .
\end{align*}
Hence we now only focus on $\Int{\mathbb{R}}{} V(x) \, \mathrm{d}\mu(x)  < +\infty$. In this case, let us begin by sending $\tau \searrow 0$, which is trivial since $J_{{\bs{\mf{r}}}}$ does not depend on $\tau$. Next let us send $M_3 \nearrow +\infty$. 
The function $\min\Big\{  (1-\kappa) V(x)  - W(x),  M_3\Big\}$ converges pointwise to  $(1-\kappa) V(x)  - W(x)$, is bounded from below by a constant, and the sequence is clearly increasing in $M_3$. Hence by monotone convergence
\beq
\Int{\mathbb{R}}{}\min\Big\{  (1-\kappa) V(x)  - W(x),   M_3\Big\} \,  \mathrm{d}\mu(x)  \tend  \Int{\mathbb{R}}{} \Big\{ (1-\kappa) V(x)  - W(x)\Big\} \,  \mathrm{d}\mu(x)  \;.
\enq
Likewise, $\vp_{\eta}(x) \, = \, \ln_{M_1}\Big( 1+ \frac{2}{\ell}  \Int{\mathbb{R}}{} f_{\kappa, \mathbf{M}}(x,y) \, \mathrm{d}\mu (y)+ \eta \Big) \searrow \vp_{0^+}(x) $
as $\eta \searrow 0^+$ and $\vp_{\eta}$ is bounded from above by a constant independent of $\eta$. Thus, by monotone convergence
$\Int{\R}{}\vp_{\eta}(x)  \dd \mu(x)  \tend \Int{\R}{}\vp_{0^+}(x) \dd \mu(x) $.

Further, as $M_2 \nearrow + \infty$, $f_{\kappa, \mathbf{M}}(x,y) \searrow \ln_{M_0}|x-y| \, - \, \tfrac{\ell \kappa }{4M_1} \big( V(x)+ V(y) \big)$. Moreover, by the growth condition on the potential \ref{potentialh2}, the sequence of functions is bounded from above by a constant independent of $M_2$. Hence, by monotone convergence and continuity of $\ln_{M_1}$,  one has pointwise in $x$ as $M_2 \nearrow + \infty$
\beq
\vp_{0^+}(x) \;  \tend  \;
\ln_{M_1}\Big( 1+ \frac{2}{\ell}  \Int{\mathbb{R}}{} \Big\{ \ln_{M_0}|x-y| \, - \, \tfrac{\ell \kappa }{4M_1} \big( V(x)+ V(y) \big)  \Big\} \, \mathrm{d}\mu (y)  \Big) \;.
\enq
This is clearly a monotone limit and, for the same reason as before, this sequence of functions is bounded from above  by a constant independent of $M_2$. Hence by monotone convergence
\begin{align*}
    \Int{\mathbb{R}}{}\vp_{0^+}(x) \, \mathrm{d}\mu(x) \tend \Int{\mathbb{R}}{}\ln_{M_1}\Big( 1+ \frac{2}{\ell}  \Int{\mathbb{R}}{} \Big\{ \ln_{M_0}|x-y| \, - \, \tfrac{\ell \kappa }{4M_1} \big( V(x)+ V(y) \big)  \Big\} \, \mathrm{d}\mu (y)  \Big) \, \mathrm{d}\mu(x)
\end{align*}
as $M_2 \nearrow + \infty$. In a similar manner one sees that $\ln_{M_0}|x-y| \searrow \ln|x-y|$ pointwise as $M_0 \nearrow  +\infty$ in a decreasing fashion, and the function $\ln_{M_0}|x-y| \, - \, \tfrac{\ell \kappa }{4M_1} \big( V(x)+ V(y) \big)$ is bounded from above by a constant independent of $M_0$. Hence, again by monotone convergence,
\begin{align*}
    &\Int{\mathbb{R}}{}\ln_{M_1}\Big( 1+ \frac{2}{\ell}  \Int{\mathbb{R}}{} \Big\{ \ln_{M_0}|x-y| \, - \, \tfrac{\ell \kappa }{4M_1} \big( V(x)+ V(y) \big)  \Big\} \, \mathrm{d}\mu (y)  \Big) \, \mathrm{d}\mu(x) \\
    &\tend \Int{\mathbb{R}}{}\ln_{M_1}\Big( 1+ \frac{2}{\ell}  \Int{\mathbb{R}}{} \Big\{ \ln|x-y| \, - \, \tfrac{\ell \kappa }{4M_1} \big( V(x)+ V(y) \big)  \Big\} \, \mathrm{d}\mu (y)  \Big) \, \mathrm{d}\mu(x)
\end{align*}
as $M_0 \nearrow +\infty$. Thus, so far we have
\begin{align*}
&\limsup_{\delta \searrow 0}\limsup_{N \to +\infty}\frac{1}{N} \ln  \overline{\Pi}_{N,\mathsf{c}}[B_{\mathsf{c}}(\mu,\delta)]\\
&\leq  \Int{\mathbb{R}}{}\ln_{M_1}\bigg\{ 1+ \frac{2}{\ell}
\Int{\mathbb{R}}{} \Big[ \ln |x - y| - \frac{\ell \kappa}{4M_1} V(x)  - \frac{\ell \kappa}{4M_1} V(y) \Big] \, \mathrm{d}\mu (y)\bigg\} \, \mathrm{d}\mu(x) \\
& - (1-\kappa) \Int{\mathbb{R}}{} V(x)  \,  \mathrm{d}\mu(x) +  \Int{\mathbb{R}}{} W(x)  \,  \mathrm{d}\mu(x) \\
&\leq \Int{\mathbb{R}}{}\ln_{M_1}\Big\{ 1+ \frac{2}{\ell} \Int{\mathbb{R}}{}  \ln |x - y|  \, \mathrm{d}\mu (y)\Big\}
\, \mathrm{d}\mu(x)  - (1-\kappa) \Int{\mathbb{R}}{} V(x)  \,  \mathrm{d}\mu(x) +  \Int{\mathbb{R}}{} W(x)  \,  \mathrm{d}\mu(x) \,
\end{align*}
where we have used the assumption that $V \geq 0$. At this stage, it is straightforward to send $\kappa \searrow 0$. It thus remains to send $ M_1 \to +\infty$.
Clearly
\beq
\Psi_{M_1}(x) :=\ln_{M_1}\Big\{ 1+ \frac{2}{\ell} \Int{\mathbb{R}}{}  \ln |x - y|  \, \mathrm{d}\mu (y)\Big\}
 \, \searrow   \ln\max\Big\{ 0, 1+ \frac{2}{\ell} \Int{\mathbb{R}}{}  \ln |x - y|  \, \mathrm{d}\mu (y)\Big\}
\nonumber
\enq
is a decreasing fashion as $M_1 \nearrow +\infty$. Moreover,
$$\Psi_{M_1}(x) \leq \ln\Big\{ 1+ \frac{2}{\ell} \ln (1+|x|) \, + \,  \frac{2}{\ell}  \Int{\R}{}  \ln (1+|y|) \dd \mu(y) \Big\} \in L^1(\mu) $$
owing to the hypothesis on the growth at infinity of $V$. Thus, by monotone convergence
$ \Int{\R}{}\Psi_{M_1}(x) \dd\mu(x) \tend \Int{\R}{}\Psi_{+\infty}(x) \dd\mu(x)$.
All-in-all, we have
\begin{align*}
&\limsup_{\delta \searrow 0}\limsup_{N \to +\infty}\frac{1}{N} \ln  \overline{\Pi}_{N,\mathsf{c}}[B_{\mathsf{c}}(\mu,\delta)]\\
&\leq \Int{\mathbb{R}}{}\ln\max \Big\{ 0,  1+ \frac{2}{\ell} \Int{\mathbb{R}}{}  \ln |x - y|  \, \mathrm{d}\mu (y)\Big\}
\, \mathrm{d}\mu(x)  - \Int{\mathbb{R}}{} V(x)  \,  \mathrm{d}\mu(x) +  \Int{\mathbb{R}}{} W(x)  \,  \mathrm{d}\mu(x) \, .
\end{align*}
Finally, recall that in the choice of $W$ given in \eqref{Wchoice} the function $\phi \in \mathrm{BL}(\mathbb{R})$ was arbitrary\symbolfootnote[2]{ We denote by
$\mathrm{BL}(\mathbb{R})$ the space of bounded Lipschitz functions on $\R$.}. We may thus optimise the above upper bound with respect to $\phi \in \mathrm{BL}(\mathbb{R})$.
\begin{align}
&\limsup_{\delta \searrow 0}\limsup_{N \to +\infty}\frac{1}{N} \ln  \overline{\Pi}_{N,\mathsf{c}}[B_{\mathsf{c}}(\mu,\delta)]\nonumber\\
&\leq \Int{\mathbb{R}}{}\ln\max \Big\{ 0,  1+ \frac{2}{\ell} \Int{\mathbb{R}}{}  \ln |x - y|  \, \mathrm{d}\mu (y)\Big\}
\, \mathrm{d}\mu(x)  -(1-\eps) \Int{\mathbb{R}}{} V(x)  \,  \mathrm{d}\mu(x)  + \ln \mathsf{Z}_{\eps V}\nonumber \\
&- \sup_{\phi \in \mathrm{BL}(\mathbb{R})} \bigg[ \Int{\mathbb{R}}{} \phi(x) \, \mathrm{d}\mu(x)
- \ln\Big\{\Int{\mathbb{R}}{} \mathrm{e}^{\phi(y)} \frac{\mathrm{e}^{-\eps V(y)}}{\mathsf{Z}_{\eps V}} \, \mathrm{d}y \Big\} \bigg]\label{labb}
\end{align}
where $\mathsf{Z}_{\eps V} = \Int{\mathbb{R}}{} \mathrm{e}^{-\eps V(x)} \,\mathrm{d}x $. However, $\mathrm{BL}(\mathbb{R})$ is dense in $\mc{C}_b^{0}(\mathbb{R})$, the space of continuous and bounded functions on the real line equipped with the supremum norm.
Furthermore
$$ \phi \mapsto \Int{\mathbb{R}}{} \phi(x) \, \mathrm{d}\mu(x)- \ln\Big\{\Int{\mathbb{R}}{} \mathrm{e}^{\phi(y)} \frac{\mathrm{e}^{-\eps V(y)}}{\mathsf{Z}_{\eps V}} \, \mathrm{d}y \Big\}$$
is continuous in the supremum norm. Hence we may replace the supremum over $\mathrm{BL}(\mathbb{R})$ with one over $\mc{C}^{0}_b(\mathbb{R})$.
Next, if we let $\mathrm{d}\nu_{\eps V}(y) = \frac{\mathrm{e}^{-\eps V(y)}}{\mathsf{Z}_{\eps V}} \, \mathrm{d}y$
then by the Donsker--Varadhan variational formula for the relative entropy (see Lemma 1.4.3 of \cite{DupuisE97} or Corollary 6.2.3 of \cite{DemboZ01}),
\begin{align*}
\sup_{\phi \in \mc{C}_b^{0}(\mathbb{R}) } \bigg[ \Int{\mathbb{R}}{} \phi(x) \, \mathrm{d}\mu(x)
- \ln\Big\{\Int{\mathbb{R}}{} \mathrm{e}^{\phi(y)} \frac{\mathrm{e}^{-\eps V(y)}}{\mathsf{Z}_{\eps V}} \, \mathrm{d}y \Big\} \bigg]  =  D(\mu \,  \| \,  \nu_{\eps V}) \, .
\end{align*}
The proof is complete since
$D(\mu \,  \| \,  \nu_{\eps V}) = - \mathrm{Ent}[\mu] + \eps \Int{\mathbb{R}}{} V \, \mathrm{d}\mu + \ln \mathsf{Z}_{\eps V} \, $ and we now observe that the dependence on $\eps$ on the {\it rhs} of \eqref{labb} entirely disappears, recovering the rate function $I[\mu]$. 
\end{proof}
\subsubsection{Weak large deviation principle for the unconstrained model}
The proof is very similar to that of the constrained model as we can derive an analogue to Proposition \ref{wldubc} saying that for 
 $\mu \in \mathcal{M}_1(\mathbb{R})$, for  $\kappa, \eta, \tau>0$ small enough and $M_a>0$  large enough, there exists $C>0$ such that for $N$ large enough
\beq
\overline{\Pi}_{N}[B(\mu,\delta)] \leq   \mathrm{e}^{C N^{\frac{7}{8}} }   \mathrm{e}^{N (J_{\bs{\mf{r}}}[\mu]+ \tau)}
\enq
for some $C >0$ and for $N$ sufficiently large. Here, $J_{\bs{\mf{r}}}$ is as introduced in \eqref{definition fnelle Jr} while $\bs{\mf{r}} \, = \, (\eta, \kappa,\mathbf{M})$. Indeed, 
the bounds on the density are exactly the same and we arrive at exactly the same upper bound as \eqref{PiBarUBound2} except that
$A_{N, \mathsf{c}}^{(p)}[ W], B_{N, \mathsf{c}}^{(p)}\big[W + \tf{\kappa V}{2} \big] $ are replaced by
$A_{N}^{(p)}[ W ], B_{N}^{(p)}\big[W + \tf{\kappa V}{2} \big]$ which we bounded in Lemma \ref{ABbound}.
The proof of \eqref{prop:wldub} then follows exactly as in the proof of Lemma \ref{lemme:wldubco}.

\subsection{Proof of exponential tightness}

In this subsection we prove Proposition \ref{exponentialtightness}. Starting from Equation \ref{PiBarBound1}
with $B_{\mathsf{c}}(\mu,\de)$ replaced by some Borel set $E \subset \mc{M}_{1,\mathsf{c}}(\R)$ we find
\begin{align*}
\overline{\Pi}_{N,\mathsf{c}}[E] &\leq  (N-1)!  \, \mathrm{e}^{CN^{\frac{3}{4}} } \Int{\mathcal{A}_N }{}
\Big\{ \frac{\Delta(\boldsymbol{\lambda}_N^+)}{\Delta(\boldsymbol{\lambda}_N^-)} \Big\}^{\f{1}{2}}
\Big\{ \mc{T}(\boldsymbol{\lambda}^+_N) \mc{T}(\boldsymbol{\lambda}^-_N) \Big\}^{\f{1}{2}}
 \pl{\ups=\pm}{} \mathrm{e}^{\frac{N}{2} J_{\mf{r}}[\op{L}_N^{(\boldsymbol{\lambda}_N^{\ups})}] } \mathbbm{1}_{E}\big( \op{L}_N^{(\boldsymbol{\lambda}_N^+)} \big) \nonumber  \\
    & \times \pl{k=1}{N} \pl{\ups=\pm}{}
   \Big\{  \mathrm{e}^{- \frac{1}{2} ( W  +  \frac{\kappa}{4}   V ) (\lambda_k^{\ups}) }  \Big\}
\; \delta\big( \ov{\bs{\lambda}}_N^+  \big)  \, \mathrm{d}\boldsymbol{\lambda}_N^+ \, .
\end{align*}
Note that by following the proof of Lemma \ref{regularisedbound} one may, in fact, set the regularisation parameter $M_3=+\infty$ in $J_{\mf{r}}$,
what is assumed from now on.

It is shown in the proof of Lemma \ref{compactlevelsets} that the set
$$F_{V,\mathsf{c}}(L) \,   = \, \Big\{ \mu \in \mathcal{M}_{1,\mathsf{c}}(\mathbb{R}) \, : \, \Int{\mathbb{R}}{} V(x) \, \mathrm{d}\mu(x) \leq L\Big\} \, $$
is compact.  By Lemma \ref{Vrelation} there is a $C > 0$ such that (for $N$ sufficiently large)
$$\Big\{ \Int{\mathbb{R}}{} V(x) \, \mathrm{d}\op{L}_N^{(\boldsymbol{\lambda}_N^+) }(x) > L \Big\}
\subset
\Big\{ \Int{\mathbb{R}}{} V(x) \, \mathrm{d}\op{L}_N^{(\boldsymbol{\lambda}_N^+) }(x) > \frac{1}{2} L-C \Big\} \bigcap
\Big\{ \Int{\mathbb{R}}{} V(x) \, \mathrm{d}\op{L}_N^{(\boldsymbol{\lambda}_N^-) }(x) > \frac{1}{2}L-C \Big\} \, .$$
Next, if we make the choice $\phi \equiv 0$, it is easy to see that
\bem
J_{\mf{r}}[\mu] \leq - (1-\kappa - \eps) \Int{\mathbb{R}}{}V(x) \, \mathrm{d}\mu(x) +\eta  + \frac{4}{\ell} \Int{\mathbb{R}}{}\ln(1+|x|) \, \mathrm{d}\mu(x)
+ \ln\Big( \Int{\mathbb{R}}{} \mathrm{e}^{-\eps V(y)}\, \mathrm{d}y \Big) \\
\leq - \frac{1}{4} \Int{\mathbb{R}}{}V(x) \, \mathrm{d}\mu(x)  + C^{\prime}
\end{multline}
for some constant $C^{\prime} >0$, if we require $0 < \eps , \kappa \leq \frac{1}{4}$. Hence there exists a $C^{\prime\prime} > 0$ such that
$$
\Big\{ \Int{\mathbb{R}}{} V(x) \, \mathrm{d}\op{L}_N^{(\boldsymbol{\lambda}_N^+)}(x)  > L \Big\} \subset
\Big\{ J_{\mf{r}}[\op{L}_N^{(\boldsymbol{\lambda}_N^{+})}] \, + \, J_{\mf{r}}[\op{L}_N^{(\boldsymbol{\lambda}_N^{-})}] \, \le  \,
- \frac{1}{4} L + C^{\prime\prime}\Big\} \, .$$
Hence taking $E = F_{V,\mathsf{c}}(L)^c$ one may upper bound the factors containing the rate function $J_{\mf{r}}$ by $\ex{- \tf{N L}{2} }$.
The remaining integral has  already been upper bounded in the proof of Proposition \ref{wldubc} by a factor of $\text{e}^{ CN^{ \tf{7}{8}} }$. Thus
\begin{align*}
 \limsup_{N \to +\infty}  \frac{1}{N}\ln  \overline{\Pi}_{N,\mathsf{c}}\big[ F_{V,\mathsf{c}}(L)^c \big] \leq  - \frac{L}{8} + C  \, .
\end{align*}
If we now let $L \to +\infty$ we arrive at Proposition \ref{exponentialtightness}. The proof for the unconstrained model goes along the same lines.





\section{Study of the rate function}\label{studygrf}

The rate function $I$ of interest to this section has been introduced in \eqref{ratefunction}. To start with, one should observe that
the expression given in \eqref{ratefunction} needs some further precision since, in principle,
each term may be infinite and so one may be left with an indeterminate expression $+\infty - \infty$.
This issue is clarified by the lemma below.
\begin{lemme}
Let $\mu \in \mathcal{M}_1(\mathbb{R})$ be such that $\Int{\mathbb{R}}{} V(x) \, \mathrm{d}\mu(x) < +\infty$.
Then there are constants $C^{\prime}, C^{\prime\prime} \in \mathbb{R}$ such that
the following lower bounds are satisfied
\begin{align}
&\frac{1}{2}\Int{\mathbb{R}}{}V(x) \, \mathrm{d}\mu(x) - C^{\prime}
 \leq \Int{\mathbb{R}}{} V(x) \, \mathrm{d}\mu(x) - \Int{\mathbb{R}}{} \ln \max\Big\{ 0 ,  1+ \frac{2}{\ell} \Int{\mathbb{R}}{} \ln|x-y| \, \mathrm{d}\mu(y) \Big\} \, \mathrm{d}\mu(x) \\
&\text{and } -\frac{1}{4} \Int{\mathbb{R}}{} V(x) \, \mathrm{d}\mu(x) - C^{\prime\prime}     \leq   - \mathrm{Ent}[\mu] \, .
\end{align}
Thus in particular, $\Int{\mathbb{R}}{} V(x) \, \mathrm{d}\mu(x) < +\infty$ implies that $I[\mu] = +\infty$ if and only if $- \mathrm{Ent}[\mu] = +\infty$.
Combining these two we see that there is a constant $C > 0$ such that
\begin{align}\label{ratefunctionlowerbound}
I[\mu] \geq \frac{1}{4}\Int{\mathbb{R}}{}V(x) \, \mathrm{d}\mu(x) - C \, .
\end{align}
Hence we naturally define $I[\mu] = +\infty$ whenever $\Int{\mathbb{R}}{}V(x) \, \mathrm{d}\mu(x) \, = \,  +\infty$ so that (\ref{ratefunctionlowerbound}) holds for all probability measures.
\end{lemme}
\begin{proof}
From the inequality $\ln|x-y| \leq \ln(1+|x|) + \ln(1+|y|)$, and $\ln(1+x) \leq x$ for $x\geq 0$, we have
\begin{align*}
\Int{\mathbb{R}}{} \ln \max\Big\{ 0 ,  1+ \frac{2}{\ell} \Int{\mathbb{R}}{} \ln|x-y| \, \mathrm{d}\mu(y) \Big\} \, \mathrm{d}\mu(x) \, \leq  \,
\frac{4}{\ell} \Int{\mathbb{R}}{} \ln(1+|x|) \, \mathrm{d}\mu(x) \, .
\end{align*}
By \ref{potentialh2}, there exists $C^{\prime}>0$ such that $\frac{4}{\ell} \ln(1+|x|) \leq \frac{1}{2}V(x) + C^{\prime}$, which, all-in-all, entails the first lower bound.
For the second one, we use the non-negativity of the relative entropy $D(\mu \, \| \, \nu) $, or equivalently the Donsker--Varadhan variational formula.
If we consider the probability measure $\mathrm{d}\nu_{\tf{V}{4}}(x) = \frac{ \mathrm{e}^{-\frac{1}{4}V(x)} }{ \mathsf{Z}_{\tf{V}{4}} }\, \mathrm{d}x$, where
$\mathsf{Z}_{\tf{V}{4}} = \Int{\mathbb{R}}{} \mathrm{e}^{-\frac{1}{4}V(x)} \, \mathrm{d}x$, then
$$-\mathrm{Ent}[\mu]= D(\mu \, \| \,  \nu) -\frac{1}{4}\Int{\mathbb{R}}{} V \, \mathrm{d}\mu - \ln \mathsf{Z}_{\tf{V}{4}}\,  \geq \,
-\frac{1}{4}\Int{\mathbb{R}}{} V \, \mathrm{d}\mu - \ln \mathsf{Z}_{\tf{V}{4}} \, .$$
\end{proof}
\subsection{Lower semi-continuity }
We start by proving the first part of  Proposition \ref{propositiongrf}, namely that $I$ is lower semi-continuous on $\mathcal{M}_1(\mathbb{R}) $ equipped with its weak topology.
The proof will build on two auxiliary lemmata.

\begin{lemme}\label{Itildesupremum} Assume that hypothesis \ref{potentialh2} holds. 
Let $\tilde{I} : \mathcal{M}_1(\mathbb{R}) \longrightarrow  \intof{-\infty}{+\infty}$ be defined by
$$ \tilde{I}[\mu]\overset{\mathrm{def}}{=}\Int{\mathbb{R}}{} V(x) \, \mathrm{d}\mu(x) -
\Int{\mathbb{R}}{} \ln \max\Big\{ 0 ,  1+ \frac{2}{\ell} \Int{\mathbb{R}}{} \ln|x-y| \, \mathrm{d}\mu(y) \Big\} \, \mathrm{d}\mu(x)$$
and likewise,  for $\eta_1, \eta_2 , M_1, M_2 > 0$, let $\bs{\mf{p}} = ( \eta_1, \eta_2 , M_1, M_2 ) $ and
\begin{align*}
&\tilde{I}_{\bs{\mf{p}}}[\mu] \overset{\mathrm{def}}{=} (1-2\eta_1) \Int{\mathbb{R}}{} \min\big\{ V(x) , M_1 \big\} \, \mathrm{d}\mu(x) \\
    &- \Int{\mathbb{R}}{} \ln_{\eta_2^{-1}}\Big\{   1+ \frac{2}{\ell} \Int{\mathbb{R}}{} \max\big\{ \ln|x-y| - \eta_1\eta_2 V(x) - \eta_1\eta_2 V(y), -M_2 \big\}
\, \mathrm{d}\mu(y) \Big\}
\, \mathrm{d}\mu(x) \, .
\end{align*}
with $\ln_{M}$ as introduced in \eqref{definition ln M regularise}.
Then, one has
\begin{equation}\label{Isupremum}
\tilde{I}[\mu] = \sup_{\bs{\mf{p}} \in (\R^+)^4 } \tilde{I}_{\bs{\mf{p}}}[\mu] \, .
\end{equation}
\end{lemme}
\begin{proof}
First of all, one has  $\tilde{I}[\mu] \geq \tilde{I}_{\bs{\mf{p}}}[\mu]$ owing to the chain of bounds
\beqa
\tilde{I}[\mu]
&\ge& \hspace{-3mm}   \Int{\R}{} V(x) \dd\mu(x) - \Int{\R}{} \ln_{\eta_2^{-1}} \Big\{ 1 + \frac{2}{\ell} \Int{\R}{} \ln|x-y| \dd\mu(y) \Big\}\dd \mu(x) \nonumber\\
%
%
%
%
&\ge&\hspace{-3mm} (1-2\eta_1) \Int{\R}{}  V(x)\dd\mu(x)  \nonumber \\
&& - \Int{\R}{} \ln_{\eta_2^{-1}} \Big\{ 1 +  \Int{\R}{} \big[ \frac{2}{\ell}  \ln|x-y|  -  \eta_1 \eta_2 \big(V(x) + V(y) \big) \big] \dd\mu(y)  \Big\}\dd\mu(x) \;,
\eeqa
the final inequality building on the fact that $\ln_{\eta^{-1}}$ has Lipschitz constant $\eta^{-1} > 0$. Now, the last of the lower
bounds above is bounded from below by $\tilde{I}_{\bs{\mf{p}}}[\mu]$ since taking the additional maximum involving $M_2 > 0$ and minimum involving $M_1 >0$  only makes the corresponding functions smaller.
 Therefore,
$$\tilde{I}[\mu] \geq \sup_{\bs{\mf{p}} \in (\R^+)^4 } \tilde{I}_{\bs{\mf{p}}}[\mu]  \, .$$
We now establish the reverse inequality.
When $M_1 \nearrow +\infty$,  since $V$ is bounded from below, monotone convergence ensures that
$$\Int{\mathbb{R}}{} \min\{ V(x), M_1 \} \, \mathrm{d}\mu(x) \longrightarrow \Int{\mathbb{R}}{} V(x)\, \mathrm{d}\mu(x) \, .$$
If $\Int{\mathbb{R}}{} V(x) \, \mathrm{d}\mu(x) = +\infty$, then $\underset{M_1 \nearrow +\infty}{\lim}\tilde{I}_{\bs{\mf{p}}}[\mu] = +\infty = I[\mu]$.
Hence without loss of generality we may assume $\Int{\mathbb{R}}{} V(x) \, \mathrm{d}\mu(x) < +\infty$.
Next, the function $\ln|x-y| - \eta_1 \eta_2 \big(V(x) +  V(y) \big)$ is bounded from above by a constant,
and $\max\big\{ \ln|x-y| - \eta_1 \eta_2 \big(V(x) +  V(y) \big) , - M_2 \big\}$ is decreasing as $M_2\nearrow +\infty$ increases, hence by monotone convergence
\begin{align*}
& \lim_{M_2\nearrow +\infty}
   \frac{2}{\ell} \Int{\mathbb{R}}{} \max\big\{ \ln|x-y| - \eta_1 \eta_2 \big( V(x) + V(y) \big) , -M_2 \big\} \, \mathrm{d}\mu(y)\qquad\qquad\\
   &\qquad= \mc{L}[\mu](x) \, := \, \frac{2}{\ell} \Int{\mathbb{R}}{} \Big\{ \ln|x-y| - \eta_1 \eta_2 \big(V(x) + V(y) \big) \Big\} \, \mathrm{d}\mu(y)\,.
\end{align*}
Again, since $\ln|x-y| - \eta_1 \eta_2 \big(V(x) + V(y) \big)$ is bounded from above by an $\eta_1, \eta_2$-dependent constant,
$ \mc{L}[\mu](x)$
is bounded from above by the same constant. Since the overall integrand is decreasing and bounded from above, monotone convergence entails that
\begin{align*}
    & \lim_{M_2\nearrow \infty}\Int{\mathbb{R}}{} \ln_{\eta_2^{-1}}\Big\{  1+ \frac{2}{\ell}
    \Int{\mathbb{R}}{} \max\big\{ \ln|x-y| -  \eta_1 \eta_2 \big(V(x) + V(y) \big), -M_2 \big\} \, \mathrm{d}\mu(y) \Big\} \, \mathrm{d}\mu(x)  \\
    &=\Int{\mathbb{R}}{} \ln_{\eta_2^{-1}}\Big\{   1+ \mc{L}[\mu](x)\Big\}
\, \mathrm{d}\mu(x)\,.
\end{align*}
Next, we take the  limit as $\eta_1 \searrow 0$.
By using that $\eta_2$ is the Lipschitz constant of $\ln_{\eta_2^{-1}}$, one gets
\begin{align*}
&\Big| \ln_{\eta_2^{-1}}\Big\{  1+ \frac{2}{\ell} \Int{\mathbb{R}}{} \big\{ \ln|x-y| - \eta_1 \eta_2 \big(V(x) + V(y) \big) \big\} \, \mathrm{d}\mu(y) \Big\} \\
    &  \hspace{1cm} - \ln_{\eta_2^{-1}}\Big\{   1+ \frac{2}{\ell} \Int{\mathbb{R}}{} \ln|x-y|   \, \mathrm{d}\mu(y) \Big\} \Big|
   \;  \leq  \; \frac{2}{\ell} \eta_1 \Int{\mathbb{R}}{} \big( V(x) + V(y) \big) \, \mathrm{d}\mu(y) .
\end{align*}
If we then integrate over $x$ and note that $\Int{\mathbb{R}}{}V(x)\, \mathrm{d}\mu(x) < +\infty$ \textit{ex hypothesi}, we may take $\eta_1 \searrow 0$ and conclude the desired result. Finally, we observe that
$$\ln_{\eta_2^{-1}}\Big\{   1+ \frac{2}{\ell} \Int{\mathbb{R}}{} \ln|x-y|   \, \mathrm{d}\mu(y) \Big\}$$
is decreasing as $\eta_2$ decreases and is bounded from above by $C + \frac{2}{\ell} \ln(1+|x|)$ for some constant $C$, which is $\mu$-integrable.
Hence, letting $\eta_2 \searrow 0$, by monotone convergence we arrive at (\ref{Isupremum}).
\end{proof}
\begin{lemme}\label{Itildecontinuous} Given hypothesis \ref{potentialh2}, for every $\bs{\mf{p}} = (\eta_1,\eta_2,M_1,M_2)\in (\R^+)^4$, the functional
$\tilde{I}_{\bs{\mf{p}}} : \mathcal{M}_1(\mathbb{R}) \longrightarrow \intof{-\infty}{+\infty}$ is continuous in the weak topology.
\end{lemme}
\begin{proof}
$\mu \mapsto (1-2\eta_1)\Int{\mathbb{R}}{} \min\{ V(x), M_1 \big\} \, \mathrm{d}\mu(x)$ is clearly continuous in the weak topology since the integrand is a bounded continuous function.
Hence we need only show that
\begin{align*}
&\psi[\mu] \, =  \, \Int{\mathbb{R}}{} h[\mu](x)\mathrm{d}\mu(x)\quad \mbox{ with } \quad
h[\mu](x) \, = \,  \ln_{\eta_2^{-1}}\Big\{  1+ \frac{2}{\ell} \Int{\mathbb{R}}{} g(x,y) \, \mathrm{d}\mu(y) \Big\}
\end{align*}
where $  g(x,y) := \max\big\{ \ln|x-y| - \eta_1 \eta_2 \big(V(x) + V(y) \big), -M_2 \big\}$ is continuous.
Now given two probability measures $\mu, \nu$ and fixed $x\in \R$, the Lipschitz property of $\ln_{\eta_2^{-1}}$ yields
$$ \big| h[\mu](x) - h[\nu](x) \big| \, \le \,  \frac{2}{\ell \eta_2} \Big| \Int{\R}{} g(x,y)\mathrm{d}(\mu-\nu)(y)\Big|\,.$$
Since $g$ is bounded and continuous, and is equal to the constant $-M_2$ outside of a compact subset of $\mathbb{R}^2$, it is thus uniformly continuous in the whole of $\mathbb{R}^2$. Hence, as $\mu \to \nu$ weakly
\begin{align*}
    \Big| \Int{\R}{} g(x,y)\mathrm{d}(\mu-\nu)(y)\Big| \to 0 & & \text{ \textit{uniformly} in } x \, .
\end{align*}
Then observing that $x \mapsto h[\nu](x)$ is bounded and continuous, we deduce that
\bem
\Big| \Int{\mathbb{R}}{} h[\mu](x) \, \mathrm{d}\mu(x) - \Int{\mathbb{R}}{} h[\nu](x) \, \mathrm{d}\nu(x) \Big|
\leq \Big| \Int{\mathbb{R}}{} h[\nu](x) \, \mathrm{d}(\mu-\nu)(x) \Big|  \\
\, + \,  \Big| \Int{\R}{} g(x,y)\mathrm{d}(\mu-\nu)(y)\Big| \, \mathrm{d}\mu(x) \to 0
\end{multline}
as $\mu \to \nu$ weakly.
\end{proof}

We have now sufficiently prepared the ground to establish the following.
\begin{lemme}\label{lemmalsc}
$I$ is lower semi-continuous on $\mathcal{M}_1(\mathbb{R}) $ equipped with its weak topology.
\end{lemme}

\begin{proof}
$\mu \mapsto -\mathrm{Ent}[\mu]$ is lower semicontinuous by the Donsker--Varadhan variational formula,
and lower semicontinuity is preserved under summation, hence we need only show that $\tilde{I}$ is lower semicontinuous.
This follows because it is the supremum of a family of continuous functionals, as per Lemmata \ref{Itildesupremum} and \ref{Itildecontinuous}.
\end{proof}
We next prove that the level sets of $I$ are compact.
\begin{lemme}\label{compactlevelsets} $I $ has compact level sets. That is, for all $L > 0$,
$\big\{  \mu \in \mathcal{M}_1(\mathbb{R}) \, : \, I[\mu] \leq L\big\} $ is compact.
\end{lemme}
\begin{proof} We show that $I^{-1}(]-\infty ; L])$ is compact by showing it is a closed subset of a compact set.
$I^{-1}( \intof{-\infty}{L})$ is closed by the lower semi-continuity of $I$ proven in Lemma \ref{lemmalsc}.
Next by (\ref{ratefunctionlowerbound}) we have that
\begin{equation}\label{ineqc}
I^{-1}(\intof{-\infty}{L}) \subset F_V(4L+C) \quad \e{with} \quad
F_V(L) \, = \, \Big\{ \mu \in \mathcal{M}_1(\mathbb{R}) \, : \, \Int{\mathbb{R}}{} V(x) \, \mathrm{d}\mu(x) \leq L \Big\}
\end{equation}
for some constant $C > 0$. We claim that the latter set is compact, which would prove our claim.
Indeed, it is closed since  by monotone convergence theorem, for every probability measure $\mu$,
$\underset{M > 0}{\sup}\Int{\R}{} \{ M \wedge V(x) \} \,  \mathrm{d}\mu(x) = \Int{\R}{} V(x) \, \mathrm{d}\mu(x)$, so that
$$ F_V(L)=\bigcap_{M>0}F_{ V\wedge M }(L) $$
where $F_{V\wedge M }(L)$ is closed since the function $V\wedge M $ is bounded continuous. Next,
$V(x) \to +\infty$ as $|x| \to +\infty$, hence given any $\epsilon > 0$, there is a compact set $K_\epsilon$ such that for all
$x \in K_\epsilon^c$, $V(x)\geq \frac{L}{\epsilon}$. Hence, for every probability measure $\mu$
$$\mu[K^c_{\eps}] \frac{L}{\epsilon} \leq \Int{\mathbb{R}}{} V(x) \, \mathrm{d}\mu(x)\,.$$
As a consequence,
$$F_V(L)\subset \bigcap_{\epsilon>0} \Big\{ \mu \in \mathcal{M}_1(\mathbb{R}) \, : \mu[K_\epsilon^c]\le \epsilon \Big\}$$
and so $F_V(L)$ is uniformly tight. Hence by Prokhorov's theorem $F_V(L)$ is compact.

\end{proof}
We next show that this property extends to the constrained model:
\begin{lemme} Assume \ref{potentialh2} holds with $\theta>1$.
Then, the restriction of $I$ to $\mathcal{M}_{1, \mathsf{c} }(\mathbb{R})$ has compact level sets.
That is, for $L >0$, $$\Big\{ \mu \in \mathcal{M}_1(\mathbb{R}) \, : \, I[\mu]\leq L \Big\} \cap \mathcal{M}_{1,\mathsf{c}}(\mathbb{R})$$
is compact in the subspace topology of $\mathcal{M}_{1,\mathsf{c} }(\mathbb{R})$.\end{lemme}
\begin{proof}
The subtlety here is that $\mathcal{M}_{1,\mathsf{c}}(\mathbb{R})$ is not a closed subset of $\mathcal{M}_{1}(\mathbb{R})$.
However we claim that its intersection with the compact set
$F_V(4L+C)$ is closed, see \eqref{ineqc} and the proof of Lemma \ref{compactlevelsets}.
when  \ref{potentialh2} holds with $\theta>1$. Indeed,  we  have for $\mu\in F_V(L)$
\begin{equation}\label{bo1}
\Int{\mathbb{R}}{} |x|^{\theta} \, \mathrm{d}\mu (x) \leq C_1 L + C_2
\end{equation}
for some constants $C_1 , C_2 > 0$.
We then argue that, $\mu\mapsto \int x \, \mathrm{d}\mu(x)$ is continuous on $F_V(L)$. To see this, let $ 0 \leq \varphi_{M} \leq 1$ be a continuous function,  equal to
$1$ on $[-M; M]$ and equal to $0$ on $[-M-1; M+1]^c$. Let $(\mu_n)_{n \in \mathbb{R}}$ be a sequence of probability measures converging weakly to $\mu$
and such that $\mu_n \in F_V(L)$. Since $F_V(L)$ is closed (see the proof of Lemma \ref{compactlevelsets}), we must have $\mu \in F_V(L)$. Then
\begin{align*}
\Big| \Int{\mathbb{R}}{} x \, \mathrm{d}(\mu_n - \mu)(x) \Big|  &\leq \Big| \Int{\mathbb{R}}{} x \, \varphi_M(x) \, \mathrm{d}(\mu_n - \mu)(x) \Big|
+ \Big| \Int{\mathbb{R}}{} x \, \big( 1- \varphi_M(x) \big) \, \mathrm{d}\mu_n(x) \Big| \\
    &\quad + \Big| \Int{\mathbb{R}}{} x \, \big( 1 - \varphi_M(x) \big) \, \mathrm{d}\mu(x) \Big| \, .
\end{align*}
By writing $|x| = |x|^{1-\theta}|x|^{\theta}$ we have the bound
$$
\Big| \Int{\mathbb{R}}{} x \, \big( 1 - \varphi_M(x) \big)  \, \mathrm{d}\nu(x) \Big|  \leq M^{1-\theta} (C_1 L + C_2)
$$
for all $\nu \in F_V(L)$. Given any $\eps >0$, by taking $M >0$ sufficiently large we can make
all such integrals less than $\frac{\eps}{2}$. Finally, we observe that $x \, \varphi_M(x)$ is bounded continuous, hence
\begin{align*}
    \limsup_{n \to +\infty}\Big| \Int{\mathbb{R}}{} x \, \mathrm{d}(\mu_n - \mu)(x) \Big|  \leq \eps \,  .
\end{align*}
Since $\eps > 0$ was arbitrary, we deduce the claim that $\mu\mapsto \int x \, \mathrm{d}\mu(x)$ is continuous on $F_V(L)$. It follows that
$$F_{V}(4L+C) \cap \Big\{ \mu \in \mathcal{M}_{1}(\mathbb{R}): \Int{\R}{}  x \, \mathrm{d} \mu(x)=0 \Big\}$$
is closed in the topology of $\mathcal{M}_1(\mathbb{R})$. The conclusion then follows, since
by \eqref{ineqc}, the set
$$\Big\{ \mu \in \mathcal{M}_1(\mathbb{R}) \, : \, I[\mu]\leq L \Big\} \cap \mathcal{M}_{1,\mathsf{c}}(\mathbb{R})$$
is equal to
$$ \Big\{ \mu \in \mathcal{M}_1(\mathbb{R}) \, : \, I[\mu]\leq L \Big\} \cap F_{V}(4L+C) \cap \Big\{ \mu \in \mathcal{M}_{1}(\mathbb{R}): \Int{\R}{}  x \, \mathrm{d} \mu(x)=0 \Big\}$$
which is the intersection of a compact set with a closed set and is therefore compact.

\end{proof}

Putting together \eqref{ratefunctionlowerbound} and Lemma \ref{compactlevelsets} we deduce that
\begin{cor}\label{lemmaexistsmin}
$I : \mathcal{M}_1(\mathbb{R}) \longrightarrow ]-\infty,+\infty]$ is lower semicontinuous and has compact level sets. Therefore, it
achieves its minimal value. Furthermore $I$ when restricted to $\mathcal{M}_{1,\mathsf{c}}(\mathbb{R})$ achieves its minimal value in $\mathcal{M}_{1,\mathsf{c}}(\mathbb{R})$.
\end{cor}
\subsection{Strict convexity and minimisers}\label{sec:proofmin}
We now prove Proposition \ref{propositiongrfmin}.
We claim that $I$ is strictly convex. Note that this remains true for its restriction to $\mathcal{M}_{1,\mathsf{c}}(\mathbb{R})$ since the constraint is linear.
We already know that $I$ achieves its minimal value by Corollary \ref{lemmaexistsmin}, hence by strict convexity this minimiser is unique
(in $\mathcal{M}_1(\mathbb{R})$ and $\mathcal{M}_{1,\mathsf{c}}(\mathbb{R})$ respectively).
\begin{prop}[Strict convexity of the rate functional]\label{propconv}
For any pair of probability measures $\mu_0, \mu_1 \in \mathcal{M}_1(\mathbb{R})$ (possibly equal), and any $\alpha \in \intff{0}{1}$, we have
$$I[(1-\alpha) \mu_0 + \alpha \mu_1 ] \leq (1-\alpha) I[\mu_0] + \alpha I[\mu_1] \, .$$
If $I[\mu_0], I[\mu_1] < +\infty$ and $ \mu_0 \neq \mu_1$, then for all $\alpha \in ]0;1[$ we have the strict inequality
$$I[(1-\alpha) \mu_0 + \alpha \mu_1 ] < (1-\alpha) I[\mu_0] + \alpha I[\mu_1] \, .$$
\end{prop}
\begin{proof}

$\mu \mapsto \Int{\mathbb{R}}{} V(x) \, \mathrm{d}\mu(x)$ is linear and therefore convex, hence we need only show that
$$
\mu \mapsto S[\mu] \overset{\mathrm{def}}{=} -\Int{\mathbb{R}}{} \ln\max\Big\{ 0,  1+ \frac{2}{\ell} \Int{\mathbb{R}}{} \ln|x-y| \, \mathrm{d}\mu(x) \Big\}
\, \mathrm{d}\mu(x) + \Int{\mathbb{R}}{} \ln \frac{\mathrm{d}\mu(x)}{\mathrm{d}x}\, \mathrm{d}\mu(x)
$$
is strictly convex. Without loss of generality we can assume $I[\mu] < +\infty$, which in particular implies that $\mu$ has a density which we denote $\varrho$, and furthermore that $\mu$ satisfies
\begin{align*}
1+ \frac{2}{\ell} \Int{\mathbb{R}}{} \ln|x-y| \, \mathrm{d}\mu(y) > 0 & & \text{for } \mu\text{--a.e. } x \, .
\end{align*}
We begin by remarking that $S$, in a certain sense, takes the form of a relative entropy. More precisely, if we let $f(u,v) = u \ln \frac{u}{v}$ for $u,v > 0$, then
$$S[\mu] = \Int{\mathbb{R}}{} f\bigg( \varrho(x) ,   1 + \frac{2}{\ell} \Int{\mathbb{R}}{} \ln|x-y|\, \varrho(y) \, \mathrm{d}y \bigg) \, \mathrm{d}x \, .$$
Note that the second argument of $f$ is not necessarily a normalised density, but this doesn't matter. Convexity of $S$ then immediately follows since $f$ is convex in $\mathbb{R}^+ \times \mathbb{R}^+$ (see Example 3.19 of \cite{boyd2004convex}) and $\varrho$ appears linearly in each argument of $f$.

Let us now prove strict convexity. To show this we show that if $\mu_0, \mu_1 \in \mathcal{M}_1(\mathbb{R})$ are probability measures, and if $S[(1-\alpha) \mu_0 +  \alpha \mu_1] = (1-\alpha) S[\mu_0] + \alpha S[\mu_1]$ for some $\alpha \in ]0;1[$, then $\mu_0 = \mu_1$. As before, we can assume $I[\mu_a] < +\infty$, $\mu_a$ has density $\varrho_a$, and
\begin{align*}
    1 + \frac{2}{\ell} \Int{\mathbb{R}}{}\ln|x-y| \, \mathrm{d}\mu_a(y) > 0, & & \text{for } \mu_a \text{--a.e. } x
\end{align*}
for $a \in \{ 0, 1 \}$. Let $\mu_\alpha = (1-\alpha) \mu_0 +  \alpha \mu_1$ and let $\varrho_\alpha$ be its associated density, for $\alpha \in ]0;1[$.

We first remark that, by convexity of $S$ if $S[(1-\alpha) \mu_0 +  \alpha \mu_1] = (1-\alpha) S[\mu_0] + \alpha S[\mu_1]$ for \textit{some} $\alpha \in ]0;1[$ then in fact this holds for \textit{all} $\alpha \in ]0;1[$. To see this, let $\alpha^\prime \in ]0;1[$ be any other point, and without loss of generality assume $ \alpha < \alpha^\prime$. Then $\mu_{\alpha} = (1- \frac{\alpha}{\alpha^\prime}) \mu_0 + \frac{\alpha}{\alpha^\prime} \mu_{\alpha^\prime} $ is a convex combination. Then by convexity of $S$,
\begin{align*}
    (1-\alpha) S[\mu_0] + \alpha S[\mu_1] = S[\mu_{\alpha}] &\leq (1- \frac{\alpha}{\alpha^\prime}) S[\mu_0] + \frac{\alpha}{\alpha^\prime} S[\mu_{\alpha^\prime}]  \leq (1-\alpha) S[\mu_0] +  \alpha S[\mu_1] \, .
\end{align*}
Hence we have equality throughout, and so $S[\mu_{\alpha^\prime}] = \alpha^\prime S[\mu_0] + (1- \alpha^\prime) S[\mu_1]$.

By the convexity of $f$, if $S[(1-\alpha) \mu_0 +  \alpha \mu_1] = (1-\alpha) S[\mu_0] + (1-\alpha) S[\mu_1]$ then it must be the case that
\begin{align*}
 f\bigg( \varrho_\alpha(x) ,   1 + \frac{2}{\ell} \Int{\mathbb{R}}{} \ln|x-y|\, \varrho_\alpha(y) \, \mathrm{d}y \bigg) &= (1-\alpha) f\bigg( \varrho_0(x) ,   1 + \frac{2}{\ell} \Int{\mathbb{R}}{} \ln|x-y|\, \varrho_0(y) \, \mathrm{d}y \bigg) \\
 &\quad + \alpha f\bigg( \varrho_1(x) ,   1 + \frac{2}{\ell} \Int{\mathbb{R}}{} \ln|x-y|\, \varrho_1(y) \, \mathrm{d}y \bigg)
\end{align*}
for almost every $x \in \mathbb{R}$ with respect to Lebesgue measure, and for all $\alpha \in ]0;1[$. Hence the second derivative with respect to $\alpha$
of the \textit{lhs}, which is clearly a smooth function of $\alpha$, must vanish identically.
If we denote $U[\mu](x) := \Int{\mathbb{R}}{} \ln|x-y| \, \mathrm{d}\mu(y)$,
this yields
\beq
\bigg(
\frac{\varrho_1(x) - \varrho_0(x)}{\sqrt{\varrho_\alpha(x)}}  \, - \,
\frac{\sqrt{\varrho_\alpha(x)}}{1 + \frac{2}{\ell} U[\mu_\alpha](x) } \frac{2}{\ell} U[\mu_1 - \mu_0](x)  \bigg)^2 \; = \; 0
\qquad \forall \alpha \in ]0;1[
\enq
for Lebesgue-almost every $x$. Taking the square root, multiplying by
$\big( 1 + \frac{2}{\ell} U[\mu_\alpha](x) \big) \frac{\varrho_1(x) - \varrho_0(x)}{\sqrt{\varrho_\alpha(x)}}$
and integrating with respect to Lebesgue measure yields
\begin{align*}
\Int{\mathbb{R}}{} \frac{(\varrho_1(x) - \varrho_0(x))^2}{\varrho_\alpha(x)}  \bigg\{ 1 + \frac{2}{\ell} U[\mu_\alpha](x) \bigg\} \, \mathrm{d}x
= \frac{2}{\ell} \Int{\mathbb{R}^2}{} \ln|x-y| \, \mathrm{d}(\mu_1-\mu_0)(x) \, \mathrm{d}(\mu_1-\mu_0)(y) \, .
\end{align*}
The \textit{lhs} is non-negative, where we recall that $1 + \frac{2}{\ell} U[\mu_\alpha](x) >0$ for $\mu_0$ and $\mu_1$-almost every $x$. The \textit{rhs}  is non-positive as it may be expressed as,
$$-\frac{2}{\ell} \Int{0}{+\infty} \frac{1}{t} \bigg|\Int{\R}{}  \e{e}^{\i tx} \mathrm{d}(\mu_0-\mu_1)(x)\bigg|^2 \mathrm{d} t \, ,$$
see \textit{e.g.} the proof of Lemma 2.6.2 of \cite{AndersonGZ10}. Hence both sides vanish, yielding $\mu_0=\mu_1$.
\end{proof} 

\begin{rem}
We note that the "interaction" term in the rate function $I$ (the term containing two logarithms) is not convex by itself but is only convex when combined with the Shannon entropy term. This is gesturing towards the idea that these two terms naturally belong together, and that the interaction term is a \textit{modification} of the Shannon entropy which takes into account how the scattering of solitons modifies their effective volume.
\end{rem}

Since $I$ is a good rate function and it is strictly convex, $I$ achieves its minimal value at unique probability measures $\nu_\ell$ and $\nu_{\ell, \mathsf{c}}$  on $\mathcal{M}_1(\mathbb{R})$ and $\mathcal{M}_{1,\rm{c}}(\mathbb{R})$
respectively. This completes the proof of Proposition  \ref{propositiongrfmin}. We finally prove Lemma \ref{relBeta}, which relates the minimiser of $I$ to the equilibrium measure of a high temperature $\beta$-ensemble. We begin by deriving Euler--Lagrange equations for the minimisation of $I$.
\begin{prop}[Euler--Lagrange equation for the unconstrained model] Let $\nu_\ell$ be the minimiser of the rate function $I$ \eqref{ratefunction}  in the space $\mathcal{M}_1(\mathbb{R})$. Then $\nu_\ell$ is absolutely continuous with respect to Lebesgue measure and there is a constant $C_\ell \in \mathbb{R}$ such that its density $\frac{\mathrm{d\nu_\ell}}{\mathrm{d}x}$ satisfies
    \begin{align}\label{eulerlagrange}
       &V(x) - \ln\Big( 1 + \frac{2}{\ell} \Int{\mathbb{R}}{}\ln|x-y| \, \mathrm{d}\nu_\ell(y) \Big)  -\Int{\mathbb{R}}{} \frac{\frac{2}{\ell} \ln|x-y|}{1 + \frac{2}{\ell} \Int{\mathbb{R}}{}\ln|y-z| \, \mathrm{d}\nu_\ell(z)}  \, \mathrm{d}\nu_\ell(y) + \ln \frac{\mathrm{d\nu_\ell}}{\mathrm{d}x} = C_\ell \, , \nonumber \\
    &\text{for } \mathrm{Leb}-\text{a.e. } x \in \mathbb{R}   \, . 
    \end{align}
where $1 + \frac{2}{\ell} \Int{\mathbb{R}}{}\ln|x-y| \, \mathrm{d}\nu_\ell(y) > 0$ Lebesgue almost everywhere. Furthermore, viewed as an equation in $\nu_\ell$, the solution is unique. More precisely, if there exists a constant $C_\ell$ and a measure $\nu$ such that \eqref{eulerlagrange} is satisfied (upon replacing $\nu_\ell$ by $\nu$) then in fact $\nu = \nu_\ell$.
\end{prop}
\begin{proof}
Let $\nu_\ell \in \mathcal{M}_1(\mathbb{R})$ be the minimiser of $I$ and let $\mu \in \mathcal{M}_1(\mathbb{R})$ be arbitrary save for the requirement that $I[\mu] < +\infty$. Then for any $0 < t \leq 1$, the ratio $\frac{I[(1-t)\nu_\ell + t \mu] - I[\nu_\ell]}{t} \geq 0$ and furthermore by convexity of $I$ is monotonically decreasing as $t \searrow 0$. Hence the limit  $\lim_{t \searrow 0}\frac{I[(1-t)\nu_\ell + t \mu] - I[\nu_\ell]}{t}$ exists. By monotone convergence theorem, we can take the limit under the integral sign and find
\begin{equation*}
      \Int{\mathbb{R}}{}V_{\mathrm{eff}}(x) \, \mathrm{d}(\mu - \nu_\ell)(x)  \geq 0 \, \text{ where } \, 
     \begin{aligned}[t]
         V_{\mathrm{eff}}(x) &= V(x) -\ln\Big( 1 + \frac{2}{\ell} \Int{\mathbb{R}}{}\ln|x-y| \, \mathrm{d}\nu_\ell(y) \Big) \\
         &\quad -\Int{\mathbb{R}}{} \frac{\frac{2}{\ell} \ln|x-y|}{1 + \frac{2}{\ell} \Int{\mathbb{R}}{}\ln|y-z| \, \mathrm{d}\nu_\ell(z)}  \, \mathrm{d}\nu_\ell(y)  + \ln \frac{\mathrm{d\nu_\ell}}{\mathrm{d}x}   \, .
     \end{aligned}
\end{equation*}

If we let $C_\ell = \Int{\mathbb{R}}{}V_{\mathrm{eff}} \, \mathrm{d} \nu_\ell$, then $\Int{\mathbb{R}}{}\{ V_{\mathrm{eff}} - C_\ell\} \, \mathrm{d}\mu   \geq 0$ for all probability measures $\mu$ such that $I[\mu] < +\infty$. From this it follows that $V_{\mathrm{eff}} \geq C_\ell$ Lebesgue almost everywhere. If $\frac{\mathrm{d\nu_\ell}}{\mathrm{d}x} = 0$ then $V_{\mathrm{eff}}(x) = -\infty $, and so in fact $\mathrm{supp}\, \nu_\ell = \mathbb{R}$.

Next, we argue that in fact $V_{\mathrm{eff}}(x) = C_\ell$ Lebesgue almost everywhere. Let $\phi$ be a continuous function on the real line such that $|\phi| \leq \frac{1}{2}$ and consider the measure
\begin{align*}
    \mathrm{d}\nu^{ \phi}(x) &= \frac{1}{Z_{\phi}} (1+ \phi(x))\,  \mathrm{d}\nu_\ell(x) \, , & & Z_{ \phi} = \Int{\mathbb{R}}{}(1+ \phi(x)) \,  \mathrm{d}\nu_\ell(x) \, .
\end{align*}
If we now take $\mu = \nu^{ \phi}$ we have
\begin{align*}
    \frac{1}{Z_{ \phi}}  \Int{\mathbb{R}}{} \{  V_{\mathrm{eff}}(x) - C_{\ell}\}  (1 + \phi(x)) \, \mathrm{d}\nu_{\ell}(x) \geq 0 \, .
\end{align*}
We immediately see that the $1$ term disappears and hence 
\begin{align*}
      \Int{\mathbb{R}}{} \{  V_{\mathrm{eff}}(x) - C_{\ell}\}   \phi(x) \, \mathrm{d}\nu_{\ell}(x) \geq 0 \, .
\end{align*}
Since we can replace $\phi$ with $-\phi$, we must have $\Int{\mathbb{R}}{} \{  V_{\mathrm{eff}}(x) - C_{\ell}\}   \phi(x) \, \mathrm{d}\nu_{\ell}(x) = 0$. By homogeneity, this must hold for all $\phi$ bounded continuous. Since $\nu_{\ell}$ is supported on the whole real line  we must have $V_{\mathrm{eff}}(x) - C_{\ell} = 0$ for Lebesgue almost every $x$.

Finally, let us demonstrate that \eqref{eulerlagrange} is not only a necessary but also a sufficient condition for a minimiser. We recall from convexity of $I$ we have, $$\frac{I[(1-t)\nu + t \mu] - I[\nu]}{t} \geq \lim_{t \searrow 0}\frac{I[(1-t)\nu + t \mu] - I[\nu]}{t}$$ for all $0 < t \leq 1$ and for $\mu,\nu$ distinct probability measures. Taking $t = 1$ we have $I[ \mu] - I[\nu] \geq \lim_{t \searrow 0}\frac{I[(1-t)\nu + t \mu] - I[\nu]}{t}$. If $\nu$ satisfies \eqref{eulerlagrange} for \textit{some} constant $C_\ell$, then, by our previous arguments, $\lim_{t \searrow 0}\frac{I[(1-t)\nu + t \mu] - I[\nu]}{t} = 0$. Hence $I[ \mu] - I[\nu] \geq 0$ and so $\nu$ is a minimiser of $I$. But, by strict convexity of $I$,  the minimiser of $I$ is unique, so $\nu = \nu_{\ell}$. 
\end{proof}
In a similar manner, see the end of Section 3 of \cite{GuionnetM22}, one may derive an Euler--Lagrange equation for the minimiser $\mu_P$ of the free energy $I_P^C$ \eqref{def:CoulombEntropy} of the high-temperature $\beta$-ensemble with inverse temperature $P > 0$. That is, there exists a constant $\tilde{C}_P$ such that
\begin{align}\label{betaeulerlagrange}
    V(x) - 2 P \Int{\mathbb{R}}{} \ln|x-y| \, \mathrm{d}\mu_P(y) + \ln \frac{\mathrm{d}\mu_P}{\mathrm{d} x} = \tilde{C}_P \, , & & \text{for } \mathrm{Leb}\text{--a.e. } x \in \mathbb{R} \, .
\end{align}
If $V$ satisfies \ref{potentialh1} and \ref{potentialh2} then $I_P^C$ is strictly convex, and so \eqref{betaeulerlagrange} is a necessary and sufficient condition for $\mu_P$ to be the minimiser of $I_P^C$. Using this fact, let us now establish a relationship between the two. In \cite{GuionnetM22} Guionnet and Memin show that under the hypothesis that $V$ is continuous and $V(x) = (1+ \mathrm{o}(1)) a x^{2k}$ as $|x| \to \infty$ for some $a > 0$, $P \mapsto \mu_P$ is differentiable. From this, they introduce the probability measure $\chi_P = \partial_P (P\mu_P)$. Then, from \eqref{betaeulerlagrange} they deduce that $\chi_P$ satisfies the relation
\begin{align}
    \frac{\mathrm{d}\chi_P}{\mathrm{d}x} = \Big( C_P^\prime + 2P \Int{\mathbb{R}}{} \ln|x-y| \, \mathrm{d}\chi_P(y) \Big) \frac{\mathrm{d}\mu_P}{\mathrm{d}x} \, ,
\end{align}
for some constant $C_P^\prime$. Rearranging for $\mu_P$ and inserting this into \eqref{betaeulerlagrange} we find that $\chi_P$ solves \eqref{eulerlagrange} for the choice $\ell = \frac{C_P^\prime}{P}$. By uniqueness of the solution of \eqref{eulerlagrange} we have $\chi_P = \nu_{\ell}|_{\ell = \frac{C_P^\prime}{P}}$.

Furthermore, there is a similar relationship in the other direction. Let $\nu_{\ell} \in \mathcal{M}_1(\mathbb{R})$ be the minimiser of $I$, and define the probability measure
\begin{align*}
    \mathrm{d}\tilde{\chi}_{\ell}(x) = \frac{1}{\mathsf{m}(\ell)} \frac{\mathrm{d}\nu_{\ell}(x)}{1 + \frac{2}{\ell}\Int{\mathbb{R}}{} \ln |x-y| \, \mathrm{d}\nu_{\ell}(y)}\, , & &\text{where } \,  \mathsf{m}(\ell) = \Int{\mathbb{R}}{} \frac{\mathrm{d}\nu_{\ell}(x)}{1 + \frac{2}{\ell}\Int{\mathbb{R}}{} \ln |x-y| \, \mathrm{d}\nu_{\ell}(y)} \, .
\end{align*}
Before moving on, we should demonstrate that $\mathsf{m}(\ell) < +\infty$. To do this, let us integrate both sides of \eqref{eulerlagrange} with respect to $\nu_{\ell}$. This yields the formula $I[\nu_{\ell}] - 1 + \mathsf{m}(\ell) = C_\ell$. Since $I[\nu_{\ell}]$ and $C_{\ell}$ are finite, so is $\mathsf{m}(\ell)$. If we insert $\tilde{\chi}_{\ell}$ into \eqref{eulerlagrange} we find $\tilde{\chi}_{\ell}$ solves \eqref{betaeulerlagrange} with $P = \frac{\mathsf{m}(\ell)}{\ell}$ and $\tilde{C}_P = C_\ell - \mathsf{m}(\ell)$. Thus by the uniqueness of solutions to \eqref{betaeulerlagrange} we have $\tilde{\chi}_{\ell} = \mu_P|_{P = \frac{\mathsf{m}(\ell)}{\ell}}$. This concludes the proof of Lemma \ref{relBeta}.

\subsection{Continuity along special sequences}\label{sec:continuity along  good subsequences}
In this subsection, we prove Lemma \ref{lemmeapprox}.

\begin{defin}
Let $\op{s}_{\rho}$ be the dilation with scale $\rho > 0$: $\op{s}_{\rho}(x)=\rho x$. Given $\mu \in \mathcal{M}_1(\mathbb{R})$, its pushforward by $\op{s}_{\rho}$
is the unique probability measure $\op{s}_{\rho \#}\mu$  such that
$$\Int{\mathbb{R}}{} f(x) \, \mathrm{d}\op{s}_{\rho \#}\mu( x)  \, =  \, \Int{\mathbb{R}}{} f(\rho  x) \,  \mathrm{d}\mu(x) \, .$$
Note that $\mathrm{supp} \big[ \op{s}_{\rho \#} \mu \big] \,  =  \,  \rho \, \mathrm{supp} [ \mu ]$. For a probability measure $\mu$ and $\tau>0$, we denote by $\mu_\tau$
the probability measure obtained by the convolution
\begin{equation}\label{defconv}
\mu_\tau=\frac{1}{2\tau} \mathbbm{1}_{[-\tau;\tau]} \ast \mu\,.
\end{equation}
\end{defin}
Our candidate for the desired sequence $\mu_p$ of Lemma \ref{lemmeapprox} will be of the form $\mu^K_{\rho,\tau}=(\mu^K_\rho)_\tau$
for some constants $\rho,\tau>0$ and well-tailored $K=[-K_1;K_2]$  for some constants $K_1, K_2 > 0$. Here,
 we  have set
\beq
\mu^K_{\rho}  \, =  \, \frac{ \mathbbm{1}_{ [-K_1;K_2]}   }{  \op{s}_{\rho\#} \mu[ K ] } \cdot \op{s}_{\rho\#} \mu \;.
\label{definition mesures mu K rho}
\enq

In the rest of this section we consider $\mu$ so that $I[\mu]$ is finite. We have seen that this implies that

\beq\label{cond1}
1+\frac{2}{\ell} \Int{\mathbb{R}}{} \ln|x-y| \, \mathrm{d}\mu(y) \geq  0 \qquad  \dd \mu(x)-\e{a.e.} \; \e{on} \; \;  \mathbb{R} \,,
\enq
and
\beq\label{cond2}
\Int{\R}{} |x|^{\theta} \,   \dd \mu(x) \, < \,  + \infty \quad \e{and} \quad
\Int{\R}{} x  \,  \dd \mu(x) \,= \,  0\;,
\enq
where the second condition in \eqref{cond2} holds for the constrained model.
\begin{rem}
Note that the requirement \eqref{cond1} implies that $\mu$ has no atoms and the first part of \eqref{cond2} is due to Hypothesis \ref{potentialh2}.
\end{rem}
We first remark that dilation of $\mu$ allows one to enforce the second condition in Lemma \ref{lemmeapprox}, while verifying the first. Namely

\begin{lemme}\label{weakdilation}
Let $\mu \in \mathcal{M}_1(\mathbb{R})$ and $\rho > 1$. Then $\op{s}_{\rho \#} \mu \to \mu$ as $\rho \searrow 1$ in the weak topology. Moreover,
\begin{align*}
   &\text{if } \int x \,  \mathrm{d} \op{s}_{\rho \#} \mu(x) =0   & &\text{ then } & & \int x \,  \mathrm{d} \mu(x) =0  \, , \\
   &\text{and if} \int |x|^\theta \,  \mathrm{d}\mu(x) < +\infty  & & \text{ then } & &  \int |x|^\theta \, \mathrm{d} \op{s}_{\rho \#} \mu(x) < +\infty \, .
\end{align*}
Furthermore,
if $\mu$ satisfies \eqref{cond1} then
\begin{align}\label{bound}
    1+\frac{2}{\ell} \Int{\mathbb{R}}{} \ln|x-y| \, \mathrm{d}(\op{s}_{\rho \#} \mu) (y) \, \geq \, 2\kappa := \frac{2}{\ell} \ln \rho > 0,\quad \op{s}_{\rho \#} \mu \quad a.s.
\end{align}
\end{lemme}

\begin{proof}

Let $R > 0$ be some (large) number. Then, we have
\bem
   \bigg|  \Int{\mathbb{R}}{} f(x) \, \mathrm{d}\op{s}_{\rho \#} \mu(x) - \Int{\mathbb{R}}{} f(x) \, \mathrm{d} \mu(x) \bigg|  \, = \,
   \bigg| \Int{\mathbb{R}}{} [f(\rho x) - f(x)] \, \mathrm{d}\mu(x)  \bigg|  \\
\, \leq \,  2 \mu\big[[-R;R]^c\big] \| f \|_{L^\infty(\mathbb{R})} \, + \,  R (\rho-1) \| f \|_{\mathrm{BL}}\;.
\end{multline}
Given any $\eps > 0$, we may take $R$ sufficiently large so that 
$2 \mu\big[ [-R;R]^c \big] \leq \frac{\eps}{2}$, and then take $\rho-1 > 0$ sufficiently small so that $R (\rho-1) \leq \frac{\eps}{2}$. Taking the supremum over $\| f \|_{\mathrm{BL}} \leq 1$ we have $d_{\mathrm{BL}}(\mu, \op{s}_{\rho \#} \mu) \leq \eps$. \eqref{bound} is obvious by rescaling.

\end{proof}
We next need to compactify our measures. One subtlety is that we must compactify in a way that preserves the vanishing of the first moment. This is handled by the following lemma.

\begin{lemme}\label{compactification}
Let $\mu \in \mathcal{M}_1(\mathbb{R})$ be atomless and such that there exists $\theta>1$
$$
\Int{\mathbb{R}}{} |x|^{\theta} \, \mathrm{d}\mu(x) \, < \,  +\infty \qquad  \mbox{and }  \qquad  \Int{\mathbb{R}}{} x \, \mathrm{d}\mu(x) \, = \,  0 \, .
$$
Then, there exists a constant $\eta_0(\mu)> 0$ such that for any $0 < \eta < \eta_0(\mu)$, there exist $K_1, K_2 > 0$ such that
\begin{align*}
\Int{-K_1}{K_2} x \, \mathrm{d}\mu(x) = 0 \;, \qquad \mu\big[ [-K_1; K_2]^c \big] \,  \leq \, \eta \;, \qquad
\Int{}{} \mathbbm{1}_{[-K_1;K_2]^c}(x)  |x| \, \mathrm{d}\mu(x) \,  \leq  \, \eta \, .
\end{align*}
Furthermore $\mathbbm{1}_{[-K_1;K_2]} \to 1$ as $\eta \searrow 0$ $\mu$-almost everywhere, in an increasing fashion. Finally, there exists $C>0$ such that
\beq
K_1+K_2 \, \leq \, \f{ C }{ \eta^{   \f{1}{\theta-1} }  } \;.
\label{ecriture borne sup a priori sur croissance des Ka}
\enq
\end{lemme}
Note that when we do not prescribe $\Int{-K_1}{K_2} x \, \mathrm{d}\mu=0$, the proof is straightforward by Chebyshev's inequality and holds for any $\theta>0$.
\begin{proof}
Let $M_{\tf{1}{2}} \, = \,  \Int{ \R^-}{} |x| \, \mathrm{d}\mu(x) \,  = \,  \Int{\R^+}{} |x| \, \mathrm{d}\mu(x)$. Further, given $y \geq 0$, define
\beq
    F_1(y) \, = \,  \Int{-y}{0} |x| \, \mathrm{d}\mu(x)   \qquad \e{and} \qquad     F_2(y) \, = \,  \Int{0}{y} |x| \, \mathrm{d}\mu(x) \, .
\enq
Clearly $F_1$ and $F_2$ are increasing functions,
\beq
F_1(0) \, = \,  F_2(0) \, = \,  0 \qquad \e{and} \qquad \underset{y \to +\infty}{\lim} F_1(y)  \,  = \underset{y \to +\infty}{\lim} F_2(y) \, = M_{\tf{1}{2}} \, .
\enq
Furthermore, because $\mu$ is non-atomic, $F_1$ and $F_2$ are continuous and $M_{\tf{1}{2}} > 0$.
By the intermediate value theorem there exists $\tilde{K}_1, \tilde{K}_2 > 0$
such that $F_1(\tilde{K}_1) = F_2(\tilde{K}_2) = \frac{1}{2}M_{\tf{1}{2}}$. Then given any $0 < \eta < \frac{1}{2}M_{\tf{1}{2}}$ let $K_1, K_2$ be defined by
$$M_{\tf{1}{2}} - F_1(K_1) \, = \,  M_{\tf{1}{2}} - F_2(K_2) \, =  \, \frac{ \eta }{2}  \min\{1, \tilde{K}_1, \tilde{K}_2 \} \, .$$
Clearly, one has
\beq
\Int{-K_1}{K_2} x \, \mathrm{d}\mu(x)  \, =  \, F_2(K_2) - F_1(K_1) = 0
\enq
and
\beq
 \Int{\R}{} \mathbbm{1}_{[-K_1;K_2]^c}(x) |x| \, \mathrm{d}\mu(x) \, = \,  2 M_{\tf{1}{2}} -  F_2(K_2) - F_1(K_1) \, \leq \,  \eta \, .
\enq
Finally by Chebyshev's inequality, and the fact that $K_1 \geq \tilde{K}_1$ and $K_2 \geq \tilde{K}_2$,
$$\mu\big[[-K_1; K_2]^c\big] \, \leq \,  \frac{1}{\min \{\tilde{K}_1, \tilde{K}_2, 1 \} } \Int{\R}{} \mathbbm{1}_{[-K_1;K_2]^c}(x)|x| \, \mathrm{d}\mu(x) \leq \eta \, .$$
Finally, one may estimate the growth of $K_{1},K_2$ by using Chebyshev's inequality. Indeed
\beq
\eta \, = \frac{ 2 }{\min \{\tilde{K}_1, \tilde{K}_2, 1 \} }  \, \Int{K_2}{+\infty}|x| \, \mathrm{d}\mu(x)
%
 \, \leq \,   \frac{ 2 K_2^{1-\theta }}{\min \{\tilde{K}_1, \tilde{K}_2, 1 \} } \Int{\mathbb{R}}{}|x|^{ \theta } \, \mathrm{d}\mu(x) = C K_2^{1-\theta} \;,
\enq
 and similar bounds hold for $K_1$.

\end{proof}

We remark that the requirement that $\mu$ is atomless cannot be dropped since it is possible to produce examples of atomic measures
for which \textit{no} compactification will preserve the first moment. We now apply Lemma \ref{compactification} to the measure $\op{s}_{\rho\#} \mu$ with  $\eta >0$ to yet be determined
and check that the associated measure $\mu^K_\rho$ defined in \eqref{definition mesures mu K rho} with $(K_1,K_2)$ as in
 Lemma \ref{compactification}
satisfies the hypotheses $\mathrm{i})-\mathrm{iv})$ of Lemma \ref{lemmeapprox}.
\begin{lemme}\label{weakdilationcomp}
Let $\mu \in \mathcal{M}_1(\mathbb{R})$ and $\rho > 1$, $\eta>0$. Then $\mu^K_\rho \to \op{s}_{\rho \#} \mu$ as $\eta\searrow 0$ in the weak topology. Moreover, $\mu^K_\rho$ is compactly supported,  $\int x \,  \mathrm{d} \mu^K_\rho(x)=0$ if $\int x \,  \mathrm{d} \mu(x)=0$ and $\int |x|^\theta  \, \mathrm{d} \mu^K_\rho(x)$ is finite if $\int |x|^\theta \,  \mathrm{d}\mu(x)$ is. Furthermore,
if $\ln(1+| \cdot  |)$ is $\mu$-integrable and $\mu$ satisfies \eqref{cond1}, then there exists $\eta'>0$ so that for all $\eta\in ]0;\eta']$
\begin{align}\label{bound2}
  1+\frac{2}{\ell} \Int{\mathbb{R}}{} \ln|x-y| \, \mathrm{d}\mu^K_\rho (y) \, \geq \, \kappa=\frac{1}{\ell} \ln \rho > 0,\quad \mu^K_\rho \quad a.e.
\end{align}
\end{lemme}
\begin{proof}

It is direct to see that
\beq
\e{d}_{\e{BL}}\big( \op{s}_{\rho\#} \mu , \mu^K_{\rho} \big)  \, \leq \,   \op{s}_{\rho\#}   \mu[ K^{\e{c}} ] \leq   \eta  \, .
\label{ecritrue estimee d BL sur compactification dilation mu}
\enq
Lemma \ref{compactification} implies the other statements except for \eqref{bound2}.
Furthermore, for $x \in K \cap \mathrm{supp}\big[ \op{s}_{\rho\#} \mu\big]$, we have
\beq
1+ \frac{2}{\ell} \Int{\mathbb{R}}{} \ln|x-y| \, \mathrm{d}\mu^K_{\rho}(y) \, = \, \mc{T}_1 \, + \,  \mc{T}_2 \, + \, \mc{T}_3
\enq
where
\beq
\mc{T}_1 \; = \; 1 - \frac{1}{\op{s}_{\rho\#} \mu[ K] } \, =\,  - \frac{  \op{s}_{\rho\#} \mu[ K^{\e{c}} ]  }{  \op{s}_{\rho\#} \mu[ K ] } \, \geq \,  \f{-\eta }{1-\eta} \;,
\enq
\beq
\mc{T}_2 \; = \; \frac{ 1  }{ \op{s}_{\rho\#} \mu[ K ] }  \Big\{ 1+ \frac{2}{\ell} \Int{\mathbb{R}}{} \ln|x-y| \, \mathrm{d}\op{s}_{\rho\#} \mu(y)   \Big\} \, \geq \, 2\kappa \;,
\enq
$\mathrm{d}\op{s}_{\rho\#} \mu(x)$ a.e. by \eqref{bound}. Finally,
\begin{align*}
\mc{T}_3 &\; = \;  \frac{ -2   }{ \ell \op{s}_{\rho\#} \mu[ K ] } \Int{\mathbb{R}}{} \mathbbm{1}_{K^{\e{c}} }(y) \ln|x-y| \, \mathrm{d}\op{s}_{\rho\#} \mu(y) \\
&\geq \frac{ -2   }{ \ell (1-\eta)  }\Int{\mathbb{R}}{} \mathbbm{1}_{K^{\e{c}} }(y) \Big( \ln(1+|x|) +\ln(1+|y|) \Big)  \, \mathrm{d}\op{s}_{\rho\#} \mu(y) \\
&\geq  \frac{ -2    }{ \ell (1-\eta)  } \Big\{ \eta \ln \big( 1+\e{max}\{K_1,K_2\} \big) \, + \, \Int{\mathbb{R}}{} \mathbbm{1}_{K^{\e{c}} }(y) \ln(1+|y|) \, \mathrm{d}\op{s}_{\rho\#} \mu(y) \Big\} \, . \\
\end{align*}
Since by assumption $\int \ln(1+|y|) \, \mathrm{d}\mu(y) < +\infty$, the \textit{rhs} tends to $0$ as $\eta \searrow 0$.
Putting this all together, one gets that there exists $\eta^{\prime}>0$ small enough such that
\beq
1+ \frac{2}{\ell} \Int{\mathbb{R}}{} \ln|x-y| \, \mathrm{d}\mu^K_{\rho}(y) \, \geq  \,  \kappa  \qquad  \mathrm{d}\mu^K_{\rho}(x) \; \; \e{a.e.}
\enq
whenever $0< \eta \leq \eta^{\prime}$.

\end{proof}
Finally, to obtain a measure with bounded density, for $0 < \tau \leq \frac{\delta}{2}$ with $\de>0$ small enough, let
$$\mu^K_{\rho,\tau} \, = (\mu^K_\rho)_\tau= \, \frac{1}{2\tau} \mathbbm{1}_{[-\tau;\tau]} \ast \mu^K_{\rho} \, .$$
To ensure that it still satisfies $\mathrm{ii})$ of Lemma \ref{lemmeapprox},
we will use the following two results. First, we will rely on the following Lemma
from Saff and Totik (see Section I.3, p. 43, in \cite{SaffT97}).

\begin{lemme}[Principle of Domination]
\label{Lemme Principe Domination}
Let $\mu \in \mathcal{M}_1(\mathbb{R})$ be  compactly supported and such that
$$\Int{\mathbb{R}^2}{} \ln \frac{1}{|x-y|}\, \mathrm{d}\mu(x) \, \mathrm{d}\mu(y) < +\infty \,. $$
Suppose there exists a constant $C$ such that
$$ \Int{\mathbb{R}}{} \ln|x-y| \, \mathrm{d}\mu(y) \geq C \qquad \dd \mu(x)-\e{a.e.} \; on \;\;  \mathrm{supp}[\mu] \;. $$
Then, in fact,
$$ \Int{\mathbb{R}}{} \ln|x-y| \, \mathrm{d}\mu(y) \geq C  \qquad \text{for every} \quad  x \in \mathbb{R} \, .$$

\end{lemme}
This allows us to show that
\begin{lemme}\label{principdom} Let $\mu \in \mathcal{M}_1(\mathbb{R})$ be compactly supported and such that
\begin{align}
1 + \frac{2}{\ell} \Int{\mathbb{R}}{} \ln|x-y| \, \mathrm{d}\mu(y) \geq \kappa > 0 \qquad \dd \mu(x)\;\;  \e{a.e.} \; x \in \mathbb{R} \;.
\label{ecriture borne mu ae sur integrale logarithmique}
\end{align}
Given  $\tau > 0$ set $\mu_\tau = \frac{1}{2\tau} \mathbbm{1}_{[-\tau;\tau]} \ast \mu$. Then,
\begin{align*}
1 + \frac{2}{\ell} \Int{\mathbb{R}}{} \ln|x-y| \, \mathrm{d}\mu_\tau(y) \geq \kappa  & & \forall x \in \mathbb{R} \, .
\end{align*}
\end{lemme}

\begin{proof}
By Lemma \ref{Lemme Principe Domination}, \eqref{ecriture borne mu ae sur integrale logarithmique} holds, in fact, for all $x \in \R$. It is then enough to integrate that inequality over $x$
 \textit{versus} $\frac{1}{2\tau} \mathbbm{1}_{[-\tau;\tau]}(s-x)$.

\end{proof}

We directly deduce from Lemmas \ref{weakdilationcomp} and \ref{principdom}  that
\begin{lemme}\label{weakdilationcompreg}
Let $\mu \in \mathcal{M}_1(\mathbb{R})$ and $\rho > 1$, $\eta,\delta>0$. Then $\mu^K_{\rho,\tau} \to \mu^K_{\rho}$ as $\tau\searrow 0$ in the weak topology.
Moreover, $\mu^K_{\rho,\tau}$ is compactly supported, has a bounded density,  $\int x \,  \mathrm{d} \mu^K_{\rho,\tau}(x)=0$ if
$\int x \, \mathrm{d} \mu(x)=0$ and $\int |x|^\theta \,  \mathrm{d} \mu^K_{\rho,\tau}(x)$ is finite if $\int |x|^\theta \, \mathrm{d}\mu(x)$ is. Furthermore,
if $\mu$ satisfies \eqref{cond1}, there exists $\eta'>0$ so that for $\eta\in ]0;\eta']$
\begin{align}\label{bound3}
  1+\frac{2}{\ell} \Int{\mathbb{R}}{} \ln|x-y| \, \mathrm{d}\mu^K_{\rho,\tau} (y) \, \geq \, \kappa =\frac{1}{\ell} \ln \rho > 0,\quad \forall x\in \mathbb{R}.
\end{align}
\end{lemme}
We can therefore take $\mu_p= \mu^{K^p}_{\rho_p,\tau_p}$ with $K^p_1,K^p_2$ large enough and $\rho_p,\tau_p>0$ small enough so that
$\e{d}_{BL}(\mu,\mu_p)\le 1/p$ to complete the proof of  Lemma \ref{lemmeapprox}$\mathrm{i)-v)}$.
So, we are left with verifying Lemma \ref{lemmeapprox}-$\mathrm{vi}$, which will take the rest of this section. Since $I$ is lower semi-continuous, we only need to prove that
\begin{equation}\label{contI}
\limsup_{\rho \searrow 1 }\lim_{\eta\searrow 0} \limsup_{\tau \searrow 0 } I[\mu^K_{\rho,\tau}]\le I[\mu]\end{equation}

We first take the limit $\tau\searrow 0$. To this end,
we invoke

\begin{lemme}\label{entropybound} Assume that \ref{potentialh1} holds. Then there exists universal non-negative finite constants $C_1',C_2'$
so that for every probability measure $\mu$ on $\mathbb R$ and $\tau\in [0;1]$, with $\mu_\tau$ as in \eqref{defconv}, we have

$$I[\mu_\tau]\le I[\mu]+ (C_1'I[\mu]+C_2')\tau\,.$$
\end{lemme}
\begin{proof} Note that we may assume that $I[\mu]$ is finite to prove the desired inequality.
In fact, by the convexity of $I$ proven in Proposition \ref{propconv}, we find that
$$I[\mu_\tau]\le \frac{1}{2\tau}\Int{-\tau}{\tau} I[\mf{t}_{t\#}\mu] \mathrm{d}t\,$$
where $\mf{t}_{t\#}\mu(A)=\mu(t+A)$ is the translation of the measure by $t$. We finally notice that in the rate function,
only the term from the potential varies with the translation: for every $t\in [-\tau;\tau]$
$$I[\mf{t}_{t\#}\mu]-I[\mu]= \int V(t+x)\mathrm{d}\mu(x)-\int V(x)\mathrm{d}\mu(x)\,.$$
By \ref{potentialh1}, we find that for $\tau\le 1$
$$\left|\int V(t+x)\mathrm{d}\mu(x)-\int V(x)\mathrm{d}\mu(x)\right|\le t \int (C_1 V(x)+C_2) \mathrm{d}\mu(x)$$
which completes the proof with \eqref{ratefunctionlowerbound}.

\end{proof}

We now focus on taking the $\eta \searrow 0$  limit. First, one observes that owing to $V$ being bounded from below, and $K_1, K_2 \nearrow + \infty$
as $\eta \searrow 0$, one may apply monotone convergence theorem to get that
$$ \Int{\mathbb{R}}{} V(x)  \, \mathrm{d}\mu^K_{\rho}(x) \; \longrightarrow \;  \Int{\mathbb{R}}{} V(x )  \, \mathrm{d}\op{s}_{\rho\#}\mu(x) \; . $$
Further, we focus on the entropy term. One has, upon denoting $\sg_{\rho}$ the density of  $\op{s}_{\rho \#} \mu$,
$$
\Int{\mathbb{R}}{} \ln \big[ \sg^K_{\rho}(x) \big] \, \sg^K_{\rho}(x) \, \mathrm{d}x  \,  =  \,
\frac{ 1 }{   (\op{s}_{\rho \#} \mu)[K] } \Int{ K }{} \ln  \big[ \sg_\rho(x) \big] \, \sg_\rho(x)   \, \mathrm{d}x
\, - \,  \ln \big\{ (\op{s}_{\rho \#} \mu)[K] \big\} \, .
$$
\begin{lemme}
Let $\mu \in \mathcal{M}_1(\mathbb{R})$. If $I[\mu] < +\infty$ then $\mu$ is absolutely continuous and its density satisfies
$\Int{\mathbb{R}}{} |\ln \frac{\mathrm{d}\mu}{\mathrm{d}x}(x)| \, \mathrm{d}\mu(x) < +\infty$.
\end{lemme}
\begin{proof}
Without loss of generality, assume $V \geq 0$. By \eqref{ratefunctionlowerbound}, $I[\mu] < +\infty$ implies $\Int{\mathbb{R}}{} V \, \mathrm{d}\mu < +\infty$. Furthermore, if $\mu$ was not absolutely continuous then $I[\mu] = +\infty$. Next, from the concavity of $x \mapsto x \ln{x^{-1}}$ for $x > 0$ we have
\begin{align*}
    a \ln{a^{-1}} \leq a b + \mathrm{e}^{-b-1}, & & \text{for all } a > 0, \,  b \in \mathbb{R} \, .
\end{align*}
Let us take $a = \frac{\mathrm{d}\mu}{\mathrm{d}x}$, $b = c V(x)$ for $c >0$, and then integrate both sides with respect to $\mu$ over the set
$\{ \frac{\mathrm{d}\mu}{\mathrm{d}x} < 1 \}$. This gives
\begin{align*}
\Int{\{ \frac{\mathrm{d}\mu}{\mathrm{d}x} < 1\}}{} \Big|\ln \frac{\mathrm{d}\mu}{\mathrm{d}x}(x)\Big| \, \mathrm{d}\mu(x)
\leq c \Int{\mathbb{R}}{} V(x) \, \mathrm{d}\mu(x) +  \mathrm{e}^{-1} \Int{\mathbb{R}}{} \mathrm{e}^{- c V(x) }\, \mathrm{d}\mu(x) \, .
\end{align*}
Then, using the inequality $\ln|x-y| \leq \ln(1+|x|) + \ln(1+|y|)$ we have
\begin{align*}
I[\mu] &\geq \Int{\mathbb{R}}{}\Big\{  V(x) -  \frac{4}{\ell}\ln(1+|x|) \Big\}\, \mathrm{d}\mu(x) + \int \ln \frac{\mathrm{d}\mu}{\mathrm{d}x}(x) \; \mathrm{d}\mu(x) \\
&= \Int{\mathbb{R}}{}\Big\{  V(x) -  \frac{4}{\ell}\ln(1+|x|) \Big\}\, \mathrm{d}\mu(x) +
\Int{\mathbb{R}}{} |\ln \frac{\mathrm{d}\mu}{\mathrm{d}x}(x)|\, \mathrm{d}\mu(x) -
2\Int{\{ \frac{\mathrm{d}\mu}{\mathrm{d}x} < 1\}}{} \Big|\ln \frac{\mathrm{d}\mu}{\mathrm{d}x}(x)\Big| \, \mathrm{d}\mu(x) \\
&\geq \Int{\mathbb{R}}{}\Big\{ (1-2c) V(x) -  \frac{4}{\ell}\ln(1+|x|) \Big\}\, \mathrm{d}\mu(x)
+ \Int{\mathbb{R}}{} \Big|\ln \frac{\mathrm{d}\mu}{\mathrm{d}x}(x) \Big|\, \mathrm{d}\mu(x) \\
&\quad - 2\mathrm{e}^{-1} \Int{\mathbb{R}}{} \mathrm{e}^{- c V(x) }\, \mathrm{d}\mu(x) \, .
\end{align*}
If we now take $c=\frac{1}{4}$, by \ref{potentialh2}, $\frac{1}{2} V(x) -  \frac{4}{\ell}\ln(1+|x|)$ is bounded from below by a constant and $\Int{\mathbb{R}}{} \mathrm{e}^{- \frac{1}{4} V(x) }\, \mathrm{d}\mu(x) < +\infty$. Hence $\Int{\mathbb{R}}{} |\ln \frac{\mathrm{d}\mu}{\mathrm{d}x}(x)|\, \mathrm{d}\mu(x) < +\infty$.
\end{proof}

From the above lemma, we conclude that $\ln \sg \in L^{1}(\mu)$ with $\dd \mu(x) \, = \, \sg(x) \dd x $. Furthermore, a direct calculation yields
$ \sg_{\rho}(x) =\rho^{-1} \sg\big(\tf{x}{\rho} \big)$. Thus,
\bem
 \Int{ K }{} \ln  \big[ \sg_\rho(x) \big] \, \sg_\rho(x)   \, \mathrm{d}x \, = \,  \Int{ \rho K }{} \ln \Big( \f{ \sg(x) }{ \rho} \Big)  \, \sg(x)   \, \mathrm{d}x \\
\limit{\eta}{0}  \Int{ \R }{} \ln \Big( \f{ \sg(x) }{ \rho} \Big)  \, \sg(x)   \, \mathrm{d}x \; = \; \Int{ \R }{} \ln  \sg_{\rho}(x)   \, \sg_{\rho}(x)   \, \mathrm{d}x
\end{multline}
by dominated convergence, since
\beq
\mathbbm{1}_{\rho K }(x) \ln \Big( \f{ \sg(x) }{ \rho} \Big)  \, \sg(x) \; \leq \; \sg(x) \ln \rho  \,  +\,    \sg(x) \ln  \sg(x)  \in L^{1}\big( \dd x \big) \;.
\enq
Given that $ (\op{s}_{\rho \#} \mu)[K] \tend 1$ as $\eta \tend 0$, we infer that
\beq
\Int{\mathbb{R}}{} \ln \sg^K_{\rho}(x) \, \sg^K_{\rho}(x) \, \mathrm{d} x \limit{\eta}{0} \Int{ \R }{} \ln  \sg_{\rho}(x)   \, \sg_{\rho}(x) \, \mathrm{d} x  \;.
\enq

We now estimate the $\eta \tend 0$ limit of the remaining contribution
\begin{lemme}
It holds
$$\Int{\mathbb{R}}{} \ln\bigg( 1 + \frac{2}{\ell} \Int{\mathbb{R}}{} \ln |x - y|  \, \mathrm{d}\mu^K_{\rho}(y)\bigg)  \, \mathrm{d}\mu^K_{\rho}(x)
\underset{\eta \searrow 0}{\longrightarrow}
\int \ln\bigg( 1 + \frac{2}{\ell} \Int{\mathbb{R}}{} \ln |x - y|  \, \mathrm{d}(\op{s}_{\rho\#}\mu)(y)\bigg)  \, \mathrm{d}(\op{s}_{\rho\#}\mu)(x) \, .$$
\end{lemme}
\begin{proof}
Recall that $K=\intff{-K_1}{K_2}$ is the support of $\mu^K_{\rho}$. First of all, one has
\bem
\Int{\mathbb{R}}{} \ln\bigg( 1 + \frac{2}{\ell} \Int{\mathbb{R}}{} \ln |x - y|  \, \mathrm{d}\mu^K_{\rho}(y)\bigg)  \, \mathrm{d}\mu^K_{\rho}(x)  \\
\, = \, \Int{\R}{}  \ln\bigg( 1 + \frac{2}{\ell} \Int{K}{} \ln |x - y|  \, \mathrm{d}\mu^K_{\rho}(y)\bigg)  \, \f{ \mathrm{d}(\op{s}_{\rho\#}\mu)(x) }{ (\op{s}_{\rho\#}\mu)[K] }
 \, - \, \f{ \mc{T} }{   (\op{s}_{\rho\#}\mu)[K] }\, ,
\nonumber
\end{multline}
where we have set
\beq
 \mc{T} \ = \, \Int{ K^{\e{c}} }{}  \ln\bigg( 1 + \frac{2}{\ell} \Int{K}{} \ln |x - y|  \, \mathrm{d}\mu^K_{\rho}(y)\bigg)  \,  \mathrm{d}(\op{s}_{\rho\#}\mu)(x)
\enq
Further, for $m \leq x \leq M$ it holds that $|\ln x| \, \leq \, |\ln m|+ |\ln M|$. Hence, since
\beq
\forall x \in \R \, , \quad 1 + \frac{2}{\ell} \Int{K}{} \ln |x - y|  \, \mathrm{d}\mu^K_{\rho}(y) \, \geq \, \kappa
\quad \e{and} \quad \ln|x-y| \, \leq \, \ln (1+|x|) + \ln (1+|y|)  \;,
\enq
we get that
\bem
|\mc{T}| \, \leq \,  (\op{s}_{\rho\#}\mu)[K^{\e{c}}] \big| \ln \kappa \big|\, + \,
\Int{ K^{\e{c}} }{}  \ln\bigg( 1 + \frac{2}{\ell} \ln (1+|x|)  \, + \, \f{2}{\ell} \Int{K}{} \ln (1+|y|)  \, \mathrm{d}\mu^K_{\rho}(y)\bigg)  \,  \mathrm{d}(\op{s}_{\rho\#}\mu)(x) \\
\, \leq \, \eta \big| \ln \kappa \big| \, + \, C \Int{ K^{\e{c}} }{}  \ln(1+|x|) \,  \mathrm{d}(\op{s}_{\rho\#}\mu)(x)
\end{multline}
which tends to $0$ as $\eta\searrow 0$. We now focus on the last difference:
\bem
\De \, = \,
\bigg| \Int{\R}{}  \ln\bigg( 1 + \frac{2}{\ell} \Int{K}{} \ln |x - y|  \, \mathrm{d}\mu^K_{\rho}(y)\bigg)  \,  \mathrm{d}(\op{s}_{\rho\#}\mu)(x) \\
 \hspace{3cm} \, - \, \Int{\R}{}  \ln\bigg( 1 + \frac{2}{\ell} \Int{\R}{} \ln |x - y|  \, \mathrm{d}(\op{s}_{\rho\#}\mu)(y)\bigg)  \,   \mathrm{d}(\op{s}_{\rho\#}\mu)(x)  \bigg|  \\
 \, \leq \,  \f{2}{\ell \kappa} \Int{\R}{}   \mathrm{d}(\op{s}_{\rho\#}\mu)(x)   \Big| \Int{\R}{} \ln |x - y|   \mathrm{d}\big(\mu^K_{\rho}-\op{s}_{\rho\#}\mu \big)(y) \Big| \\
\, \leq \,  \f{2}{\ell \kappa} \bigg\{  \f{ (\op{s}_{\rho\#}\mu)[K^{\e{c}}] }{ (\op{s}_{\rho\#}\mu)[K] } \Int{ \R^2 }{}  \big| \ln |x - y| \big|    \mathrm{d}^2(\op{s}_{\rho\#}\mu)(x,y) \\
\, + \, \Int{ \R^2 }{}  \big| \ln |x - y| \big|  \mathbbm{1}_{K^{\e{c}}}(y)  \f{ \mathrm{d}^2(\op{s}_{\rho\#}\mu)(x,y) }{ (\op{s}_{\rho\#}\mu)[K] } \bigg\} \, .
\end{multline}
The vanishing of $\De$ will then follow from $(x,y) \mapsto \big| \ln |x - y| \big| \in L^1\big( (\op{s}_{\rho\#}\mu)^{\otimes 2} \big)$.
Indeed, the first term will go to $0$
due to the prefactor and the second one by dominated convergence since $\mathbbm{1}_{K^{\e{c}}}(y) \tend 0$ pointwise.
Recall that it holds
\begin{align*}
1+ \frac{2}{\ell} \Int{\mathbb{R}}{} \ln|x-y| \, \mathrm{d}(\op{s}_{\rho\#}\mu)(y) \geq 0, & & \op{s}_{\rho\#}\mu-\text{a.e. } \,  x \in \mathbb{R},
\end{align*}
and so $\Int{\mathbb{R}}{} \ln|x-y| \, \mathrm{d}(\op{s}_{\rho\#}\mu)(x) \, \mathrm{d}(\op{s}_{\rho\#}\mu)(y) \geq - \frac{\ell}{2} $. Thus
\bem
\Int{\mathbb{R}}{} |\ln|x-y|| \, \mathrm{d}(\op{s}_{\rho\#}\mu)(x) \, \mathrm{d}(\op{s}_{\rho\#}\mu)(y)  \\
= 2 \Int{\mathbb{R}^2}{}\ln|x-y| \, \mathbbm{1}_{|x-y| \geq 1} \, \mathrm{d}(\op{s}_{\rho\#}\mu)(x) \, \mathrm{d}(\op{s}_{\rho\#}\mu)(y)
 - \Int{\mathbb{R}^2}{} \ln|x-y| \, \mathrm{d}(\op{s}_{\rho\#}\mu)(x) \, \mathrm{d}(\op{s}_{\rho\#}\mu)(y) \\
 \leq 4 \Int{\mathbb{R}^2}{}\ln(1+|x|) \, \mathrm{d}(\op{s}_{\rho\#}\mu)(x) + \frac{\ell}{2} < +\infty \, .
\end{multline}
\end{proof}

We have now reached the final step of the proof, namely to take the $\rho \searrow 1$ limit. Because $\ln$ is an increasing function, we have that
\bem
\Int{\mathbb{R}}{} \ln\bigg( 1 + \frac{2}{\ell} \Int{\mathbb{R}}{} \ln |x - y|  \, \mathrm{d}(\op{s}_{\rho\#}\mu) (y)\bigg)  \, \mathrm{d}(\op{s}_{\rho\#}\mu)(x) \\
= \Int{\mathbb{R}}{} \ln\bigg( 1 + \f{2}{\ell} \ln \rho \, + \,  \frac{2}{\ell} \Int{\mathbb{R}}{} \ln |x - y|  \, \mathrm{d}\mu (y)\bigg)  \, \mathrm{d}\mu(x) \\
 \geq  \Int{\mathbb{R}}{} \ln\bigg( 1   \, + \,  \frac{2}{\ell} \Int{\mathbb{R}}{} \ln |x - y|  \, \mathrm{d}\mu (y)\bigg)  \, \mathrm{d}\mu(x)   \, .
\end{multline}
With regard to the entropy, given that the densities of $\op{s}_{\rho\#}\mu$ and $\mu$ are related as $\sg_\rho(x) \, = \,  \rho^{-1}\sg(\tf{x}{\rho} )$ one has
$$ \Int{\mathbb{R}}{} \ln \big[ \sg_\rho(x) \big]\, \sg_\rho(x) \, \mathrm{d}x  \,  = \,  - \ln \rho \,  +\,
\Int{\mathbb{R}}{} \ln \big[ \sg(x) \big] \, \sg(x) \, \mathrm{d}x \underset{\rho\searrow 1}{\longrightarrow}
\Int{\mathbb{R}}{} \ln \big[\sg(x) \big] \, \sg(x) \, \mathrm{d}x \, .$$

It thus remains to establish the limit for the term involving the potential $V$.
\begin{lemme}
Let $V$ be continuous and satisfying \ref{potentialh2}, and suppose $\mu \in \mathcal{M}_1(\mathbb{R})$ is such that $\Int{\mathbb{R}}{} V \, \mathrm{d}\mu <+\infty$. Then,
$$\lim_{\rho\searrow 1} \Int{\mathbb{R}}{} V(\rho x) \, \mathrm{d}\mu(x)  = \Int{\mathbb{R}}{} V( x) \, \mathrm{d}\mu(x) \, .$$
\end{lemme}
\begin{proof}
 By \ref{potentialh2}, there exists $C_1, \tilde{C}_1 \in \mathbb{R}$ and $C_2, \tilde{C}_2 > 0$ such that
$$C_1 + C_2 |x|^\theta \leq V(x) \leq \tilde{C}_1 + \tilde{C}_2 |x|^\theta \, .$$ Next, let $R > 0$ be a (large) constant. Then
\begin{align*}
\Int{[-R;R]^{\e{c}}}{} V(\rho x) \, \mathrm{d}\mu(x) &\leq \tilde{C}_1 \mu\big[ [-R;R]^{\e{c}}\big]
+ \tilde{C}_2\rho^\theta  \Int{[-R;R]^{\e{c}}}{} |x|^\theta \, \mathrm{d}\mu(x) \\
&\leq (\tilde{C}_1 - C_1 C_2^{-1} \rho^\theta) \mu\big[ [-R;R]^{\e{c}}\big]
+ \tilde{C}_2 C_2^{-1}\rho^\theta \hspace{-3mm} \Int{[-R;R]^{\e{c}}}{} \hspace{-3mm} V(x) \, \mathrm{d}\mu (x) \, .
\end{align*}
Then given any $\eps > 0$
\beq
\limsup_{\rho \searrow 1} \hspace{-3mm}   \Int{[-R;R]^{\e{c}}}{}  \hspace{-3mm}  V(\rho x) \, \mathrm{d}\mu(x)
\leq (\tilde{C}_1 - C_1 C_2^{-1}) \mu\big[ [-R;R]^{\e{c}}\big]
+ \tilde{C}_2 C_2^{-1}  \hspace{-3mm} \Int{[-R;R]^{\e{c}}}{} \hspace{-3mm}  V(x) \, \mathrm{d}\mu(x) \, \leq  \, \eps
\nonumber
\enq
for $R > 0$ sufficiently large. By uniform continuity on compact sets $$\Int{-R}{R} V(\rho x) \, \mathrm{d}\mu(x) \overset{\rho \searrow 1}{\longrightarrow}  \Int{-R}{R} V(x) \, \mathrm{d}\mu(x)$$ by uniform continuity of $V$ on compact sets. Since $\eps >0$ was arbitrary we conclude the result.
\end{proof}





\section*{Acknowledgement}

K.K.K. acknowledges support from  CNRS and ENS de Lyon.  A.G., K.K.K. and A.L. are supported  by the ERC Project LDRAM : ERC-2019-ADG Project 884584.
K.K.K. and A.L. are supported by the joint AND-DFG TSF24 project ANR-24-CE92-0033
 T.G. acknowledges the support of PRIN 2022 (2022TEB52W) "The charm of
integrability: from nonlinear waves to random matrices"-– Next Generation EU grant – PNRR Investimento M.4C.2.1.1 - CUP: G53D23001880006; the GNFM-INDAM group and the research project Mathematical Methods in NonLinear Physics (MMNLP), Gruppo 4-Fisica Teorica of INFN.
The authors wish to thank Alexander Its and Herbert Spohn for the many insightful comments during the preparation of this manuscript.





\appendix





\section{Conditions on the roots}
\label{section:condroots}

In our rate function $I$, a necessary condition for $I[\mu]< +\infty$ is that $1+ \frac{2}{\ell} \Int{\mathbb{R}}{} \ln|x-y| \, \mathrm{d}\mu(y) \geq 0$
for $\mu$-almost every $x \in \mathbb{R}$. This unusual-looking non-local constraint arises from the requirement
that both $P-2\varepsilon_N$ and $P+2\varepsilon_N$ have all real roots (where $P$ is defined in \eqref{defP1}).
In this section we prove a finite-$N$ version of this inequality. Namely, we consider how the reality of the roots
of $P\pm 2\varepsilon_N$ implies a lower bound on
\begin{align}\label{Petamin}
    \min_{k \in \intn{1}{N}} \left( 1- \frac{1}{\ln(2\veps_N)} \ln |P^\prime(\eta_k)| \right)
\end{align}
and vice versa, where we recall that the entries of the vector $\boldsymbol{\eta}_N$ are given by the roots of $P$. We begin by establishing a characterisation of the reality of the roots of $P\pm 2\varepsilon_N$.

\begin{prop}\label{etalambdaequivalence}
Let $\boldsymbol{\lambda}^+_N \in \mathbb{R}^N_{<}$. Then the following are equivalent.
\begin{enumerate}
    \item $\boldsymbol{\lambda}^+_N \in \mathcal{A}_N$.
    \item There exists a vector $\boldsymbol{\eta}_N \in \mathbb{R}^N_{<}$ such that the following holds. Let
$$P(x) := \prod_{k=1}^N ( x - \eta_k)$$ and let $\zeta_1, \dots, \zeta_{N-1} \in \mathbb{R}$ be the roots of $P^\prime$. Then $|P(\zeta_k)| > 2\veps_N$ for all $k \in \intn{1}{N-1}$ and $\lambda_1^+, \dots, \lambda_N^+$ are the roots of $P - 2\veps_N$.
\end{enumerate}
Note that this vector $\boldsymbol{\eta}_N \in \mathbb{R}^N_{<}$, if it exists, is unique since we fix the ordering.
\end{prop}
\begin{proof}
Let $\boldsymbol{\lambda}^+_N \in \mathbb{R}^N_{<}$, and define
\begin{align*}
P^+(x) = \prod_{k=1}^N ( x  - \lambda_k^+) \, , & & P(x) = P^+(x) + 2\veps_N \, .
\end{align*}
Suppose that ii) holds. Then by the intermediate value theorem, for every $k \in \intn{1}{N-1}$, there is a zero of $P^\prime$ in the interval $]\eta_k ; \eta_{k+1}[$. Let us call this zero $\zeta_k$, and we see by counting zeroes that $\zeta_1, \dots, \zeta_{N-1}$ are \textit{all} the zeros of $P^\prime$. Suppose that $P < 0$ on the interval $]\eta_k ; \eta_{k+1}[$.

Then by the intermediate value theorem, since $P(\zeta_k) < - 2\veps_N$, there must be a pair of solutions $\lambda_k^-$, $\lambda_{k+1}^-$ to the equation $P(\mu) = - 2\veps_N$ such that $$\eta_k < \lambda_k^- < \zeta_k < \lambda_{k+1}^- < \eta_{k+1} \, .$$ If $N$ is even there are $\frac{N}{2}$ such intervals, and hence we have a complete set of real distinct roots for the polynomial $P + 2\veps_N$. If $N$ is odd there are $\frac{N-1}{2}$ such intervals, which gives us $N-1$ real distinct roots. In this latter case, an additional root is found $\lambda_1^- < \eta_1$, since $P(\eta_1) = 0$ and $P(\mu)  \to -\infty$ as $\mu \to -\infty$. Thus $P + 2\veps_N$ has $N$ distinct real roots. A completely symmetrical argument shows that the roots of $P^+ = P- 2\veps_N$ are distinct.

For the converse statement, suppose by way of contradiction that i) is true but ii) is false. By similar reasoning as before, the intermediate value theorem
tells us there exists a series of roots of $P^\prime$ which interlace with the roots of $P^+$,
$$\lambda_k^+ < \zeta_k < \lambda_{k+1}^+ \, .$$
On every interval $]\lambda_k^+ ; \lambda_{k+1}^+[$ on which $P^+$ is negative, $P^- = P^+ + 4\veps_N$ can have at most two simple roots, since $P^+$
first monotonically decreases and then monotonically increases on this interval. Since $\boldsymbol{\lambda}^+_N \in \mathcal{A}_N$, by counting roots in a similar
manner as before, we must have pairs of simple roots of $P^-$,  $\lambda_k^-$, $\lambda_{k+1}^-$, on all of these intervals. Hence
$$\lambda_k^+ < \lambda_k^- < \zeta_k < \lambda_{k+1}^- < \lambda_{k+1}^+ \, .$$
Then, by the intermediate value theorem, $P = P^+ + 2\veps_N$ must have  roots $\{ \eta_k \}_{k=1}^N$ such that
$$\lambda_k^+ < \eta_k < \lambda_k^- < \zeta_k < \lambda_{k+1}^- < \eta_{k+1} < \lambda_{k+1}^+ \, .$$
Finally, suppose by way of contradiction that $|P(\zeta_k)| \leq 2 \veps_N$ for some $k \in \intn{1}{N-1}$. If $|P(\zeta_k)| = 2 \veps_N$,
then $P^+(\zeta_k)P^-(\zeta_k) = P(\zeta_k)^2 - 4 \veps_N^2 = 0$, and so $\zeta_k$ is a root of either $P^+$ or $P^-$, which would contradict the simplicity,
\textit{i.e.} distinctness, of the roots. If $|P(\zeta_k)| < 2 \veps_N$ then, since $\zeta_k$ is a local extremum, we have $|P(\mu)| < 2 \veps_N$ for all
$\mu \in ]\lambda_k^+; \lambda_{k+1}^+[ \cup ]\lambda_k^-; \lambda_{k+1}^-[ $, and so $P(\mu) = \pm 2 \veps_N$ has no solutions on this interval. But this contradicts
the existence of two roots on this interval.
\end{proof}
From the above proposition, the idea is to deduce a lower bound on $\min_{j \in \intn{1}{N}} |P^\prime(\eta_j)|$ from the assumption of
a lower bound on $\min_{k \in \intn{1}{N-1}} |P(\zeta_k)|$, since the latter is equivalent to $\boldsymbol{\lambda}^+_N \in \mathcal{A}_N$,
or equivalently to $P\pm 2\varepsilon_N$ having full sets of real roots. Such a relation is given by the following inequality.
\begin{prop}\label{inequality}

Let $\bs{\eta}_N \in \R^N_{<}$, then one has
\begin{align*}
 \left( \min_{k \in \intn{1}{N-1}} |P(\zeta_k)|\right)^{\frac{N-1}{N}} \leq N \min_{j \in \intn{1}{N}} |P^\prime(\eta_j)| \, .
\end{align*}
\end{prop}
\begin{proof}
To begin with, assume $k \not\in \{ 1 ,N \}$. Then by the interlacing property there is an $\alpha \in ]0;1[$ such that
$\eta_k =  \a \zeta_{k-1} + (1-\alpha) \zeta_k$. Then by Young's inequality (or Jensen's, equivalently)
\begin{align*}
&|P(\zeta_{k-1})|^{\alpha} |P(\zeta_{k})|^{1-\alpha} \\
&\leq |\zeta_{k-1} - \eta_k|^\alpha |\zeta_{k} - \eta_k|^{1-\alpha} |P^\prime(\eta_k)| \\
&\leq N \Big(  |\zeta_{k-1} - \eta_k| \wedge  |\zeta_{k-1} - \eta_{k-1}|  \Big)^\alpha
\Big(  |\zeta_{k} - \eta_k|  \wedge  |\zeta_{k} - \eta_{k-1}| \Big)^{1-\alpha} |P^\prime(\eta_k)| \\
&\leq N |P(\zeta_{k-1})|^\frac{\alpha}{N} |P(\zeta_{k})|^{\frac{1-\alpha}{N}} |P^\prime(\eta_k)|
\end{align*}
where in the second inequality we have used Lemma \ref{Lemme espacement racine lambda pm et zeta j}.
Thus,
$$N |P^{\prime}(\eta_k)| \geq   \big( |P(\zeta_{k-1})|^{\alpha} |P(\zeta_{k})|^{1-\alpha} \big)^{ (N-1)/N }$$
For $k=1$ we may write
$$|P(\zeta_1)| \leq |\zeta_1 - \eta_1| |P^\prime(\eta_1)| \leq |P(\zeta_1)|^\frac{1}{N}|P^\prime(\eta_1)| $$
and similarly for $k=N$. Then, taking the minimum of both sides yields the result.
\end{proof}
Put together,  Propositions \ref{etalambdaequivalence}-\ref{inequality} imply that
\begin{align}
    \min_{k \in \intn{1}{N}} \Big( 1- \frac{1}{\ln(2\veps_N)} \ln |P^\prime(\eta_k)| \Big) \geq - \mathrm{O}\Big( \frac{\ln N}{N} \Big)\,.
\end{align}
Thus, although we are not able to show that \eqref{Petamin} is strictly positive
for any $\bs{\eta}_N$ subordinate to $\bs{\la}^+_N\in \mc{A}_N$, it is bounded from below by a negative number which tends to $0$. Let us also remark, though it is not an observation we use, that one can show that $1- \frac{1}{\ln(2\veps_N)} \ln |P^\prime(\eta_k)| $ is positive \textit{on average}; more precisely we have the following identity.
\begin{align*}
\prod_{k=1}^N |P^\prime(\eta_k)| = N^{N} \prod_{j=1}^{N-1} |P(\zeta_j)| \, .
\end{align*}
This follows from the formulas $P^\prime(\eta_k) = N \prod_{j=1}^{N-1}(\eta_k - \zeta_j)$ and $P(\zeta_j) = \prod_{k=1}^N (\zeta_j - \eta_k)$. Taking the logarithm of both sides we find
\begin{align}
1 - \frac{1}{N} \sum_{k=1}^N \frac{1}{\ln(2\veps_N)} \ln |P^\prime(\eta_k)| \geq - \frac{ \ln  N }{\ln 2\veps_N}+   \frac{1}{N} > 0
\end{align}
where we note that $\ln 2\veps_N < 0$ and is of order $N$, so both terms are positive and are overall of order $\mathrm{O}( \frac{\ln N }{N})$.

Next, we may ask if a converse statement is true: namely, if one has a lower bound on \eqref{Petamin}, whether that implies that $P\pm 2\varepsilon_N$ have complete sets of real roots. The answer to this, strictly speaking, is no. In addition to a lower bound on $\min_{k \in \intn{1}{N}} |P^\prime(\eta_k)|$ one must also have a lower bound on the interparticular spacing. More precisely, we have the following inequality.
\begin{prop}\label{Pzetalowerbound2} $\sqrt{|P^\prime(\eta_k)P^\prime(\eta_{k+1})|} \leq 4 \frac{1}{|\eta_k - \eta_{k+1}|}|P(\zeta_k)|$.
\end{prop}
\begin{proof} 
We note from the AM-GM inequality that
\begin{align*}
&\sqrt{|P^\prime(\eta_k)P^\prime(\eta_{k+1})|}= |\eta_k-\eta_{k+1}| \prod_{\substack{j = 1 \\ j \neq k,k+1}}^{N}\sqrt{|\eta_k-\eta_j||\eta_{k+1}-\eta_j|} \\ &\leq |\eta_k-\eta_{k+1}| \prod_{\substack{j = 1 \\ j \neq k,k+1}}^{N}\big|\frac{\eta_k+\eta_{k+1}}{2}-\eta_j \big| = 4|\eta_k-\eta_{k+1}|^{-1}\big|P\big(\frac{\eta_k+\eta_{k+1}}{2}\big)\big| \, .
\end{align*}
The conclusion follows from the fact that $\zeta_k$ maximises $|P|$ in $[\eta_k  ; \eta_{k+1}]$.
\end{proof}
\begin{cor}\label{Pzetalowerbound} Let $\delta > 0$ and assume $|\eta_i - \eta_{i+1}| \geq N^{-1 - \delta}$ for all $i = 1, \dots, N-1$. Then
\begin{align}
\frac{1}{4} N^{-1 - \delta} \min_{j\in \intn{1}{N}} |P^\prime(\eta_j)| \leq \min_{k \in \intn{1}{N-1}} |P(\zeta_k)| \, .
\end{align}
\end{cor}
\begin{lemme}
\label{Lemme espacement racine lambda pm et zeta j}
The following inequalities hold.
\begin{alignat*}{2}
\frac{1}{N} |\lambda_j^\pm - \zeta_j | &\leq |\lambda_{j+1}^\pm - \zeta_j | &&\leq N |\lambda_j^\pm - \zeta_j | \, , \\
\frac{1}{N} |\zeta_k - \eta_{k+1}| &\leq  \; \;   |\zeta_k - \eta_k| &&\leq N |\zeta_k - \eta_{k+1}|  \; .
\end{alignat*}
\end{lemme}
\begin{proof}
Since $\zeta_j$ is a simple root of $(P^\pm)^\prime$, one has
\begin{equation}\label{simple}
0 = \frac{(P^\pm)^\prime(\zeta_j)}{P^\pm(\zeta_j)} = \sul{k=1}{N} \frac{1}{\zeta_j - \lambda_k^\pm}\,.\end{equation}
 Then
\begin{equation}\label{simple2}\frac{1}{|\zeta_j - \lambda_j^\pm|} \leq \sum_{k=1}^j \frac{1}{|\zeta_j - \lambda_k^\pm|} = \sum_{k=j+1}^N \frac{1}{|\zeta_j - \lambda_k^\pm|}\leq N
\frac{1}{|\zeta_j - \lambda_{j+1}^\pm|}  \, .\end{equation}
A similar argument with reverse inequalities yields the other inequality.
The case of the roots $\eta_k$ is dealt with analogously.

\end{proof}

\begin{lemme}\label{reciprocalbound} If $\ups_j = -\vsg $, $\vsg \in \{\pm\}$, then one has the upper bound
$$\sul{i=j+1}{N} \frac{ |\lambda_j^{-\vsg} - \lambda_j^{\vsg} |}{\lambda_i^{\vsg} - \lambda_j^{-\vsg} }
+ \sum_{i=1}^{j} \frac{|\lambda_{j+1}^{-\vsg} - \lambda_{j+1}^{\vsg} | }{ \lambda_{j+1}^{-\vsg} - \lambda_i^{\vsg} } \,  \leq N  \, \, .$$
\end{lemme}
\begin{proof}
Let $\zeta_1 < \dots < \zeta_{N-1}$ be the roots of $(P^+)^\prime = (P^-)^\prime$. Then by the interlacing properties of the roots we have
\begin{align*}
\sum_{i=j+1}^{N} \frac{|\lambda_j^{-\vsg} - \lambda_j^{\vsg}|}{\lambda_i^{\vsg} - \lambda_j^{-\vsg} }
+ \sum_{i=1}^{j} \frac{|\lambda_{j+1}^{-\vsg} - \lambda_{j+1}^{\vsg}|}{\lambda_{j+1}^{-\vsg} - \lambda_i^{\vsg} }
\, \leq  \, \sum_{i=j+1}^{N} \frac{|\lambda_j^{-\vsg} - \lambda_j^{\vsg}|}{\lambda_i^{\vsg} - \zeta_j }
+ \sum_{i=1}^{j} \frac{|\lambda_{j+1}^{-\vsg} - \lambda_{j+1}^{\vsg}|}{\zeta_{j} - \lambda_i^{\vsg} } \, .
\end{align*}
Then from \eqref{simple}, we have
\begin{equation}\label{simple3}
\sum_{k=1}^j \frac{1}{|\zeta_j - \lambda_k^{\vsg}|} = \sum_{k=j+1}^N \frac{1}{|\zeta_j - \lambda_k^{\vsg}|}\,.
\end{equation}
Hence our quantity is bounded by
$$\sum_{i=1}^{j} \underbrace{\frac{|\lambda_j^{-\vsg} - \lambda_j^{\vsg}|}{\zeta_j - \lambda_i^{\vsg}  } }_{\leq 1} + \sum_{i=j+1}^{N}
\underbrace{\frac{|\lambda_{j+1}^{-\vsg} - \lambda_{j+1}^{\vsg}|}{ \lambda_i^{\vsg} - \zeta_{j}  }}_{\leq 1} \leq j + (N-j) = N \, .$$
\end{proof}

\begin{lemme}
\label{Lemme borne sup somme lambda pm pour borne des Bj sigma}
Given $\vsg\in \{\pm\}$ and $j \in \intn{1}{N-2}$, one has the upper bound
$$\om_j^{\vsg} \,= \,  \mathbbm{1}_{\{ \ups_j = -\vsg\}} \sum_{i=1}^j  \frac{|\lambda_{j+1}^{\vsg} - \lambda_{j+1}^{-\vsg} |}{\lambda_{j+1}^{-\vsg} - \lambda_{i}^{\vsg}}
+  \mathbbm{1}_{\{ \ups_{j+1}= -\vsg\}}
\sum_{i=j+2}^N \frac{|\lambda_{j+1}^{-\vsg} - \lambda_{j+1}^{\vsg}|}{\lambda_i^{\vsg} - \lambda_{j+1}^{-\vsg} }
\leq N+1 \, .$$
\end{lemme}

\begin{proof} By the same reasoning as in the proof of Lemma \ref{reciprocalbound}
\begin{align*}
\om_j^{\vsg} &\leq \mathbbm{1}_{\{ \ups_j = -\vsg\}} \sum_{i=1}^j  \frac{|\lambda_{j+1}^{\vsg} - \lambda_{j+1}^{-\vsg} |}{\zeta_j - \lambda_{i}^{\vsg}}
+  \mathbbm{1}_{\{ \ups_{j+1}= -\vsg \}}  \sum_{i=j+2}^N \frac{|\lambda_{j+1}^{-\vsg} - \lambda_{j+1}^{\vsg}| }{ \lambda_i^{\vsg} - \zeta_{j+1} } \\
&\leq \mathbbm{1}_{\{ \ups_j = -\vsg \}} \sum_{i=j+1}^N  \frac{|\lambda_{j+1}^{\vsg} - \lambda_{j+1}^{-\vsg} |}{ \lambda_{i}^{\vsg} - \zeta_j }
+  \mathbbm{1}_{\{ \ups_{j+1}= -\vsg \}}  \sum_{i=1}^{j+1} \frac{|\lambda_{j+1}^{-\vsg} - \lambda_{j+1}^{\vsg}|}
{\zeta_{j+1} - \lambda_i^{\vsg}  } \leq N+1 \, .
\end{align*}

\end{proof}

\section{Jacobian between the roots}

\begin{prop}\label{openness}
Let $t \in \mathbb{R}$, $\boldsymbol{x}_N \in \R^{N}_{<} $
and denote by $P_{\boldsymbol{x}_N}(\lambda) = \prod_{k=1}^N (\lambda- x_k)$ the monic polynomial of degree $N$ whose $N$ distinct real roots
are given by the coordinates of $\bs{x}_N$. Then
$$\mathcal{A}_N(t) := \big\{ \boldsymbol{x}_N \in \mathbb{R}^N_{<} \, : \,
P_{\boldsymbol{x}_N} - t \text{ has } N \text{distinct real roots} \big\} $$
is an open subset of the Weyl chamber $\R^{N}_{<}$.
\end{prop}
\begin{proof}
The case $t = 0$ is trivial. We shall discuss the proof when $t > 0$, the $t < 0$ case can be treated in the same way.
If $\mathcal{A}_N(t)=\emptyset$, then there is nothing more to do. Thus, let $\bs{x}_N \in \mathcal{A}_N(t)$.
Having simple roots, $P_{\boldsymbol{x}_N}$ alternates in sign on the intervals between the roots, so that  $P_{\boldsymbol{x}_N}(\lambda) > 0$
on
\beq
\intoo{-\infty}{x_1} \, \bigcup\limits_{p=1}^{\tf{N}{2} -1 } \, \intoo{ x_{N-2p} }{ x_{N-2p+1}  } \, \bigcup  \, \intoo{x_N}{+\infty}\, ,
\enq
resp.
\beq
\bigcup\limits_{p=1}^{\tf{(N-1)}{2} } \, \intoo{ x_{N-2p} }{ x_{N-2p+1}  } \, \bigcup  \, \intoo{x_N}{+\infty} \;,
\enq
for $N$-even, resp. $N$-odd. Thus, the real roots of $ P_{\boldsymbol{x}_N}-t$ belong only to these domains.
There is always a pair of
roots in $\intoo{-\infty}{x_1}$ and $\intoo{x_N}{+\infty}$ for $N$ even, and a single root in $\intoo{x_N}{+\infty}$ for $N$ odd.
Thus, in order to have $N$ simple roots, it follows that $N-2$, resp. $N-1$, of these have to belong to the union of bounded
intervals above. Now, denote by $\th_1(\bs{x}_N) <  \dots < \th_{N-1}(\bs{x}_N)$  the simple zeros of $P_{\boldsymbol{x}_N}^\prime$ which
interlace with the original roots as $x_k < \th_k(\bs{x}_N) < x_{k+1}$ for all $k \in \intn{ 1 }{ \lfloor \tfrac{N-1}{2} \rfloor  }$.
By the intermediate value theorem, $P_{\boldsymbol{x}_N}-t$ will have a pair of distinct roots on the interval $]x_{N - 2k}, x_{N-2k+1}[$ if and
only if $t < P_{\boldsymbol{x}_N}(\th_{N-2k}(\bs{x}_N))$.
%
%
%
%
%
%
%
%
Now, the coefficients of $P_{\boldsymbol{y}_N}^{\prime}$ are smooth in $\bs{y}_N\in \mathbb{R}^N_{<} $. Since  $P_{\boldsymbol{x}_N}^{\prime}$
has only simple roots, by the implicit function theorem, there exists a small open neighbourhood $U$ of $\bs{x}_N$ in $\R^N_{<}$
such that $\bs{y}_N \mapsto \bs{\th}_{N-2k}(\bs{y}_N)$ is smooth in $U$. Thus, $\bs{y}_N \mapsto
P_{\boldsymbol{y}_N}\big( \th_{N-2k}(\bs{y}_N) \big)$ is smooth in $U$. This ensures that there exists an open neighbourhood
of $\bs{x}_N$ in $\R^N_{<}$ such that $t < P_{\boldsymbol{y}_N}\big( \th_{N-2k}(\bs{y}_N) \big)$ on it, thus proving that
$\mathcal{A}_N(t)$ is open.
\end{proof}

\begin{prop}\label{diffeo}
Fix $t  \in \mathbb{R}$, and let $\boldsymbol{x}_N \in \mathcal{A}_N(t)$. Let $\boldsymbol{y}_N^t(\bs{x}_N) \in \mathbb{R}^N_{<}$ be the roots of
$P_{\boldsymbol{x}_N}(\lambda)-t$. Then $\boldsymbol{y}_N^t(\bs{x}_N)  \in \mathcal{A}_N(-t)$ and
the map $ \boldsymbol{y}_N^t: \mathcal{A}_N(t) \to \mathcal{A}_N(-t)$ is a diffeomorphism onto.
\end{prop}

\begin{proof}

By the implicit function theorem, simple roots of a polynomial are smooth functions of its coefficients locally, thus ensuring that
$\boldsymbol{y}_N^t$ is a local diffeomorphism. The fact that $\boldsymbol{y}_N^t(\bs{x}_N)  \in \mathcal{A}_N(-t)$ is clear.
One thus needs to establish the global character of the diffeomorphism.

It is direct to see that given $\bs{u}_N \in \mc{A}_N(-t)$, $P_{ \boldsymbol{y}_N^{-t}(\bs{u}_N) }=P_{\bs{u}_N}+t$. Since
$\boldsymbol{y}_N^{-t}(\bs{u}_N) \in \mc{A}_N(t) $, one has that $\boldsymbol{y}_N^{t} \circ\boldsymbol{y}_N^{-t}(\bs{u}_N) \in \mc{A}_N(-t)$
and $P_{ \boldsymbol{y}_N^{t}\circ \boldsymbol{y}_N^{-t}(\bs{u}_N) }=P_{\bs{u}_N}$. Thus,
any $\bs{u}_N \in \mc{A}_N(-t) $ is given by the image of $\boldsymbol{y}_N^{-t}(\bs{u}_N) \in \mc{A}_N(t) $ under $\boldsymbol{y}_N^{t}$.
This entails surjectivity.

As for injectivity, assume that there exists $\bs{x}_N, \bs{x}_N^{\prime}$ such that there exists $\bs{y}_N \in \mc{A}_N(-t)$
satisfying $P_{\bs{x}_N}-t=P_{\bs{y}_N}= P_{\bs{x}_N^{\prime}}-t$. However, then, since $\bs{x}_N, \bs{x}_N^{\prime} \in \R^N_{<}$
and $P_{\bs{x}_N}=P_{\bs{x}_N^{\prime}}$, one has $\bs{x}_N =\bs{x}_N^{\prime}$, which entails injectivity.

\end{proof}

Propositions \ref{openness} and \ref{diffeo} ensure that the Jacobians 
of $\boldsymbol{\lambda}_N^\pm(\boldsymbol{\eta}_N)$ and
of$\boldsymbol{\lambda}_N^+(\boldsymbol{\lambda}_N^-)$ exist. We now compute those explicitly.

\begin{prop}[Jacobian for change of variables between sets of roots]\label{jacobian}

It holds
\begin{align}
&\Big| \det_N \Big[ \op{D}_{ \boldsymbol{\lambda}_N^- } \boldsymbol{\lambda}_N^+ \Big] \Big| \, =  \,
\frac{\Delta(\boldsymbol{\lambda}_N^-)}{\Delta(\boldsymbol{\lambda}_N^+)} ,
& &\Big| \det_N \Big[ \op{D}_{ \boldsymbol{\eta}_N } \boldsymbol{\lambda}_N^+ \Big] \Big|
= \frac{\Delta(\boldsymbol{\eta}_{N})}{\Delta(\boldsymbol{\lambda}_N^+)} \, .
\end{align}
\end{prop}
\begin{proof}
The proof for both claims is the same, so that we only focus on the first identity.
By definition we have $P^-(\lambda_i^+) = 4\veps_N$ or equivalently
$$\sum_{k=1}^N \ln |\lambda_i^+ - \lambda_k^-| = \ln(4\veps_N) \, .$$
Differentiating with respect to $\lambda_j^-$ yields
$$-\frac{1}{ \lambda_i^+ - \lambda_j^-} + \sum_{k=1}^N \frac{1}{\lambda_i^+ - \lambda_k^-} \frac{\partial \lambda_i^+}{\partial \lambda_j^-} = 0  \, . $$
Thus, upon denoting $Q_i := \sul{k=1}{N}  \frac{1}{\lambda_i^+ - \lambda_k^-}\, = \, \frac{(P^-)^\prime(\lambda_i^+)}{P^-(\lambda_i^+)} $
one has
$$\frac{\partial \lambda_i^+}{\partial \lambda_j^-} = \frac{1}{Q_i}\frac{1}{ \lambda_i^+ - \lambda_j^-} \, .$$
Hence
\begin{equation}\label{jaco1} \det_N \Big[ \op{D}_{ \boldsymbol{\lambda}_N^- } \boldsymbol{\lambda}_N^+ \Big]
\, = \,  \pl{i=1}{N}\Big\{ Q_i \Big\}^{-1}  \, \det_N \Big[ \frac{1}{ \lambda_i^+ - \lambda_j^-} \Big] \, .\end{equation}
The claim then follows upon  observing that since $(P^+)^\prime = (P^-)^\prime$,
$$\pl{i=1}{N} (P^-)^\prime(\lambda_i^+) = (-1)^{\frac{N(N-1)}{2}}\Delta(\boldsymbol{\lambda}_N^+)^2$$
and using the Cauchy determinant formula 
$$
\det_N \Big[ \frac{1}{ \lambda_i^+ - \lambda_j^-} \Big] \, = \,  (-1)^{\frac{N(N-1)}{2}}
\Delta(\boldsymbol{\lambda}_N^+)\Delta(\boldsymbol{\lambda}_N^-) \pl{a=1}{N}  \Big\{P^-(\lambda_a^+)  \Big\}^{-1}\,.
$$
%
%
%
%
%
%
%
%
%
%
%
\end{proof}





\section{Bounding the integral $C_N^{(p)}[U]$}





In this section we establish a suitable upper bound on the $N$-fold integral introduced in \eqref{definition integrale multiple CNp}.
First, however, we need an auxiliary lemma

\begin{lemme}[Chernoff bound]\label{chernoff}

Let $\{X_i\}_{i=1}^{n-1}$ be a collection of $n-1$ random uniformly distributed iid variables on $\intff{0}{1}$. Then, for any  $\xi \geq 1$,
it holds
$$\mathbb{P}\bigg[ \sum_{j=1}^{n-1} \big|\ln X_i \big| \geq (n-1)\xi \bigg] \,  \leq  \, \big( \xi \mathrm{e}^{1- \xi} \big)^{n-1}\, .$$
\end{lemme}
\begin{proof} For $0 \leq \alpha < 1$, by Markov's inequality
\begin{align*}
\mathbb{P}\bigg[ \sum_{j=1}^{n-1} \big|\ln X_i \big| \geq (n-1)\xi \bigg]  &= \mathbb{P}\bigg[ \mathrm{e}^{\alpha \sum_{j=1}^{n-1} |\ln X_i| }  \geq \mathrm{e}^{\alpha (n-1)\xi } \bigg] \leq \bigg( \frac{\mathbb{E}[\mathrm{e}^{\alpha |\ln X | }] } {\mathrm{e}^{\alpha \xi}}\bigg)^{n-1} \, .
\end{align*}
A straightforward calculation shows that $\mathbb{E}[\mathrm{e}^{\alpha |\ln X | } ]  = \frac{1}{1-\alpha}$ (for $0 \leq \alpha < 1$). Then, setting
$\alpha = 1 - \xi^{-1}$ allows one to conclude.
\end{proof}

The proof of the upper bound, will  rely on the following concept of $\eps$-cluster.

\begin{defin}[$\epsilon$-cluster structure] Let $\eps>0$ and $\bs{n}_K=(n_1, \dots, n_K)\in (\mathbb{N}^*)^K$ with $\ov{\bs{n}}_K=\sum_{k=1}^K n_k = N$. One says
that $\bs{x}_N \in \R^{N}_{<}$ has the $\epsilon$-cluster $\bs{n}_K$  if the following is true.
Define $i_1=1$ and $i_k = n_1 + \dots + n_{k-1}+1$ for $k\geq 2$, so that $x_{i_k}$ is the leftmost particle in the $k^{\e{th}}$ cluster.

\begin{enumerate}
\item For all $k = 1, \dots, K$,  $x_{i_k + j} - x_{i_k + j-1} \leq \epsilon$ for all $j = 1, \dots, n_k -1$.
\item For all $k = 2, \dots, K$,  $x_{i_{k}}  - x_{i_{k} -1} > \epsilon$.
\end{enumerate}
We let $\mc{E}_{\eps}( \bs{n}_K)$ be the set of $\bs{x}_N \in \R^{N}_{<}$ that have $\epsilon$-cluster $\bs{n}_K$.

\end{defin}

It is easy to see that every element of $\R^{N}_{<}$ has exactly one $\eps$-cluster structure, so that the sets $\mc{E}_{\eps}( \bs{n}_K)$ partition $\R^{N}_{<}$.

\begin{prop}\label{CNbound}
Let $p \geq 1$, $U\in \mc{C}^0(\R)$ be such that  $\Int{\mathbb{R}}{}\mathrm{e}^{-U(x)} \, \mathrm{d}x = 1$.
 Then, upon fixing the notation $\mc{L}n_0(x)= \ln \big( x  \wedge 1\big) $ for $x \geq 0$, one has that the $N$-fold integral
$$
C_N^{(p)} [U] = \Int{\mathbb{R}^N }{}
\varpi_p(\bs{x}_N)
\pl{k=1}{N} \mathrm{e}^{- U(x_k)}\, \mathrm{d}\boldsymbol{x}_N \;, \quad
\varpi_p(\bs{x}_N) \, = \, \pl{j=1}{N} \bigg\{ 1 + \frac{1}{ N^{\frac{5}{4}}} \sum_{\substack{i=1 \\ i \neq j}}^N \mc{L}n_0^2|x_i - x_j| \bigg\}^p
$$
admits the upper bound $C_N^{(p)} [U] \leq \mathrm{e}^{C N^{\frac{3}{4}}\ln^2 N}$ for some $C>0$ and any $N$ large enough. $C$ may depend on $p$ and $U$.
\end{prop}

\begin{proof}

One starts by changing the integration domain from $\R^N$ to $\R^{N}_{<}$, thus producing the additional factor of $N!$.
Pick $ \epsilon$ positive and small and consider the associated resolution of unity subordinate to having a given $\bs{n}_K$
$\eps$-cluster
$$1 = \sum_{K=1}^N \sum_{\substack{ \bs{n}_K \in (\mathbb{N}^*)^K \\ \ov{\bs{n}}_K = N}} \hspace{-3mm} \mathbbm{1}_{ \mc{E}_{\eps}( \bs{n}_K)}(\bs{x}_N) \,. $$
 This decomposes the original integral as
\begin{align*}
C_N^{(p)} [U] =  N! \sum_{K=1}^N \sum_{\substack{ \bs{n}_K \in (\mathbb{N}^*)^K \\ \ov{\bs{n}}_K = N}}
\Int{ \R^N_{<}}{}
\varpi_p(\bs{x}_N)
\pl{k=1}{N} \Big\{ \mathrm{e}^{ - U(x_k) } \Big\}
\mathbbm{1}_{\mc{E}_{\eps}( \bs{n}_K) }(\bs{x}_N)
\, \mathrm{d}\bs{x}_N, .
\end{align*}
Then decomposing products with respect to the cluster structure $\pl{j=1}{N} = \pl{k=1}{K} \pl{j=i_k}{i_k + n_k -1}$ leads to
\beqa
\varpi_p(\bs{x}_N)  & = &
\prod_{k=1}^K \prod_{j=i_k}^{i_k + n_k -1} \bigg\{ 1 + \frac{1}{ N^{\tf{5}{4}}} \sum_{\substack{i=1 \\ i \neq j}}^N \mc{L}n_0^2|x_i - x_j| \bigg\}^p \\
&\leq & \prod_{k=1}^K \prod_{j=i_k}^{i_k + n_k -1}
\bigg\{ 1 + \frac{1}{ N^{\tf{5}{4}}} \hspace{-2mm}\sum_{\substack{i=i_k \\ i \neq j}}^{i_k+n_k-1} \hspace{-3mm} \mc{L}n_0^2|x_i - x_j| \bigg\}^p
\Big\{ 1+ \frac{\ln^2 \epsilon}{N^{\tf{1}{4}}}\Big\}^{pN} \, .
\eeqa
Further, upon denoting $\mf{u} = \norm{  \mathrm{e}^{-U} }_{L^{\infty}(\R)}$, given $\bs{x}_N \in\mc{E}_{\eps}(\bs{n}_K)$, one has the upper bound
\beq
\pl{k=1}{N}\mathrm{e}^{-U(x_k) } \leq \mf{u}^{N-K} \pl{k=1}{K} \mathrm{e}^{- U(x_{i_k})  }\, .
\label{product decomposition}
\enq
Obviously, when $n_k = 1$, one has
$$\prod_{j=i_k}^{i_k + n_k -1} \bigg\{ 1 + \frac{1}{ N^{\tf{5}{4}}} \hspace{-3mm} \sum_{\substack{i=i_k \\ i \neq j}}^{i_k+n_k-1}  \hspace{-2mm} \mc{L}n_0^2|x_i - x_j| \bigg\}^p = 1 \, .$$
Further, observe that when $\bs{x}_N \in\mc{E}_{\eps}(\bs{n}_K)$  it holds that
$$ x_{i_k} < x_{i_k+1} < \dots < x_{i_{k}+n_k-1} < x_{i_k} + (n_k-1)\epsilon \leq x_{i_k} + N \epsilon \, . $$
Moreover, the product involving the logarithmic terms only depends on the differences of the variables so that the blocs become independent. Thus, the upper bound \eqref{product decomposition}
and the  change of variables $x_{i_k+p}=x_{i_k}+y_{k,p}$ lead to
$$
C_N^{(p)} [U]  \leq   N! \Big\{ 1+ \frac{\ln^2 \epsilon}{N^{\tf{1}{4}}}\Big\}^{pN} \sum_{K=1}^N \sum_{\substack{ \bs{n}_K \in (\mathbb{N}^*)^K \\ \ov{\bs{n}}_K = N}}  \mf{u}^{N-K}
\Int{ \R^K_{<}}{} \pl{k=1}{K} \Big\{  \mathrm{e}^{- U(x_{i_k})  } \dd x_{i_k} \Big\}  \times \pl{s=1}{K} \mc{J}_{n_s}
$$

where we set $\mc{J}_1 = 1$ while, for $n \geq 2$, since $|\mc{L}n_{0}(|x|)| \leq |\ln |x||$,
\bem
\mc{J}_n  :=
\Int{0}{ N \epsilon }  \mathbbm{1}_{\R^{n-1}_{<}}(\bs{x}_{n-1}) \,
\bigg\{ 1 + \frac{1}{N^{\tf{5}{4}}} \sum_{i=1}^{n-1} \ln^2|x_i| \bigg\}^p  \\
\hspace{5mm} \times \pl{j=1}{n-1} \bigg\{  1 +  \frac{1}{N^{\tf{5}{4}}}\sul{ \substack{i = 1 \\ i \neq j } }{n-1} \ln^2|x_i - x_j| + \frac{ \ln^2 |x_j| }{N^{\tf{5}{4}}}   \bigg\}^p \, \mathrm{d}\bs{x}_{n-1} \, .
\end{multline}
The integrals over the Weyl chamber involving the $x_{i_k}$ produce, after symmetrisation, a $\tf{1}{K!}$ contribution.
Thus, one arrives at
\begin{align}\label{boundonC}
C_N^{(p)} [U] \leq \Big\{ 1+ \frac{\ln^2 \epsilon}{N^{\tf{1}{4}}}\Big\}^{pN} \sum_{K=1}^N \, \sum_{\substack{ \bs{n}_K \in (\mathbb{N}^*)^K \\ \ov{\bs{n}}_K = N}}
\hspace{-3mm} \mf{u}^{N-K}  \f{ N!  }{ K! }   \pl{k=1}{K} \mc{J}_{n_k} \, .
\end{align}
It remains to upper bound $\mc{J}_n$.
One starts by symmetrising the integration domain, which incurs a $\tf{ 1 }{ (n-1)! }$ factor,
and rescaling the variables by $N\epsilon$. Given the upper bound $\ln^2|N\epsilon x_i| \leq 2 \ln^2|x_i| +2 \ln^2 |N \epsilon|$
and the inequality $(1+a+b)\leq (1+a)(1+b)$ for $a,b\geq 0$, one gets for $n \geq 2$
\beq\label{boundonC2}
\mc{J}_n \leq \frac{  (N \epsilon)^{n-1}  }{(n-1)!} \Big\{ 1 + 2(n-1) \frac{ \ln^2 |N \epsilon|}{N^{\tf{1}{4}}}\Big\}^{pn} \wt{\mc{J}}_n
\enq
where we have set
\beq
\wt{\mc{J}}_n \, = \,  \Int{0}{1}  \bigg\{ 1 + 2 \sul{i=1}{n-1} \frac{  \ln^2|x_i| }{N^{\tf{5}{4}}} \bigg\}^p \,
\prod_{j=1}^{n-1} \bigg\{ 1 + 2 \frac{  \ln^2 |x_j|  }{N^{\tf{5}{4}}} + \frac{2}{N^{\tf{5}{4}}}\sul{ \substack{i = 1 \\ i \neq j} }{n-1} \ln^2|x_i - x_j|  \bigg\}^p \, \mathrm{d}\bs{x}_{n-1}  \, .
\enq
Applying Hölder's inequality to each of the functions occurring in the integrand of  $\wt{\mc{J}}_n$ leads to
\begin{align*}
\wt{\mc{J}}_n  &\leq \bigg[ \Int{0}{1} \bigg\{ 1 + 2 \sul{i=1}{n-1} \frac{ \ln^2|x_i| }{N^{\tf{5}{4}}} \bigg\}^{np}  \, \mathrm{d}\bs{x}_{n-1} \bigg]^{\frac{1}{n}} \\
&\times
\pl{j=1}{n-1} \bigg[ \Int{0}{1}  \Big\{ 1+ 2 \frac{ \ln^2 |x_j| }{ N^{\tf{5}{4}} }  + \frac{2}{ N^{\tf{5}{4}} }
\sul{\substack{i = 1 \\ i \neq j}}{n-1} \ln^2|x_i - x_j|   \Big\}^{np} \, \mathrm{d}\bs{x}_{n-1} \bigg]^{\frac{1}{n}}
\end{align*}
In the $j^{\e{th}}$ integral occurring in the second line, one shifts $x_i \to x_{i}+x_j$ for all $i \in \intn{1}{n-1} \setminus \{ j \}$.
The integration region for $x_i$ thus becomes $\intff{-x_j}{ 1- x_j} \subset [-1,1]$.
By using that the resulting integrand is even with respect to $x_i$ and strictly positive, one arrives to
\beq
\wt{\mc{J}}_n  \, \leq  \, 2^{n-1} \!  \Int{0}{1} \bigg\{ 1 + 2 \sul{i=1}{n-1} \frac{ \ln^2|x_i| }{N^{\tf{5}{4}}} \bigg\}^{np}  \, \mathrm{d}\bs{x}_{n-1}\, .
\enq
At this stage, the inequalities $\sul{i=1}{n-1} \ln^2(x_i) \leq (\sul{i=1}{n-1} |\ln(x_i)|)^2$ and $1+x^2 \leq (1+x)^2$ yield
 \beq
\wt{\mc{J}}_n  \, \leq  \, 2^{n-1} \!  \Int{0}{1} \bigg\{ 1 +    \sqrt{2}
 \sul{i=1}{n-1} \frac{ |\ln(x_i)|  }{ N^{\tf{5}{8}} } \bigg\}^{2np}  \, \mathrm{d}\bs{x}_{n-1}\,  .
\enq
This term can be estimated by using the concentration inequalities
provided by Lemma \ref{chernoff}. For $n \geq 2$ one sets $X = \frac{1}{n-1}\sum_{j=1}^{n-1} |\ln(x_i)|$ so that by using $n-1 \leq N$,
one is led by integration by parts to
 \begin{align*}
\wt{\mc{J}}_n &\leq 2^{n-1} \mathbb{E}\bigg[  \Big\{ 1 +  \sqrt{2} \frac{n-1}{N^{\tf{5}{8}} }X \Big\}^{2np}  \bigg]  \\
%
%
%
%
%
&\leq  2^n  \Big\{ 1 + \sqrt{2}N^{  \frac{3}{8}} \Big\}^{2np}   + 2^{n+\f{1}{2}}np \frac{n-1}{ N^{\tf{5}{8}} }
\Int{1}{+\infty} \Big\{ 1 +  \sqrt{2} \frac{n-1}{N^{\tf{5}{8}} } \xi \Big\}^{2np-1}
 \mathbb{P}\big[ \{X \geq \xi\} \big] \, \mathrm{d}\xi \, .
\end{align*}
By using that $\xi \mathrm{e}^{1-\frac{1}{2}\xi} \leq 2$, the bound provided by Lemma  \ref{chernoff}  leads to
\beq
\wt{\mc{J}}_n  \, \leq  \, 2^n  \Big\{ 1 + \sqrt{2}N^{  \frac{3}{8}} \Big\}^{2np}   + 2^{2n-\f{1}{2}}np \frac{n-1}{ N^{\tf{5}{8}} } \,  \mc{W}
\enq
where, for $\eta>0$ and small enough
\beq
\mc{W} \, = \, \Int{1}{+\infty} h_{n}(\xi) \mathrm{e}^{-\frac{\eta }{2} \xi } \, \mathrm{d}\xi
\qquad \e{with} \qquad h_n(\xi) \, = \, \Big\{ 1 +  \sqrt{2} \frac{n-1}{N^{\tf{5}{8}} } \xi \Big\}^{2np-1}\mathrm{e}^{-\frac{1}{2}(n-1-\eta)\xi} \, .
\enq
This last integral can be upper bounded by studying the variations of $h_n$. One may show that $h_n$ achieves its maximum on $\intfo{1}{+\infty}$ at
\beq
\xi_* \, = \, \e{max} \bigg\{ 1,   \f{ \sqrt{2} N^{\tf{5}{8}} }{ (n-1)(n-1-\eta) } \Big[ \sqrt{2} \f{ 2np-1 }{  N^{\tf{5}{8}} } \, - \, \f{n-1-\eta}{2} \Big] \bigg\}
\enq
It is then direct to check that there exists $C>0$ such that, uniformly in $n\in\intn{1}{N}$, $h_n(\xi_*)\leq \big( C N^{\tf{3}{8}} \big)^{2np-1} $. Thus,
$\mc{W} \leq c^{\prime} \big( C N^{\tf{3}{8}} \big)^{2np-1} $. Upon readjusting $C$ if need be, this entails that
$\wt{\mc{J}}_n  \leq \big( C N^{\tf{3}{8}} \big)^{2np} $.

Note that for $n \geq 2$, one has  $n \leq 2(n-1)$
so that one may write up an upper bound valid up to $n=1$: $\wt{\mc{J}}_n  \leq \big( C N^{\tf{3}{8}} \big)^{4(n-1)p}$.
In its turn, this translates into
$$\mc{J}_n \leq \frac{1}{(n-1)!} \Big\{ 1 + 2 (n-1) \frac{ \ln^2 (N \epsilon) }{ N^{\tf{5}{4}}}\Big\}^{pn}
\Big[ N  \epsilon \big( C N^{ \tf{3}{8}} \big)^{4p} \Big]^{n-1}\, .$$
Substituting this in \eqref{boundonC}, we find, for some constant $C > 0$ depending only on $p$
\begin{align*}
C_N^{(p)} [U] \leq \Big\{ 1+  \frac{ 2 }{N^{\tf{1}{4}} }  \big[ \ln^2( N \epsilon)+\ln^2(\epsilon) \big] \Big\}^{2pN}  \sum_{K=1}^N \frac{N! }{K!}  \Big[ C N  \epsilon \big( N^{ \tf{3}{8}} \big) ^{4p} \Big]^{N - K}
\hspace{-4mm} \sum_{\substack{ \bs{n}_K \in (\mathbb{N}^*)^K \\ \ov{\bs{n}}_K = N}} \prod_{k=1}^K \frac{1}{(n_k-1)!  }  \, .
\end{align*}
The last summation can be computed in closed form by the multinomial expansion
\begin{align*}
 \sum_{\substack{ \bs{n}_K \in (\mathbb{N}^*)^K \\ \ov{\bs{n}}_K = N}} \prod_{k=1}^K \frac{1}{(n_k-1)!  }
 =  \sum_{\substack{ \bs{n}_K \in (\mathbb{N})^K \\ \ov{\bs{n}}_K = N-K}}  \prod_{k=1}^K \frac{1}{n_k ! } = \frac{K^{N-K}}{(N-K)!} \leq \frac{N^{N-K}}{(N-K)!} \;.
\end{align*}
Thus overall, after summing up the remaining binomial expansion,
\begin{align*}
C_N^{(p)} [U] \leq \Big\{ 1+ \frac{ 2 }{N^{\tf{1}{4}} }  \big[ \ln^2( N \epsilon)+\ln^2(\epsilon) \big] \Big\}^{32pN}  \Big\{ 1 +  C  \epsilon  N^{2+ \f{3p}{4}}  \Big\}^N \;.
\end{align*}
At this stage, one takes $\epsilon = N^{-3 -\f{3p}{4}}$ thus ensuring the existence of $C^{\prime}, C^{\prime \prime}>0$ such that
\begin{align*}
C_N^{(p)} [U] \leq  C^{\prime} \Big\{ 1+ C^{\prime \prime} \frac{\ln^2 N}{N^{\frac{1}{4}}}\Big\}^{2pN}   \, .
\end{align*}
 This concludes the proof of Proposition \ref{CNbound}.

\end{proof}





\section{Auxiliary determinant identities}





\begin{lemme}
\label{Lemme invertibilite matrice A}

Let $\bs{x}_N\in \R^N$ have pairwise distinct entries. The matrix $\op{A}(\boldsymbol{x}_N)$ introduced in \eqref{definition matrice A2}
is invertible, and
\beq
\big(\op{A}^{-1}(\boldsymbol{x}_N)\big)_{ij} =  \mathbbm{1}_{i \leq j} \prod_{\substack{k=j+1 }}^N (x_i - x_k) \, .
\label{ecriture A inverse}
\enq
\end{lemme}

\begin{proof}
Let $B_{ij} = \mathbbm{1}_{i \leq j} \prod_{\substack{k=j+1 }}^N (x_i - x_k)$.
Then
$$\sum_{\ell=1}^N A_{i \ell} B_{\ell j} = \mathbbm{1}_{i \leq j} \sum_{\ell=i}^j \prod_{\substack{m= i \\ m \neq \ell }}^j \frac{1}{x_\ell - x_m}
= \Oint{ \Ga( \{x_a\} ) }{} \pl{m=i}{j}  \frac{1}{z- x_m} \, \frac{   \mathrm{d}z }{ 2 \pi \i } \, =  \, \delta_{ij}$$
where $\Ga( \{x_a\} )$ is a small counterclockwise index one loop around $x_1, \dots, x_N$. The final equality follows from expanding the contour to infinity.
\end{proof}
 By reversing the order of the matrix product, one gets the identity
\begin{cor}\label{unusual}
\begin{align*}
\mathbbm{1}_{i \leq j} \sum_{\ell=i}^j \prod_{m=\ell+1}^N (x_i - x_m) \prod_{\substack{m = \ell \\ m \neq j}}^N \frac{1}{x_j - x_m}
= \delta_{ij} & & i,j=1,\dots, N\, .
\end{align*}
\end{cor}

\begin{lemme}
\label{Lemme ecriture action inverse A}
Let  $\bs{\la}^{\pm}_N\in \R^N$ have pairwise distinct entries,
$q_{i-1}^\pm$ be as in \eqref{definition polynome qi pm}, $P^{\pm}$ as in \eqref{definition P pm} and
$(\op{A}^{-1}(\boldsymbol{\lambda}^\pm_N))_{ij}$ as defined through \eqref{ecriture A inverse}. Then, it holds
\begin{align*}
\sum_{\ell=1}^N \big( \op{A}^{-1}(\boldsymbol{\lambda}^\pm_N) \big)_{i\ell}  \frac{q_{\ell-1}^\pm(\mu)}{P^\pm(\mu)}  = \frac{1}{\mu - \lambda_i^\pm} \, .
\end{align*}

\end{lemme}

\begin{proof}
The identity reduces to the claim that
\begin{align*}
\sul{\ell=i}{N}   \frac{  \prod_{\substack{k=\ell+1 }}^N (\lambda_i^\pm - \lambda_k^\pm)    }{\prod_{k=\ell}^N (\mu - \lambda_k^\pm)}
=  \frac{1}{\mu - \lambda_i^\pm} \, .
\end{align*}
Both sides tend to $0$ as $\mu \to \infty$ and both have only simple poles. Hence one only needs to verify that the residues on both sides equal.
\begin{align*}
\mathrm{Res} \bigg(
\sum_{\ell=i}^{N}     \frac{ \prod_{k=\ell+1 }^N (\lambda_i^\pm - \lambda_k^\pm) }{\prod_{k=\ell}^N (\mu - \lambda_k^\pm)} \dd \mu,
\mu = \lambda_j^\pm  \bigg)
&=\mathbbm{1}_{i \leq j}  \sum_{\ell=i}^j   \frac{  \prod_{\substack{k=\ell+1 }}^N (\lambda_i^\pm - \lambda_k^\pm)  }{\prod_{\substack{k=\ell \\ k \neq j}}^N (\lambda_j^\pm - \lambda_k^\pm)} \\
&= \delta_{ij}
\end{align*}
by Corollary \ref{unusual}.
\end{proof}

\printbibliography


\end{document}